\documentclass[10pt,reqno]{amsart}

\usepackage{color}
\usepackage[dvipsnames]{xcolor}

\usepackage{setspace} 
\usepackage{marginnote}

\usepackage[abbrev,nobysame]{amsrefs}
\usepackage{enumitem}
\usepackage{amssymb}
\usepackage{bbm}

\usepackage[colorlinks=true,linkcolor=blue,citecolor=blue]{hyperref}
\usepackage[capitalize,sort]{cleveref}

 \usepackage[abbrev,nobysame]{amsrefs}

\usepackage{float}
\usepackage{makecell}
\usepackage{ctable}
\usepackage{placeins}
\usepackage{caption} 
\makeatletter
\newtheorem*{rep@theorem}{\rep@title}
\newcommand{\newreptheorem}[2]{%
\newenvironment{rep#1}[1]{%
 \def\rep@title{#2 \ref{##1}}%
 \begin{rep@theorem}}%
 {\end{rep@theorem}}}
\makeatother

\numberwithin{equation}{section}

\theoremstyle{plain}
\newtheorem{theorem}{Theorem}[section]
\newreptheorem{theorem}{Theorem}
\newreptheorem{definition}{Definition}
\Crefname{theorem}{Theorem}{Theorems}
\crefname{theorem}{theorem}{theorems}
\crefformat{theorem}{Theorem~#2#1#3}
\crefmultiformat{theorem}{Theorems~#2#1#3}{ and~#2#1#3}{, #2#1#3}{, and~#2#1#3}

\newtheorem{proposition}[theorem]{Proposition}
\newreptheorem{proposition}{Proposition}
\Crefname{proposition}{Proposition}{Propositions}
\crefname{proposition}{proposition}{propositions}
\crefformat{proposition}{Proposition~#2#1#3}
\crefmultiformat{proposition}{Propositions~#2#1#3}{ and~#2#1#3}{, #2#1#3}{, and~#2#1#3}

\newtheorem{claim}[theorem]{Claim}
\Crefname{claim}{Claim}{Claims}
\crefname{claim}{Claim}{Claims}
\crefformat{claim}{Claim~#2#1#3}

\newtheorem{lemma}[theorem]{Lemma}
\newreptheorem{lemma}{Lemma}
\Crefname{lemma}{Lemma}{Lemmas}
\crefname{lemma}{lemma}{lemmas}
\crefformat{lemma}{Lemma~#2#1#3}
\crefmultiformat{lemma}{Lemmas~#2#1#3}{ and~#2#1#3}{, #2#1#3}{, and~#2#1#3}

\newtheorem{conjecture}[theorem]{Conjecture}
\newtheorem{definition}[theorem]{Definition}

\newtheorem{corollary}[theorem]{Corollary}
\Crefname{corollary}{Corollary}{Corollaries}
\crefname{corollary}{corollary}{corollaries}
\crefformat{corollary}{Corollary~#2#1#3}
\crefmultiformat{corollary}{Corollaries~#2#1#3}{ and~#2#1#3}{, #2#1#3}{, and~#2#1#3}

\newtheorem{custheorem}{Theorem}

\Crefname{custheorem}{Theorem}{Theorems}
\crefname{custheorem}{theorem}{theorems}
\crefformat{custheorem}{Theorem~#2#1#3}
\crefmultiformat{custheorem}{Theorems~#2#1#3}{ and~#2#1#3}{, #2#1#3}{, and~#2#1#3}

\theoremstyle{definition}

\crefformat{definitions}{Definition~#2#1#3}

\newtheorem{remark}[theorem]{Remark}

\newtheorem{assumption}[theorem]{Assumption}

\crefformat{assumption}{Assumption~#2#1#3}
\crefmultiformat{assumption}{Assumptions~#2#1#3}{ and~#2#1#3}{, #2#1#3}{, and~#2#1#3}

\theoremstyle{remark}

\newcounter{CASE}
 \newenvironment{CASE}[1][\unskip]{\refstepcounter{CASE}\em
 \medskip \noindent Case \theCASE\ #1.\ }{\unskip\upshape}
 \renewcommand{\theCASE}{\arabic{CASE}}
 \crefformat{CASE}{Case~#2#1#3}
 \Crefname{CASE}{Case}{Cases}
 \crefname{CASE}{case}{cases}

\newcounter{step}
\crefformat{step}{Step~#2#1#3}
\crefmultiformat{step}{Steps~#2#1#3}{ and~#2#1#3}{, #2#1#3}{, and~#2#1#3}

\def\bs{\backslash}

\newcommand\Folner{F{\o}lner }
\def\folner{\Folner}

\renewcommand{\hat}{\widehat}

\newcommand{\interval}[1]{\widehat{#1}}
\newcommand{\inv}{^{-1}}
\newcommand{\Lie}[1]{\mathfrak{\lowercase{#1}}}
\def\su{\mathfrak{su}}

\newcommand{\so}{\mathfrak{so}}
\newcommand{\liesl}{\mathfrak{sl}}
\newcommand{\liesp}{\mathfrak{sp}}

\newcommand{\liea}{\Lie a}

\newcommand{\lieu}{\Lie u}

\newcommand{\lieg}{\Lie g}
\newcommand{\lieh}{\Lie h}
\newcommand{\liek}{\Lie k}
\newcommand{\liel}{\Lie l}
\newcommand{\lien}{\Lie n}
\newcommand{\liep}{\Lie p}
\newcommand{\lieq}{\Lie q}

\newcommand{\quot}[1]{\widehat{#1}}
\newcommand{\sm}{\smallsetminus}
\newcommand{\td}{\tilde}
\DeclareMathOperator{\len}{len}
\newcommand{\wtd}{\widetilde}

\DeclareMathOperator{\diam}{diam}

\def\C{\mathbb{C}} 
\newcommand{\K}{\mathbb {K}}
\newcommand{\N}{\mathbb {N}}
\renewcommand{\P}{\mathbb {P}}
\newcommand{\Q}{\mathbb {Q}}
\newcommand{\R}{\mathbb {R}}

\newcommand{\Z}{\mathbb {Z}}

\def\calA{\mathcal A}
\def\calK{\mathcal K}
\def\calB{\mathcal B}
\def\calC{\mathcal C}

\def\calE{\mathcal E}

\def\calG{\mathcal G}

\def\calZ{\mathcal Z}
\newcommand{\Isom}{\mathrm{Isom}}

\DeclareMathOperator{\Ad}{Ad}
\DeclareMathOperator{\ad}{ad}

\DeclareMathOperator{\Diff}{Diff}
\def\diff{\Diff}
\DeclareMathOperator{\Gl}{GL}
\def\GL{\Gl}
\DeclareMathOperator{\inj}{inj}
\DeclareMathOperator{\rank}{rank}
\DeclareMathOperator{\Sl}{SL}
\renewcommand{\sl}{\Lie{sl}}

\DeclareMathOperator{\Sp}{{Sp}}
\DeclareMathOperator{\So}{{SO}}

\DeclareMathOperator{\supp}{supp}

\newcommand{\fullcref}[2]{\cref{#1}\pref{#1-#2}}
\newcommand{\fullCref}[2]{\Cref{#1}\pref{#1-#2}}
\newcommand{\pref}[1]{(\ref{#1})}

\newcommand{\bfG}{\mathbf{G}}
\newcommand{\bfL}{\mathbf{L}}

\newcommand{\bfP}{\mathbf{P}}
\newcommand{\bfS}{\mathbf{S}}

\newcommand{\bfU}{\mathbf{U}}
\def\bfF{\mathbf{F}}
\def\bfH{\mathbf{H}}
\def\bfGL{\mathbf{GL}}

\newcommand{\fakeSSS}[1]{\medskip \noindent{\it #1.}}
 \newcommand{\restrict}[2]{{#1}{|_{{ #2}}}}
 \newlist{enumlemma}{enumerate}{3}
\setlist[enumlemma]{label*={(\alph*)}, ref= {(\alph*)} }

\usepackage{makecell}

 \providecommand{\Davethm}{$\langle$Run \LaTeX\ again to get the statement of this result$\rangle$}
 
 \providecommand{\UnipSubgrp}{$\langle$Run \LaTeX\ again to get the statement of this result$\rangle$}
 \providecommand{\cuspgroup}{$\langle$Run \LaTeX\ again to get the statement of this result$\rangle$}
 \providecommand{\Qrkone}{$\langle$Run \LaTeX\ again to get the statement of this result$\rangle$}

 \makeatletter
\newcommand{\laterdef}[2]{%
 \protected@write\@auxout{}{\gdef\string#1{#2}}}
\makeatother

\newcommand{\fund}{\mathcal{F}}
\renewcommand \top{\mathrm{top}}
\newcommand\Sieg{\mathfrak{S}}
\DeclareMathOperator{\Rad}{Rad}
\newcommand\1{\mathbf{1}}
\renewcommand{\ae}{a.e.\ }

\renewcommand\epsilon{\varepsilon}
\def\split{\mathrm{split}}
\def\comp{\mathrm{anis}}

\begin{document}

\title[Zimmer's conjecture for non-uniform lattices]{Zimmer's conjecture for non-uniform lattices: escape of mass and growth of cocycles}

\author[A.~Brown]{Aaron Brown}
\address{Northwestern University, Evanston, IL 60208, USA}
\email{awb@northwestern.edu}

\author[D.~Fisher]{David Fisher}
\address{Indiana University, Bloomington, Bloomington, IN 47401, USA}
\email{fisherdm@indiana.edu}

\author[S.~Hurtado]{Sebastian Hurtado}
\address{University of Chicago, Chicago, IL 60637, USA}
\email{shurtados@uchicago.edu}

\thanks{DF was partially supported by NSF Grants DMS-1607041 and DMS-1906107 and a Fellowship from the Simons Foundation.
AB was partially supported by NSF Grant DMS-1752675. The authors would like to thank UoC, IPAM, IUB for their hospitality at various points.
}

\begin{abstract}
We establish finiteness of low-dimensional actions of lattices in higher-rank semisimple Lie groups and establish Zimmer's conjecture for many such groups. This builds on previous work of the authors handling the case of actions by cocompact lattices and of actions by $\Sl(n,\Z)$. While the results are not sharp in all cases, they do dramatically improve all known results. The key difficulty overcome in this paper concerns escape of mass when taking limits of sequences of measures. Due to a need to control Lyapunov exponents for unbounded cocycles when taking such limits, quantitative controls on the concentration of mass at infinity are need and novel techniques are introduced to avoid ``escape of Lyapunov exponent."
\end{abstract}

\renewcommand{\emph}[1]{{\bf #1}}

\date{\today}
\maketitle

\tableofcontents

\section{Introduction and main results}
 \subsection{Motivation}
We begin by stating a sample theorem that is a very special case of our main results on Zimmer's conjecture.

\begin{theorem}
\label{thm:sample}
Let $\Gamma\subset \Sl(n,\R)$ be a lattice subgroup for $n\ge 3$. Let $M$ be a compact manifold and let $\alpha \colon \Gamma \to \Diff^{1+\beta}(M)$ be a homomorphism for $\beta>0$.

\begin{enumlemma}
\item If $\dim(M)< n-1$ then the image $\alpha(\Gamma)$ is finite.
\item If $\dim(M)\le n$ and if $\alpha(\Gamma)$ preserves a volume form then the image $\alpha(\Gamma)$ is finite.
\end{enumlemma}
\end{theorem}

\noindent For  cocompact $\Gamma$, this was proven by the authors in \cite{BFH}.  When  $\Gamma$ is a finite index subgroup of $\Sl(n,\Z)$, we established this in \cite{BFH-SLnZ}. The crucial contribution of this paper is to establish Theorem \ref{thm:sample} without any additional assumptions on $\Gamma$ and to extend all theorems in \cite{BFH} to non-uniform lattices in the relevant higher-rank Lie groups.

 While there are common arguments across these three papers, fundamental new ideas are needed in this paper to handle the case of general lattices.  These ideas occur primarily in Section \ref{SubexpForRank1Subgroups} and concern controlling the growth of cocycles in the ends of the non-compact space $(G \times M)/\Gamma$.  The non-compactness of this space comes entirely from the quotient $G/\Gamma$ and controlling growth of cocycles in the ends of this space will be central in future dynamical questions concerning $\Gamma$ actions. In \cref{subsec:comp} below, we discuss in more detail what is novel in the approach in this paper compared to the approaches in \cite{BFH, BFH-SLnZ}.

\subsection{History and background}
Before stating our most general results, we give some history related to Zimmer's conjecture that motivates our work. It is well known that lattices in $\Sl(2,\R)$ are rather flexible whereas lattices in $\Sl(n,\R)$ are very rigid when $n\ge 3$. Such rigidity properties are shared by lattices in more general higher-rank simple Lie groups and, to a lesser extent, by irreducible lattices in higher-rank semisimple Lie groups. Some key results demonstrating such rigidity phenomena include the strong rigidity theorem of Mostow \cite{MR0385004}, Kazhdan's property (T) \cite{MR0209390}, and the superrigidity theorem of Margulis \cite{MR1090825}.
For instance, Margulis's superrigidity theorem states, roughly, that every finite-dimensional representation of an irreducible lattice in a higher-rank semisimple Lie group extends, modulo compact groups and passing to covers, to a continuous representation of the ambient group.

Inspired by known rigidity results for linear representations, Zimmer proposed in the 1980s that certain ``nonlinear representations'' of higher-rank lattices should also exhibit some rigidity properties. Specifically, Zimmer laid out a program to study $C^\infty$ volume-preserving actions of such groups on compact manifolds; that is, for a compact manifold $M$ equipped with a smooth volume form $\omega$, Zimmer aimed to study homomorphisms $\alpha \colon \Gamma \to\Diff^\infty(M,\omega)$ from higher-rank lattice subgroups $\Gamma$ to $\Diff^\infty(M,\omega)$, the group of $C^\infty$, $\omega$-preserving diffeomorphisms. In later work of Zimmer and other authors, the case of non-volume-preserving actions is also considered.

The most ambitious goal of this program is to show that all such non-linear representations are derived from certain homogeneous spaces and algebraic actions induced by linear representations. In particular, all volume-preserving actions $\alpha \colon \Gamma \to\Diff^\infty(M,\omega)$ are expected to be of an ``algebraic origin;'' more general actions $\alpha \colon \Gamma \to \Diff^\infty(M)$ are expected to have an ``algebraic quotient." The \emph{Zimmer program} refers to a number of precise conjectures towards this goal. See for instance the surveys of Fisher and Labourie \cite{F11,MR1648087} for overviews of conjectures and results in this area and \cite{F18} for an update with more recent developments.

As a first step in the program, Zimmer conjectured a lower bound on the dimension of a  compact manifold $M$ admitting a nonlinear, volume-preserving representation $\alpha\colon \Gamma\to \Diff^\infty(M,\omega)$ with infinite image. See \cref{conjecture:zimmer} below for a precise formulation. This lower bound is computed in terms of algebraic data associated with the ambient higher-rank group $G$ and its representations. To motivate this conjecture, we recall the following immediate consequence of Margulis's superrigidity and arithmeticity theorems: for $n\ge 3$, if $\Gamma$ is a lattice subgroup in $\Sl(n,\R)$ and if $d<n$ then every homomorphism $\rho\colon \Gamma\to \Gl(d,\R)$ has finite image. Zimmer's conjecture, \cref{conjecture:zimmer}, is the natural non-linear analogue of this fact.

The main heuristic for Zimmer's conjecture is Zimmer's cocycle superrigidity theorem \cite{MR0776417}.
An immediate corollary of Zimmer's cocycle superrigidity theorem is the following: given a lattice $\Gamma$ in a higher-rank simple Lie group and a volume-preserving action $\alpha\colon \Gamma\to \Diff^\infty(M,\omega)$ on a compact manifold $M$ of sufficiently small dimension, there exists a measurable, $\alpha(\Gamma)$-invariant Riemannian metric $g$ on $TM$.
If this measurable metric can be promoted to a continuous metric then the image $\alpha(\Gamma)$ is contained in the compact Lie group of isometries $\Isom_g(M)$; finiteness of the action $\alpha(\Gamma)$ follows if $\dim(M)$ is sufficiently small by known results on representations of such $\Gamma$ into compact Lie groups.

While the above heuristic requires the action $\alpha$ to preserve a volume-form, results for actions on the circle \cite{MR1198459,MR1911660, MR1703323}, results for real-analytic actions \cite{MR1666834, MR1254981}, and results on projective quotients \cite{MR1933077} suggest analogous lower bounds should exist on the dimension of manifolds $M$ admitting nonlinear representations $\alpha\colon \Gamma\to \Diff^\infty(M)$ with infinite image.

\subsection{Zimmer's conjecture and invariant Riemannian metrics}
To state the simplest version of the conjecture, given a semisimple Lie group we define $v(G)$ to be the minimal dimension of a homogeneous space $G/H$, where $H$ ranges over all proper, closed subgroups of the identity component $G^\circ$ of $G$. We define $n(G)$ to be the minimal dimension of a non-trivial linear representation of the Lie algebra of $G$. Given these dimension bounds, we have the following.

\begin{conjecture} \label{ZimmerConjNonVol}
Let $G$ be a connected semisimple Lie group such that every non-compact factor of $G$ has real rank at least $2$. Let $\Gamma$ be a lattice subgroup in $G$.  Let $M$ be a compact manifold and let $\alpha \colon \Gamma \to \Diff^\infty(M)$ be a homomorphism.
\begin{enumlemma}
\item\label{Aa} If $\dim(M)< v(G)$ then $\alpha(\Gamma)$ preserves a smooth Riemannian metric on~$M$.
\item\label{Ab} If $\dim(M)< n(G)$ and if $\alpha(\Gamma)$ preserves a volume form then $\alpha(\Gamma)$ preserves a smooth Riemannian metric on~$M$.
\end{enumlemma}
\end{conjecture}
Using that the isometry group $\Isom(M)$ of a compact Riemanniam manifold is compact, in the case that $\Gamma$ is non-uniform, the conclusion of \cref{ZimmerConjNonVol} implies that
the image $\alpha(\Gamma)$ is finite. Using Margulis's superrigidity theorem for representations into compact groups, from \cref{ZimmerConjNonVol} we obtain an upper bound on the dimension $\dim(M)$ below which the image $\alpha(\Gamma)$ is finite; see \cref{conjecture:zimmer} and discussion in \cref{rem:step3}.

The main results of this paper establishes  \Cref{ZimmerConjNonVol}\ref{Aa} and some cases of \Cref{ZimmerConjNonVol}\ref{Ab} for $C^{1+\text{H\"older}}$ actions of lattices $\Gamma$ in $\R$-split simple Lie groups.  See \cref{Rsplit}.  We also establish partial results towards the conjecture for other groups.  See \cref{thm:main}.

\subsection{Results for lattices in $\R$-split simple Lie groups} \label{sec:12}
We establish  \Cref{ZimmerConjNonVol}\ref{Aa}  for $C^{1+\text{H\"older}}$ actions of lattices $\Gamma$ in $\R$-split simple Lie groups.
For \Cref{ZimmerConjNonVol}\ref{Ab} our method of proof only yields the optimal conjectured upper bound for groups with Lie algebra $\sl(n,\R)$ or $\mathfrak{sp}(2n,\R)$ though we obtain partial results for all other $\R$-split groups.

Given $0<\beta\le 1$ we recall that a diffeomorphism $f\colon M\to M$ of a compact manifold $M$ is $C^{1+\beta}$ if it is $C^1$ and the variation of its derivative is $\beta$-H\"older in local charts. We write $\Diff^{1+\beta}(M)$ for the group of $C^{1+\beta}$ diffeomorphisms of $M$.

 In the statement of all results, we abuse terminology and call a connected Lie group $G$ \emph{simple} if its Lie algebra $\lieg$ is simple; that is, $G$ is simple if it is connected and has no non-trivial connected normal subgroups. In particular, we allow for the case that $G$ has infinite center.

A connected simple Lie group $G$ is \emph{$\R$-split} if its Lie algebra $\lieg$ is a split real form of a simple complex Lie algebra $\lieg^\C$. All $\R$-split simple Lie groups with finite center are isogenous to either $\Sl(n,\R)$, $\Sp(2n,\R)$, $\So(n,n)$, or $\So(n,n-1)$ or to one of the five exceptional  $\R$-split simple Lie groups.

\begin{custheorem} \label{Rsplit} Let $G$ be a connected $\R$-split simple Lie group with real rank at least $2$ and let $\Gamma\subset G$ be a lattice subgroup. Let $M$ be a compact manifold and let $\alpha \colon \Gamma \to \Diff^{1+\beta}(M)$ be a homomorphism for $\beta>0$.

\begin{enumlemma}
\item\label{uglyface} If $\dim(M)< v(G)$ then $\alpha(\Gamma)$ is finite.
\item\label{uglyface2} If $\dim(M)\le v(G)$ and if $\alpha(\Gamma)$ preserves a volume form then $\alpha(\Gamma)$ is finite.
\end{enumlemma}
\end{custheorem}

The special case of \cref{Rsplit} (and of \cref{thm:main} below) for cocompact $\Gamma$ was proved by the authors in \cite{BFH}. As already remarked above, this paper primarily concerns the case that $\Gamma$ is not cocompact and develops new techniques and arguments to overcome the lack of compactness of the homogeneous space $G/\Gamma$. The standard example of a non-uniform lattice is the subgroup $\Sl(n,\Z)$ in $\Sl(n,\R)$; for this special case, \cref{Rsplit} was established by the authors in \cite{BFH-SLnZ}. Although this paper extends and generalizes many of the ideas from \cite{BFH,BFH-SLnZ} to the setting of general non-uniform lattices, important new ideas and techniques are needed here. The difficulty for non-uniform lattices is various subgroup trajectories escaping to infinity in $G/\Gamma$ or even spending too much time ``near infinity''. Particularly there is a need to control the growth of certain cocycles over such trajectories.  We introduce genuinely new approaches to handle the ensuing issues in this paper, primarily in Section \ref{SubexpForRank1Subgroups}.

\subsection{Critical dimensions, full conjecture, and general result} \label{GeneralConjSect}
We state the full version of Zimmer's conjecture in order to put our results in context.
As before, given a semisimple Lie group $G$, let $n(G)$ denote the minimal dimension of a non-trivial real representation of the Lie algebra $\lieg$ of $G$ and let
$v(G)$ denote the minimal codimension of all proper subgroups $H$ of $G$; we note that this is equivalent to the minimal codimension of all proper parabolic subgroups $Q$ of $G$. (See \cite[Lemma 3.7]{BFH}.)
Let $d(G)$ denote the minimal dimension of all non-trivial homogenous spaces $K/C$ as $K$ varies over all compact real-forms of all simple factors of the complexification of $G$. Given the numbers $n(G), d(G),$ and $ v(G)$ the full conjecture is the following.

\begin{conjecture}[Zimmer's conjecture]\label{conjecture:zimmer}
Let $G$ be a connected  semisimple Lie group, all of whose simple factors have real rank at least $2$, and let $\Gamma\subset G$
be a lattice subgroup. Let $M$ be a compact manifold and let $\omega$ be a volume form on $M$.
\begin{enumlemma}
\item\label{conjecture:zimmer-a} If $\dim(M) < \min \bigl( d(G), v(G) \bigr)$ then any homomorphism $\alpha\colon \Gamma \rightarrow \Diff(M)$ has finite image.
\item\label{conjecture:zimmer-b} If $\dim(M) < \min \bigl( d(G), n(G) \bigr)$ then any homomorphism $\alpha\colon \Gamma \rightarrow \Diff(M, \omega)$ has finite image.
\item\label{conjecture:zimmer-c} If $\dim(M) < n(G)$ then for any homomorphism $\alpha\colon \Gamma \rightarrow \Diff(M, \omega)$, the image $\alpha(\Gamma)$ preserves a Riemannian metric.
\item\label{conjecture:zimmer-d} If $\dim(M) < v(G)$ then for any homomorphism $\alpha\colon \Gamma \rightarrow \Diff(M)$, the image $\alpha(\Gamma)$ preserves a Riemannian metric.
\end{enumlemma}
\end{conjecture}

\begin{remark}\label{rem:step3}
\
\begin{enumerate}
\item The finiteness of the image $\alpha(\Gamma)$ in parts \ref{conjecture:zimmer-a} and~\ref{conjecture:zimmer-b} of Conjecture \ref{conjecture:zimmer} follow from the conditions in parts \ref{conjecture:zimmer-c} and~\ref{conjecture:zimmer-d} and Margulis's superrigidity theorem for representations into compact groups. If $g$ is an $\alpha(\Gamma)$-invariant continuous Riemannian metric on $M$ then the isometry group $K$ of $(M,g)$ is a compact Lie group and $\alpha(\Gamma)$ is contained in $K$. Margulis's superrigidity theorem implies that either $\alpha(\Gamma)$ is finite or $K$ contains a compact form of a simple factor of $G$. The latter option is ruled out by dimension considerations. See \cite[\S 2.3]{BFH} for a detailed discussion.

\item For a non-uniform lattice, Margulis's superrigidity theorem also implies that any homomorphism from $\Gamma$ to a compact Lie group $K$ has finite image. In our context, this implies that whenever $\alpha(\Gamma)$ preserves a Riemannian metric, $\alpha(\Gamma)$ is finite.

\item The tools in this paper are very well adapted to study parts \ref{conjecture:zimmer-a} and~\ref{conjecture:zimmer-b} of \cref{conjecture:zimmer}. In situations in which parts \ref{conjecture:zimmer-c} and~\ref{conjecture:zimmer-d} are not subsumed by \ref{conjecture:zimmer-a} and~\ref{conjecture:zimmer-b}, certain technical arguments used in our proof seem insufficient to give a complete answer but do yield the best known partial results.
 \end{enumerate}
 \end{remark}

For Lie groups that are not $\R$-split, we obtain partial results towards \cref{conjecture:zimmer}. To each simple Lie algebra $\lieg$, we associate a number $r(\Sigma_\lieg)$, computed in terms of the
restricted root system $\Sigma_\lieg$ of $\lieg$,
and define  $$r(G) = r(\lieg)= r(\Sigma_\lieg).$$ The number $r(\lieg)$ was defined in terms of combinatorics of root systems in \cite{BFH}; that definition is reviewed in \cref{def:resCoD} below.
For $\R$-split Lie groups $G$, we always have $r(G) = v(G)$. More generally, for simple $G$ it follows from \cite[Theorem 7.2]{MR0207712} that $r(G) = v(G')$, where $G'$ is any maximal connected, $\R$-split, semisimple subgroup of~$G$ (with the same reduced restricted root system.)  See \cref{r(G)=v(split)} below.
Given a semisimple Lie algebra $\lieg$, we have $r(\lieg)=\min r(\lieg')$ where the minimum is taken over every non-compact, non-trivial simple ideal $\lieg'\subset \lieg$.

With the above notation, we have the following extension of \cref{Rsplit}.
\begin{custheorem}\label{thm:main}
Let $G$ be a semisimple Lie group such that every non-compact factor has real rank at least 2. Let $\Gamma \subset G$ be a lattice subgroup and let $M$ be a compact manifold. Fix $0<\beta\le 1$.
\begin{enumerate}
\item If $\dim(M) < r(G)$ then any homomorphism $\alpha\colon \Gamma \to \Diff^{1+\beta}(M)$ has finite image.
\item If $\omega$ is a volume form on $M$ and if $\dim(M) \le r(G)$ then any homomorphism $\alpha\colon \Gamma \to \Diff^{1+\beta}(M, \omega)$ has finite image.
\end{enumerate}
\end{custheorem}

See Tables \ref{T1}--\ref{T4} in \cref{appendix} for computations of the numerology $v,n,d,$ and $r$ associated to simple real Lie groups.

\subsection{Actions by $C^1$ diffeomorphisms}
Certain arguments used in this paper require the action to be by diffeomorphisms that are at least $C^{1+\beta}$ for $\beta>0$; specifically, \cref{prop:invprinc} below uses tools from smooth ergodic theory in its proof which  requires the action to be at least $C^{1+\beta}$. However, most arguments in this paper work for $C^1$ actions and an analogue of \cref{thm:C} below holds in the $C^1$ setting under more restrictive dimension assumptions. Using results from \cite{BrownDamjanovicZhang} one obtains the following.
\begin{theorem}[{c.f. {\cite[Theorem 1]{BrownDamjanovicZhang}}}]\label{thm:FiniteImage}
Let $G$ be a connected simple Lie group with real rank at least $2$. Let $\Gamma \subset G$ be a lattice subgroup and let $M$ be a compact manifold.
\begin{enumerate}
\item If $\dim(M) < \rank_\R(G)$ then any homomorphism $\alpha\colon \Gamma \to \Diff^{1}(M)$ has finite image.
\item If $\omega$ is a volume form on $M$ and if $\dim(M) \le \rank_\R(G)$ then any homomorphism $\alpha\colon \Gamma \to \Diff^{1}(M, \omega)$ has finite image.
\end{enumerate}
\end{theorem}

\noindent It is easy to check that the dimension bound given by Theorem \ref{thm:FiniteImage} only matches the conjectured bounds in \cref{conjecture:zimmer} for lattices in covers of $\Sl(n, \R)$.

In the case of $C^1$ actions, we replace \cref{prop:invprinc} below with \cite[Proposition 3]{BrownDamjanovicZhang}

\subsection{Proof of \texorpdfstring{\cref{thm:main}}{Theorem B}}
We end the introduction by reducing \cref{thm:main} to the first technical result of the paper, \cref{thm:mainreal} below. The proof of \cref{thm:main} from \cref{thm:mainreal} follows the same outline as in \cite{BFH}.

\fakeSSS{Step 1: Subexponential growth of derivatives}
We recall that lattice subgroups in semisimple Lie groups are finitely generated. Given a finitely generated group $\Gamma$, fix a finite, symmetric generating set $S\subset \Gamma$ and let $\len_\Gamma\colon \Gamma\to \N$ denote the word-length norm on $\Gamma$ relative to $S$. The following definition is the key property we establish in this paper. \index{uniform subexponential growth of derivatives}
\begin{definition}\label{def:USEGOD}
Let $M$ be a compact manifold and let $\alpha\colon \Gamma\to \Diff^1(M)$ be an action. We say the action $\alpha$ has \emph{uniform subexponential growth of derivatives} if for all $\epsilon>0$ there is  $C_\epsilon>0$ such that for all $\gamma\in \Gamma$,
\begin{equation}\label{eq:USEGOD}
\sup_{x\in M} \|D_x \alpha(\gamma)\|\le C_\epsilon e^{\epsilon \len_\Gamma(\gamma)}.
\end{equation}

 \end{definition}
In the sequel, we often restrict $\alpha$ to certain subgroups $\Lambda\subset \Gamma$. If $\Lambda$ is a subgroup of $\Gamma$, we say that the restriction $\restrict {\alpha}{\Lambda}$ has \emph{uniform $\len_\Gamma$-subexponential growth of derivatives} if \eqref{eq:USEGOD} holds for all $\gamma\in \Lambda$. We note in this definition that we always measure the word-length of $\gamma\in \Lambda$ relative to the ambient group $\Gamma$ rather than any intrinsic word-length on the subgroup~$\Lambda$.

Our first technical result of this paper is the following which shows under the dimension constraints in \cref{thm:main} that every action has uniform subexponential growth of derivatives.
\begin{custheorem}\label{thm:USEGoD}\label{thm:mainreal}\label{thm:C}
Let $G$ be a connected semisimple Lie group such that every non-compact simple factor has real rank at least 2. Let $\Gamma \subset G$ be a lattice subgroup, let $M$ be a compact manifold, and let $0<\beta\le 1$.
\begin{enumerate}
\item If $\dim(M) < r(G)$ then any homomorphism $\alpha\colon \Gamma \to \Diff^{1+\beta}(M)$ has uniform subexponential growth of derivatives.
\item If $\omega$ is a volume form on $M$ and if $\dim(M) \le r(G)$ then any homomorphism $\alpha\colon \Gamma \to \Diff^{1+\beta}(M, \omega)$ has uniform subexponential growth of derivatives.
\end{enumerate}
\end{custheorem}

\fakeSSS{Step 2: Strong property (T)}
We apply \cite[Theorem 2.4]{BFH} and de la Salle's recent result establishing strong property $(T)$ for nonuniform lattices \cite[Theorem 1.1]{MR4018265} to conclude that any action $\alpha$ satisfying the conclusions of \cref{thm:USEGoD} preserves a continuous Riemannian metric $g$.
In particular, the image $\alpha(\Gamma)$ is contained in the compact Lie group $K=\Isom_g(M)$. For completeness, we recall these two results.
\begin{theorem}[{\cite[Theorem 1.1]{MR4018265}}]
Assume $G$ is a connected semisimple Lie group such that the real rank of every simple factor is at least two and let  $\Gamma$ be a lattice in~$G$. Then  $\Gamma$ has strong property~$(T)$.
\end{theorem}
For the special case of cocompact lattices, the analogous result was established in \cite{MR2423763,MR3407190}.

\begin{theorem}[{\cite[Theorem 2.4]{BFH}}]\label{thm:sobolostintranslation}
Let  $\Gamma$ be a finitely generated group and let $M$ be a compact manifold. For $k\ge 2$, let $\alpha \colon \Gamma \to \Diff^k(M)$ be an action. If $\alpha$ has uniform subexponential growth of derivatives and if $\Gamma$ has strong property $(T)$ then $\alpha(\Gamma)$ preserves a Riemannian metric that is $C^{k-1-\delta}$ for all $\delta > 0$.
\end{theorem}
For actions
by $C^{1+\beta}$ diffeomorphism, the proof of {\cite[Theorem 2.4]{BFH}} can be adapted to establish an analogue of \cref{thm:sobolostintranslation} to obtain a continuous invariant Riemannian metric. For actions by $C^1$ diffeomorphisms, an analogue of \cref{thm:sobolostintranslation} is obtained in \cite[Proposition 5]{BrownDamjanovicZhang}.

\fakeSSS{Step 3: Superrigidity with compact codomains}
As discussed in \cref{rem:step3}, the finiteness of the image $\alpha(\Gamma)$ follows from Margulis's superrigidity theorem with compact codomain. This relies on explicit case by case computation of $d(G)$ that yields that $r(G)<d(G)$, and in fact $v(G)<d(G)$, for all $G$. To do the computation, one uses Margulis' theorem to find all compact groups $K$ into which $\Gamma$ admits dense image homomorphisms and then computes the minimal dimension of a homogeneous space of the form $K/C$.  \cref{thm:main} then follows.

\subsection*{Acknowledgement}  The authors owe a profound debt to Dave Witte Morris.  In particular, Theorem
\ref{QIBddGenByRank1} (and its extension \cref{QIBddGenByRank1lift}) was proven by Witte Morris in \cite{MR4074402} in response to a question from the authors.  In addition, Witte Morris helped the authors a great deal with Section \ref{sec:3} and in particular explained to us the proof of Lemma \ref{LNoCent}.

\section{Reduction to and outline of proof of main result. Theorem \ref{thm:whatweprove}}
\subsection{Reductions and standing hypotheses}\label{hyp1}\index{first reductions on $\Gamma$}
Without loss of generality, we may assume the following hypotheses for the proof of \cref{thm:C}:
\begin{enumerate}[ ]
\item {\it $G$ is a connected semisimple Lie group with real rank at least $2$;}
\item {\it $G$ is simply connected (as a topological group);}
\item {\it $\Gamma$ is an irreducible lattice subgroup of $G$;}
\item {\it $G$ has no compact factors; i.e.\ $G$ has no non-trivial connected compact normal subgroups.}
\end{enumerate}
Indeed, we first note that \cref{thm:C}
 holds if it holds for finite-index subgroups $\Gamma'\subset \Gamma$ or for subgroups $\Gamma''$ containing $\Gamma$ as a subgroup of finite index (where $\Gamma''$ acts on the compact manifold $(\Gamma''\times M)/\Gamma$).\index{action of $\Gamma$ induces an action of any commensurable $\Gamma''$} In particular, we may replace $\Gamma$ with a group commensurable with $\Gamma$ in $G$. Furthermore, if $\Gamma$ is reducible (see \cite[(5.9) Definition, p.~133]{MR1090825}) and the conclusion holds for the restriction of $\alpha$ to each irreducible component of $\Gamma$, then the conclusion holds for the action of $\Gamma$.

Given an arbitrary connected semisimple Lie group $G$ as in \cref{thm:C}, we may replace $G$ with its topological universal cover $\wtd G$; then the preimage of $\Gamma$ under the map $\wtd G\to G$ is a lattice subgroup $\wtd \Gamma$ of $\wtd G$ and any action of $\Gamma$ extends to an action of $\wtd \Gamma$. The conclusion of \cref{thm:C} then follows for an action of $\Gamma$ if it is verified for the induced action of $\wtd \Gamma$.

Assume now that $G$ is simply connected. Let $K\subset G$ and $L\subset G$ denote, respectively, the product of all compact (resp.\ noncompact) normal subgroups. As $G$ is simply connected, $G$ is a direct product of $K$ and $L$. Then $G/L=K$ is compact, hence linear, whence the image of $\Gamma$ in $G/L$  contains a torsion-free subgroup of finite index. Replacing $\Gamma$ with a finite-index subgroup, we may assume $\Gamma\cap K$ is a singleton. Then the map $G\to G/K$ is one-to-one on $\Gamma$ and we may replace $G$ with $G/K$ and $\Gamma$ with its image in $G/K$.

\subsection{Main technical theorem and proof of \cref{thm:mainreal}}
To establish \cref{thm:mainreal} we have the following theorem which gives a dichotomy for $C^1$ actions of higher-rank lattices: either the  action $\alpha\colon \Gamma\to \Diff^1(M)$ has uniform subexponential growth of derivatives\ or there exists an $A$-invariant measure (where $A$ is a split Cartan subgroup) for the induced $G$-action with certain dynamical properties. To prove \cref{thm:mainreal} we show the properties of this measure contradict known results.

The following is the main technical result of this paper. To state the theorem, we refer the reader to \cref{InducedSect} for the definition of the induced $G$-space $X$ and to \cref{subsection:Lyap} for the definition of the average top Lyapunov exponent. Let $\calZ$ denote the center of $G$. See also \cref{RealParabSect} for the definition of a split Cartan subgroup $A\subset G$.
 The cocycle $\calA$ is taken to be the fiberwise derivative cocycle for the induced $G$-action on $X$.

\begin{custheorem}\label{thm:whatweprove}\label{mainthm}
Let $G$ be a connected semisimple Lie group without compact factors and with $\rank_\R G \ge 2$. Let $\Gamma$ be an irreducible lattice subgroup in~$G$, let $M$ be a compact manifold, and let $\alpha\colon \Gamma\to \Diff^{1}(M)$ be an action.
Let $X = (G \times M)/\Gamma$ denote the induced $G$-space and let $\calA$ denote the corresponding fiberwise derivative cocycle on the bundle $(G \times TM)/\Gamma \to X$.

If $\alpha$ fails to have {uniform subexponential growth of derivatives} then there exists a split Cartan subgroup~$A$ of~$G$ and a Borel probability measure~$\mu$ on~$X$ such that
\begin{enumerate}
\item $\mu$ is $(\calZ A)$-invariant,
\item $\mu$ projects to the Haar measure on $G/\Gamma$,
	and
	\item for some $a \in \calZ A$,
	the average top Lyapunov exponent $\lambda_{\top,a,\mu, \calA}$ is positive.
	\end{enumerate}
\end{custheorem}
We remark that there are no constraints on the dimension of $M$ in the statement of \cref{mainthm}. Additionally, in \cref{mainthm} we do not assume every factor of $G$ is higher rank; in particular, the theorem applies to actions by irreducible lattices in semisimple Lie groups with rank-1 factors. Finally, we note that if we assume $\calZ$ is finite, we may conclude $\lambda_{\top,a,\mu, \calA}>0$ for some $a\in A$ as in the conclusion of the analogous results appearing in \cite{BFH,BFH-SLnZ}. For groups with infinite center, we can only conclude that $\lambda_{\top,a,\mu, \calA}>0$ for some (necessarily infinite order) element $a\in \calZ A$.

We also remark that in the case that $\Gamma$ is cocompact in $G$, the proof of \cref{mainthm} is mostly contained in \cite{BFH}. Assuming that $G$ has finite center and that $G$ is simple or that every simple factor of $G$ is of higher rank, the proof of \cref{mainthm} is immediate from arguments in \cite{BFH}. If $\Gamma$ is an irreducible lattice in a higher-rank semisimple Lie group $G$ with rank-1 factors or if $G$ has infinite center, one needs some minor modifications to establish \cref{mainthm}. See \cref{rem:cocomp} below for discussion of the required modifications.

Finally, we remark that for general $C^1$ actions on non-compact spaces, the definition and existence of Lyapunov exponents becomes complicated without additional integrability hypotheses on the norm-growth of the cocycle. In the statement of \cref{mainthm}, the choice of norms on the bundle $(G \times TM)/\Gamma $ are as constructed in \cref{sec:norms} below. These norms are well adapted to the geometry of $\Gamma $ in $G$ and---combined with fact that the measure $\mu$ in the conclusion of in the statement of \cref{mainthm} projects to Haar measure on $G/\Gamma$---ensures the cocycle $\calA$ satisfies a $\log$-integrability criterion. See \cref{CocycIsL1} below.

The remainder of this paper is devoted to establishing \cref{mainthm}. See \cref{ProofOutline} for an outline of the proof.
As discussed above, for actions by cocompact lattices, \cref{thm:whatweprove} is essentially contained in \cite{BFH}; in this case, one first produces an $A$-invariant Borel probability measure with positive exponents by a fairly soft argument in \cite[Proposition 4.6]{BFH}. Improving the measure to one that projects to Haar is more difficult and occupies much of \cite{BFH}. In the context of actions of non-uniform lattices, the arguments that construct an $A$-invariant measure with positive exponents is as difficult as finding one that projects to Haar. As the arguments in this paper use in an essential way all arguments from \cite{BFH} and many of those from \cite{BFH-SLnZ}, the reader may find it easier to read those papers first. An expository account of some of the arguments from \cite{BFH} with more detailed background may be found in the lecture notes by Brown \cite{Brown}; see also the expository account of many ideas from  \cite{BFH} in \cite{MR4093195}.

 \cref{thm:mainreal} follows immediately from \cref{mainthm} following the same as the arguments as in \cite{BFH}. We briefly recall the argument.

\begin{proof}[Proof of \cref{thm:mainreal}]
By \cref{mainthm}, if $\alpha\colon \Gamma\to \Diff^{1+\beta}(M)$ fails to have uniform subexponential growth of derivatives then there exists a split Cartan subgroup $A\subset G$ and a $(\calZ A)$-invariant, Borel probability measure~$\mu$ on the suspension space $X= (G\times M)/\Gamma$ such that $\mu$ projects to the Haar measure on $G/\Gamma$ and the average top Lyapunov exponent $\lambda_{\top,a,\mu, \calA}$ is positive for some $a \in \calZ A$. Arguing exactly as in \cite[Section 5.5]{BFH}, using that the action is $C^{1+\text{H\"older}}$, \cite[Proposition 5.1]{AWBFRHZW-latticemeasure} and the fact that $\dim(M)$ is sufficiently small implies the following.

\begin{proposition}[{\cite[Proposition 5.1]{AWBFRHZW-latticemeasure} and \cite[Proposition 3.5]{BFH}}]
\label{prop:invprinc}
Suppose either that 
\begin{enumerate}
	\item $\dim(M) < r(G)$ or
	\item $\dim(M) \le r(G) $ and $\alpha(\Gamma)$ preserves a smooth volume form on $M$ 
\end{enumerate}
Then, any $A$-invariant Borel probability measure on~$X$ projecting to the Haar measure on $G/\Gamma$ is $G$-invariant.
\end{proposition}

We also have the following well-known corollary of Zimmer's cocycle superrigidity theorem \cite{MR2039990}. In the statement, $n(G)$ denotes the smallest dimension of a nontrivial linear representation of (the Lie algebra of) $G$.

\begin{lemma}
\label{lemma:fromzcsr}
Let $\mu$ be any $G$-invariant measure for the induced action on~$X$ projecting to the Haar measure on $G/\Gamma$. If $\dim (M) < n(G)$ then $\lambda_{\top,g,\mu,\calA} = 0$ for every $g \in G$.
\end{lemma}
To derive the lemma, we remark that $X$ is a fiber-bundle with fibers diffeomorphic to $M$.  Then the fiberwise derivative cocycle is measurably a $\dim(M)$-dimensional linear cocycle; Zimmer's cocycle superrigidity theorem and dimension constraints force the cocycle to be cohomologous to a compact-group-valued cocycle and the conclusion follows.

In particular, assuming Theorem \ref{thm:mainreal} is false, Theorem \ref{thm:whatweprove} produces a $(\calZ A)$-invariant Borel probability measure with non-zero Lyapunov exponents. By Proposition \ref{prop:invprinc} this measure is $G$-invariant, contradicting Lemma \ref{lemma:fromzcsr}.
\end{proof}

\subsection{Review of the proof of \texorpdfstring{\cite[Theorem B]{BFH-SLnZ}}{the theorem for SL(n,Z)}} \label{sec:review}

To motivate the outline in the next subsection, we recall the proof of \cite[Theorem B]{BFH-SLnZ}---the analogue of \cref{thm:USEGoD}
for actions of $\Gamma= \Sl(n,\Z)$ for $n\ge 4$. We recall here that our proof in \cite{BFH-SLnZ} does not in fact cover the case of $\Sl(3,\Z)$ and it's finite index subgroups.  We let $\Gamma\!_{i,j}$ be the subgroup of $\Gamma$ generated by the elementary matrices $E_{i,j}$ and $E_{j,i}$. Then $\Gamma\!_{i,j}$ is isomorphic to $\Sl(2,\Z)$ and differs from the identity matrix only in the $(i,i), (i,j), (j,i),$ and $(j,j)$ entries. Write $H_{i,j}\subset \Sl(n,\R)$ for the corresponding copy of $\Sl(2,\R)$. Let $A\subset G$ be the group of diagonal matrices and $A_{i,j}=\{a_{i,j}^t\}$ for the 1-parameter group of diagonal matrices in $H_{i,j}$
\subsubsection*{Quasi-isometric bounded generation}
	 \label{out:1}As shown in \cite{MR1244421}, the group $\Sl(n,\Z)$ is \emph{quasi-isometrically boundedly generated} by the subgroups $\{\,\Gamma\!_{i,j} \mid 1\le i,j\le n \,\}$. Thus, to show the action $\alpha\colon \Gamma\to \diff(M)$ has {uniform subexponential growth of derivatives} it suffices to show the restriction $\restrict{\alpha}{\Gamma\!_{i,j}} \colon \Gamma\!_{i,j}\to \Diff(M)$ of $\alpha$ to each $\Gamma\!_{i,j}$ has uniform subexponential growth of derivatives.

\subsubsection*{{Uniform subexponential growth of derivatives} for the action of unipotent subgroups}
The first key proposition established in \cite{BFH-SLnZ} shows that unipotent elements of $\Gamma= \Sl(n,\Z)$ act with subexponential growth: for every unipotent $\gamma\in \Gamma$ and $\epsilon>0$ there is a $C>0$ such that for all $n\in \Z$, $$\sup_{x\in M} \|D_x\alpha(\gamma^n)\|\le Ce^{\epsilon \len_\Gamma(\gamma^n)}.$$ We recall that $\len_\Gamma(\gamma^n)$ denotes the word-length of $\gamma^n$ measured in~$\Gamma$ which, in particular, grows logarithmically in~$n$.

\subsubsection*{{Uniform subexponential growth of derivatives} for the action of $\Gamma\!_{i,j}$}
We now explain why the action of $\Gamma\!_{i,j}$ has uniform subexponential growth of derivatives
\begin{enumerate}
	\item We pass to the suspension space $X=(G\times M)/\Gamma$. If $\restrict{\alpha}{\Gamma\!_{i,j}} \colon \Gamma\!_{i,j}\to \Diff(M)$ fails to have uniform subexponential growth of derivatives, we may find a sequence of finite $\{a_{i,j}^t\}$-orbits $\{\, a_{i,j}^t\cdot x_n \mid 0\le t\le t_n \,\}$ with $t\to \infty$ and each $x_n$ contained in the ``thick part'' of $X$ over $H_{i,j}/\Gamma\!_{i,j}$ such that, in the limit, we see positive exponential growth of the fiberwise derivative.

\item
Using that the ``cusp group'' of the rank-1 subgroup $\Lambda_{i,j}\simeq \Sl(2,\Z)$ of $\Sl(n,\Z)$ is generated by a single unipotent element we show that, restricted to the subbundle of $X$ over $H_{i,j}/\Gamma\!_{i,j}$ any collection of finite $\{a_{i,j}^t\}$-orbits $\{\, a_{i,j}^t\cdot x_n \mid 0\le t\le t_n \,\}$ with $t\to \infty$ that approximates the maximal exponential growth rate of the fiberwise derivatives gives a family of empirical measures $\mu_n$ that is primarily concentrated over the ``thick part'' of $H_{i,j}/\Gamma\!_{i,j}$.

\item We may average the family of measures $\mu_n$ over \Folner sets in a unipotent subgroup $N$ normalized by $H_{i,j}$ that meets $\Gamma$ in a lattice to obtain a new family of measures $\td \mu_n$. Using quantitative non-divergence of unipotent flows, or an explicit computation as in \cite[Subsection 5.3]{BFH-SLnZ}, we control the amount of mass of each $\td \mu_n$ at $\infty$; in particular, we rule out escape of mass. Also, using that the $N\cap \Gamma$ contains only unipotent
    elements, it follows that these measures see no ``escape of Lyapunov exponent.''

\item There is a one-parameter family of diagonal matrices $\{b^s\}$ such that $N$ is the horospherical subgroup of $\{b^s\}$ for $s>0$. Using exponential mixing of $b^s$, or an explicit computation as in \cite[Subsection 5.4]{BFH-SLnZ}, there is a sequence of $s_n\to \infty$ such that the family of measures $\{b^{s_n}\cdot \td \mu_n\}$ is uniformly tight and, in fact, have uniformly exponentially small mass at $\infty$ (see \cref{def:masscusp}) when projected to $G/\Gamma$; moreover, we are able to maintain positive exponential growth of fiberwise derivative for the action of $a_{i,j}^t$.
\item It follows that any limit measure of $\{b^{s_n}\cdot \td \mu_n\}$ projects to the Haar measure on has $G/\Gamma$ and has a non-zero fiberwise Lyapunov exponent for the action of $a_{i,j}^t$. We can then obtain an $A$-invariant measure on $X$ that projects to the Haar measure on  $G/\Gamma$ and has a non-zero fiberwise Lyapunov exponent.  We then obtain a contradiction with Zimmer's cocycle superrigidity as in the proof of \cref{thm:mainreal}.
\end{enumerate}

 \subsection{Outline of proof of \cref{mainthm} and discussion of new tools} \label{ProofOutline}
We now outline the proof of \cref{mainthm} and highlight the new tools developed in this paper when compared with the proof of \cite[Theorem B]{BFH-SLnZ}. For lattices of higher $\Q$-rank, the proof of \cref{mainthm} has the same main lines as in \cite{BFH-SLnZ} though with considerable additional technical difficulty, particularly in the case where the $\Q$-rank is $2$.  In the case of  $\Q$-rank-$1$, there are genuinely new complications that require ideas not in any way present in  \cite{BFH-SLnZ}.

For simplicity of this outline, we may assume $G$ is linear (in particular has finite center) and defined over~$\Q$. Moreover, as \cite{BFH} essentially handles the case that $\Gamma$ is cocompact, we may assume $\Gamma$ is nonuniform.  By Margulis's arithmeticity theorem (see discussion in \cref{hyp2}) we may   assume $\Gamma$ is commensurable with $G_\Z$.

\subsubsection{Quasi-isometric bounded generation}
In \cite{BFH-SLnZ}, we used the result \cite{MR1244421} of  Lubotzky, Mozes, and Raghunathan  that $\Sl(n,\Z)$ is \emph{quasi-isometrically boundedly generated} by the subgroups $\{\, \Gamma\!_{i,j} \mid 1\le i,j\le n \,\}$; this was  used in \cite{MR1244421} to show $\Gamma$ is quasi-isometrically embedded in $G$ (see \cref{LMR} below).  
In the setting of this paper, we need an analogous result for general arithmetic groups $\Gamma$. In the general version \cite{MR1828742}, Lubotzky, Mozes and Raghunathan establish the general quasi-isometric embedding result via a
different outline; in particular, they do not establish the analogous quasi-isometrically bounded generation result we require.
However, very recently, Witte Morris established a  quasi-isometrically bounded generation result sufficient for our proof.  We replace the canonical copies of $\Sl(2,\Z)\simeq \Gamma_{i,j}$ from  \cite{BFH-SLnZ} with lattice subgroups in  \emph{standard $\Q$-rank-1 subgroups}; see \cref{StandardSubgrpDefn}.  We then have the following key definition and \lcnamecref{QIBddGenByRank1} provided to us by Dave Witte Morris \cite{MR4074402}:

\def\qibdddef{
Let $F= \bfF(\R)$ be a semisimple algebraic $\Q$-group and let $\hat \Gamma$ be a subgroup commensurable with $F_\Z$.
We say that $\hat \Gamma$ is {\bf quasi-isometrically boundedly generated by standard $\Q$-rank-1 subgroups}   if there are constants $r = r(\bfF,\hat \Gamma) \in \N$ and $C = C(\bfF,\hat \Gamma) \in \R^+$, and a finite subset $\hat \Gamma\!_0 =\hat \Gamma\!_0(\bfF,\hat\Gamma)$ of~$\hat\Gamma$, and a finite collection $\mathcal{L} = \mathcal{L}(\bfF,\hat \Gamma)$ of $\Q$-subgroups of $G$ such that
	\begin{enumerate}
	\item each $L \in \mathcal{L}$ is a standard $\Q$-rank-1 subgroup of~$F$;
	\item
	 every element~$\gamma$ of~$\hat \Gamma$ can be written in the form $\gamma = s_1 s_2 \cdots s_r$ where either
	\begin{enumerate}
	\item $s_i \in \hat\Gamma\!_0$
	or
	\item $s_i \in \hat \Gamma\!_L$ for some $L \in \mathcal{L}$ and $\log \|s_i\| \le C \log \|\gamma\|$.
	\end{enumerate}
	\end{enumerate}
}

\begin{repdefinition}{QIBddGenDefnF}
Let $G= \bfG(\R)$ be a semisimple algebraic $\Q$-group and let $ \Gamma$ be a subgroup commensurable with $G_\Z= \bfG(\Z)$.
We say that $  \Gamma$ is {\bf quasi-isometrically boundedly generated by standard $\Q$-rank-1 subgroups}   if there are constants $r = r(\bfG,  \Gamma) \in \N$ and $C=C(\bfG,  \Gamma)>1$ , and a finite subset $  \Gamma\!_0 =  \Gamma\!_0(\bfG,\Gamma)$ of~$ \Gamma$, and a finite collection $\mathcal{L} = \mathcal{L}(\bfG,\Gamma)$ of $\Q$-subgroups of $F$ such that
	\begin{enumerate}
	\item each $L \in \mathcal{L}$ is a standard $\Q$-rank-1 subgroup of~$G$;
	\item
	 every element~$\gamma$ of~$  \Gamma$ can be written in the form $\gamma = s_1 s_2 \cdots s_r$ where either
	\begin{enumerate}
	\item $s_i \in  \Gamma\!_0$
	or
	\item $s_i \in   \Gamma\!_L$ for some $L \in \mathcal{L}$ and $\log \|s_i\| \le C \log \|\gamma\|$.
	\end{enumerate}
	\end{enumerate}\end{repdefinition}
Above and in  following, we write $\Gamma\!_L = L\cap \Gamma$.
The following is the main result of \cite{MR4074402}.
\begin{reptheorem} {QIBddGenByRank1}
{Every arithmetic subgroup $  \Gamma$ of a $\Q$-isotropic, almost $\Q$-simple $\Q$-group is quasi-isometrically boundedly generated by standard $\Q$-rank-1 subgroups.}
\end{reptheorem}

\subsubsection{Uniform subexponential growth of derivatives for unipotent subgroups}
As in \cite{BFH-SLnZ}, we again show that unipotent subgroups $\Delta\subset \Gamma$ have subexponential growth. This is done in Section \ref{UnipotentSubgroups}. While argument is similar in outline to one in \cite[Section 4]{BFH-SLnZ}, the fact that we work in much greater generality in the end requires substantial new arguments.

\begin{repproposition}{UnipSubgrp}
\UnipSubgrp
\end{repproposition}

To establish \cref{UnipSubgrp}, we work with the induced $G$-action. It is easy to see that subexponential growth of derivatives for $\Gamma\!_U$ is equivalent to subexponential growth of derivatives for a closed $U$ orbit in $X:=(G \times M)/\Gamma$. The proof of the proposition proceeds by working inside compact orbits of a solvable group of the form $\wtd A \ltimes U$ where $U$ is a horospherical  $\Q$-group and $\wtd A$ is part of a $\Q$-anisotropic torus. We first show   that generic trajectories for elements of $\wtd A$ have subexponential growth of derivatives. Combining this with the fact that $U$ is normalized by $\wtd A$, we obtain subexponential growth of derivatives for a large set of elements in $U$.  Using that $U$ is nilpotent, a sumset argument implies subexponential growth of derivatives for every element of $U$.  To establish subexponential growth of derivatives for generic orbits of elements of $\wtd A$, we exploit that $U$ is horospherical and use an exponential mixing argument to upgrade an $\tilde A$-invariant measure supported on some closed $\tilde A \ltimes U$ orbit to an $A$-invariant measure projecting to the Haar measure on $G/\Gamma$.

\subsubsection{{Uniform subexponential growth of derivatives} for ``cusp groups'' of standard $\Q$-rank-1 subgroups}\label{subsubsec:cusps}   In \cite{BFH-SLnZ}, we heavily used that the fundamental group of the  cusp in $\Sl(2,\R)/\Sl(2,\Z)$ is generated by the single unipotent element $\left(\begin{array}{cc}1 & 1 \\0 & 1\end{array}\right)$. For a standard $\Q$-rank-1 subgroup $H$, we cannot expect the  fundamental group of a cusp in $H/\Gamma_H$ to be generated by a single unipotent element; moreover, the fundamental group of a cusp need not be a unipotent subgroup. However,  the structure of such groups (the semidirect product of a cocompact lattice and a unipotent subgroup) allows us to still show the restriction of the action to the  fundamental group of a cusp  in $H/\Gamma_H$ for a standard $\Q$-rank-1 subgroup $H$   has subexponential growth.

Given a standard $\Q$-rank-1 subgroup $H$, to establish that the restriction of the action to the  fundamental group of a cusp in  $H/\Gamma_H$ has subexponential growth, we exploit reduction theory and the construction of Siegel fundamental domains to relate growth properties of the fundamental group of a cusp in  $H/\Gamma_H$ with subexponential growth for minimal parabolic $\Q$-subgroups in $G$.  We point the reader to \cref{SiegelForSubgroup} and especially \cref{SubexpInCusps} in   \cref{TemperedSect}.

To complete this step, it remains to establish subexponential growth for minimal parabolic $\Q$-subgroups in $G$ which occupies \cref{SubexpForParabolic}. 

\begin{repproposition}{prop:cuspgroup}
\cuspgroup
\end{repproposition}

The proof of \cref{prop:cuspgroup} uses that $\Gamma\!_P$ is the semidirect product of  a cocompact lattice in a reductive ($\Q$-anisotropic) $\Q$-subgroup $L$ and a (cocompact) lattice in a unipotent $\Q$-subgroup $U$. Subexponential growth for the restriction to $\Gamma\!_U$ follows from \cref{UnipSubgrp}.  Following arguments from \cite{BFH}, we show that if $\Gamma\!_L$ fails to have subexponential growth of the derivatives, then there exists a Borel probability measure on $X$ which is invariant under a torus in $L$ and has a positive Lyapunov exponent.
Using \cref{UnipSubgrp}, we  can find a measure that is also invariant under $U$ and also has a positive Lyapunov exponent.
We then use exponential mixing to find a measure that projects to the Haar measure on $G/\Gamma$ with a positive Lyapunov exponent which can be made $A$-invariant by additional averaging.

\subsubsection{{Uniform subexponential growth of derivatives} for $\Q$-rank-1 subgroups}\label{subsubsec:qrank1} By \cref{QIBddGenByRank1}, to show \cref{mainthm} it is enough to show the restriction of $\alpha$ to $\Gamma\!_H\subset \Gamma$ has uniform subexponential growth of derivatives for any standard $\Q$-rank-1 subgroup $H\subset G$.  We recall that $H\subset G$ is a subgroup defined over $\Q$ and that $\Gamma\!_H = H\cap \Gamma$   is a lattice in $H$.

Having established \cref{UnipSubgrp,prop:cuspgroup,SubexpInCusps},  \cref{SubexpForRank1Subgroups} is devoted to the following.
\begin{repproposition}{thm:mainQR1}
\Qrkone
\end{repproposition}

The proof of \cref{thm:mainQR1} contains the main new ingredients of the paper which have no analogue in  \cite{BFH-SLnZ}.   This is especially true in the case that $\Gamma$ is $\Q$-rank 1.
When $\Gamma$ has  $\Q$-rank 1, we have that $G=H$ is  the only standard $\Q$-rank-1 subgroup; in particular, we cannot find a subgroup of $G$ of the form $H\ltimes N$ where $H$ is a standard $\Q$-rank-1 subgroup and $N$ is a (non-trivial) horospherical subgroup of $G$.  Having a horospherical subgroup $N$ normalized by $H$ was a key ingredient used to establish subexponential growth for the restriction to the canonical copies of $\Sl(2,\Z)$ in $\Gamma= \Sl(m,\Z)$ in  \cite{BFH-SLnZ}.
 For the sake of completeness, when $\Gamma$ has $\Q$-rank at least 2, we outline a proof of  \cref{thm:mainQR1} in \cref{AlternateHighRank} that is somewhat closer to the argument from  \cite{BFH-SLnZ}.

 \cref{SubexpForRank1Subgroups} provides a uniform argument when $\Gamma$ is $\Q$-rank-1 or has higher-$\Q$-rank.
 Fix a $\Q$-rank-1 subgroup $H\subset G$.  When $\Gamma$ has higher $\Q$-rank we also find a horospherical subgroup $N$ normalized by $H$; if $\Gamma$ has $\Q$-rank 1, we take $N= \{\1\}$.  We study growth of the fiberwise derivative cocycle for the induced $H$-action on the bundle $X_{HN}:= (HN\times M)/\Gamma_{HN}$.
We fix a 1-parameter ($\R$-split) subgroup $\{a^t\}$  in $H$  and   consider orbit segments of that start and end in some fixed ``thick'' compact  part of $X_{HN}$.  Using  \cref{prop:cuspgroup} and \cref{UnipSubgrp}, a sequence of empirical measures supported on such orbits and limiting to the maximal growth rate of the fiberwise derivative cocycle (for this 1-parameter subgroup) will, in fact, produce a uniformly tight sequence of measures (see \cref{lem:maxpathstight}).
Passing to a limit, we obtain an $a^t$-invariant Borel probability measure $\mu$ on   $X_{HN}$.
While we are able to avoid escape of mass, we are not able to rule out ``escape of Lyapunov exponent.''  Specifically, the fiberwise derivative cocycle is not bounded and there is no reason the fiberwise derivative cocycle is $\log$-$L^1$ for the limiting measures $\mu$.  Furthermore, even if  fiberwise derivative cocycle happens to be $\log$-$L^1$ for the limiting measures $\mu$ and the empirical measures see exponential growth of the cocycle, we need not have any semi-continuity properties of the top exponent.

To remedy this problem, we use a combination of cut-off functions and time averages to modify the fiberwise derivative cocycle in the cusp; see \cref{TAcoc} and   \eqref{psiellDefn} in \cref{s4.2}.  These new cut-off cocycles will be bounded but will fail to be continuous; however, for the cut-off cocycle, the set of discontinuities is rather tame and, for the analysis in  \cref{SubexpForRank1Subgroups}, these cocycles acts much like a bounded continuous cocycle.
Using these modified cocycles, we define analogues of Lyapunov exponents; we then show these ``fake Lyapunov exponents" behave well under various averaging operations (using that the discontinuity set of the cut-off cocycle has zero measure for all limiting measures considered).
We can then use averaging techniques (following either \cite{BFH} or \cite{BFH-SLnZ}) to upgrade to a measure whose projection to $G/\Gamma$ is the Haar measure on $G/\Gamma$ while maintaining a positive ``fake Lyapunov exponent.''
Since the fiberwise derivative cocycle is $\log$-$L^1$ for measures projecting to the Haar measure (see \eqref{HaarSmallCusps} and \cref{CocycIsL1}), we check the  ``fake top Lyapunov exponent'' coincides with the actual top Lyapunov exponent of the original cocycle once we are considering a measure which projects to Haar.

\begin{remark}[Benefits of working with $\Q$-rank-1 subgroups]
We remark that the reduction to $\Q$-rank-1 subgroups $\Gamma\!_H$ is primarily to control the behavior of the cocycle outside of compact sets. If $\Gamma\subset G$ is an arithmetic lattice with $\rank_\Q(\Gamma) = 1$, then there is some $\ell>0$ such that for any $1$-parameter subgroup $\{a^t\}$ of $G$, if $d(a^t\cdot x, \1\Gamma)\ge\ell$ for all $0\le t\le t_0$ then, lifting the path $\{\,a^t\cdot x \mid 0\le t\le t_0\,\}$ in $G/\Gamma$ to the path $\{\,a^t\cdot \td g \mid 0\le t\le t_0\,\}$ in $G$, for any  element $\gamma_x\in \Gamma$ such that $g$ and $a^tg\gamma\inv$ are in the same Siegel set, $\gamma_x$ is in the fundamental group of the cusp of $H/\Gamma\!_H$.

This fails for lattices of higher $\Q$-rank. Indeed, for higher-rank groups, there is no collection of disjoint cusps. In $\Q$-rank one, the complement of the thick part is  a disjoint union of cusps, each of the form $\R^+ \times N$ where $N$ is a compact manifold.  In higher rank, the complement of the thick part is a union of thickened fans, where each fan is of the form $S \times N$ where $S$ is a Weyl chamber and $N$ is a compact manifold. These fans can intersect non-trivially in lower dimensional sets and a path may leave the thick park of $G/\Gamma$ through one fan and return through another fan which intersects the first in a lower-dimensional object. This is a consequence of the fact that in this case the rational Tits building is connected. In this case, the collection of possible monodromy elements $\gamma_x$ of a path $\{\,a^t\cdot x \mid 0\le t\le t_0\,\}$ does not have a structure we can directly exploit.
\end{remark}

\subsection{Comparison to our earlier papers on Zimmer's conjecture}
\label{subsec:comp}
We end this section with a summary of similarities and differences  between the techniques used here and those used in our previous papers \cite{BFH,BFH-SLnZ} and emphasize the new ideas developed in this paper. 

An analogue of Theorem \ref{thm:whatweprove} occurs in both \cite{BFH} and \cite{BFH-SLnZ}, though it is never explicitly stated as a separate theorem.
The primary difficulty in both this paper and in \cite{BFH-SLnZ} is in establishing this theorem: the failure of subexponential
growth of derivatives for the $\Gamma$-action can be witnessed as an $A$-invariant measure on the suspension space with positive top Lyapunov exponent and
which projects to Haar measure on $G/\Gamma$.  Besides Theorem \ref{thm:whatweprove}, the only substantial change from the outline in \cite{BFH} needed to study actions of non-uniform lattices is 
de la Salle's result that non-uniform
higher-rank lattices have strong property $(T)$ \cite{MR4018265}.

To prove \cref{thm:whatweprove} in the cocompact case, establishing that there is some  $A$-invariant measure with positive
Lyapunov exponent is a relatively soft argument; the main technical difficulty in \cite{BFH} is finding a procedure to average such measures to obtain  a new $A$-invariant measure that projects to the Haar measure and maintains a positive top Lyapunov exponent.
All averaging procedures in \cite{BFH} heavily use that 
the suspension space is compact whence the space of Borel probability measures is
weak-$*$ sequentially compact and the top Lyapunov
exponent is automatically upper semi-continuous when restricted to the space of invariant  Borel probability measures.
In \cite{BFH-SLnZ}, for actions by $\Sl(n,\Z)$ ($n\ge 4$) we were able to overcome the failure of weak-$*$ compactness and establish a version of upper semicontinuity of the top exponent which yielded  a special case of Zimmer's conjecture following the general approach of \cite{BFH}.  As described above, this is done by iteratively finding certain subgroups along which failure of subexponential growth of derivatives can be used to construct a measure as in Theorem \ref{thm:whatweprove}.
When the $\mathbb{Q}$-rank of $\Gamma$ is at least $2$, the large scale structure of this  paper mostly resembles that of \cite{BFH-SLnZ} with several significant improvements and modifications which we briefly describe.  First, the arguments in   \cite{BFH-SLnZ} required  
the  $\mathbb{Q}$-rank to be at least $3$;   
this caused no loss of generality in \cite{BFH-SLnZ} since Zimmer's conjecture was already known for finite index subgroups $\Sl(3,\Z)$.   The proof in  \cite{BFH-SLnZ} also depended (see discussion in  \cref{subsubsec:cusps}) on the fact that ``cusp groups'' of the $\Q$-rank-1 subgroups are generated by a single unipotent. 
Finally, \cite{BFH-SLnZ} used an ad hoc argument---inspired by exponential mixing---to obtain limiting measures projecting to the Haar measure on $G/\Gamma$ as a final step in  the proof.  In the current paper, we replace this ad hoc argument by a more robust argument following exponential mixing with careful optimization of all constants involved.  These arguments are used to establish subexponential growth of derivatives for the restriction to unipotent subgroups and cusp groups of $\Q$-rank-1 subgroups.  The arguments here can be used to write a more efficient proof of the results in \cite{BFH-SLnZ}.

In the special case of groups of $\mathbb{Q}$-rank $1$, this paper substantially diverges from 
our  previous arguments in  \cite{BFH-SLnZ}  and requires substantial new ideas.
In this case, we still establish subexponential growth of derivatives iteratively along unipotent subgroups and  cusp groups of $\Q$-rank-1 subgroups as described in the previous paragraph.
The end-game using those results is entirely different from the endgame in  \cite{BFH-SLnZ}  as  the structure of $\Gamma$ does not allow us access to the use of exponential mixing or the ad hoc variant of it used in \cite{BFH-SLnZ}.  This requires the introduction of entirely new ideas to extend our arguments so we can deal with averages of unbounded cocycles as described in  \cref{subsubsec:qrank1} above.

\section{Algebraic groups and reductions for nonuniform $\Gamma$} \label{sec:3}
We present the basic terminology and facts from algebraic groups that will be used in the sequel. We end the section with the main reductions we use when working with nonuniform lattices $\Gamma$ and the consequences of these reductions.
\subsection{Basic terminology}
We denote by $\bfGL(d)=\Gl(d,\C)$ the affine algebraic group of $d\times d$ invertible matrices with complex entries.
A \emph{(linear) algebraic group} is a Zariski closed subgroup $\bfH$ of $\bfGL(d)$ for some $d$; that is $\bfH$ is a subgroup that coincides with the common zeros of a set of polynomial functions in the matrix entries of $A\in \bfGL(d)$.
Working over $\C$, a linear algebraic group $\bfH$ is connected in the Zariski topology if and only if it is connected in the analytic topology.

Let $\K \in \{\R, \Q\}$. An algebraic subgroup $\bfH\subset \bfGL(d)$ is said to be  \emph{defined over $\K$} or is said to be \emph{$\K$-group} if $\bfH$ is algebraic and the polynomial functions defining $\bfH$ take coefficients only in $\K$.
The \emph{radical} (resp.\ \emph{unipotent radical}) of a linear algebraic subgroup $\bfH$ {defined over $\K$} is the maximal connected solvable (resp.\ unipotent) algebraic subgroup $\Rad(\bfH)$ (resp.\ $\Rad_u(\bfH)$) of $\bfH$. These are $\K$-subgroups of $\bfH$.
The algebraic group $\bfH$ is \emph{semisimple} (resp.\ \emph{reductive}) if the radical (resp.\ unipotent radical) vanishes.
We say that $\bfH$ is \emph{almost $\K$-simple} if it has no connected, proper, normal $\K$-subgroups. If $\bfH$ is noncommutative, almost $\K$-simple, and connected, then $\bfH$ is semisimple.

Given a linear algebraic subgroup $\bfH\subset \bfGL(d)$ {defined over $\K$}, let $\bfH(\K)=\bfH\cap\GL(d,\K)$ denote the $\K$-points of $\bfH$. If $\K=\R$, the group $H=\bfH(\R)$ is a Lie group which, in general, might not be connected even if $\bfH$ is connected but has only finitely many connected components (in the analytic topology); we write $H^\circ$ for the connected component (in the analytic topology) of the identity of $\bfH(\R)$.
When $\bfH$ is simply connected (as an algebraic group) or if $\bfH$ is a unipotent $\K$-group, then $\bfH(\R)$ is connected in the analytic topology. See \cite[Chapter 1, Theorem 2.3.1(c)]{MR0315007} or \cite[Proposition 7.6]{MR1278263}.

Given a connected algebraic group $\bfH$ defined over $\K$, we will at times abuse terminology and also say that $H= \bfH(\R)$ and $H^\circ$ are defined over $\K$. We have that $\Rad(\bfH)(\R)$ (resp.\ $\Rad_u(\bfH)(\R)$) is the radical (resp.\ unipotent radical) of the Lie group $H$.
If $\bfH$ is semisimple (resp.\ reductive) then $H=\bfH(\R)$ is semisimple (resp.\ reductive) as a Lie group

We collect the following important facts about unipotent subgroups.
\begin{lemma} \label{BasicUnip}
Let $U$ be a connected unipotent subgroup of $\Sl(d,\R)$.
	\begin{enumerate}
	\item \label{BasicUnip-exp} \label{BasicUnip-homo}
	$U$ is simply connected. In fact, the exponential map from the Lie algebra~$\lieu$ of~$U$ to~$U$ is a bijective polynomial map and its inverse is also a polynomial map
	\cite[Theorem~8.1.1, p.~107]{MR620024}.
	This implies that every continuous homomorphism from~$U$ to any unipotent group is a polynomial.
	\item The abelianization $U/[U,U]$ of~$U$ is simply connected (and in fact, it is a unipotent linear algebraic group).
	\end{enumerate}
\end{lemma}

\subsection{Algebraic tori and $\K$-rank}\label{sec:tori}\index{$\K$-tori}
Let $\bfH\subset \bfGL(d)$ be a connected linear algebraic group defined over $\K \in \{\R,\Q\}$.
	\begin{enumerate}
	\item We say that $\bfH$ is a \emph{$\K$-torus} if it is commutative, defined over $\K$, and conjugate in $\bfGL(d)$ to a subgroup of diagonal matrices
	\item A $\K$-torus $\bfH$ is \emph{$\K$-split} if there exists $g \in \Sl(d, \K)$, such that $g \bfH g^{-1}$ is a subgroup of diagonal matrices. If $\bfH$ is a $\K$-split torus then $\bfH(\R)$ is isomorphic to~$(\R^*)^\ell$ for some~$\ell$.
	\item \label{split} Let $\bfS$ be a torus defined over $\K$. There exists unique maximal $\K$-tori $\bfS',\bfS''\subset \bfS$ such that $\bfS= \bfS'\cdot \bfS''$ where $\bfS'$ is $\K$-split and $\bfS''$ contains no $\K$-split subtorus. (See \cite[Proposition 13.2.4]{MR2458469}.)
	\item If $\bfH$ is defined over~$\K$, then all maximal $\K$-split tori of~$\bfH$ are conjugate by elements of $\Sl(d, \K)$, and therefore have the same dimension. (See \cite[(0.25)]{MR1090825} or \cite[Theorem 15.2.6]{MR2458469}.) This common dimension is the \emph{$\K$-rank} of~$\bfH$, denoted by $\rank_\K \bfH$.
	\item We say that $\bfH$ is \emph{$\K$-anisotropic} if $\rank_\K \bf H = 0$; otherwise, $\bfH$~is \emph{$\K$-isotropic}. An $\R$-group $\bfH$ is an $\R$-anisotropic if and only if $\bfH(\R)$ is compact.
	\item If $\bfH$ is {reductive} then it is the almost-direct product of a semisimple $\K$-group and a $\K$-torus: $\bfH = \bfL \bfS$ where $\bfL$~is semisimple, $\bfS$~is a torus that centralizes~$\bfL$, and $\bfL \cap \bfS$ is finite. The subgroups $\bfL$ and~$\bfS$ are unique.
\item If $H= \bfH(\R)$ and if $\bfS$ is a maximal $\K$-torus (resp.\ maximal $\K$-split torus) we call $S= \bfS(\R)$ a maximal $\K$-torus (resp.\ maximal $\K$-split torus) of $H$.
\item From \eqref{split}, every one-parameter subgroup~$\{a^t\}$ of a $\K$-torus~$S$ has a \emph{real Jordan decomposition} $a^t = a^t_{\split} \, a^t_{\comp}$, where $a^t_{\split}$ is $\R$-diagonalizable, $a^t_{\comp}$ is contained in a compact subgroup,  and $a^t_{\split}$ and $a^t_{\comp}$ are contained in~$S$.
\item If $\K= \R$ and if $\bfS$ is a maximal $\R$-split torus in $\bfH$, we call the connected component of the identity $A= \bfS(\R)^\circ$ a
 \emph{split Cartan subgroup} of~$H= \bfH(\R).$
	\end{enumerate}

\subsection{Levi and Langlands decompositions} \label{sec:levilang}
For $\K \in \{\R,\Q\}$, let $\bfH$ be a connected linear algebraic group over~$\K$.
	\begin{enumerate}
	\item $\bfH$ can be written as a semidirect product $\bfH = \bfF \ltimes \bfU$ where $\bfU= \Rad_u(\bfH)$ and $\bfF$ is a connected reductive $\K$-group. This is a \emph{Levi decomposition} and $\bfF$ is a \emph{Levi subgroup}. All Levi subgroups of~$\bfH$ are conjugate by elements of $\bfU(\K)$. See \cite[pp.~200--201]{MR0092928}\index{Levi Decomp}
	\item We may further decompose a Levi $\K$-subgroup as $\bfF = \bfL \times \bfS$ where $\bfL$~is a connected reductive $\K$-subgroup with $\K$-anisotropic center and $\bfS$~is a connected $\K$-split torus.

Write $H= \bfH(\R)$ and let
	$$L= \bfL(\R), \ \ S= \bfS(\R), \ \ A = S^\circ, \ \ \text{and } \ \ U= \Rad_u(\bfH)(\R). $$
	 Then $H= F\ltimes U$ and $H=(L\times A)\ltimes U$ are called, respectively, a \emph{Levi decomposition}
 and a \emph{Langlands decomposition} of $H$. Note that $U$ and $A$ are connected and we also call $H^\circ = (L^\circ \times A)\ltimes U$ a Langlands decomposition of $H^\circ$.
\item If $\bfH$ is defined over $\Q$ then all Levi subgroups $F$ of $H= \bfH(\R)$ are conjugate by $U$ regardless if we view $\bfH$ as a $\Q$-group or as a $\R$-group.

However, a Langlands decomposition of $H$ as a $\Q$-group need not be conjugate to any Langlands decomposition of $H$ as an $\R$-group. We refer to a choice of Langlands decomposition of $H$ as a $\Q$-group as a \emph{rational Langlands decomposition}. We remark that in our definition of a rational Langlands decomposition $H=(L\times A)\ltimes U$, we will always take the subgroups $L$ and $A= S^\circ$ to be defined over $\Q$; this is different from the terminology used for instance in \cite{MR2189882} which only requires that $L$ and $A$ be conjugate to groups defined over $\Q$.  \index{rational Langlands decomp. and discussion of non-conventional terminology}
	\end{enumerate}

Given a connected linear algebraic $\Q$-group $\bfH$, let $H=\bfH(\R)$ and write $\bfH(\Z) = H_\Z:= H\cap \Sl(d,\Z)$.
\begin{proposition} \label{LatticeIff}
Let $\bfH$ be a connected linear algebraic $\Q$-group and let $H=\bfH(\R)$.
	\begin{enumerate}
	\item \label{LatticeIff-cocompact}
	$H_\Z$ is a cocompact lattice in~$H$ if and only if some (and hence every) Levi subgroup $\bfF$ of~$\bfH$ defined over $\Q$ is $\Q$-anisotropic. (See \cite[Theorem~4.12, p.~210]{MR1278263}.)
	\item \label{LatticeIff-latt}
	$H_\Z$ is a lattice in~$H$ if and only if the subgroup~$A=S^\circ$ is trivial in some (and hence every) rational Langlands decomposition $H = (L \times A) \ltimes U$. (See \cite[Theorem~4.13, p.~213]{MR1278263}.)
	\end{enumerate}
\end{proposition}

\subsection{Restricted $\K$-roots and parabolic subgroups of algebraic groups}
Let $\bfG$ be a connected semisimple linear algebraic group defined over $\K \in \{\R,\Q\}$ and let $G= \bfG(\R)$.

Let $\bfS$ be a maximal $\K$-split torus in $\bfG$.
Let $A= \bfS(\R)^\circ$ and let $\Phi(A,G)$ denote the set of weights for the adjoint action of $A$ on the Lie algebra~$\lieg$ of~$G$. These are the \emph{$\K$-roots} of~$G$. Each root is a homomorphism from~$A$ to the multiplicative group~$\R^+$.
The $\K$-roots $\Phi(A,G)$ extend to $\K$-characters on $\bfS$. The $\K$-characters form a finitely generated abelian group and, as is standard, we write the group operation on characters and roots additively. Choose an ordering of $\Phi(A,G)$ and let $\Delta$ be the corresponding set of simple roots: a positive root~$\alpha$ is \emph{simple} if it is not the sum of two other positive roots. Let $N$ be the connected unipotent subgroup of~$G$ whose Lie algebra is the sum of all the root spaces corresponding to positive roots. This is a maximal unipotent $\K$-subgroup of~$G$.

A \emph{Borel subgroup} of $\bfG$ is a maximal Zariski connected solvable subgroup. A \emph{parabolic $\K$-subgroup} of $\bfG$ is a $\K$-subgroup $\bfP$ that contains a Borel subgroup. When $\K=\Q$, Borel subgroups of $\bfG$ need not be defined over $\Q$. However, all minimal parabolic $\Q$-subgroups are conjugate over $\bfG(\Q)$. (See \cite[Theorem 15.4.6(ii)]{MR2458469}.)
Every parabolic $\K$-subgroup $\bfP$ is connected. If $G= \bfG(\R)$, we also refer to $P =\bfP(\R)$ as a parabolic $\K$-subgroup of $G$.

We have the following characterization of parabolic $\K$-subgroups of $G$:
Fix a subset~$\Delta_0$ of the set~$\Delta$ of simple $\K$-roots of~$\Phi(A,G)$ and let\index{paremetrization of parabolics}
	$$ S = \bigcap \nolimits_{\alpha \in \Delta_0} \ker \alpha \subseteq A .$$
Then $S$ is the identity component of a $\K$-split $\K$-torus and the group \begin{equation}\label{eq:parahoop}P_{\Delta_0} := C_G(S) N\end{equation} is a parabolic $\K$-subgroup of~$G$. Moreover, all parabolic $\K$-subgroups of~$G$ arise as in \eqref{eq:parahoop} for some a choice of maximal $\K$-split torus, ordering of the roots, and set of simple roots.

It is clear from the definition of~$S$ that $\dim S = \rank_\K G - |\Delta_0|$.

\begin{remark} \label{MaxSplitInP}
We have $A \subseteq C_G(S) \subseteq P_{\Delta_0}$. Since all maximal connected $\K$-split tori of~$P_{\Delta_0}$ are conjugate via $(P_{\Delta_0})_\K$, this implies that every maximal connected $\K$-split torus of~$P_{\Delta_0}$ is a maximal connected $\K$-split torus of~$G$.
\end{remark}

A {minimal} parabolic $\K$-subgroup is a parabolic $\K$-subgroup that does not contain any other parabolic $\K$-subgroup. These occur by taking $\Delta_0 = \emptyset$ above and, as previously remarked, all minimal parabolic $\K$-subgroups of~$G$ are conjugate via $G_\K=\bfG(\K)$.

The following well-known observation is immediate from the definitions.
\begin{lemma} \label{MinPar}
Let $P$ be a minimal parabolic $\K$-subgroup of~$G$, and let $P = (L \times S^\circ) \ltimes U$ be a Langlands decomposition (as a $\K$-group). Then:
	\begin{enumerate}
	\item \label{MinPar-S}
	$S^\circ$ is the identity component of a maximal $\K$-split torus in~$G$
	and
	\item \label{MinPar-L}
	$L$ is $\K$-anisotropic.
	\end{enumerate}
\end{lemma}

At the other extreme, a parabolic $\K$-subgroup~$Q$ is said to be \emph{maximal} if $Q$ is a proper subgroup of~$G$ which is not contained in any other proper, parabolic $\K$-subgroup of~$G$. In the notation of \eqref{eq:parahoop}
this means that $\dim S = 1$ or, equivalently, that $|\Delta_0| = \rank_\K G - 1$.

\begin{remark}\label{rem:sclevi}
Suppose $Q = LA N$ is a Langlands decomposition (as a $\K$-group) of a parabolic $\K$-subgroup of~$G$. It is well known that for $G$ simply connected as an algebraic group, the semisimple subgroup $[L,L]$ is simply connected as an algebraic group \cite[Exercise 8.4.6(6), p.~149]{MR2458469}.
\end{remark}

\begin{definition} \label{ExpandingHoroDefn}
Let $\{b^s\}$ be a one-parameter subgroup of~$G$. The corresponding \emph{expanding horospherical subgroup} of~$\{b^s\}$ is
	$$ \mathcal U^+(b^s) := \{\, u \in G \mid \text{$b^{s} u b^{-s} \to \1$ as $s \to -\infty$} \,\} .$$
This is a connected closed, unipotent subgroup of~$G$.
\end{definition}

 The following elementary observation can be proven by taking $\{b^s\}$ to be in the interior of the positive Weyl chamber of~$A$. As above $\K \in \{\R, \Q\}$.

\begin{proposition} \label{UIsHoro}
Let $P = (L \times A) \ltimes U$ be the Langlands decomposition of a parabolic $\K$-subgroup of~$G$. Then $U$ is the expanding horospherical subgroup of some  $\R$-diagonalizable one-parameter subgroup  $\{b^s\}$ of~$A$.
\end{proposition}

In the following results, we consider a parabolic $\Q$-subgroup $Q$ of a $\Q$-group $G$ with rational Langlands decomposition $Q = LA U$.
The group $L$ acts by conjugation on $U$. The action preserves the commutator subgroup $[U,U]$ and thus induces an action of $L$ and its subgroups on $U/[U,U]$. Identifying the abelian group $U/[U,U]$ with some $\R^n$, this action is a linear representation.
The following result is probably well known (see related results such as \cite[Theorem~2]{MR1047327} and \cites{MR1213978,MR1219660}) but we could not find a reference.

\begin{lemma} \label{LNoCent}
Let $\bfG$ be an almost $\Q$-simple $\Q$-group with $\rank_\R(\bfG)\ge 2$ and $\rank_\Q(\bfG)\ge 1$.
Let $Q$ be a maximal parabolic $\Q$-subgroup of~$G=\bfG(\R)^\circ$ with rational Langlands decomposition $Q = LA U$.
Let $L^\dagger$ be the product of all noncompact almost simple factors of~$L$ and the maximal connected $\R$-split torus in the center of~$L$.

Then the induced action of $L^\dagger$ on $U/[U,U]$ does not contain the trivial representation; that is, the only fixed point for the action of $L^\dagger$ on $U/[U,U]$ is the identity coset.
\end{lemma}

\begin{proof}
We make a number of preliminary observations.
Fix the rational Langlands decomposition $Q = LA U$. Since $Q$ is a maximal parabolic $\Q$-subgroup, we have $\dim(A)=1$.
Fix a minimal parabolic $\Q$-subgroup  $P_\Q\subset Q$ and a minimal parabolic $\R$-subgroup $P_\R\subset P_\Q$.  Fix maximal connected $\Q$-split torus $S$ in $P_\Q$  and a maximal connected $\R$-split torus $T$ in $P_\R$ with $$A\subset S\subset T.$$
Let $\Delta_\Q$ be a collection of simple $\Q$-roots for $\Phi(S,G)$ determined by $P_\Q$
and let $\Delta_\R$ be a collection of simple $\R$-roots for $\Phi(T,G)$ determined by $P_\R$.  Note that every element of $\Delta_\Q$ is the restriction of an element of $\Delta_\R$ to $S$.

Let $\lieu$ be the Lie algebra of $U$. The exponential map $\lieu\to U$ gives an identification between $\lieu/[\lieu,\lieu]$ and $U/[U,U]$.
Since $Q$ is a maximal parabolic $\Q$-subgroup, there is a simple $\Q$-root $\alpha\in \Delta_\Q$ such that any $\Q$-root space~$\lieg^\beta$ is contained in~$\lieu$ if and only if $\beta$ is a non-negative integer combination of elements of $\Delta_\Q$ and the coefficient of $\alpha$ is positive. The map $\lieu\to \lieu/[\lieu,\lieu]$ is injective on the $\Q$-root space $\lieg^\beta$ if and only if the coefficient of $\alpha$ in $\beta$ is $1$; the kernel of $\lieu\to \lieu/[\lieu,\lieu]$ is spanned by all $\Q$-root spaces $\lieg^\beta$ such that the coefficient of $\alpha$ in $\beta$ is at least $2$.

We observe that $T = A\cdot (T\cap L) = A\cdot (T\cap L^\dagger)$ and thus $T\cap L^\dagger $ has codimension-1 in $T$.
There are at most 2 (proportional) positive real root ${\beta_0}, 2{\beta_0} \in \Phi(T,G)$ for which the group $T\cap L^\dagger $ acts trivially by conjugation on the real root spaces $\lieg^{\beta_0}$ and $\lieg^{2{\beta_0}}$. Since the action on all other root spaces is by scalar multiplication, if there is $x\in \lieu/[\lieu,\lieu]$ which is fixed under conjugation by $T\cap L^\dagger $, we conclude that $x\in \lieg^{\beta_0} \mod [\lieu,\lieu]$.

To finish the proof of the lemma, we consider separately the cases that $Q$ is or is not maximal as a parabolic $\R$-subgroup.

\setcounter{CASE}{0}
\begin{CASE}
Suppose $Q$ is a maximal parabolic $\R$-subgroup.
\end{CASE}
As we assume $Q$ is a maximal parabolic $\R$-subgroup, there is a simple $\R$-root~$\alpha \in \Delta_\R$ such that a real root space~$\lieg^\phi$ is in~$\lieu$ if and only if $\phi$ is a non-negative integer combination of elements of $\Delta_\R$ and the coefficient of $\alpha $ is positive. Additionally, for every simple root $\beta\in \Delta_\R\sm\{\alpha\}$, the restriction of $\beta$ to $A$ vanishes.

Note that $G$ may have rank-1 factors as an $\R$-group. In particular, the Dynkin diagram for $\Phi(T,G)$ may have isolated nodes. We claim that the node associated to $\alpha $ is not isolated in the Dynkin diagram; in particular, $\alpha $ is not a root associated to a rank-1 real factor. Indeed, if $\alpha$ were isolated then we would have $\lieu = \lieg^\alpha$ or $\lieu= \lieg^\alpha\oplus \lieg^{2\alpha}$. Since $\lieu$ is the expanding horospherical subgroup for some 1-parameter subgroup, it is well-known that the subalgebra $\lieh$ generated by $ \lieg^\alpha$ and $ \lieg^{-\alpha}$ is an ideal. On the other hand, $U$ is defined over $\Q$ whence the analytic subgroup $H$ tangent to $\lieh$ is the connected component of a $\Q$-subgroup. (A standard $\Q$-rank-1 subgroup in the terminology of \cref{StandardSubgrpDefn} below.) This contradicts the assumption that $\bfG$ is almost $\Q$-simple.

Now let $x\in \lieu/[\lieu,\lieu]$ be fixed by conjugation for $L^\dagger$. As discussed above, there is an $\R$-root $\beta_0$ such that $x\in \lieg^{\beta_0} \mod [\lieu,\lieu]$. It follows from the preceding paragraph that there is a simple real root $\beta \neq \alpha$ such that either $\beta_0+\beta$ or $\beta_0-\beta$ is a root.
For every non-zero $x \in \lieg^{\beta_0}$, this implies that either $[x,\lieg^\beta] \neq \{0\}$ or $[x,\lieg^{-\beta}] \neq \{0\}$ (c.f.\ \cite[Lemma~7.75, p.~477]{MR1920389}).
Both $\lieg^{\beta_0+\beta}$ and~$\lieg^{\beta_0-\beta}$ inject into $\lieu/[\lieu,\lieu]$.
Since $\lieg^{\pm\beta} $ are elements of the lie algebra $ \liel^\dagger$ of $L^\dagger$, this implies the only element of $\lieu/[\lieu,\lieu]$ fixed under conjugation by $L^\dagger$ is $0$ as claimed.

 \begin{CASE} \label{LNoCentPf-NotMax}
Suppose $Q$ is not maximal as a parabolic $\R$-subgroup.
\end{CASE}
Let $Z$ be the central torus of~$LA $. Then $Z $ is a $\Q$-subgroup of $G$ and $Z= A Z_a$ where $Z_a$ is the center of $L$ and hence is an anisotropic $\Q$-torus. Since $\dim(A) =1$, $Z_a$ has codimension~$1$ in~$Z$.

We work in the complexified Lie algebra~$\lieg_\C$.
Let $\Delta_\C$ be the set of simple $\C$-roots, and let $\rho = \sum_{\alpha \in \Delta_\C} \alpha$ be the sum of these simple roots.  After a maximal torus has been fixed, there is a natural action of the Galois group $\mathrm{Gal}(\bar{\Q}/\Q)$ on the set of all $\C$-roots since the Lie algebra splits over $\bar{\Q}$. The set~$\Delta_\C$ is usually not invariant under this action, but all possible choices of simple roots are conjugate under the Weyl group. Therefore, composing the action of each element of the Galois group with an appropriate element of the Weyl group yields an action of $\mathrm{Gal}(\bar{\Q}/\Q)$ that leaves $\Delta_\C$ invariant. This is usually called the $*$-action of $\mathrm{Gal}(\bar{\Q}/\Q)$ \cite[\S 2.3, p.~39]{MR0224710}.

For any $\sigma \in \mathrm{Gal}(\bar{\Q}/\Q)$ and any $\alpha \in \Delta_\C$, applying $\sigma$ to~$\alpha$ by the usual action of the Galois group will usually have a different result than by the $*$-action. However, since $Z$ is a $\Q$-torus, both results have the same restriction to~$Z$ {\cite[Proposition~6.7, p.~107]{MR0207712}.} Since $\rho$ is clearly invariant under the $*$-action, we conclude that the restriction $\rho|_Z$ is invariant under the usual action of the Galois group; in particular, $\rho$ defines a $\Q$-character on $Z$.
This $\Q$-character must vanish on the $\Q$-anisotropic torus~$Z_a$. On the other hand, the restriction $\rho|_Z$ of $\rho$ to $Z$ is nontrivial, because the restriction of each element of~$\Delta_\C$ to~$A $ is either a simple $\Q$-root or~$0$ and not all restrictions are~$0$ \cite[Proposition~6.8, p.~107]{MR0207712}. Since $Z_a$ has codimension~$1$, we conclude that $Z_a = \bigl( \ker(\rho|_Z) \bigr)$.

Recall $T$ is the maximal $\R$-split torus in~$G$. Let $r = \dim T\cap Z$  and let $\alpha_1,\ldots,\alpha_r\in \Delta_\R$ be the simple real roots that are nontrivial on~$T\cap Z$. From the description of parabolic subgroups in \eqref{eq:parahoop}, $\dim (T\cap Z)$ is the same as the number of such simple real roots. From the preceding paragraph and the fact that the simple real roots are precisely the nonzero restrictions of elements of~$\Delta_\C$ to a maximal $\R$-split torus \cite[Proposition~6.8, p.~107]{MR0207712}, there are positive integers $k_1,\ldots,k_r$ such that $T \cap Z_a = \ker \bigl( \sum_{i=1}^r k_i \alpha_i |_{T\cap Z} \bigr)$. Since we assume that $Q$ is not maximal as a parabolic $\R$-subgroup, we have $r \ge 2$.
Since the roots $\{\alpha_1, \dots, \alpha_r\}$ are linearly independent on $T\cap Z$, they are also linearly independent with the restriction of $\rho$ to $T\cap Z$. In particular, the restriction of each~$\alpha_i$ to $(T \cap Z_a)^\circ$ is nontrivial.

We claim that if a real root space~$\lieg^\phi$ injects into $\lieu/[\lieu,\lieu]$ then there is some~$\alpha_i$ such that $\phi|_{T\cap Z} = \alpha_i|_{T\cap Z}$. Indeed, if we enumerate $\Delta_\R$ as $\Delta_\R= \{\alpha_1, \dots, \alpha_r, \alpha_{r+1},\dots, \alpha_\ell\}$ then the restrictions of each $\alpha_{r+1}, \dots, \alpha_\ell$ to $T\cap Z$ vanishes. A real root space~$\lieg^\phi$ injects into $\lieu/[\lieu,\lieu]$ if and only if the root $\phi$ is of the form $$\phi = \sum_{i=1}^r c_i \alpha_i +\sum_{i=r+1}^\ell c_i \alpha_i$$
where $c_i\ge 0$, $c_i\in \{0,1\}$ for all $1\le i\le r$, and $c_i=1$ for exactly one $1\le i\le r$. The claim then follows.
By the preceding paragraph, the restriction of~$\phi$ to $(T \cap Z_a)^\circ$ is then nontrivial.
Since $(T \cap Z_a)^\circ \subseteq L^\dagger$, we conclude that the centralizer of~$L^\dagger$ in $\lieu/[\lieu,\lieu]$ is trivial, as desired.
\end{proof}

\begin{corollary} \label{ExpandersGenerate}
Assume that $Q$ is a maximal parabolic $\Q$-subgroup of~$G$ with rational Langlands decomposition $Q = LS^\circ U$.
Then there is a finite collection~$\mathcal{C}$ of one-parameter subgroups of~$L$, such that
	\begin{enumerate}
	\item \label{ExpandersGenerate-inA}
	each $\{a^t\} \in \mathcal{C}$ is contained in a $\Q$-anisotropic $\Q$-torus~$\wtd A$ of~$L$ with $a^1 \in \hat \Gamma\!_{\wtd A}$, 
	and
	\item $U/[U,U]$ is generated by the associated expanding horospherical subgroups $$\{\, \mathcal{U}^+(a_{\split}^t) \mid \{a^t\} \in \mathcal{C} \,\}.$$
	\end{enumerate}
\end{corollary}

\begin{proof}
Let $L^\dagger$ be as in \cref{LNoCent}.
Since the center of~$L$ is $\Q$-anisotropic, there is a maximal connected $\R$-split torus $A$ of~$L$ that is contained in a $\Q$-anisotropic torus~$\wtd A$ in~$L$.
 See \cite[p.~211]{MR0492074} or \cite[Theorem 2.13]{MR0302822}.
Since $L_\Q$ is Zariski dense in~$L$ \cite[\S13.7 and Theorem~18.2(ii), pp.~29 and 218]{MR1102012}, there are finitely many $L_\Q$-conjugates $A_1,\ldots,A_k$ of~$A$ that generate a dense subgroup of $L^\dagger$. Each $A_i$ is contained in a $\Q$-anisotropic torus~$\wtd A_i$ of~$L$.
Since each $A_i$ is $\Q$-anisotropic, each $\wtd A_i/(\wtd A_i)_\Z $ and thus each $\wtd A_i/\hat \Gamma\!_{\wtd A_i}$
 is compact (see \cref{LatticeIff}\pref{LatticeIff-cocompact}) and thus identified with $\R^q/\Z^q$ for some~$q$. Then each $\wtd A_i$ is generated by finitely many one-parameter subgroups $\{a_{i,1}^t\}, \ldots, \{a_{i,p}^t\}$, such that $a_{i,j}^1 \in \Gamma\!_{\wtd A}$ for all~$i,j$. Then the $\R$-diagonalizable parts $\{(a_{i,j}^t)_{\split}\}$ generate~$L^\dagger$.
 Since $L^\dagger$ is reductive, the induced representation on $U/[U,U]$ is totally reducible.
From \cref{LNoCent}  no element of $U/[U,U]$ is centralized by every $\{(a_{i,j}^t)_{\split}\}$. Thus if we assume that the collection $\{a_{i,j}^t\}$ is closed under inverses, this implies that $U/[U,U]$ is generated by the associated expanding horospherical subgroups, as desired.
\end{proof}

\begin{remark}
In many cases, it is possible to take the collection~$\mathcal{C}$ in \cref{ExpandersGenerate} to be a singleton. However, there are cases where this is not possible because the weight~$0$ occurs in the representation of~$L^\dagger$ on $U/[U,U]$. This happens for instance for one of the maximal parabolic subgroups of the $\Q$-split group $\Sp(4, \R)$ of type $C_2\simeq B_2$.
\end{remark}

\subsection{Iwasawa decomposition, restricted roots, and parabolic subgroups}
 \label{RealParabSect}\index{Cartan involution for general connected Lie}\index{Iwasawa decomp. for general connected Lie}
Fix $G$ to be a connected semisimple Lie group with Lie algebra $\lieg$.  We do not assume $G$ has finite center.
Choose a Cartan involution~$\theta$ of~$\lieg$; this induces a global Cartan involution of $G$ which we also denote by $\theta$. We decompose $\lieg= \liek\oplus \liep$ where $\liek$ and $\liep$ are, respectively, the $+1$ and $-1$ eigenspaces of $\theta$. Let $\liea$ be a maximal abelian subspace of $\liep$, and let $\lieg = \liek \oplus \liea\oplus \lien$ be the associated Iwasawa decomposition. If $K,A$ and $N$ are the analytic subgroups corresponding to $\liek, \liea,$ and $ \lien$, respectively, then $G = KAN$ is the corresponding \emph{Iwasawa decomposition.} Here $A$ is a maximal connected abelian $\ad$-$\R$-diagonalizable subgroup, and $N$ is a maximal $\ad$-unipotent subgroup of $G$ normalized by~$A$. We call $A$ a \emph{split Cartan subgroup} of~$G$.
 Furthermore, $K$ and~$A$ are $\theta$-invariant; more precisely, $K$~is the set of fixed points of~$\theta$, and $\theta$~inverts every element of~$A$. The subgroup $K$~contains the center of $G$ and is a maximal compact subgroup of~$G$ if (and only if) the center of $G$ is finite. All choices of $\liea$ and $A$ above of conjugate by $K$.  See \cite[Theorem 6.31]{MR1920389} and \cite[Theorem 6.46]{MR1920389} for details.

Write $M = C_K(A)$ for the centralizer of~$A$ in~$K$. The subgroup $P = MAN$ is a \emph{minimal parabolic subgroup} of~$G$. The subgroup~$P$ and the decomposition $P=MAN$ depends on the choice of $\theta$, $A$, and~$N$, but all minimal parabolic subgroups are conjugate in~$G$. Any subgroup of~$G$ that contains a minimal parabolic subgroup is said to be a \emph{parabolic subgroup}.

Write $\Phi(A,G)$ for the set of weights for the adjoint action of $A$ on~$\lieg$. These are the \emph{real restricted roots} of~$G$ (relative to the choice of $A$). We write the group operation on characters of $A$ additively.
We say two real roots $\beta,\hat \beta\in \Phi(A,G)$ are \emph{positively proportional} if there is some $c>0$ with
$$\hat\beta = c\beta.$$
Note that $c$ takes values only in $\{\frac 1 2, 1, 2\}$ and that $c = 1$ unless the root system~$\Phi(A,G)$ contains a component of type $BC_\ell$.
A \emph{coarse root} is an equivalence class of roots under this equivalence relation.
Write $\hat \Phi(A,G)$ for the collection of coarse roots.

For $[\beta] \in \hat \Phi(A,G)$, we let $\lieg^{[\beta]} = \sum_{\alpha \in [\beta]} \lieg^\alpha$. This is the Lie algebra of a connected unipotent subgroup $U^{[\beta]}$ of~$G$, which is called the \emph{coarse root group} corresponding to~$[\beta]$.
We remark that every parabolic subgroup that contains the minimal parabolic subgroup $P = C_G(A) N$ is saturated by coarse root groups of~$A$.

\begin{definition}\label{def:resCoD}\
\begin{enumerate}
\item A \emph{parabolic} Lie subalgebra of~$\lieg$ is the Lie algebra of a parabolic subgroup.
\item The \emph{resonant codimension}, $\bar r (\lieq)$, of a parabolic Lie subalgebra~$\lieq$ is defined to be the cardinality of the set $$\{\,[\beta] \in \hat \Phi(A,G) \mid \lieg^{[\beta]}\not\subset \lieq\,\}.$$
\item The \emph{minimal resonant codimension} of $\lieg$ or $G$, denoted by $r(\lieg)$ or $r(G)$, is defined to be the minimal value of the resonant codimension $\bar r (\lieq)$ of $\lieq$ as $\lieq$ varies over all (maximal) proper parabolic subalgebras of~$\lieg$.
\end{enumerate}
\end{definition}

\begin{remark} \label{r(G)=v(split)}
Recall that $v(G)$ is the minimal dimension of a homogeneous space $G/H$, where $H$ ranges over all proper, closed subgroups of~$G$. We have $r(G) = v(G')$, where $G'$ is any maximal connected, $\R$-split, semisimple subgroup of~$G$ by \cite[Theorem.~7.2, p.~117]{MR0207712}.
\end{remark}

\subsection{Standard \texorpdfstring{$\Q$}{Q}-rank-1 subgroups}
We return to the case that $G= \bfG(\R)$ where $\bfG$ is a connected semisimple linear algebraic group defined over $\Q$. Let $\bfS$ be a maximal $\Q$-split torus in $\bfG$ and write $A = \bfS(\R)^\circ$.
For each $\Q$-root $\alpha\in \Phi(A,G)$, such that $\frac{1}{2}\alpha \not\in \Phi(A,G)$, there is a unique connected unipotent subgroup~$U^{[\alpha]}$ with Lie algebra $\lieg^{\alpha}$ or $\lieg^\alpha \oplus \lieg^{2\alpha}$.
Since $-\alpha$ is also a $\Q$-root, we also obtain the subgroup~$U^{[-\alpha]}$.
 The following construction is well known, but there does not seem to any established terminology for referring to this subgroup.

 \begin{definition} \label{StandardSubgrpDefn}
For $\Q$-root $\alpha \in \Phi(A,G)$ such that $\frac{1}{2}\alpha \not\in \Phi(A,G)$,
the subgroup~$H_\alpha$ of $G$ generated by $U^{[\alpha]}$ and~$U^{[-\alpha]}$ is called a \emph{standard $\Q$-rank-1 subgroup}.

If $\bfG$ is simply connected, then there is an almost simple $\Q$-subgroup $\bfH$ of~$\bfG$ with $\rank_\Q(\bfH)= 1$ and $H_\alpha= \bfH(\R)$; for general $G$, we have $H_\alpha = \bfH(\R)^\circ$.
\end{definition}

 To justify the last remark of Definition \ref{StandardSubgrpDefn}, we view $\alpha$ as the restriction of $\alpha \in \Phi(\bfS, \bfG)$; then $U^{[\alpha]}= \bfU^{[\alpha]}(\R)$ where $\bfU^{[\alpha]}$ is the root subgroup in $\bfG$ associated to $\alpha$ and $2\alpha$. Let $\bfH$ be the algebraic group generated by $\bfU^{[\alpha]}$ and $\bfU^{[-\alpha]}$. Then $\bfH$ is defined over $\Q$ and has $\rank_\Q(\bfH)=1$.
Choose an ordering on the roots so that $\alpha$ is simple and let $\bfP$ be the parabolic $\Q$-subgroup that contains $\bfU^{[\alpha]}$, $\bfU^{[-\alpha]}$, and no other root groups associated to negative roots. That is, $\bfP = \bfP_{\Delta_0}$ where $\Delta_0= \{ \alpha\}$. Then $\bfH\subset \bfP$ and we may select a Levi $\Q$-subgroup $\bfF\subset \bfP$ which contains $\bfH$. Let $\bfL =[\bfF,\bfF]$. Then $\bfH\subset \bfL$. Moreover, $\bfL$ is semisimple and decomposes as an (almost) direct sum of almost $\Q$-simple factors; of these $\bfH$ is the unique $\Q$-isotropic, almost $\Q$-simple factor of $\bfL$. By \cref{rem:sclevi} we have that $\bfL$ and hence $\bfH$ are simply connected.

The above also gives the following characterization of standard $\Q$-rank-1 subgroups.

\begin{lemma} \label{Standard=Levi}
If $H$ is a standard $\Q$-rank-1 subgroup of~$G$ then there is a parabolic $\Q$-subgroup~$Q$ of~$G$ with rational Langlands decomposition $Q = (L \times A) \ltimes U$ such that $H$~is (the identity component of) the unique $\Q$-isotropic, almost $\Q$-simple factor of~$L$.
\end{lemma}

\subsection{Standing hypotheses when $\Gamma$ is nonuniform}\label{hyp2}\index{standing hyp and second reductions}
Recall the standing hypotheses and reductions \cref{hyp1}.
When $\Gamma\subset G$ is nonuniform, in the proof of \cref{mainthm} we will assume without loss of generality the following additional standing hypotheses:
\begin{enumerate}
\item {\it $G$ is a connected simply connected Lie group without compact factors and with $\rank_\R(G)\ge 2$;}
\item {\it $\Gamma$ is a nonuniform, irreducible lattice subgroup of $G$ which contains the center $\calZ$ of $G$;}
\item {\it there is a connected, algebraically simply connected, linear algebraic group $\bfF$ defined over $\Q$ which is almost $\Q$-simple, non-commutative, $\Q$-isotropic, and has $\rank_\R(\bfF)\ge 2$ such that $F= \bfF(\R)$ is a connected Lie group and $F_\Z= \bfF(\Z)$ is a lattice in $F$;}
\item {\it there is a continuous, surjective Lie group morphism $\phi\colon G\to F$ and a finite-index subgroup $\hat \Gamma\subset F_\Z$ such that:}
 \begin{enumerate}
	\item $\ker \phi$ is contained in the center $\calZ$ of $G$;
	\item the only torsion elements of $\hat \Gamma$ are central;
	\item $\phi (\Gamma) = \hat \Gamma$.
 \end{enumerate}
\end{enumerate}

To see there is no loss of generality,  let $\calZ\subset G$ denote the center of $G$. As $G$ has no compact factors, $\Gamma$ has finite index in $\Gamma\cdot \calZ$ (see \cite[Chapter IX, Lemma (6.1)]{MR1090825}). Without loss of generality, we may assume $\calZ\subset \Gamma$. Let $\overline G= \Ad(G)=G/\calZ$ and $\overline\Gamma = \Gamma/\calZ$ denote the images of $G$ and $\Gamma$ under the adjoint representation. Margulis's arithmeticity theorem (see \cite{MR1090825} Introduction, Theorem 1', or Chapter IX, Theorem 1.11) guarantees there exist:
\begin{enumerate}
\item a $\Q$-simple, linear algebraic $\Q$-subgroup $\bfF$ of $\bfGL(d)$, and
\item a surjective homomorphism $\overline\phi\colon \bfF(\R)^\circ\to \overline G$ with compact kernel such that $\overline\phi(F_\Z) $ is commensurable with $\overline\Gamma$.
\end{enumerate}
Replacing $\bfF$ with a finite cover, we may assume $\bfF$ is algebraically simply connected \cite[Proposition 2.24(ii)]{MR0315007} whence $\bfF(\R)$ is connected as a Lie group.

Since $\overline\Gamma$ is nonuniform, it follows that $F= \bfF(\R)$ has no compact factors (see \cite[Corollary 5.3.2]{MR3307755}). It follows that the kernel of $\overline \phi$ is discrete and hence contained in the center of $F$.
In particular, $F$ is a covering space of $\overline G $. As $G$ is the universal cover of $\overline G $ it follows that the natural map $G\to \overline G $ factors through the map $\overline \phi \colon F\to \overline G $. In particular, $\overline \phi$ lifts to a surjective moprhism of Lie groups $\phi\colon G\to F$. Moreover we have $\hat \Gamma := \phi(\Gamma)$ is commensurable with $F_\Z$. Passing to finite-index subgroups, we may assume $\bar \Gamma$ is a torsion-free and thus the only torsion elements in $\hat \Gamma$  are central.  

 \subsection{Lifting $\Q$-structures from $F$ to $G$} \label{sec:QonG} We keep the standing hypotheses of \cref{hyp2}. \index{$\Q$-groups in $G$}
In our proof of \cref{mainthm} we will often want to identify subgroups of $G$ associated with $\Q$-subgroups of $F$.
Let $\bfL$ be an algebraic subgroup of $\bfF$ defined over $\Q$ and let $L= \bfL(\R)^\circ \subset F$.
Associated to $L$, there is a unique connected Lie subgroup $\wtd L:= (\phi\inv(L))^\circ \subset G$ with $\phi (\wtd L) = L.$ We will abuse terminology and say that $\wtd L$ is a $\Q$-subgroup of $G$.
In the sequel, we will write, equivalently, $$\Gamma\!_{\wtd L} =\Gamma\!_L:= \Gamma \cap \wtd L.$$
We also write $$\hat \Gamma\!_L =\hat \Gamma\cap L.$$
We will be particularly interested in the case that $\bfL$ is either unipotent, parabolic, or a standard $\Q$-rank-1 subgroup.
 We remark that if $U\subset F$ is a unipotent $\Q$-subgroup then $\wtd U\to U$ and $\Gamma\!_U\to \hat \Gamma\!_U$ are isomorphisms.
 If $P$ is a parabolic $\Q$-subgroup and if $H$ is a standard $\Q$-rank-1 subgroup then the maps $$ \wtd P\to P^\circ, \ \ \Gamma\!_P\to \hat \Gamma\!_P, \ \ \wtd H\to H, \ \ \Gamma\!_H\to \hat \Gamma\!_H$$ have kernel contained in the center of $G$.
We will call $ \wtd U$, $ \wtd P$ and $\wtd H$, respectively, unipotent, parabolic, and standard $\Q$-rank 1 $\Q$-subgroups of $G$.

We will also abuse terminology with the following definitions:   
\begin{enumerate}
 \item Let $G_\Q:= \phi\inv (F_\Q)$ denote the preimage of $F_\Q$ under $\phi\colon G\to F$.
Given a connected $\Q$-subgroup $H$ of $ G$, we write $H_\Q:= G_\Q\cap H$.
\item A 1-parameter subgroup $\{a^t\}$ of $G$ is $\R$-diagonalizable if its image in $F$ (equivalently, its image under the adjoint representation) is $\R$-diagonalizable.
\item A connected subgroup $U$ of $G$ is unipotent if its image in $F$ (equivalently, its image under the adjoint representation) is unipotent.

\end{enumerate}

\subsection{Consequences of standing hypotheses: quasi-isometric properties} \label{secQI}
Let $G$, $\Gamma$, $F=\bfF(\R)$, and $\hat \Gamma$ be as in the standing hypotheses in \cref{hyp2}.

Fix a Cartan involution $\theta$ of $\lieg$ and also write the induced global Cartan involutions on $F$ and $G$ as $\theta$. Let $K$ and $ \wtd K$ denote, respectively, the subgroups of $\theta$-fixed points of $F$ and $G$. We recall that $ \wtd K$ contains the center $\calZ$ of $G$.
Equip $F$ with a right-invariant and left-$K$-invariant metric and equip $G$ with the pulled-back, right-$G$-invariant and left-$ \wtd K$-invariant metric. Let $d_F$ and $d_G$ denote, respectively, the induced distances on $F$ and $G$.  We remark that all right-invariant metrics on $F$ are bi-Lipschitz  equivalent  and thus all right-invariant, bi-$\calZ$-invariant metrics on $G$ are  bi-Lipschitz equivalent.  In particular, all results stated below are   independent of the choice of $\theta$.

\subsubsection{Bi-Lipschitz and quasi-isometric properties of central extensions}\label{BiLipSec}
Using that $G\to F$ and $\Gamma\to \hat \Gamma$ are central extensions, we collect a number of   bi-Lipschitz and quasi-isometric  estimates used throughout the sequel. We note for the remainder of \cref{BiLipSec} that we make no assumptions on the rank of $F$ or arithmeticity of $\hat \Gamma$.

Let $X= K\backslash F = \wtd K\backslash G $ denote the globally symmetric space associated with $G$ and equip $X$ with a right-invariant metric.
Equip $K$ and $ \wtd K$ with their intrinsic  Riemannian metrics.
Recall the Iwasawa decomposition $ \wtd K \times A\times N\to G$,  $(k,a,n)\mapsto kan$ of $G$ introduced in \cref{RealParabSect}.
We begin with the following well-known fact; for completeness, we include a proof adapted directly from \cite[Lemma 3.6.3]{MR2415834}.
\begin{lemma}\label{niceBiLip}
The diffeomorphism $\psi\colon G\to \wtd K\times X$, $\psi\colon kan \mapsto (k, \wtd Kan)$ is uniformly bi-Lipschitz.
\end{lemma}
\begin{proof}
We make three observations. First $AN$ acts isometrically on $G$ and $X$ on the right and also acts isometrically on $ \wtd K \times X $ by acting only in the second coordinate. Second, $\calZ$ acts isometrically on $G$ and on $ \wtd K\times X $ acting only on the first coordinate.
Finally, the map $\psi\colon G\to \wtd K\times X$ is $(AN)$-equivariant and $\calZ$-equivariant.

 The function $g\mapsto C_g$, $$C_g = \max \{\|D_g\psi\| , \|(D_g\psi)\inv\|\},$$ is continuous. By $(AN)$-isometric-equivariance of $\psi$, we have $C_g = C_{gs}$ for all $s\in AN$. Thus $\sup \{C_g: g\in G\} = \sup \{C_k: k\in \wtd K\}$. We also have $C_g = C_{zg}$ for all $z\in \calZ$.  Since $\calZ$ is cocompact in $\wtd K$ we have $\sup \{C_k: k\in \wtd K\}<\infty$.
\end{proof}

We record a number of corollaries of \cref{niceBiLip}.

 As a finitely generated abelian group, equip $\calZ$ with an intrinsic word metric relative to some choice of finite symmetric generating set.
 Recall the center $\calZ$ of $G$ is contained in $ \wtd K$ and thus $\calZ$ acts isometrically on $\wtd K$ with compact quotient.
\cref{niceBiLip}
 implies $ \wtd K$ is quasi-isometrically embedded in $G$ and, in particular, the following.
\begin{corollary}\label{centerisQIG}
$\calZ$ is quasi-isometrically embedded in $G$.
\end{corollary}
We similarly have the following.
\begin{corollary}\label{centerisQI}
 $\calZ$ is quasi-isometrically embedded in $\Gamma$.
\end{corollary}
\begin{proof}
Select a finite symmetric generating set for $\Gamma$ that includes a finite generating set for $\calZ$. Then, with respect to the word-length induced by these generating sets, for all $z\in \calZ$ we have $ \len _\Gamma(z)\le \len_{\calZ}(z)$. On the other hand, there is $C>1$ such that for any $\gamma\in \Gamma$ we have $$d_G(\1, \gamma) \le C \len_\Gamma(\gamma).$$
Since $\calZ$ is quasi-isometrically embedded in $G$, it follows that there are $A,B$ such that $\len_{\calZ}(z)\le A\len_\Gamma(z) + B$ for all $z\in \calZ$.
\end{proof}

Given a connected Lie subgroup $L\subset F$, recall we write $\wtd L= (\phi\inv L)^\circ.$
\begin{corollary}\label{everyliftisQI}
Let $L\subset F$ be a connected Lie subgroup. Then
$(\wtd L,d_G)$ is quasi-isometric to $(L,d_F)\times (\calZ\cap \wtd L)$.
\end{corollary}
\begin{proof}
We have that $(\wtd L, d_G)$ is bi-Lipschitz equivalent to $$\bigl(\wtd L\cap \wtd K\bigr)\times \bigl((\wtd K \cap \wtd L)\backslash \wtd L \bigr)= \bigl(\wtd L\cap \wtd K\bigr)\times \bigl((  K \cap   L)\backslash L \bigr)$$ through the map $\psi$.
Since $K$ is compact, $K\cap L$ is a closed, hence compact subgroup of $L$ and thus $(L,d_F)$ is quasi-isometric to its image $(K\cap L)\backslash L$ in $X$ and $\calZ\cap \wtd L$ is cocompact in $ \wtd K\cap\wtd L$ whence $ \wtd K\cap\wtd L$ is quasi-isometric to $\calZ\cap\wtd L$.
\end{proof}

We will frequently consider the case that $L\subset F$ is a $\R$-split torus or a unipotent subgroup.  In these settings, we have the following special case of the above.
\begin{corollary}\label{simpconisbilip}
Let $L\subset F$ be a connected Lie subgroup such that $\wtd L\cap \wtd K= \1$. Then \begin{enumerate}\item the map $\phi\colon (\wtd L,d_G) \to (L, d_F)$ is uniformly bi-Lipschitz and \item $(\calZ\cdot \wtd L, d_G)$ is quasi-isometric to $\calZ\times (\wtd L, d_G)$. \end{enumerate}
\end{corollary}

Let $A\subset F$ be a choice of a $\theta$-invariant split Cartan subgroup.
 By abuse of notation, we also identify $A$ with a subgroup of $G$ and observe that $\phi\colon G\to F$ restricts to a bi-Lipschitz isomorphism between the two copies of $A$.  %
From \cref{simpconisbilip}, $(\calZ\cdot A,d_G)$ is quasi-isometric to $\calZ\times (A,d_G)$.

Let $K_0$ be a compact subset of $ \wtd K$ such that $\calZ K_0 = K_0 \calZ = \wtd K$.
We generalize the well-known $KAK$ decomposition of $F$ by the following.
\begin{corollary}\label{kaka} \index{generalized kak decomposition}
Given $g\in G$  we may write $$g= k'zak''$$ where $k',k''\in K_0$, $a\in A$, and $z\in \calZ$. Moreover,
there are $A_1, B_1\ge 1$ such that for any $g\in G$, relative to the above decomposition,
\begin{align*}
 \frac{1}{A_1 }d(g,\1) - B_1
 \le
\max\{ d(a,\1), d(z,\1)\}
\le A_1 d(g,\1) + B_1.
\end{align*}
\end{corollary}

\def\cartfoot{{\footnote{See for instance \cite[Theorem 7.3]{MR69830}. More directly, observe that a maximal compact subgroup $K'$ of $H/\calZ$ is contained in a maximal compact subgroup $K$ of $F/\calZ$. The subgroup $K$ determines a Cartan involution $\theta$ on $F/\calZ$. With respect to the Killing form on the Lie algebra $\lieg$, the orthogonal complement to the Lie algebra of $K'$ in the Lie algebra of $H$ is contained the $-1$ eigenspace for $\theta$. Then $\theta$ leaves the Lie algebra of $H$ in $\lieg$ invariant.} }}

\begin{remark}[Generalized KAK decomposition]\label{kakaReductive}
Consider a connected reductive  $\R$-subgroup $L\subset F$ and let $H = \wtd {L} := (\phi \inv (L))^\circ$.  By abuse of terminology, we say that $H$ is a connected reductive   $\R$-subgroup of $G$.  Write $\calZ(H)$ for the center of $H$ and let $Z= \calZ(H)^\circ$.

The image $\phi(Z)$ is the connected component of  an $\R$-torus in $F$.  As discussed in \cref{sec:tori}, we have that $$Z= \calZ( H)^\circ= Z_{\comp}\cdot Z_{\split}$$ where $\phi (Z_{\comp})$ is $\R$-anisotropic and $\phi (Z_{\split})$ is  $\R$-split.
Write $ H'=[ H, H]$ for the derived group of $ H$. Then $$H = H' \cdot Z = H' \cdot Z_\comp \cdot Z_\split.$$

 We may select the Cartan involution $\theta$ on the Lie algebra $\lieg$ of $F$ and $G$ so that $H$ is $\theta$-invariant.\cartfoot
Let $A\subset G$ be a $\theta$-invariant split Cartan subgroup and let $\wtd K\subset G$ be the subgroup of $\theta$-fixed points in $G$.
Write $A_H:= A\cap H$, and $ \wtd K_H:= \wtd K\cap H$. We collect a number of observations. \begin{enumerate}
\item $A_H$ is a maximal, connected, abelian subgroup of $\ad$-$\R$-diagonalizable elements in $H$; thus $Z_{\split} \subset A_H\subset A$.
\item $\phi (\wtd K_H)$ is compact and, moreover, is a maximal compact subgroup of $\phi (H)$. It follows that $Z_{\comp}\subset \wtd K_H$.
\item $\calZ\cap H= \calZ\cap (H' \cdot Z_{\comp})= \calZ\cap \wtd K_H$ and $(\calZ\cap H)$ is cocompact in $\wtd K_H$.
\end{enumerate}

There is a compact $K_{H,0}\subset \wtd K_H$ such that $(\calZ\cap H)\cdot K_{H,0} = \wtd K_H$.
We may then decompose $$ H= K_{H,0}\cdot (\calZ\cap H) \cdot A_H \cdot K_{H,0}.$$ In particular, we may write any $h\in H$ as $$h = k'zak''$$ where $k',k''\in K_{H,0}$, $a\in A_H$, and $z\in \calZ\cap H$. Moreover, as in \cref{kaka}, we similarly have that there are $A_1, B_1\ge 1$ such that for all $h\in H$, relative to the above decomposition,
$$
 \frac{1}{A_1 }d(h,\1) - B_1
 \le
\max\{ d(a,\1), d(z,\1)\}
\le A_1 d(h,\1) + B_1. $$
\end{remark}

 \subsubsection{Quasi-isometric embedding of arithmetic subgroups} We return to the standing hypotheses in \cref{hyp2}.
The subgroup $\hat \Gamma$ of $F$ is finitely generated and thus may be equipped with a choice of word-length function $\len_{\hat \Gamma}$ induced by a choice of finite symmetric generating set; the word-length function induces a word-metric    on $\hat \Gamma$.
We write $(\hat \Gamma, \len_{\hat \Gamma})$ for this metric space.
 As a discrete subset of $F$, $\hat \Gamma$ also inherits a metric $d_F$ by restricting the right-invariant Reimannian distance $d_F$ on $F$ to $\hat \Gamma$. We similarly obtain a word-length $\len_\Gamma$ and associated word-metric $\Gamma$ and a metric $d_G$ on $\Gamma $ by restricting the Riemannian distance $d_G$ on $G$ to $\Gamma$.

For general groups $F$, the metrics $\len_{\hat \Gamma}$ and $d_F$ on $\hat \Gamma$ need not be comparable; indeed this holds for $ \Sl(2,\Z)$ in $\Sl(2,\R)$.   However, as we have assumed $\rank_\R(F)\ge 2$ and $\hat \Gamma$ irreducible, we obtain the following.
\begin{theorem}[\cite{MR1828742}]\label{LMR}
For $F$ and $\Gamma$ as in the standing hypotheses in \cref{hyp2} the word-metric and Riemannian metric on $\hat \Gamma$ are quasi-isometric: there are $A_0,B_0$ such that for all $\gamma,\gamma'\in \hat \Gamma $,
$$\frac 1 {A_0 }d_F(\gamma,\gamma') -B_0 \le \len_{\hat \Gamma}(\gamma\inv\gamma')\le {A_0 }d_F(\gamma,\gamma') +B_0.$$
\end{theorem}

\subsubsection{Quasi-isometric embedding of lattices in the topological universal cover}
We assert that the conclusion of \cref{LMR} also holds for the lattice subgroup $\Gamma$   of $G$ under the standing hypotheses in \cref{hyp2}.
Although this is probably well known, we do not know a reference and thus include a proof.

Recall $\hat \Gamma$ is a lattice subgroup of $F$ and $\Gamma=\phi\inv(\hat \Gamma)$ where $\phi\colon G\to F$ is the covering map.  We have $\ker \phi \subset \calZ$ and recall we also assumed that $\calZ\subset \Gamma$.

In the following, all discrete groups are equipped with a word-length metric.
 \begin{lemma}\label{liftQI}\label{cor:QIlift}
For $G$ and $\Gamma$ as in the standing hypotheses of \cref{hyp2} the following hold:
\begin{enumerate}
\item \label{point11} $\Gamma$ is quasi-isometrically embedded in $G$.
\item \label{point22} $\Gamma$ is quasi-isometric to $\calZ\times \widehat \Gamma$.
\end{enumerate}
\end{lemma}
\begin{proof}
We recall the bi-Lipschitz map $\psi$ in \cref{niceBiLip}.
 Given $\gamma\in  \Gamma$, write   $\hat\gamma = \phi( \gamma)$ for the image of $\gamma$ in $ \hat \Gamma$.
 Take a finite generating set $S_{\calZ}$ for the center $\calZ$ and a finite generating set $S$ of $\Gamma$ containing $S_{\calZ}$.  Let $\hat S = \phi(S)$ be the image of $S$ in $\hat \Gamma$; then $\hat S$ is a finite generating set of $\hat \Gamma$.

Fix a precompact fundamental domain $K_0$ for the $\calZ$-action on $\wtd K$.
Given $\gamma\in \Gamma$, write  $\gamma = z_\gamma k_\gamma a_\gamma n_\gamma$ where $z_\gamma\in \calZ, k_\gamma\in K_0,  a_\gamma\in Z,$ and $n_\gamma\in N$.  We naturally identify the cosets $\wtd K\cdot \gamma =\wtd K a_\gamma n_\gamma $ and $ K\cdot \hat \gamma$  in $X= \wtd K \bs G= K\bs F$.
Note $(\wtd K, d_K)$ is quasi-isometric to $(\calZ,\len_{S_\calZ})$ and $X$ is quasi-isometric to $(F, d_F)$.
Through the bi-Lipschitz map $\psi$ in \cref{niceBiLip}, the  map $\gamma\mapsto (z_\gamma, \hat \gamma)$ is a quasi-isometry  $$(\Gamma, d_G)\to (\calZ,\len_{{S_\calZ}}) \times (\hat \Gamma, d_F).$$ By \cref{LMR}, this induces a quasi-isometry
 $$(\Gamma, d_G)\to (\calZ,\len_{S_\calZ}) \times (\hat \Gamma, \len_{\hat S}).$$

It remains to compare $(\Gamma, d_G)$ and $(\Gamma, \len_S)$.
Take $C_1:= \max_{s \in S} d_G(s,1)$.   Then for all $\gamma\in \Gamma$, $$\frac{1}{C_1}d_G(\gamma, 1) \leq \len_S(\gamma)$$
To complete the proof of the lemma, we claim  there is   $\hat C>1$ such that
\begin{equation} \label{finaleq} \len_S(\gamma)   \le \hat   C d_G(\gamma, 1)  + \hat   C .\end{equation}

To establish \eqref{finaleq}, fix $\gamma \in \Gamma$ and write $\hat \gamma =  \widehat{\gamma_1 } \hat{\gamma_2} \dots \hat{\gamma_n} $ where $\gamma_i \in S$, $\hat{ \gamma_i }= \phi(\gamma_i)\in \hat S$, and $n = \len_{\hat{S}}(\hat{\gamma})$.
By \cref{niceBiLip}, there is $C_2>0$ such that \begin{equation}\label{eq999}d_F(\hat \gamma,1) \le C_2 d_G(\gamma,1).\end{equation}
We have  $\gamma = z \gamma_1\gamma_2 \dots \gamma_n$ for some $z \in \calZ$ and $$d_G(z, 1) \leq d_G(\gamma, 1) + d_G(1, z\inv  \gamma)\le d_G(\gamma, 1) + C_1n = d_G(\gamma, 1) + C_1\len_{\hat S}(\hat \gamma).$$
Applying \cref{LMR} and \eqref{eq999},
there is $C_3>1$ such that
$$d_G(z, 1) \leq C_3d_G(\gamma, 1) + C_3.$$
By \cref{centerisQIG},   $\len_{S_\calZ}(z) \leq C_4 d_G(z,1) + C_4$, and thus
\begin{equation}\label{ecu31}
\len_{S_\calZ}(z) \leq C_5d_G(\gamma, 1) + C_5
\end{equation}
for some constants $C_4, C_5 > 1$. Finally, by \cref{LMR},  \eqref{eq999}, and \eqref{ecu31}, we have
\begin{align*}\len_S(\gamma) &= \len_S(z \gamma_1\gamma_2 \dots \gamma_n)  \le \len_S(z) + \len _S(\gamma_1\gamma_2 \dots \gamma_n)\\&
\leq \len_{S_\calZ}(z) + \len_{\hat S}(\hat{\gamma})  \leq  \len_{S_\calZ}(z) + C_6d_F(\hat {\gamma}, 1) + C_6 \\& \leq C_7d_G(\gamma, 1) + C_7\end{align*} for some constants $C_6, C_7>1$ establishing \eqref{finaleq}.
\end{proof}

\subsubsection{Quasi-isometric properties of parabolic and unipotent $\Q$-subgroups}
Using \cref{cor:QIlift} we formulate additional quasi-isometric relationships that will be used in the sequel.  We assume the standing hypothesis in
\cref{hyp2}.

\begin{lemma} \label{GammaCap1}
Let $L$ be a connected reductive $\Q$-subgroup of $F$ and let $U$ be a connected unipotent $\Q$-subgroup of $F$ normalized by $L$.  Suppose that  $L\cap U = \{\1\}$.   Then the following hold:
\begin{enumerate}
\item \label{QI33312222} the map $(L,d_F)\times (U,d_F)\to (LU, d_F)$, $ (g,u) \mapsto gu$, is quasi-isometric relative to $d_F$.
\item \label{QI3331} $L_\Z \ltimes U_\Z$ has finite index in $(LU)_\Z$ and the map $L_\Z\times U_\Z\to (LU)_\Z$, $ (g,u) \mapsto gu$, is quasi-isometric relative to $\len_{F_\Z}$.
\item \label{QI3331555} the map $(\wtd L,d_G)\times (\wtd U,d_G)\to (\wtd {L U}, d_G)$, $ (g,u) \mapsto gu$, is quasi-isometric relative to $d_F$.
\item \label{QI4441} the map $\Gamma\!_L\times \Gamma\!_U\to \Gamma\!_{LU}$, $ (g,u) \mapsto gu$, is quasi-isometric relative to $\len_{\Gamma}$.
\end{enumerate}
\end{lemma}
Indeed \eqref{QI33312222} follows using that the Riemannian distance from $g\in F$ to the identity is comparable to $\max\{\log \|g\|, \log \|g\inv\|\}$ and that $\log\|gu\|$ is comparable to $\max\{\log \|g\|, \log \|u\|\}.$   \eqref{QI3331} then follows   from \cref{LMR}.
For \eqref{QI3331555},  we have $\wtd {LU} =\wtd L U$ (identifying $U\subset F$ with its connected preimage in $G$). By \cref{everyliftisQI}  we have $\wtd LU$ is quasi-isometric to $(\calZ\cap( \wtd LU))\times ( LU, d_F)$ and  $( \wtd {L}, d_G)$ is quasi-isometric to $(\calZ\cap( \wtd L))\times ( {L}, d_F)$.
 The result follows from \eqref{QI3331} and \cref{simpconisbilip}.        \eqref{QI4441} then follows from  \cref{cor:QIlift}.

We have the following special case of the above setup.
\begin{lemma} \label{GammaCap}
Let $P$ be a parabolic $\Q$-subgroup of $F$ with rational Langlands decomposition $P = (L \times S^\circ) \ltimes U.$ Then
\begin{enumerate}
\item $S^\circ_\Z$ is trivial whence $P_\Z= (LU)_\Z$;
\item \label{QI333} the map $L_\Z\times U_\Z\to P_\Z$, $ (g,u) \mapsto gu$, is quasi-isometric relative to $\len_{F_\Z}$.
\item \label{QI444} the map $\Gamma\!_L\times \Gamma\!_U\to \Gamma\!_P$, $ (g,u) \mapsto gu$, is quasi-isometric relative to $\len_{\Gamma}$.
\end{enumerate}
\end{lemma}

\begin{lemma}[{\cite[(3.14)]{MR1828742}}] \label{GenByMaxParab}
Let $\Lambda$ be a unipotent subgroup of ~$F_\Z$. Then there are finitely many maximal parabolic $\Q$-subgroups $Q_1,Q_2,\ldots,Q_r$ of~$G$ such that if we let $U_i = \Rad_u Q_i$, then $\hat \Gamma\!_{U_1}, \ldots, \hat \Gamma\!_{U_r}$ $\len_{\hat \Gamma}$-quasi-isometrically boundedly generate a subgroup of ~$\hat \Gamma$ that contains a finite-index subgroup of~$\Lambda$.
\end{lemma}

\begin{proof}
The set of unipotent elements of $\Sl(q,\R)$ is Zariski closed, so the Zariski closure~$\overline{\Lambda}$ of~$\Lambda$ is a unipotent $\Q$-subgroup of~$F$. Therefore $\overline{\Lambda}$~is contained in the unipotent radical~$U$ of some minimal parabolic $\Q$-subgroup~$P$ \cite[Prop.~3.1]{MR0294349}.
Choose a maximal $\Q$-split torus~$S$ in~$P$, let $\Delta$ be the set of simple roots in $\Phi(S,F)$, with respect to the ordering in which roots occurring in the Lie algebra of~$U$ are positive, and let $U^{[\alpha]}$ be the (coarse) root group corresponding to the root~$\alpha$. Then $U = \bigl\langle\, U^{[\alpha]} \mid \alpha \in \Delta \,\bigr\rangle$. Since $U$ is nilpotent (and simply connected), this implies that $\bigl\langle\, \hat \Gamma\!_{U^{[\alpha]}} \mid \alpha \in \Delta \,\bigr\rangle$ generates a finite-index subgroup of~$\hat \Gamma\!_U$ by, for example, \cite[Prop.~2.5]{MR0507234}. Furthermore, for each $\alpha \in \Delta$, the subgroup $U^{[\alpha]}$ is contained in the unipotent radical of the maximal parabolic subgroup~$P_{\Delta_\alpha}$, where $\Delta_\alpha = \Delta \sm \{\alpha\}$.
By induction on the derived length of~$U$, it is easy to see that if a collection of subgroups generates a finite-index subgroup of the nilpotent group~$U$, then they $\len_{\hat \Gamma}$-quasi-isometrically boundedly generate a finite-index subgroup of~$U$.
\end{proof}

\begin{corollary}
Let $U\subset F$ be a unipotent $\Q$-subgroup.
Then there are finitely many maximal parabolic $\Q$-subgroups $Q_1,Q_2,\ldots,Q_r$ of~$G$ such that if we let $U_i = \Rad_u Q_i$, then $\hat \Gamma\!_{U_1}, \ldots, \hat \Gamma\!_{U_r}$ $\len_{\hat \Gamma}$-quasi-isometrically boundedly generate a subgroup of ~$\hat \Gamma$ that contains a finite-index subgroup of~$\hat \Gamma\!_U$.
\end{corollary}

\subsection{Consequences of standing hypotheses: quasi-isometric bounded generation}
Let $\phi\colon G\to F$ and $\Gamma\to \hat \Gamma$ be as in the standing hypotheses in \cref{hyp2}. In particular, $F$ is a $\Q$-simple algebraic $\Q$-group, $\hat \Gamma$ is commensurable with $F_\Z$, and $\Gamma\to \hat \Gamma$.
Given an algebraic $\Q$-group $\bfL\subset \bfF$ recall we let $\wtd L:= (\phi\inv(L))^\circ$ denote the connected subgroup in the lift of $L=L(\R)^\circ$ to $G$. We then write $\Gamma\!_L= \Gamma\cap \wtd L$.

We summarize the main definitions from \cite{MR4074402}.
\begin{definition}[{\cite[Definition 1.1]{MR4074402}}] \label{QIBddGenDefnF}
\qibdddef
\end{definition}

Returning to the lattice subgroup $\Gamma$ in $G$, we may lift each $L\in \mathcal{L}(F,\hat \Gamma)$ to $G$ and ask that $\Gamma$ be efficiently generated by elements from $\Gamma\!_L$.
\begin{definition}[{\cite[Definition 6.1]{MR4074402}}] \label{QIBddGenDefn}
We say that $\Gamma$ is \emph{quasi-isometrically boundedly generated by standard $\Q$-rank-1 subgroups} if there are constants $r = r(G,\Gamma) \in \N$ and $C = C(G,\Gamma) >1$, a finite subset $\Gamma\!_0 = \Gamma\!_0(G,\Gamma)$ of~$\Gamma$, and a finite collection $\mathcal{L} = \mathcal{L}(G,\Gamma)$ of $\Q$-subgroups of $G$ such that
	
	\begin{enumerate}
	\item each $L \in \mathcal{L}$ is a standard $\Q$-rank-1 subgroup of~$G$,
	\item \label{QIGenDefns-o-product}
	 every element~$\gamma$ of~$\Gamma$ can be written in the form $\gamma = s_1 s_2 \cdots s_r$ where either
	\begin{enumerate}
	\item $s_i \in \Gamma\!_0$,
	or
	\item $s_i \in \Gamma\!_L$ for some $L \in \mathcal{L}$ and $\len_\Gamma( s_i) \le C \len_\Gamma(\gamma)$.
	\end{enumerate}
	\end{enumerate}
\end{definition}

The main result of \cite{MR4074402} guarantees quasi-isometrically bounded generation by $\Q$-rank-1 subgroups for all lattice subgroups under consideration.
\begin{theorem} [{\cite[Theorem 1.2]{MR4074402}}]\label{QIBddGenByRank1}
\laterdef\Davethm{Every arithmetic subgroup $\hat \Gamma$ of a $\Q$-isotropic, almost $\Q$-simple $\Q$-group $F$ is quasi-isometrically boundedly generated by standard $\Q$-rank-1 subgroups.}
\Davethm
\end{theorem}

As discussed in \cite[\S 6]{MR4074402}, the efficient bounded generation property above lifts to lattices subgroups of Lie groups with infinite center such as $\Gamma$ in $G$ satisfying the hypotheses of \cref{hyp2}.
\begin{theorem} [{\cite[Corollary 6.2]{MR4074402}}] \label{QIBddGenByRank1lift}
Let $\Gamma$ be a non-uniform, irreducible lattice in a connected semisimple Lie group $G$ without compact factors. Then $\Gamma$ is quasi-isometrically boundedly generated by standard $\Q$-rank-1 subgroups.
\end{theorem}

\section{Averaging measures and control of mass near \texorpdfstring{$\infty$}{infinity}}
For this section, take $G$ to be a connected semisimple Lie group which acts continuously on a locally compact, second countable metric space~$X$. The results stated in this section all consider the case $X= G/\Gamma$; in the sequel, we will consider the case that $X$ is a fiber bundle over $G/\Gamma$ with compact fibers.  Let $\Gamma$ be a lattice subgroup of $G$.    In \cref{sec:trans} we will assume $\Gamma$ is irreducible   and in \cref{sec:nondiv} below we will further assume $G=\bfG(\R)$ is a $\Q$-algebraic group and $\Gamma$ is commensurable with $\bfG(\Z)$.

While many results discussed below hold for more general Lie groups (and locally compact topological groups) $G$ and for more general discrete subgroups $\Gamma$, we only formulate results in the settings which will be used in the sequel.
\subsection{Notation}\label{sec:not}
 Let $H$ be a closed subgroup of $G$ and let $E$ be a precompact set of positive Haar measure in $H$. 
Write $m_H$ for a choice of Haar measure on~$H$ and $m_E$ for the Haar measure on~$H$ normalized so that $m_E(E) = 1$. Given any Borel probability measure~$\mu$ on~$X$ we write
	$$ E \ast \mu
	:= \frac{1}{m_H(E)} \int_E g_* \mu \, dm_H(g)
	= \int_E g_* \mu \, dm_E(g)
	.$$

We will often take $E$ to range over the sets in a \folner sequence, such as in the following construction:
Assume that $\{a^t\}$ is an $\ad$-$\R$-diagonalizable, one-parameter subgroup of~$G$ and that $U$ is a connected $\ad$-unipotent subgroup of~$G$ that is normalized by~$\{a^t\}$.
Fix a norm $\| \cdot \|$ on Lie algebra~$\lieu$ of~$U$.
For each $\kappa  > 0$ and $t > 0$, set
	\begin{align*}
	U_\kappa (t) &:= \exp_\lieu\bigl(\{\, Z\in \lieu \mid \|Z\|\le e^{\kappa  t}\,\}\bigr) \\
	\intertext{and}
	F_\kappa (t) &:= \{\, a^s \mid 0\le s\le t \,\} \cdot U_\kappa (t)
	. \end{align*}
If $\kappa $ is sufficiently large and if $\{t_n\}$ is a sequence tending to infinity, then $\{F_\kappa (t_n)\}$ is a (left) \Folner sequence in the solvable group $\{\,a^t \mid t\in \R\,\}\ltimes N$; for instance, we may take
	$$\kappa  > 10 \max \bigl\{ \| {\Ad_\lieu(a)}\|, \| {\Ad_\lieu(a\inv)}\| \bigr\}.$$

\subsection{Control of mass near \texorpdfstring{$\infty$}{infinity}}
We often impose strong control on the distribution of a measure at infinity. We have the following definition.
\begin{definition}[{\cite[Section 3.2]{BFH-SLnZ}}]\label{def:masscusp}
Let $(X,d)$ be a complete, second countable, metric space and let $\mu$ be a finite Borel measure on $X$. We say that $\mu$ has \emph{exponentially small mass at $\infty$} if there is $\tau_\mu>0$ such that for all $0<\tau<\tau_\mu$,
\begin{equation}\label{eq:SmallMassIneq}\int_X e^{\tau d(x_0, x)} \ d \mu(x) <\infty\end{equation}
for some choice of base point $x_0\in X$.
We say that a collection $\mathcal M$ of probability measures on $X$ has \emph{uniformly exponentially small mass at $\infty$} if there is $\tau_0>0$ such that for all $0<\tau<\tau_0$,
\begin{equation}\label{eq:SmallMassIneq2}\sup_{\mu\in \mathcal M} \left\{\int e^{\tau d(x_0, x)} \ d \mu(x)\right\} <\infty.\end{equation}
\end{definition}
We remark that \eqref{eq:SmallMassIneq2} holds if and only if
there is $C>1$ such that for every sufficiently large $\ell$,
$$ \sup_{\mu\in \mathcal M} \mu \bigl( \{ x \mid d(x_0, x) \ge \ell \}\bigr) <C e^{-\tau \ell}.$$
It is well known that the Haar measure on $G/\Gamma$ has exponentially small mass at $\infty$; for discussion in the case of arithmetic lattices $\Gamma$, see \eqref{HaarSmallCusps} below.

\subsection{Translates of horospherical orbits and mass at \texorpdfstring{$\infty$}{infinity}}\label{sec:trans}
In this section we assume that the lattice $\Gamma < G$ is irreducible.

Given an expanding horospherical subgroup $U = \mathcal U^+(b^s)$ for some $\R$-diagonalizable, one-parameter subgroup~$\{b^s\}$ in $G$, it is well known that the normalized Haar measure on a precompact open subset of~$U$ equidistributes to the Haar measure~$m$ on $G/\Gamma$ when translated by~$b^s$.
It is useful to have a more quantitative version of this well-known fact that includes a rate of equidistribution.
This ``quantitative equidistribution'' follows from exponential mixing, and is stated in \cite[Proposition 2.4.8]{MR1359098} and \cite[Theorem 1.1]{MR2867926} for flows on $\Sl(n,\R)/\Sl(n,\Z)$ and without complete optimization of all constants.  Exponential mixing follows from each simple factor of $G$ having a spectral gap in the representation on $L^2_0(G/\Gamma)$; see \cite[Section 2.4]{MR1359098}.  The required spectral gap is proven for irreducible non-uniform lattices by Kleinbock and Margulis in \cite[Theorem 1.12]{MR1719827}.  The result for all irreducible lattices follows easily from arguments in that paper and Clozel's proof of Property $\tau$ \cite[Theorem 3.1]{MR1953260}.  Clozel actually proves a uniform spectral gap over all congruence latices, deducing a spectral gap for a single lattice, congruence or not follows easily as in the proof of \cite[Theorem 1.12]{MR1719827}.

 We state an adaptation, referred to as the ``effective equidistribution property'' in {\cite[Theorem~2.5]{MR4125693}}, which holds on general $G/\Gamma$ and optimizes all constants.   To state this result, given $x \in G/\Gamma$ write $\inj x$ for the injectivity radius of $G/\Gamma$ at~$x$.

\begin{theorem}[{\cite[Theorem~2.5]{MR4125693}}]\label{ExpMixing}
Let $U$ be the expanding horospherical subgroup of a one-parameter $\R$-diagonalizable subgroup $\{b^s\}$ of~$G$. There exist constants $C, \lambda > 0$, and $k \in \N$, such that, for all $x \in G/\Gamma$, all $s > \max\{C, C \log (1/{\inj x}) \}$, all $\varphi \in C^\infty_2(G/\Gamma)$, and all $f \in C^\infty_{\text{cpct}}(U)$ supported in the unit ball centered at the identity, we have
	\begin{align*}
	\Biggl| \int_U f(u) \, \varphi(b^s u x) \, d m_U(u)
		&- \int_U f \, d m_U
		 \int_{G/\Gamma} \varphi \, d m_{G/\Gamma} \Biggr| \\
		&\notag
		< C
		\cdot \max \bigl( \| \varphi \|_{C^1}, \| \varphi \|_{k,2} \bigr)
		\cdot \| f \|_{C^k}
		\cdot e^{-\lambda s}
		.\end{align*}
\end{theorem}

Although it is not used directly in the paper, we state the following consequence of the above.
\begin{corollary} \label{MixingSmallCusps}
Suppose that
	$U$ is the expanding horospherical subgroup of a one-parameter, $\R$-diagonalizable subgroup $\{b^s\}$ of $G$
	and that
	$\mu$ is a $U$-invariant Borel probability measure on $G/\Gamma$ with exponentially small mass at $\infty$.

Then $(b^{s})_* \, \mu \to m_{G/\Gamma}$ as $s \to +\infty$ and the family $\{(b^s)_* \, \mu\}_{s \ge 0}$ has uniformly exponentially small mass at $\infty$.
\end{corollary}

In the proof of \cref{MixingSmallCusps} and later in \cref{OrbitClassification}, we will need the following construction.
 Consider the function $\hat h\colon G/\Gamma \to [0,\infty) $ given by $$\hat h\colon g\Gamma \mapsto d(g \Gamma, \1 \Gamma).$$ The function $\hat h$ is proper and 1-Lipschitz but need not be differentiable. At times it will be convenient to replace $\hat h$ with a $C^\infty$ approximation. Convolving $\hat h$ with a non-negative, symmetric, $C^\infty$ function on $G$ with sufficiently small support, we obtain a proper,    uniformly Lipschitz, $C^\infty$ function $h\colon G/\Gamma \to [0,\infty)$ with
\begin{equation} \label{heightDefn} \text{$d(g \Gamma, \1 \Gamma) -1 \le h(g\Gamma) \le d(g \Gamma, \1 \Gamma) + 1$ for all $g \in G$} .\end{equation}

\begin{proof}[Proof of \cref{MixingSmallCusps}]
The convergence to Haar is clear.

Let $x_0= \1\Gamma \in G/\Gamma$. To obtain uniform estimates on the mass at $\infty$, it is sufficient to find $\tau'>0$ such that for all sufficiently large~$\ell$ and all $s>0$ we have
	$$\bigl((b^{s})_* \mu\bigr) \bigl( \{\, x\in G/\Gamma \mid d(x, x_0)\le \ell \,\}\bigr) \ge 1 - e^{-\tau'\ell }.$$

By assumption, there is $\tau>0$ such that for all sufficiently large $\ell>0$,
	$$\mu \bigl( \{\, x\in G/\Gamma \mid d( x, x_0)\le \ell\,\}\bigr) \ge 1 - e^{-\tau \ell}. $$
Also, there is $\hat \tau>0$ such that for all sufficiently large $\ell>0$,
	$$m_{G/\Gamma} \bigl( \{\, x\in G/\Gamma \mid d( x, x_0)\le \ell\,\}\bigr) \ge 1 - e^{-\hat \tau \ell}. $$
We may assume $\{b^t\}$ is unit-speed so that $b^t= \exp (tY)$ where $\|Y\|= 1$.

Fix any~$\ell$. For $0\le s\le \ell / 2$ we have the uniform estimate
\begin{align*}
\bigl( (b^{s})_* {\mu} \bigr) &\bigl( \{\, x\in G/\Gamma \mid d(x, x_0)\le \ell\,\}\bigr)
\\&= {\mu}\bigl( \{\, x\in G/\Gamma \mid d(b^s x, x_0)\le \ell\,\}\bigr)
\\&\ge {\mu}\bigl( \{\, x\in G/\Gamma \mid d( x, x_0)\le \ell / 2 \,\}\bigr)
\\&\ge 1 - e^{-\tau \ell / 2}.
\end{align*}
It remains to consider all $s \ge \ell / 2$.
Fix constants $C$ and~$\lambda$ satisfying \Cref{ExpMixing}.

Given $t>0$, let
	$$B_{t} := \{\,x\in G/\Gamma \mid C \log( 1/\inj x)\ge t \,\} . $$
We have $x\mapsto - \log (\inj x)$ is bounded above by a linear function of $d(x,x_0)$; in particular there is $c_1 > 0$ such that for sufficiently large $t$
	\begin{align*}\mu \bigl( B_{t}\bigr)
	&\le \mu \bigl( \{\, x\in G/\Gamma \mid d(x, x_0) \ge c_1 t \,\}\bigr)
	\le e^{-\tau c_1 t}. \end{align*}

Fix a $C^\infty$ function $\varphi \colon \R\to [0,1]$ with $\varphi(y)= 1 $ for all $y\le 0$ and $\varphi(y) = 0$ for all $y\ge 1$. Recall the function $h$ satisfying \eqref{heightDefn}. Given $\ell>0$, set $\varphi_\ell\colon G/\Gamma\to [0,1]$ to be
	$$\varphi_\ell(x) = \varphi \left( h(x) -\ell + 2 \right) .$$
Then $\varphi_\ell(x)=1$ when $d(x,x_0)\le \ell-3$ and $\varphi_\ell(x)=0$ when $d(x,x_0)\ge \ell.$
For every $k$, we clearly have that $\| \varphi_\ell\|_{k,2}$ and $\| \varphi_\ell\|_{k}$ are bounded uniformly in~$\ell$.
Fix some nonnegative $f_U \in C^\infty(U)$ supported in the unit ball centered at the identity with $\| f_U \|_1 = 1$.
Then for $s > \ell / 2 > C$ we have
	\begin{align*}
	\bigl((b^{s})_* \mu\bigr) & \bigl( \{\, x\in G/\Gamma \mid d(x, \1\Gamma)\le \ell \,\}\bigr)\\
	& \ge ((b^{s})_* \mu )(\varphi_\ell)
	\\&= \int_{G/\Gamma} \varphi_\ell(b^s x) \, d \mu(x)
	\\&= \int_{G/\Gamma} \varphi_\ell(b^s x) \ d \left[\int_U f(u) \, u_*\mu \, d m_U(u)\right](x)
		\\&\ge \int_{(G/\Gamma) \sm B_{\ell/2}} \varphi_\ell(b^s x) \ d \left[\int_U f(u) \, u_*\mu \, d m_U(u)\right](x)
	\\&= \int_{(G/\Gamma) \sm B_{\ell/2}} \int_U f(u) \, \varphi_\ell(b^s ux)\, dm_U(u)
 \, d \mu(x)
	\\&\ge \int_{(G/\Gamma) \sm B_{\ell/2}} \left[ \int_{G/\Gamma} \varphi_\ell \, dm_{G/\Gamma} - \hat C e^{-\lambda s} \right] \, d \mu(x)
	\\&> \bigl(1- e^{-\tau c_1 {\ell}/{2}} \bigr) \bigl(1 - e^{-\hat \tau( \ell-3 )} - \hat C e^{-\lambda {\ell}/{2}} \bigr)
	 \end{align*}
where the second equality follows from the $U$-invariance of~$\mu$ and the next to last inequality uses \Cref{ExpMixing}.
\end{proof}

In the proof of our main results, we only use  the following special case of \cref{MixingSmallCusps}.
We note that this corollary can be proved  using  classical exponential mixing results rather than the more quantitative version in \cref{ExpMixing}.

\begin{corollary} \label{MixingSmallCuspsCom}
Suppose that
	$U$ is the expanding horospherical subgroup of a one-parameter $\R$-diagonalizable subgroup~$\{b^s\}$ of~$G$
	and that
	$\mu$ is a $U$-invariant Borel probability measure on $G/\Gamma$ with compact support.
Then $(b^{s})_* \, \mu$ converges to the Haar measure on $G/\Gamma$ as $s \to +\infty$ and the family $\{(b^s)_* \, \mu\}_{s \ge 0}$ has uniformly exponentially small mass at $\infty$.
\end{corollary}

\subsection{Non-divergence of unipotent averages of measures}\label{sec:nondiv}
We emphasize in this section that we assume $G=\bfG(\R)$ is a $\Q$-algebraic group and $\Gamma$ is commensurable with $\bfG(\Z)$.
 In particular, we may view $G\subset \Sl(N,\R)$ for some $N\in \N$.  We equip $\Sl(N,\R)$ with the metric associated with the Cartan involution $X\mapsto -X^T$.  Although $G$ need not by invariant under this Cartan involution, any choice of metric on $G$ is uniformly Lipschitz comparable with the restriction of the ambient metric on $\Sl(N,\R)$ to $G$.  The results and references quoted below are primarily stated for unipotent orbits in $\Sl(N,\R)/\Sl(N,\Z)$; however, using that distance estimates in $\Sl(N,\R)/\Sl(N,\Z)$ give uniformly Lipschitz comparable estimates in $G/\Gamma$ gives analogous results for unipotent orbits in $G/\Gamma$.

A key property of unipotent dynamics first established in \cite{MR0470140} and later made quantitative in \cite{MR530631} is the non-divergence of orbits under unipotent flows.
As we require non-divergence of orbits of certain subsets of higher-dimensional unipotent subgroups, we make the following definition.

 \begin{definition}\label{def:UnipInterval}
Let $U$ be a $k$-dimensional unipotent subgroup of $G$. A basis $\calB=  \{X_1, \dots, X_k\}$ of the  Lie algebra $\lieu$ of $U$ is \emph{regular} (or \emph{triangular}) if for each $1\le i\le k$,
\begin{enumerate}
	\item $\{X_1, \dots, X_i\}$ spans an $i$-dimensional subalgebra  of $\lieu$;
	\item $X_i$ normalizes the subalgebra spanned by $\{X_1, \dots, X_{i-1}\}$.
\end{enumerate}

Fix a regular basis $\calB =  \{X_1, \dots, X_k\}$ and let  $\{u_j^t\} = \{\exp(t X_j)\}$ be the associated one-parameter subgroups of $U$.
Given nonempty, bounded intervals $I_1,\ldots,I_k$ in~$\R$, the product
	$ u_k^{I_k} \, u_{k-1}^{I_{k-1}} \, \cdots \, u_1^{I_1} \, $
	$$ \{\, u_k^{t_k} \mid t_k \in I_k \} \cdot \{\, u_{k-1}^{t_{k-1}} \mid t_{k-1} \in I_{k-1} \} \cdots \{\, u_1^{t_1} \mid t_1 \in I_1 \}$$
is an \emph{interval} (relative to the regular basis $\calB$) in~$U$ where $u_i^{I_i} = \{\, u_i^t \mid t \in I_i \,\}$.

An interval (relative to the basis $\calB$) $\interval {U}= u_k^{I_k} \, u_{k-1}^{I_{k-1}} \, \cdots \, u_1^{I_1}$ is \emph{centered} if $\1\in \interval U$; that is, a $k$-dimensional interval $\interval U= u_k^{I_k} \, u_{k-1}^{I_{k-1}} \, \cdots \, u_1^{I_1}$ is {centered} if $0\in I_i$ for each $1\le i\le k$.
\end{definition}

\begin{remark} \label{rem:interval} (Centered) $k$-dimensional intervals $\interval{U} = u_k^{I_k} \, u_{k-1}^{I_{k-1}} \, \cdots \, u_1^{I_1}$ in~$U$ have a number of nice properties relative to dynamics and choice of base point. In particular, we have the following:
	\begin{enumerate}
	\item For any $g \in N_G(U)$, the conjugate $g \interval{U} g^{-1}$ is also a (centered) $k$-dimensional interval in~$U$ (possibly relative  to a different basis).
	\item For any $u$ in the subgroup generated by~$\interval{U} $, the translates $u \interval{U} $ and $\interval{U} u$ are also $k$-dimensional intervals in~$U$; moreover, they are centered if $u\inv\in \interval{U}$.
	\end{enumerate}
\end{remark}

Recall that a collection $\mathcal{M}$ of Borel probability measures on a locally compact space~$X$ is \emph{uniformly tight} if for every $\delta > 0$ there exists a compact subset~$\mathcal{K}$ of~$X$, such that $\eta(\mathcal{K}) > 1-\delta$ for every $\eta \in \mathcal{M}$. We recall that every uniformly tight family of Borel probability measures is pre-compact in the space of Borel probability measures equipped with the weak-$*$ topology; in particular, any uniformly tight sequence of Borel probability measures has subsequential limit points in the space of Borel probability measures.
We frequently pass to compact extensions and make use of the following observation: if $f \colon X_1 \to X_2$ is a proper continuous map, then a collection $\mathcal{M}$ of Borel probability measures on~$X_1$ is uniformly tight if and only if $f_* \mathcal{M}$ is a uniformly tight collection of measures on~$X_2$.

Results of the following type are well known, but we do not know of a reference for this precise fact.

\begin{lemma}\label{lem:unipotenttight}
Let $U\subset G$ be a unipotent subgroup and let $\{F_j\}$ a sequence of centered intervals (relative to a regular basis $\calB$) in~$U$.
Let $\mu$ be a Borel probability measure on $G/\Gamma$. Then the family $\{F_j \ast \mu\}$ is uniformly tight.\end{lemma}

\begin{proof}
By induction on~$\dim(U)$, it suffices to show that if $\mathcal{M}$ is a uniformly tight family of Borel probability measures on $G/\Gamma$, $\{u^t\}$~is a one-parameter unipotent subgroup of~$G$, and $\{I_n\}$ is a sequence of centered intervals in~$\R$, then the family $$\{\,u^{I_n} \ast \mu \mid n \in \N, \mu \in \mathcal{M} \,\}$$ is uniformly tight.

Given $\epsilon > 0$, there is some compact subset~$\mathcal{K}_1$ of $G/\Gamma$, such that $\mu(\mathcal{K}_1) > 1 - \epsilon$ for all $\mu \in \mathcal{M}$. It follows from the uniform estimate for the quantitative nondivergence of unipotent flows established by Kleinbock and Margulis in \cite[Theorem~5.2]{MR1652916} (see also improvements in \cite{MR2434296}) that there is a compact subset~$\mathcal{K}_2$, such that for every $x \in \mathcal{K}_1$, and every centered interval~$I$ in~$\R$, we have
	$$ m_I\bigl( \{\ t \in I \mid u^t x \in \mathcal{K}_2 \,\} \bigr) > 1 - \epsilon .$$
Then $(u^{I_n} \ast \mu)(\mathcal{K}_2) > 1 - 2 \epsilon$ for every $n \in \N$ and $\mu \in \mathcal{M}$.
\end{proof}

We also have the following more quantitative version of tightness which follows from {\cite[Lemma 3.3]{BFH-SLnZ}} by induction.

\begin{lemma}[{\cite[Lemma~3.3]{BFH-SLnZ}}] \label{UnipSmallCusp}
Suppose
	\begin{enumerate}
	\item $\{\mu_n\}$ is a sequence of probability measures on~$G/\Gamma$ with uniformly exponentially small mass at $\infty$ and
	\item $\{\interval{U}_n\}$ is a sequence of centered intervals (relative to a fixed regular $\calB$) in a unipotent subgroup~$U$ of~$G$.
	\end{enumerate}
Then the family of measures $\{\interval{U}_n * \mu_n\}$ has uniformly exponentially small mass at $\infty$.
\end{lemma}

\subsection{Equidistribution and invariance}
We recall that unipotent flows on $G/\Gamma$ exhibit a number of additional surprising properties as exhibited by Ratner's theorems on unipotent flows \cite{MR1403920}. In the sequel, we will only use a handful of these important results.

The first result  asserts that a measure on $G/\Gamma$ that is invariant under certain subgroups are automatically invariant under larger related subgroups.

\begin{proposition}[{\cite[Proposition 2.1]{MR1135878}}]\label{thm:opproot}
Let $A$ be a split Cartan subgroup of~$G$, let $\beta \in \Phi(A,G)$, and let $\mu$ be a Borel probability measure on $G/\Gamma$. If $\mu$ is invariant under both~$A$ and the coarse root group~$U^{[\beta]}$ then $\mu$ is also invariant under the coarse root group~$U^{[-\beta]}$.
\end{proposition}

For the second result, we recall that averaging an $A$-invariant Borel probability measure along a \Folner sequence in a subgroup that centralizes (or normalizes)~$A$, any weak-$*$ limit point is an $A$-invariant Borel probability measure. Ratner's equidistribution theorem implies that, on the homogeneous space $G/\Gamma$, {$A$-invariance} is preserved when passing to limits of averages by unipotent subgroups normalized by $A$.

 \begin{proposition}[{\cite[Proposition 6.2(b)]{BFH} and \cref{lem:unipotenttight}}]\label{thm:averaginghomo2}
Let $A$ be a split Cartan subgroup of~$G$,
let $\mu$ be an $A$-invariant Borel probability measure on $G/\Gamma$,
let $U$ be a unipotent subgroup that is normalized by $A$,
and
let $\{F_j\}$ be a \Folner sequence of centered intervals in~$U$.
Then
\begin{enumerate}
\item the family $\{F_j\ast \mu\}$ is uniformly tight, and
\item \label{jlkjlkj} every weak-$*$ subsequential limit of $\{F_j\ast \mu\}$ is $A$-invariant.
\end{enumerate}
\end{proposition}
 We remark that conclusion \pref{jlkjlkj} employs the equidistribution theorem of Ratner and its extension by Shah; see \cite{MR1291701,MR1106945}.

\section{Siegel sets, cocycles, and induced actions}
\label{SiegelSetSect}
We return to the standing hypotheses on $G,\Gamma, F$, and $\hat \Gamma\subset F_\Z$ as in \cref{hyp2}.
\subsection{Siegel sets and Siegel fundamental sets}
We recall the well-known construction of fundamental sets for arithmetic lattices by the use of Siegel sets; see for instance \cite[\S12--\S13, pp.~85--94]{MR0244260}, \cite[III.1--III.2]{MR2189882}, or \cite[II.4]{MR2590897} for discussion and background.
Fix a maximal $\Q$-split torus $\bfS$ in $\bfF$ and fix a minimal parabolic $\Q$-subgroup~$\bfP$ of~$\bfF$  containing $\bfS$.
Let $P= \bfP(\R)$ and $S= \bfS(\R)$.   All such choices of $P$ and $S$ are conjugate via $F_\Q$.
 Fix a Cartan involution $\theta$ of $F= \bfF(\R)$.  There exists a unique conjugate (over $N_P:= \Rad_u(P)$) of $S^\circ$ in $P$ which is $\theta$-invariant; denote this conjugate by $A_P= A_{\bfP, \theta}$. In most applications in later sections, we will   select $\theta$ so that $S$ is $\theta$-invariant and thus $A_{P}= S^\circ$.  However, for the constructions in this section, this is not strictly required.

We may  decompose $P = (L \times A_P) \ltimes N_P$ where $N_P=\Rad_u(P)$ and $(L\times A_P)$ are conjugate by $N_P$ to $\Q$-subgroups appearing in a choice of rational Langlands decomposition of $P$.  If $S=\bfS(\R)$ is assumed $\theta$-invariant, then $(L \times A_P) \ltimes N_P$ is a rational Langlands decomposition (i.e.\ a decomposition into $\Q$-subgroups) as defined in \cref{sec:levilang}.
From \cref{MinPar}, the minimality of~$P$ implies that any conjugate of $L$ which is defined over $\Q$ will be $\Q$-anisotropic.
Recall that $\Phi(A_P,F)$ denotes the set of weights of~$F$ with respect to~$A_P$. Fix an order on $\Phi(A_P,F)$ determined by $N_P$; that is, select the order so that root spaces corresponding to positive roots $\Phi^+(A_P,F)$ are contained in the Lie algebra of~$N_P$.  Recall that each $\beta\in \Phi(A_P,F)$ is a homomorphism $\beta\colon A_P\to \R^+$.

For $t\in \R^+$, write
	$$A_t := \{\, a\in A_P \mid \text{$\beta(a) < t$ for all $\beta\in \Phi^+(A_P,F)$} \,\} .$$
Let $K = \{ g\in F: \theta(g) = g\} \subset F $ be the maximal compact subgroup determined by $\theta$.
A \emph{Siegel set} (see \cite[Definition 12.3]{MR0244260} or \cite[Definition 4.1]{MR2590897}) in $F$ relative to the choice of $P$ and~$\theta$ is a set of the form
$$\Sieg:= K A_t \, \mathcal C$$
where $t \in \R^+$ and $\mathcal C$~is any compact subset of $L^\circ \ltimes N_P$. We remark that we may alternatively define Siegel sets where $\mathcal C$ is taken to be any compact subset of $P^\circ$; in this case, one may take $A_t= A_1$ to be a (negative) fundamental chamber. See discussion in \cite[p.\ 68]{MR2590897} or \cite[12.4]{MR0244260}.

Fix a choice of minimal parabolic $\Q$-subgroup~$P$ and Cartan involution~$\theta$ as above.
  Siegel sets in $F$ constructed relative to $P$ and $\theta$ have the following key properties:
\begin{proposition}[{\cite[Theorems~13.1 and~15.4, pp.~90 and~103]{MR0244260}}] \label{sieg}
\ 
\begin{enumerate}
\item\label{sieg-cover} (Covering property)
There exists a Siegel set~$\Sieg$ and a finite set $C\subset F_\Q$ such that
	$F = \Sieg \, C \, \hat \Gamma $.

\item  \label{sieg-separation}(Separation property)
For any Siegel sets~$\Sieg_1$ and~$\Sieg_2$ in $F$, any $g\in F_\Q$, and any finite set $C\subset F_\Q$, the set
$$ \left\{\gamma\in F_\Z : (\Sieg_1\, g)\cap \left( \Sieg_2 \, C \, \gamma\right)\neq \emptyset \right\}$$
is finite.
\end{enumerate}
\end{proposition}
 If necessary, we may take a larger compact set $\mathcal C\subset L^\circ \ltimes N_P$ in the construction of the Siegel set~$\Sieg$ so that $\mathcal C$ is the closure of its interior $\mathrm{int}(\mathcal C)$. Then there is a finite set $C\subset F_\Q$ such that
$D= \Sieg \cdot C$ has the following properties:
\begin{enumerate}
\item $F= \mathrm{int}(D)\cdot \hat \Gamma$
\item $\{ \gamma\in F_\Z : D\cap D\cdot \gamma\neq \emptyset\}$ is finite.
\end{enumerate}
We call such a set $D$ a \emph{Siegel fundamental set} in $F$.
Note that any such $D$ contains a fundamental domain for the right action of $\hat \Gamma $ on $F$.

\begin{remark}\label{AWBcantdokindergardenalgebra}If $P'$ is another minimal parabolic $\Q$-subgroup, there is $g\in \bfF(\Q)$ such that $P' = gPg\inv$.   Write $g = klan$ where $k\in K$, $l\in L$, $a\in A_P$, and $n\in N_P$.  We have that $A_{P'}'=kA_Pk\inv$ is the unique $\theta$-invariant subgroup of  $P'$ conjugate to a maximal $\Q$-split torus in $P'$ and $N_{P'}:= kN_Pk\inv $ is the unipotent radical of $P'$.   Write $A'= A_{P'}$ and let $\Phi^+(A', F)$ be determined by $N_{P'}$.  Then $kA_tk\inv = A'_t$.

A standard calculation shows that if $\Sieg:= K A_t \, \mathcal C$ is a Siegel set constructed relative to $P$ and $\theta$  (where $\mathcal C$ is a compact subset of $P^\circ$)
then   $\Sieg \cdot g\inv $ is a Siegel set relative to $P'$ and $\theta$.
Indeed
  \begin{align*}
 \Sieg \cdot g\inv  &= \Sieg \cdot n\inv a\inv l\inv k\inv \\
 &= K A_t   (\calC n\inv l\inv a\inv)k\inv \\
 &= K A_{t}  \calC' k\inv \\
  &= K (k A_{t} k\inv) (k \calC' k\inv) \\
  &= K  A'_{t}  (g \calC'' g\inv)
 \end{align*}
 where $\calC'$ and $\calC''$ are compact in $P^\circ$ whence  $(g \calC'' g\inv)$ is compact in $(P')^\circ$.

Thus, if $P'$ is a minimal parabolic $\Q$-subgroup and if $D'$ is a Siegel fundamental set constructed relative to $P'$ and $\theta$ then there is a Siegel fundamental set $D$ constructed relative to $P$ and $\theta$ with $D'\subset D$.  In particular, the construction of Siegel fundamental sets in $F$ depends only on the choice of Cartan involution $\theta$.
\end{remark}

As before, let $\wtd K\subset G$ be the connected subgroup in $G$ projecting to $K$ under $G\to F$. Recall that $\wtd K$ contains the center $\calZ$ of $G$. Fix a compact subset $ K_0\subset \wtd K$ that projects onto $K$ under the map $G\to F$. Recall that both $A_P$ and $N_P$ are simply connected subgroups of $F$; we abuse notation and identify $A_P$ and $N_P$ with the associated connected subgroups of $G$ via the map $G\to F$. Let $\wtd L$ be the connected subgroup of $G$ projecting to $L^\circ$.
 A \emph{Siegel set} in $G$ relative to the choice of $P$ and~$\theta$ is a set of the form
$$\wtd \Sieg:= K_0 A_t \, \mathcal C$$
where $t \in \R^+$ and $\mathcal C$~is any compact subset of $\wtd L\ltimes N_P$. For any Siegel set $\Sieg$ in $F$, there exists a Siegel set $\wtd \Sieg$ in $G$ whose image under the map $G\to F$ contains $\Sieg$.
We may similarly construct a Siegel set~$\wtd \Sieg$ in $G$ and a finite set $C\subset G_\Q$ such that $\wtd \Sieg\cdot C$ satisfies the covering and   separation properties of \cref{sieg}. Then $$\wtd D= \wtd \Sieg C$$ is a {Siegel fundamental set} in $G$ for the right-action of $\Gamma$.

We note that given a connected reductive $\Q$-subgroup $H$ of $F$, it is possible to construct Siegel sets for $H$ that are contained in and  well-adapted to Siegel fundamental sets for $F$.  See \cite[Theorem~4.1]{MR3803710} for a general formulation of this fact.
In \cref{SubexpForRank1Subgroups}, we will need this fact when we study the action of (the restriction of a cocycle to)  $\Gamma\!_H$ where $H$ is a standard $\Q$-rank-$1$ subgroup of~$F$.
For the special case of $\Q$-rank-1 subgroups, we give an elementary proof of a stronger fact:  Siegel sets in $H$ may be chosen inside Siegel sets for $F$ relative to some choice of $\theta$ and $P$.

\begin{lemma} \label{SiegelForSubgroup}
Let $H$ be a standard $\Q$-rank-$1$ subgroup of~$F$. Then a Cartan involution~$\theta$ of~$F$ and a maximal $\Q$-split torus $S$ in $F$ can be chosen with the following properties:
	\begin{enumerate}
	\item \label{SiegelForSubgroup-theta}
	$H$ is $\theta$-invariant and the restriction $\restrict{\theta}{H}$ is a Cartan involution of~$H$;
	\item \label{SiegelForSubgroup-PSK}
	$P_H = P \cap H$ is a minimal parabolic $\Q$-subgroup of~$H$,
	$S_H = S \cap H$ is a $\restrict{\theta}{H}$-invariant maximal $\Q$-split torus of~$P_H$,
	and
	$K_H = K \cap H$ is the $\restrict{\theta}{H}$-invariant maximal compact subgroup of~$H$;
	\item \label{SiegelForSubgroup-Siegel}\label{Siegel}
	each Siegel set of~$H$ relative to $P_H$ and $\restrict{\theta}{H}$ is contained in some Siegel set of~$F$ relative to~$P$ and~$\theta$.
	\end{enumerate}
Moreover, if $Q$ is the parabolic subgroup in \cref{Standard=Levi} then $S\subset Q$.
\end{lemma}
We caution the reader that  if $Q$ is the parabolic subgroup in \cref{Standard=Levi} we typically cannot take $P\subset Q$.

\begin{proof}[Proof of \cref{SiegelForSubgroup}]
Fix a maximal $\Q$-split torus $S$ in $F$ such that $H= H_\beta$ is generated by the root subgroups $U^{[\pm\beta]}$ for some rational root $\beta\in \Phi(A, F)$ relative to $A= S^\circ$.
Let $S_H := S \cap H$ and $A_H:= S_H^\circ = A\cap H$.
Since $A_H$ is one-dimensional, acting by the Weyl group of $\Phi(A, F)$ we may select an ordering and associated collection of simple roots $\Delta$ for $\Phi(A, F)$ such that $A_H$ is contained in the union of the closed fundamental chambers in $A$,
 \begin{align*}A^+&:=\{a\in A: \text{$ \alpha(a)\ge 1 $ for all $\alpha\in \Delta$}\},\\ A^-&:=\{a\in A: \text{$ \alpha(a)\le 1 $ for all $\alpha\in \Delta$}\}.\end{align*}
 Up to reversing the order, we may assume $\beta\in \Phi(A,F)$ is a positive root.
Let $P$ be the minimal parabolic $\Q$-subgroup~$P$ of~$F$ with rational Langlands decomposition $P = (L \times A ) \ltimes N_P$
where the Lie algebra of $N_P$ consists of roots spaces associated with positive roots given by the ordering associated with the $\Delta$.
Choose a Cartan involution~$\theta$ on $F$ such that $LS$~is $\theta$-invariant.

With the above choices, we verify the conclusions of the lemma.

\pref{SiegelForSubgroup-theta}~The Cartan involution~$\theta$ acts by inversion on every $\R$-split torus it leaves invariant. It follows that $\theta$ interchanges $U^{[\beta]}$ and~$U^{[-\beta]}$ and thus $H= H_\beta$ and $S_H$ are $\theta$-invariant.

\pref{SiegelForSubgroup-PSK}~If $K$ is a maximal compact $\theta$-invariant subgroup of $F$ then, since $H$ is $\theta$-invariant, $K \cap H$ is a maximal compact subgroup of~$H$.

The minimal parabolic $\Q$ subgroup~$P$ contains the root spaces corresponding to all positive roots associated to $\Delta$ and thus contains $U^{[\beta]}$.
We note that $S = S_H \times C_S(H)$. Thus $C_H(S_H)\subset C_H(S)$
and so $C_H(S_H) \subseteq C_G(S) \subseteq P$. Therefore $P$ contains the minimal parabolic subgroup $C_H(S_H) \, U^{[ \beta]}$ of~$H_\beta$.

\pref{SiegelForSubgroup-Siegel}~From our choice of simple roots $\Delta$, for any $t\in \R^+$ the ray $$A_{H,t}:=\{\, a\in A_H \mid \text{$\beta(a) < t$ for all $\beta\in \Phi^+(A_P,F)$} \,\} $$
 is contained in $A_{t'}$ for some $t'$. Since we also have $K_H \subseteq K$ and ${P_H} \subset P$, the conclusion follows.
\end{proof}

\subsection{Coarse geometry and partitions of unity} \label{PartUnitySect}
As before, fix a reference right-$F$-invariant Riemannian metric on $F$ and induced distance function $d$ on $F$. Also assume the metric is bi-invariant by the center of $F$.  We lift the metric to a right-invariant, $\calZ$-bi-invariant metric on $G$ with induced distance $d$. All definitions and results below hold up to bi-Lipschitz equivalence and thus are independent of the choice of $d$.

As above, let $\len_\Gamma \colon \Gamma\to \N$ be the word-length function on $\Gamma$ relative to some fixed symmetric, finite generating set of~$\Gamma$. This induces a word-metric $d_\Gamma$ on $\Gamma$.
Recall from
\cref{cor:QIlift} there is a constant $C > 1$ such that for all $\gamma, \gamma'\in \Gamma$,
\begin{equation}\label{eq:LMR}\frac 1 C \, d(\gamma, \gamma') \le d_\Gamma(\gamma, \gamma') \le C \, d(\gamma, \gamma') .\end{equation}

It is well known that $m_F(\Sieg)<\infty$ for any Siegel set $\Sieg$ in $F$; see for instance \cite[Proposition 4.11, p.~215]{MR1278263}. Moreover, the proof of this fact shows there is some $\tau>0$ such that
	\begin{align} \label{HaarSmallCusps}
	\int_\Sieg e^{\tau d(g,\1) }\ d m_F(g)<\infty
	. \end{align}
This implies that the measure $m_{G/\Gamma}$ on the homogeneous space $F/\hat \Gamma= G/\Gamma$ has exponentially small mass at $\infty$. (See \cref{def:masscusp}).

If $D$ is a Siegel fundamental set in $F$ then the map $D \to F/\hat \Gamma=G/\Gamma$ does not drastically distort distances.  In particular, there  are $C>1$ and $B>0$ such that for all $g\in D$,
$$	d(g,\1) \le C d(g\hat \Gamma, \1 \hat \Gamma) + B .$$
Moreover, if the metric on $F$ is taken left $K$-invariant we can take $C=1$ above. See for example \cite[Theorem C]{MR338456}.
For the reader's convenience we sketch a proof.  One can argue in a single Siegel set.  If $\calC$ is a compact subset of $L^\circ \ltimes N_P$ then, for any $t\ge 1$, the choice of chamber in $A$ ensures the diameter of $a\calC a\inv $ is uniformly bounded over the choice of $a\in A_t$.
Then for $w\in \calC$, $k\in K$, and $a\in A_t$,
$$d(kaw,\1)= d(kawa\inv a,\1)\le \diam K+\diam (a\calC a\inv) +d(a,\1).$$
One then argues (see discussion in \cite[\S III.21]{MR2189882} or \cite[\S5]{MR1906482} especially \cite[Proposition 5.12]{MR1906482}) that  for some $0<t_0<1$ sufficiently small, the partial orbit $x_0 \cdot A_{t_0}$ is a convex subset of $X/\Gamma$ where $x_0=K$ is the origin in $X=K\bs F$.
For a much stronger result see \cite[Theorem~5.7]{MR2066859} and discussion in \cite[III.21.15]{MR2189882}.
Similarly, using that $A_P\subset G$ is chosen $\theta$-invariant and thus geodesic, the bound
\begin{equation}\label{SiegelAlmIsom} d(g,\1) \le C d(g \Gamma, \1 \Gamma) + B. \end{equation}
also holds for all $g$ in any Siegel fundamental set $\wtd D$ in $G$.

For the following construction, we fix a minimal parabolic $\Q$-subgroup~$P$ and Cartan involution~$\theta$ of $G$.
 Let $K$ be the maximal compact subgroup determined by $\theta$ and let $D$ be a Siegel fundamental set in $F$ constructed relative to $P$ and $\theta$. Let $\wtd D$ be a Siegel fundamental set in $G$ relative to $P$ and $\theta$ whose image under $G\to F$ contains $D$.

Let $X$ denote the globally symmetric space $X:= K\bs F= \wtd K\bs G$.
 The locally symmetric space $X/\Gamma$ admits compactification $\overline{X/\Gamma}^{\mathrm{BS}} = {}_\Q\overline{X}^{\mathrm{BS}} \! /\Gamma$, know as the Borel-Serre compactification, as a real-analytic manifold with corners.
 See \cite[\S III.9, pp.~326--338]{MR2189882} 
for details.
In this compactification, neighborhoods of points at infinity are of the form {$(\R_{\ge 0})^{k_1} \times \R^{k_2}$} for some $k_1, k_2$; see \cite[Lemma~III.9.13, p.~362, and Proposition III.9.17, p.~365]{MR2189882}.
 We note that such coordinate neighborhoods of points at infinity in the partial compactification  ${}_\Q\overline{X}^{\mathrm{BS}}$ can be taken to be in the closure in ${}_\Q\overline{X}^{\mathrm{BS}}$ of a single Siegel set in $K\bs F$.
 Since the only torsion elements of $\hat \Gamma$ are contained in $K$, $\hat \Gamma$ acts freely on ${}_\Q\overline{X}^{\mathrm{BS}}$.
In particular, we may form an open cover $\overline{X/\Gamma}^{\mathrm{BS}}$ consisting of sets $\{U_p: p = 1, \dots, \ell\}$ such that every $U_p$ is the injective image of either
 \begin{enumerate}
\item an open set $\wtd U_p \subset X$, or
\item an open set $\wtd U_p$ in the closure of a Siegel set in ${}_\Q\overline{X}^{\mathrm{BS}}$.
 \end{enumerate}

 We may then select a $C^\infty$ partition of unity $\{ \hat \psi_p \mid i=1,\dots, \ell \}$ on $\overline{X/\Gamma}^{\mathrm{BS}}$
 subordinate to the cover $\{U_p\}. $
Let $\rho\colon {}_\Q\overline{X}^{\mathrm{BS}} \to \overline{X/\Gamma}^{\mathrm{BS}}$ be the canonical projection.   For each $1\le i\le p$, the function $\hat\psi_p\circ \rho$ is right-$\Gamma$ invariant and is supported in the open set $\wtd U_p \cdot \Gamma$;  select $\gamma_p$ so that $\wtd U_p\cdot \gamma_p $ meets the image  in $X$ of the Siegel fundamental set $D$ in $F$ fixed above and let $\psi_p = \mathbbm 1_{\wtd U_p \cdot \gamma_p} \cdot (\hat \psi_p \circ \rho)$ denote the lift of $\hat \psi_p$ supported in $\wtd U_p\cdot \gamma_p$.
Similarly select an open cover of $K$ consisting of sets $\{V_q\}$, each of which is the injective image of an open set $\wtd V_q\subset \wtd K$. Let $\{\hat h_q \mid j = 1, \dots, m\}$ be a partition of unity on $K$ subordinate to $\{V_q\}$ and for each $q$, choose a lift $h_q\colon \wtd K\to [0,1]$ whose support intersects a choice of compact fundamental domain for $\wtd K\to K$.
Given $g\in G$, write $g= kan$ and let $$\phi_{p,q}(g) = h_q(k) \psi_p( \wtd K an).$$
Let $\{\phi_i : 1\le i\le N\}$ denote the collection of functions $\{\phi_{p,q}\}$.

 We have the following properties of the above construction which we use in constructions in the next subsections.
\begin{enumerate}
\item For each $i$ and $\gamma \in \Gamma$, write $\phi_{i,\gamma}\colon G\to [0,1]$ for the function
	$$\phi_{i,\gamma}(g) = \phi_i(g\gamma\inv) .$$
 From the assumptions on the supports of $\phi_i$ we have
that the supports satisfy $\supp(\phi_{i,\gamma})\cap \supp(\phi_{i,\gamma'})= \emptyset$ whenever $\gamma\neq \gamma'$.
\item Since the coordinate neighborhoods of points at infinity are in the closures of single Siegel sets, the  separation property of Siegel sets (see also the proof of \cite[Proposition~III.9.17, p.~362]{MR2189882}) shows that, for the Siegel fundamental set~$\wtd D$ in $G$ fixed above, the set
$$\{\, \gamma\in \Gamma \mid \supp (\phi_{i,\gamma})\cap \wtd D\neq \emptyset \text{ for some $i$}\,\} $$
is finite.
\item The collection $\{\,\phi_{i,\gamma} \mid i\in \{1,\dots, N\}, \gamma\in \Gamma\,\}$ is a locally-finite, $\Gamma$-invariant, partition of unity on $G$.
\end{enumerate}

\subsection{Induced actions} \label{InducedSect}
We make precise a number of constructions appearing in the statement of \cref{thm:whatweprove}.

Given a $C^1$ action~$\alpha$ of~$\Gamma$ on a compact manifold~$M$ there is an induced $C^0$ action of~$\Gamma$ by vector-bundle automorphisms on the tangent bundle $TM$. We equip $TM$ with some background $C^\infty$ norm.
We introduce the notion of the induced action and associated constructions in this and in a somewhat more general setting.

To generalize the above setup, fix a compact metric space $(X_0, d_0)$ and let $\calE_0$ be a continuous, finite-dimensional, normed vector bundle over~$X_0$.
We suppose $\Gamma$ acts continuously on~$X_0$ and that the action lifts to a continuous linear cocycle $\calA_0\colon \Gamma\times \calE_0\to \calE_0$. We write $\calA_0(\gamma,x)$ for the linear map between fibers $\calE_0(x)$ and $\calE_0(\gamma\cdot x)$.

For any closed subgroup~$H$ of~$G$, the action of~$\Gamma$ restricts to a continuous action of $\Gamma\!_H$ on~$X_0$. 
We then obtain actions of $\Gamma\!_H$ and~$H$ on $H \times X_0$ given by
	\begin{align*}
	\text{$(g,x)\cdot \gamma = (g\gamma, \gamma\inv \cdot x)$
	{and}
	$g'\cdot(g,x) = (g'g,x)$}
	 \end{align*}
	 for $g'\in H$ and $\gamma\in \Gamma\!_H$.
As these actions commute, the $H$-action on $H\times X_0$ descends to a well-defined $H$-action, the \emph{induced $H$-action}, on the quotient $X_H = (H \times X_0)/\Gamma\!_H$.
We let $\calE_H = (H\times \calE_0)/\Gamma\!_H$ be the similarly defined finite-dimensional vector bundle over~$X_H$ with induced $H$-action.
Write $\pi_{X_H}\colon X_H \to H/\Gamma\!_H$ and $\pi_{\calE_H}\colon \calE_H \to X_H$ for the canonical projections.
 Note that $X_H$ is a fiber-bundle over $X/\Gamma\!_H$ with fiber homeomorphic to $X_0$.

When $H = G$, we omit subscripts and write $X = X_G$ and $\calE=\calE_G$.
The bundle $\calE$ admits a continuous linear cocycle $\calA\colon G\times \calE\to \calE$ which factors over the $G$-action on~$X$. Given closed subgroups $N\subset H\subset G$, we have a well-defined restriction of the cocycle $\calA$ to $N\times \calE_H\to \calE_H$.

Given $g\in G$ and $ x\in X_0$, we write $[g,x]\in X$ for the equivalence class in $X= (G\times X_0 )/\Gamma$.

\subsection{Construction of norms on $\calE$ and metric on $X$}\label{sec:norms}
We fix some choice of minimal parabolic $\Q$-subgroup~$P$ and Cartan involution~$\theta$ of $F$.
We recall the  partition of unity $\{\phi_{i,\gamma}\}$ of $G$ built relative to the Borel-Serre compactification in \cref{PartUnitySect}.
 Relative to these objects, we construct a norm on the bundle $\calE$ that is well adapted to the dynamics of $\calA$ on $\calE$.
We note that while this norm depends on the choice of $P$ and $\theta$, the dynamical properties of the cocycle $\calA$ will be independent of this choice; see \cref{lem:nottoobad}.

 First, let $\|\cdot\|_0$ denote the norm on $\calE_0$; given $x\in M$, denote by $\| \cdot \|_{0,x}$ the norm on the fiber $\calE_0(x)$. Given $g\in G,x\in X_0$, and $v\in \calE_0(x)$ set
	 	$$ \|v\|_{g,x} := \sum_{i=1}^N\sum_{\gamma\in \Gamma} \phi_{i,\gamma}(g) \, \| \calA(\gamma,x)(v)\|_{0,\alpha(\gamma)x} . $$
Let $\|\cdot \|'$ denote the norm on $G\times\calE_0$ induced by the collection of norms $\{\|\cdot \|_{g,x}\}$.
We collect a number of observations:
\begin{claim} \label{SeveralObservations} \
\begin{enumerate}
\item For any compact subset~$\mathcal{K} \subset G$, $\supp (\phi_{i,\gamma})\cap \mathcal{K} = \emptyset$ 
 for all but finitely many $(i,\gamma)$.

\item \label{SeveralObservations-comparable} \label{comparable}
The norms $\|\cdot \|_{g,x}$ are uniformly comparable on any Siegel fundamental set~$\wtd S$ in $G$: there is a constant $C > 0$, such that for all $g_1,g_2 \in D$, $x \in X_0$, and $v \in \calE_0(x)$, we have
	$$ \frac{1}{C} \| v \|_{g_1,x} \le \| v \|_{g_2,x} \le C \| v \|_{g_1,x} . $$
\item \label{SeveralObservations-33443}
For any $g\in G$, $\hat \gamma \in \Gamma$, and $(x,v)\in \calE_0$
we have $$\|\calA(\hat \gamma,x) v\|_{g,\alpha(\hat \gamma)(x)} = \|v\|_{g\hat \gamma,x}.$$

 In particular, $\Gamma$ acts by isometries on the fibers of  $G\times \calE_0$ whence $\|\cdot \|'$ descends to a norm $\|\cdot \|$ on the  bundle $\calE$ over $X$.
\item \label{SeveralObservations-smooth}
With respect to the norm $\|\cdot\|$ on $\calE$, for every $x\in X$ and any $v\in \calE(x)$, the function $$g\mapsto \|\calA(g,x)(v)\|$$ is $C^\infty$.
\end{enumerate}
\end{claim}
\begin{proof}
The only assertion that is not by definition is \eqref{SeveralObservations-33443}.
We verify
\begin{align*}
\|v\|_{g\hat \gamma,x}
&=\sum_{i=1}^N\sum_{\gamma\in \Gamma} \phi_{i,\gamma}(g\hat\gamma) \| \calA_0(\gamma,x)(v)\|_{0,\alpha(\gamma)(x)}\\
&=\sum_{i=1}^N\sum_{\gamma\in \Gamma} \phi_{i,\gamma\hat\gamma\inv}(g) \| \calA_0(\gamma,x)(v)\|_{0,\alpha(\gamma)(x)}\\
&=\sum_{i=1}^N\sum_{\gamma\in \Gamma} \phi_{i,\gamma\hat\gamma\inv}(g) \| \calA_0(\gamma\hat\gamma\inv\hat\gamma,x)(v)\|_{0,\alpha(\gamma\hat \gamma \inv)(\alpha (\hat \gamma)(x))}\\
&=\sum_{i=1}^N\sum_{\gamma\in \Gamma} \phi_{i,\gamma\hat\gamma\inv}(g) \| \calA_0(\gamma\hat\gamma\inv,\alpha \bigl( \hat\gamma)(x) \bigr)\calA_0(\hat\gamma,x)(v)\|_{0,\alpha(\gamma\hat \gamma \inv)(\alpha (\hat \gamma)(x))}\\  
&=\|\calA(\hat \gamma,x) v\|_{g, \alpha(\hat \gamma)(x)}
. \qedhere
\end{align*}
\end{proof}

The uniform comparability of the norms on Siegel fundamental sets  in  \cref{SeveralObservations}\pref{comparable} can be reformulated as  follows: for any (relative to $\theta$ and $P$) Siegel fundamental set  $D$, there is $C>1$ so that for all $g\in D$ and $x\in X_0$,
\begin{equation}\label{SiegGrowth}
\frac 1 C\le m\big(\calA(g, [\1,x])\big) \le \|\calA(g, [\1,x])\|\le C.
\end{equation}
We also have the following direct corollary of \cref{SeveralObservations}\pref{comparable}.  Let $\wtd K$ be the fixed point set of $\theta$ so that $\wtd K/\calZ$ is a maximal compact $\theta$-invariant subgroup of $\Ad(G)$.
\begin{claim}\label{moregrowth} \
\begin{enumerate}
\item For any compact set $\calK\subset \wtd K$ there is $C>1$ so that for all $x\in X$ and $k\in \calK$, $$\|\calA(k,x)\|\le C.$$
\item There is $C>1$ such that for all $x\in X$ and $z\in \calZ$, $$\|\calA(z,x)\|\le C\sup_{x_0\in X_0}\|\calA_0(z,x_0)\|.$$
\end{enumerate}
\end{claim}

\begin{remark}[Dependence  on the choice of $P$ and $\theta$]
In the sequel, we study growth properties the cocycle $\calA$ relative to the family of norms constructed above.
We note that if $D$ is a Siegel fundamental set constructed relative to a choice of  $P$ and $\theta$ then the norm on $\|\cdot \|$  does not depend strongly on the choice of fundamental domain $\fund \subset D$ or partitions of unity used in the construction; indeed, the strong separation properties of Siegel sets in \cref{sieg}\pref{sieg-separation} ensure that all norms on $\calE$ constructed relative to such choices are uniformly comparable.  From \cref{AWBcantdokindergardenalgebra}, this implies, up to multiplication by uniformly bounded functions, the construction of the norm on $\calE$ is independent of the choice of  minimal parabolic $\Q$-subgroup $P$.

 On the other hand, the choice of norms depends (in a non-uniformly comparable way) on the choice of Cartan involution $\theta$.  %
However, the following lemma guarantees that dynamical quantities including the values of Lyapunov exponents associated with the cocycle $\calA$ are independent  of the choice of $\theta$; see \cref{rem:Lyapok}.
To state the lemma, let $\theta$ and $\theta'$ be Cartan involutions and let   $P$ and $P'$ be minimal parabolic $\Q$-subgroups of $G$.
 Let $\|\cdot\|$ and $\|\cdot\|'$ denote, respectively, the norms constructed as above with respect to $(P, \theta)$ and $(P',\theta')$.

 \begin{lemma}\label{lem:nottoobad} For any right-invariant distance on $G$, there exists $C>1$ and $k>0$ such that for every $x\in X$ and $v\in \calE(x)$, writing $\hat x = \pi_X(x)$ for the image of $x$ in $G/\Gamma$,	$$ {\|v\|}\le Ce^{k d(\hat x, \1\Gamma)} \|v\|'.$$
\end{lemma}

\begin{proof}
Let $D$ and $D'$ be, respectively, Siegel fundamental sets in $G$ relative to the choice of $(P, \theta)$ and $(P',\theta')$. Consider any $x\in X$ and $v\in \calE(x)$. There are lifts $\wtd x_1 \in D\times X_0$ and $\wtd x_2 \in D'\times X_0$ of $x$. In particular, there are $g_0\in D$, $t_0\in X_0$, and $\wtd v\in \calE_0(t_0)$ with $$\wtd x_1 = (g_0,t_0) \text{ and } \|v\| = \|\wtd v \|_{g_0,t_0}.$$
There is $\gamma \in \Gamma$ with $g_0\gamma\in D'$. Let $t_1 = \alpha(\gamma\inv)(t_0)$. Then $\wtd x_2$ may be chosen so that $$\wtd x_2 = \big(g_0\gamma, t_1 \big)\text{ and }
\|v\|' = \|\calA_0(\gamma\inv,t_0)(\wtd v) \|'_{g_0\gamma ,t_1} .
$$
The families of norms $\|\cdot \|_{g,t_0}$ and $\|\cdot \|'_{g,t_1}$ are uniformly comparable with $\|\cdot \|_{0,t_0}$ and $\|\cdot \|_{0,t_1}$, respectively, when restricted to $g\in D$ and $g\in D'$, respectively. It follows that there exist constants $ C_1, C_2>1$ (independent of $x$, $v$, and $\gamma$) such that
\begin{align*}
\|v\| &= \|\wtd v\|_{g_0 , x} \le C_1 \|\wtd v\|_{0,t_0}\\
&\le C_1\|\calA(\gamma ,t_1)\|_0 \, \|\calA_0(\gamma\inv,t_0)\wtd v \|_{0 ,t_1 }\\
&\le C_1^2\|\calA(\gamma ,t_1)\|_0 \, \|\calA_0(\gamma\inv,t_0)\wtd v \|'_{g_0\gamma ,t_1} \\
& \le C_1^2 C_2^{\, \len_\Gamma ( \gamma)}\|v\|'.\end{align*}

From \eqref{eq:LMR} and \eqref{SiegelAlmIsom}, there are $C_3>1$ and $B>0$ so that $$\len_\Gamma ( \gamma) \le C_3d(g_0\Gamma,\1\Gamma) + B$$
and the result follows.
\end{proof}
\end{remark}

Using the construction above, we may similarly equip $G\times X_0$ with a metric  $\td d$ with the following properties:
\begin{enumerate}
\item the restriction of the metric $\td d$ to $G$-orbits in $G\times X_0$ projects to a fixed right-invariant metric $d_G$ on $G$;
\item $\Gamma$ acts by isometries relative to $\td d$;
\item the restrictions of the metric $\td d$ to each $\{g\}\times X_0$ are uniformly comparable when $g$ is restricted to a Siegel fundamental set.
\end{enumerate}
We then obtain a metric $d_X$ on $X= (G\times X_0)/\Gamma$.
 Using that the projection $\pi_X\colon X\to G/\Gamma$ is proper and that the fibers are uniformly comparable, the metric on $X$ satisfies the following properties:
\begin{enumerate}
\item there exists $B>0$ such that
$d_X(x,y) - B\le d_{G/\Gamma} (\pi_X(x), \pi_X(y))\le d_X(x,y)$
for all $x,y\in X$;
\item a sequence of Borel probability  measures $\{\mu_j\}$ on $X$ has uniformly exponentially small mass at $\infty$ if and only if the  projections $\{(\pi_X )_*\mu_j\}$  form a sequence of Borel probability measures on  $G/\Gamma$ with uniformly exponentially small mass at $\infty$;
\item a sequence of Borel probability  measures $\{\mu_j\}$ on $X$ is uniformly tight  if and only if the  projections $\{(\pi_X)_*\mu_j\}$ form a uniformly tight sequence of Borel probability measures on  $G/\Gamma$.
\end{enumerate}
In the sequel, we will typically ignore the specific metric $d_X$ on $X$ and only consider coarse metric properties of points in $X$ under the
projection $\pi_X\colon X\to G/\Gamma$.

\subsection{Growth estimates on the cocycle $\calA$} \label{TemperedSect}
We fix $P$ and $\theta$ as above and an associated norm $\|\cdot\|$ on $\calE$ constructed as in \cref{sec:norms} relative these choices.
Let $D\subset G$ be an associated Siegel fundamental set for $\Gamma\subset G$.
 We present a number of definitions and lemmas that allow us to control and relate growth of the cocycles $\calA$ and $\calA_0$.
We remark that the strong estimates on the growth of $\calA$ obtained in \cref{SiegelImpliesSubexp,SubexpInCusps}  below are only true relative to the norm constructed relative to the choice of Cartan involution $\theta$.

We begin with the following definition.
\begin{definition}
Given a Borel fundamental domain~$\fund$ for the right $\Gamma$-action on $G$, the \emph{return cocycle}  $\beta_\fund\colon G \times G/\Gamma \rightarrow \Gamma$ is defined as follows: given $\hat x\in G/\Gamma$, take $\wtd x$ to be the unique
lift of $\hat x$ in $\fund$ and define $\beta_\fund(g,\hat x)$ to be the unique $\gamma\in \Gamma$ such that
$g\wtd{x}\gamma\inv \in \fund$.
\end{definition}
One verifies that $\beta_\fund$ is, in fact, a Borel-measurable cocycle:
	$$ \beta_\fund(g_1 g_2,\hat x) = \beta_\fund(g_1,g_2\hat x) \, \beta_\fund(g_2,\hat x) .$$
A second choice of fundamental domain $\fund'$ for~$\Gamma$ defines a cohomologous cocycle~$\beta_{\fund'}$.

For the remainder, we will always assume that the fundamental domain~$\fund$ is contained in a Siegel fundamental set~$D$ in $G$.

\begin{lemma}
\label{lemma:fromlmr}
For a fundamental domain $\fund$ contained in a Siegel fundamental set in $G$ there is a $C$ such that for every $g\in G$ and $\hat x\in G/\Gamma$,
$$\len_\Gamma \bigl( \beta_\fund(g,\hat x) \bigr) \leq C d(g,\1)+ C d(\hat x, \Gamma) + C.$$

In particular, for any $1\le p<\infty$ and any compact $\mathcal K \subset G$ the function
$$ \hat x\mapsto \sup_{g\in \mathcal K} \len_\Gamma \bigl( \beta_\fund (g, \hat x) \bigr)$$
is in $L^p(G/\Gamma, m_{G/\Gamma})$.
\end{lemma}
Indeed, write $\hat x = \td x\Gamma$ for $\td  x\in \fund$.  Then $g\td  x = x' \beta_\fund(g,\hat x)$ for some $x'\in \fund$ and
\begin{align*}
d(\1, \beta_\fund(g,\hat x))
	&\le d(\1,\td  x) + d(\td  x, g\cdot \td x) + d (g\cdot \td  x, \beta_\fund(g,\hat x))\\
	&= d(\1,\td  x) + d(\1, g) + d (\td  x' \beta_\fund(g,\hat x), \beta_\fund(g,\hat x))\\
	&= d(\1,\td  x) + d(\1, g) + d (\td  x' , \1).
\end{align*}
We have $$d(\td  x' \Gamma, \Gamma) =d(g \hat x, \Gamma) \le d(g \hat x, \hat x) + d( \hat x, \Gamma) \le d(g, \1) + d( \hat x, \Gamma).$$
From \eqref{SiegelAlmIsom},  there is a constant $\hat C$ with   $d(\1, \td  x) \le \hat C d(\hat x, \Gamma) + \hat C$
and $$d(\1, \td  x') \le \hat  C d(g\cdot \hat  x, \Gamma) +\hat  C \le      \hat  C d(g, \1) + \hat C d(\hat  x, \Gamma) +\hat  C.$$
\cref{lemma:fromlmr} then follows from  \eqref{eq:LMR} and \eqref{HaarSmallCusps}.

 From the uniformly comparability of the norms across a Siegel fundamental sets  in \cref{SeveralObservations}\pref{comparable} and the finite generation of $\Gamma$ we obtain the following: there are constants $ C>1$ and $k>0$ so that for all $g \in G$ and $x \in X$, writing $\hat x= \pi_X(x)\in G/\Gamma$, we have
\begin{equation}\label{BoundCalAByBeta}
 \| \calA(g,x) \| \le Ce^{k \len_\Gamma ( \beta_\fund (g, \hat x) ) }
\end{equation}
and similarly
\begin{equation}\label{BoundCalAByBeta2}
m(\calA(g,x) ) \ge \frac 1 C e^{- k \len_\Gamma ( \beta_\fund (g, \hat x) ) }.
\end{equation}
Above, $m(A) = \|A\inv\|\inv$ denotes the conorm of an invertible linear operator $A$.

Recall in \cref{sec:trans} we built a proper, $C^\infty$, uniformly Lipschitz function $h\colon G/\Gamma \to [0,\infty)$ satisfying
\begin{equation} \label{heightDefnII} \text{$d(g \Gamma, \1 \Gamma) -1 \le h(g\Gamma) \le d(g \Gamma, \1 \Gamma) + 1$ for all $g \in G$} .\end{equation}
We lift $h$ to functions on $X$ and $\calE$ by precomposing with the canonical projections.
	
For any function $h$ satisfying \eqref{heightDefnII}, combining \eqref{BoundCalAByBeta}, \eqref{BoundCalAByBeta2}, and \cref{lemma:fromlmr} we obtain the following.
\begin{proposition} \label{tempered}
The cocycle $\calA\colon G\times \calE\to \calE$ is \emph{$h$-tempered:} there are constants $C > 1$ and $k> 0$ such that for all $g\in G$ and $x \in X$,
	$$ \|\calA(g,x)\| \le C e^{k h(x) + k d(g,\1)} \quad \text{and} \quad m(\calA(g,x))\ge \frac 1 C e^{- k h(x) - k d(g,\1)}.$$
\end{proposition}

Recalling \cref{thm:C} and \cref{def:USEGOD}, we formulate the following   controls on the growth of the cocycles $\calA_0$ and $\calA$.
\begin{definition}\label{defn:tempsubexp}
Fix closed subgroups $N\subset H\subset G$.  Let $\phi\colon H/\Gamma\!_H\to [0,\infty)$ be a proper continuous function  and extend to $\phi\colon X\!_H\to [0,\infty)$ by precomposition with the projection.
\begin{enumerate}
\item We say that the restriction of the cocycle~$\calA_0$ to~$\Gamma\!_H$ has \emph{(uniform) $\len_\Gamma$-subexponential growth} if, for every $\epsilon>0$, there exists $C>0$ such that for every $\gamma\in \Gamma\!_H$ and $x \in X_0$ we have
 \begin{align*}
	\|\calA_0(\gamma,x)\| &\le C^{\epsilon \len_\Gamma(\gamma) }
	. \end{align*}
	\item \label{uifse} We say that the restriction of the cocycle $\calA$ to $N\times \calE_H$ has \emph{(uniform) subexponential growth} if, for every $\epsilon>0$, there exists $C>0$ such that for every $g\in N$ and $x \in X_H$ we have
	\begin{align*}
	\|\calA(g,x)\| &\le Ce^{ \epsilon d(g,\1) }
	 \end{align*}
	and
	\begin{align*}
	m(\calA(g,x)) &\ge Ce^{ - \epsilon d(g,\1) }. \end{align*}

	\item \label{useg} We say that the restriction of the cocycle $\calA$ to $N\times \calE_H$ has \emph{$\phi$-tempered subexponential growth} if there exist $\omega>0$ such that  for every $\epsilon>0$ there exists $C_\epsilon>0$ such that for every $g\in N$ and $x \in X_H$ we have
	\begin{align*}
	\|\calA(g,x)\| &\le C_\epsilon e^{  \omega  \phi(x) + \epsilon d(g,\1) }
	 \end{align*}
	and
	\begin{align*}
	m(\calA(g,x)) &\ge C_\epsilon e^{ -  \omega \phi(x)  - \epsilon   d(g,\1) }
	. \end{align*}
	\end{enumerate}
\end{definition}

We collect a number of crucial estimates relating the above definitions.  For notational convenience in the following proofs,  we  introduce the   subadditive cocycles $\calB^\pm\colon G\times G/\Gamma \to \R$,
\begin{align}\calB^+(g, x) &:= \sup_{y\in \pi_X\inv (x)} \log \|\calA(g, y)\|\label{eq:pooface1}\\
\calB^-(g, x) :
&=\calB^+(g\inv , g\cdot x) \label{eq:pooface2} \\&
= \sup_{y\in \pi_X\inv (x)} \log \| \calA(g, y)\inv \| \notag \\&= \sup_{y\in \pi_X\inv (x)} - \log m( \calA(g, y) )\notag
.\end{align}
By \cref{tempered} there are constants  $C$ and $k$ with
$$ \calB^\pm (g,x) \le {k h(x) + k d(g,\1)} +  C.$$

If $H$ is a subgroup of $G$ and $N$ is a unipotent $\Q$-group normalized by $H$ we write $\wtd H= HN$.  We observe that $\Gamma\!_N$ is cocompact in $N$ and  that $\wtd H/\Gamma\!_{\wtd H}$ is a smooth fiber-bundle over $H/\Gamma$ with compact fibers.  Let $\rho\colon \wtd H/\Gamma\!_{\wtd H} \to H/\Gamma\!_{\wtd H}$ be the canonical projection.  In the following lemma and in \cref{horseforlunch} below, the function $\phi$ which tempers the growth estimates is  the map  $  \phi\colon \wtd H/\Gamma\!_{\wtd H} \to [0,\infty) $ given by \begin{equation}\label{eqtdphi}    \phi(x) = h(\rho(x))\end{equation} where $h\colon G/\Gamma\to [0,\infty)$ is the $C^\infty$ approximation to $d(\cdot, \1\Gamma)$ as in \eqref{heightDefnII}.

\begin{lemma} \label{SiegelImpliesSubexp}\label{eggface}\label{Ngrowthcorrect}
Let $N$ and  $H$ be  connected $\Q$-subgroups of $G$ such that $N$ is unipotent and $H$ normalizes $N$.
Suppose there is a fundamental domain $\fund_H$ for $\Gamma\!_H$ in $H$ contained in a Siegel fundamental set~$D$ in $G$ relative to the choice of $P$ and~$\theta$.

Write $\wtd H = HN$.  Then  the restriction of~$\calA_0$ to~$\Gamma\!_N$ has $\len_\Gamma$-subexponential growth if and only if the restriction of~$\calA$ to $N \times \calE_{\wtd H}$ has  $\phi$-tempered subexponential growth.
\end{lemma}

\begin{proof}
We first assume that the restriction of~$\calA_0$ to~$\Gamma\!_N$ has $\len_\Gamma$-subexponential growth.
Fix a compact fundamental set $\fund_N$ for $\Gamma\!_N$ in $N$.
Then $\fund_H\fund_N$ contains a fundamental domain for $\Gamma\!_{\wtd H}$.
Indeed for all $g \in HN$ there exists $\gamma_H \in \Gamma\!_H$ such that $g \gamma_H \in \fund_H  N$. Then there is  $\gamma_N \in\Gamma\!_N$ with $g \gamma_H \gamma_N \in \fund_H  \fund_N$.
Fix $ gn_0 \in \fund_H\fund_N$.
Given $n\in N$, using that $H$ normalizes $N$ we have $$ngn_0 = g(g\inv n g) n_0 = gn_1\gamma_N$$
for some $\gamma_N\in \Gamma\!_N$ and $n_1\in \fund_N$.

Fix $\epsilon>0$.  As $n = gn_1\gamma_N n_0\inv g\inv$, we have
 \begin{align*}\calB^\pm(n, gn_0\Gamma) & = \calB^\pm( gn_1\gamma_N n_0\inv g\inv, gn_0 \Gamma) \\
& \le \calB^\pm( n_0\inv g\inv, gn_0 \Gamma)  +
\calB^\pm(\gamma_N,  \1 \Gamma) +
\calB^\pm( gn_1, \1 \Gamma) \\
& = \calB^\mp( gn_0, \1 \Gamma)  +
\calB^\pm(\gamma_N,  \1 \Gamma) +
\calB^\pm( gn_1, \1 \Gamma)\\
&\le  2k h(1\Gamma) + k d(gn_0,\1) + k d(gn_1,\1) + \epsilon \len_\Gamma(\gamma_N) +  \hat C +C_\epsilon
\end{align*}
for some constant $\hat C$ independent of $\epsilon$.

As $\Gamma$ is quasi-isometrically embedded in $G$, there are  $A_0$ and $B_0$ independent of $g$ and $n$  such that $$\len_\Gamma(\gamma_N)\le A_0\big( d(\gamma_N,\1) \big) + B_0.$$
  We further bound the right-hand side using  $d(\gamma_N,\1) \le  2d(g,\1) + d(n,\1)+ 2 \diam(\fund_N)$.
Indeed
\begin{align*}
d(\1, \gamma_N)
	&\le d(\1, gn_0) + d(gn_0,ngn_0) + d(ngn_0, \gamma_N) \\
	&= d(\1, gn_0) + d(gn_0,ngn_0) + d(gn_1\gamma_N, \gamma_N) \\
	&= d(\1, gn_0) + d(\1,n) + d(gn_1 , \1) \\
	&\le  d(\1, n_0) + d(n_0, gn_0) + d(\1,n) + d(gn_1 , n_1) + d(n_1 , \1) \\
	&\le  2d(g,\1) + d(n,\1)+ 2 \diam(\fund_N).
\end{align*}

The  map $G\to F$ is $1$-Lipschitz by definition and is uniformly bi-Lipschitz when restricted to $N$ by \cref{simpconisbilip}.   Since the distance from $g$ to the identity in the linear group $F$ is quasi-isometric to $g\mapsto \log \max \{\|g\|, \|g\inv\|\}$, it follows there are constants $A_1, B_2$ such that for all $\hat n \in \fund_N$,
$$ d\big(g\hat n g\inv ,\1\big) \le A_1 d (g, \1) + B_1.$$
Then \begin{align*}
d(g\hat n,\1) &=  d\big((g\hat ng\inv )g,\1\big)
\\ & \le   d\big((g\hat ng\inv) g,g\big) + d(g, \1) \\&
= d (g\hat ng\inv ,\1 ) + d(g, \1)\\
&\le ( A_1 +1) d (g, \1) + B_1.
\end{align*}

We recall that $\fund_H$ is contained in a Siegel fundamental domain.  It follows there are $A_2$ and $B_2$ such that for all $g\in \fund_H$, $$d(g,\1)\le A_2d(g\Gamma, \1 \Gamma)+ B_2 \le A_2 h(g\Gamma)+ B_2.$$
Assembling the above estimates, there are constants  $c_1, c_2$, and $C_3$ and, for any $0<\epsilon\le 1$, a constant $C_\epsilon$ such that
\begin{align*}
\calB^\pm(n, gn_0\Gamma)
	&\le c_1 d(g\Gamma,\1\Gamma) + c_2 \epsilon d(n,\1) + C_3 + C_\epsilon\\
	&\le c_1 h(g\Gamma) + c_2 \epsilon d(n,\1) + C_3 + C_\epsilon\\
	&=  c_1  \phi(gn_0\Gamma) + c_2 \epsilon d(n,\1) + C_3 + C_\epsilon
\end{align*}
where we recall that $   \phi$ is  constant on $N$-orbits in  $\wtd H/\Gamma\!_{\wtd H}$ whence $$\phi(gn_0\Gamma) = \phi((gn_0g\inv) g\Gamma)= \phi(g\Gamma).$$
We then take $\omega = c_1.  $

For the reverse implication, by hypothesis there are constants $\hat C>1$ and, for every $\epsilon>0$, there is $C_\epsilon$ such that given $\gamma \in \Gamma\!_N$ and $x \in X_0$, we have
	$$ \|\calA_0(\gamma,x)\|
	\le \hat C \bigl\| \calA \bigl(\gamma,[\1 ,x] \bigr) \bigr\|
	\le \hat C C_\epsilon e^{ \omega  \phi( \1\Gamma) +\epsilon  d(\gamma,\1) }. $$
The conclusion then follows from \eqref{eq:LMR}.
\end{proof}

In the case that $\Gamma\!_H$ is cocompact in $H$ we note that any compact set in $G$ is contained in a   Siegel fundamental set $D$ (for any choice of $P$ and $\theta$).  We then immediately obtain the following corollary.
\begin{corollary} \label{SubexpForCompact}
Let $N \subset H$ be closed connected subgroups of $G$ with $N$ unipotent, $N$ normal in $H$, and $\Gamma\!_H$   cocompact in $H$.
Then the restriction of~$\calA_0$ to~$\Gamma\!_N$ has $\len_\Gamma$-subexponential growth if and only if the restriction of~$\calA$ to $N \times \calE_H$ has subexponential growth.
\end{corollary}

In \cref{SubexpForRank1Subgroups}, we will also use the following variant of \cref{SiegelImpliesSubexp} that is special for the case that $H$ is a standard $\Q$-rank-1 subgroup.

Let $  \phi\colon \wtd H/\Gamma\!_{\wtd H}\to [0, \R)$ be  as in \eqref{eqtdphi}.  The following lemma is important for passing from subexponential growth on unipotent subgroups to subexponential growth for trajectories that remain in special subgroups and away from the thick part.

\begin{lemma}[$\phi$-tempered subexponential growth at $\infty$] \label{SubexpInCusps}\label{horseforlunch}
Let $H$ be a standard $\Q$-rank-$1$ subgroup of~$G$ and choose $P$ and~$\theta$ as in \cref{SiegelForSubgroup}. Let $N$ be a connected unipotent $\Q$-subgroup of $G$ 
that is normalized by~$H$.
Assume that the restriction of~$\calA_0$ to ~$\Gamma\!_N$ has $\len_\Gamma$-subexponential growth and that the restriction of~$\calA_0$ to~$\Gamma\!_{P'}$ has $\len_\Gamma$-subexponential growth for every minimal parabolic $\Q$-subgroup~$P'$ of~$G$.

Write $\wtd H = HN$.  Then there exist $\omega>1$ and $\ell_0>0$ such that for any $\epsilon>0$  there is a $C_\epsilon >0$ such that for  any continuous path $t\mapsto g_t$ in  $H$ with $g_0 = \1$, the following hold:
\begin{enumerate}
\item ({\bf Weak $\phi$-tempered subexponential growth in the cusp on $X_{\wtd H}$}.)
For any  $x\in \wtd H/\Gamma\!_{\wtd H}$  and any $T > 0$ such that $\phi (g_t\cdot x)\ge \ell_0$ for all   $0\le t \le T$,
$$ \calB^{\pm}(g_t, x) \le \omega \phi(x) + \omega \phi (g_t\cdot x) + \epsilon d(\1 ,g_t)  + C_\epsilon .$$

\item ({\bf Strong $\phi$-tempered subexponential growth  in the cusp  on $X_{H}$}.)
For any  $x\in   H/\Gamma\!_{ H}$  and any $T > 0$ such that $\phi (g_t\cdot x)\ge \ell_0$ for all   $0\le t \le T$,
$$ \calB^{\pm}(g_t, x) \le \epsilon \big( \phi (x) +  d(\1 ,g_t)  \big)+ C_\epsilon .$$
\end{enumerate}
\end{lemma}

\begin{proof}
By the choice of~$P$ and~$\theta$ as in \cref{SiegelForSubgroup}, we may select a Siegel fundamental set $D_H$  in~$H$ for $\Gamma\!_H$ such that $D_H$ is contained in a Siegel fundamental set $D$ in $G$ for $\Gamma$. Fix a fundamental domain~$\fund_H$ for~$\Gamma\!_H$ in~$D_H$.
Select a precompact fundamental domain $\fund_N$ for the $\Gamma\!_N$-action on $N$. As in the proof of \cref{eggface}, $\fund_H\fund_N$ contains a fundamental domain $\fund_{\wtd H}$ for the $\Gamma\!_{\wtd H}$-action on $\wtd H$.

Let $\pi_H \colon H \to H/\Gamma\!_H$ be the natural projection.
Since $\rank_\Q H = 1$, there is a compact subset~$\mathcal{K}$ of $H/\Gamma\!_H$ such that for each connected component $\mathcal{C}$ of $\pi_H^{-1} \bigl( (H/\Gamma\!_H) \sm \mathcal{K} \bigr)$, there is a minimal parabolic $\Q$-subgroup $P_{\mathcal{C}}$ of~$H$ such that $\mathcal C$ is contained in a $H_\Q$-translate of a Siegel set constructed relative to $P_{\mathcal{C}}$ and $\restrict {\theta}{H}$. In particular, for each such $\calC$,  the set
	$ \{ \gamma \in \Gamma \mid \mathcal{C} \gamma \cap \mathcal{C} \neq \emptyset \}$
is contained in~$\Gamma\!_{P_{\mathcal{C}}}$; see \cite[Theorem~13.10, p.~202]{MR0507234} or \cite[Proposition III.2.19, p.~307]{MR2189882}.
We remark that every minimal $\Q$-parabolic subgroup in $H$ is the intersection of $H$ with a minimal parabolic $\Q$-subgroup of $G$; indeed all minimal $\Q$-parabolic subgroup in $H$ are conjugate by $H_\Q$ so the conclusion follow from \cref{SiegelForSubgroup}.
 It follows that the restriction of $\calA_0$ to each $\Gamma\!_{P_{\mathcal{C}_i}}$, $1\le i\le r$, has uniform $\len_\Gamma$-subexponential growth.

Fix  $\ell_0> \sup\left\{ \phi(x): x\in \calK\right\}$.
Since $\fund_H$ is contained in a Siegel fundamental set, $\pi_H^{-1} \bigl(  \phi\inv  ([\ell_0, \infty))) \cap \fund_H$ is contained in the union of finitely many connected components $\mathcal{C}_1,\ldots,\mathcal{C}_r$ of $\pi_H^{-1} \bigl( (H/\Gamma\!_H) \sm \mathcal{K} \bigr)$.
Then, since  $ \phi$ is $N$-invariant, if $g_0n_0 \in \fund_H\fund_N$ and if $ \phi(g_0n_0\Gamma) \ge \ell_0$ then $g_0n_0 \in \calC_i\fund _N$ for some $1\le i\le r$.

Consider  a continuous 1-parameter path $t\mapsto g_t $ in $H$ with $g_0 = \1$.  Consider any $g'n_0\in \fund_{H}\fund _{N}$ such that for some $T>0$ and all $0\le t\le T$,
$$ \phi (g_t g' n_0 \Gamma) \ge \ell_0. $$
Fix  $0\le t\le T$ and write $g_tg' n_0 = g'' n_1 \gamma_H \gamma_N$ where $g'' n_1 \in \fund_{H}\fund _{N}$.
By our choice of $\ell_0$, we have $\gamma_H\in \Gamma\!_{P_{\mathcal{C}_i}}$ for some $1\le i\le r$.

As in the proof of \cref{eggface},
\begin{align*}\calB^\pm(g_t, g'n_0\Gamma)
& \le  \calB^\mp( g'n_0, \1 \Gamma)  +
\calB^\pm(\gamma_H,  \1 \Gamma) +
\calB^\pm(\gamma_N,  \1 \Gamma) +
\calB^\pm( g''n_1, \1 \Gamma).  \end{align*}
Let $\gamma = \gamma_H\gamma_N$.
As in the proof of \cref{eggface},
$$d(\gamma,\1) \le  d(g',\1)+ d(g'',\1) + d(g_t,\1)+ 2 \diam(\fund_N).$$
Indeed
\begin{align*}
d(\1, \gamma)
	&\le d(\1, g'n_0) + d(g'n_0,g_tg'n_0) + d(g_tg'n_0, \gamma) \\
	&= d(\1, g'n_0) + d(g'n_0,g_tg'n_0) + d(g''n_1\gamma, \gamma) \\
	&= d(\1, g'n_0) + d(\1,g_t) + d(g''n_1 , \1) \\
	&\le  d(\1, n_0) + d(n_0, g'n_0) + d(\1,g_t) + d(g''n_1 , n_1) + d(n_1 , \1) \\
	&\le  d(g',\1) + d(g'',\1)+ d(g_t,\1)+ 2 \diam(\fund_N).
	\end{align*}
\cref{cor:QIlift} and the hypotheses then imply  there are  $C_0$ and for any $\epsilon>0$ there exists  $C_\epsilon$ such that
$\calB^\pm(\gamma_H,  \1 \Gamma) $ and $\calB^\pm(\gamma_N,  \1 \Gamma)$ are bounded above by $$\epsilon C_0d(\1 , \gamma) + C_\epsilon.$$
As in the proof of \cref{eggface} there are constants $k,B$ so that each $\calB^\pm( g_in_i, \1 \Gamma)$ is bounded above by $$k d(g_i,\1) + B.$$
Since $g_i\in \fund _H$ there are $C_1$ and $B_1$ so that $$d(g_i,\1) \le C_1 d(g_i\Gamma ,\1\Gamma ) + B_1 \le C_1  \phi (g_i\Gamma) + B_1.$$
Since $ \phi (g_tg'n_0 \Gamma )=  \phi (g''n_1 \Gamma )$ this establishes the first conclusion.

For the second conclusion, if $x\in H/\Gamma\!_H$ then $x$ has a representative of $g'\in \fund_H$.  We have that  $\calB^\pm( g, \1 \Gamma)$ is uniformly bounded over all $g\in \fund_H$ and thus for some constants $C_2, C_3, c_2$, and $c_3$,
\begin{align*}
\calB^\pm(g_t, g'n_0\Gamma)
&\le C_1+ c_1\epsilon (d(g',\1)+ d(g'',\1) + d(g_t,\1)) + C_\epsilon\\
&\le C_2+ c_2\epsilon ( \phi(g'\Gamma)+  \phi (g_t g \Gamma) + d(g_t,\1)) + C_\epsilon\\
&\le C_3+ c_2\epsilon (2  \phi(g'\Gamma)+ 2 d(g_t,\1)) + C_\epsilon
\end{align*}
 where we use that
	\[ \phi(g_tg'\Gamma) \le d(g_tg'\Gamma,1\Gamma) + 1 \le
	d(g_t, \1) + d(g'\Gamma,1\Gamma) + 1 \le
	d(g_t, \1) +  \phi (g'\Gamma) + 1 .\qedhere\]
\end{proof}

\section{Lyapunov exponents} \label{LyapunovPrelims}
As in \cref{InducedSect}, we continue to assume that $\Gamma$ acts on a compact metric space $(X_0,d_0)$ and the action lifts to a continuous cocycle $\calA_0$ on a finite-dimensional normed vector bundle~$\calE_0$ over $X_0$. We have the induced actions of~$G$ on $X = (G \times X_0)/\Gamma$ and $\calE = (G \times \calE_0)/\Gamma$. The induced linear cocycle on $\calE$ is denoted by $\calA$. Equip $\calE$ with a norm as in \cref{sec:norms}; recall from \cref{tempered} we have that $\calA$ is $h$-tempered.

\subsection{Lyapunov exponents for unbounded cocycles}
\label{subsection:Lyap}

We have the following claim which follows immediately from \cref{tempered} and guarantees the
existence and finiteness of Lyapunov exponents for the cocycle $\calA$.
Equip $\calE$ with a norm constructed as in \cref{sec:norms}.

\begin{claim}\label{CocycIsL1}
Let $\mu$ a probability measure on $X$ with exponentially small mass at $\infty$.
Then for any compact $E \subset G$, the functions
$$x\mapsto \sup_{g\in E} \log \left \| \calA(g, x)\right \| \quad \text{and} \quad x\mapsto \inf_{g\in E}\log m( \calA(g, x)) $$
are $L^1(\mu)$.
\end{claim}

Given $a\in G$ and an $a$-invariant Borel probability measure $\mu$ on~$X$ we define the \emph{average top} (or \emph{leading}) \emph{Lyapunov exponent of $\calA$} to be
\begin{equation} \label{eq:toplyap}
\lambda_{\top,a,\mu, \calA} := \liminf _{n\to \infty} \frac 1 n \int \log \|\calA(a^n, x)\| \ d \mu (x).\end{equation}
This is finite when the function $x\mapsto \log \|\calA(a, x)\|$ is
integrable; in particular, for any
$a$-invariant Borel probability measure $\mu$ on $X$ with exponentially small mass at $\infty$,
the average top Lyapunov exponent $\lambda_{\top,a,\mu, \calA}$ of~$\calA$ is finite.
Note also that the $a$-invariance of the measure $\mu$ implies that the sequence $$n\mapsto \int \log \|\calA(a^n, x)\| \ d \mu (x)$$ is subadditive whence the $\liminf$ in \eqref{eq:toplyap} can be replaced with either a limit or an infimum.

\begin{remark} \label{rem:Lyapok}
By \cref{lem:nottoobad}, the value of $\lambda_{\top,a,\mu, \calA}$ is independent of the choice of norm constructed in \cref{sec:norms} whenever the $\{a^t\}$-invariant measure $\mu$ on $X$ projects to a measure on $G/\Gamma$ for which the function $g\Gamma\mapsto d(g\Gamma, \1\Gamma)$ is $L^1$. In particular, for any $\{a^t\}$-invariant measure $\mu$ on $X$ projecting to the Haar measure on $G/\Gamma$, the value of $\lambda_{\top,a,\mu, \calA}$ is defined independent of the choice of $P$ and $\theta$ chosen in the construction of the norm on $\calE$.
\end{remark}

With the notation established above and in \cref{InducedSect,sec:norms}, we state the following reformulation and generalization of \cref{mainthm}.
\begin{theorem}\label{thm:whatweprove2}\label{mainthm2}
Let $G$ be a connected semisimple Lie group without compact factors and with $\rank_\R G \ge 2$. Let $\Gamma$ be an irreducible lattice subgroup in~$G$ and let $\calA_0\colon \Gamma\times \calE_0\to \calE_0$ be a cocycle as above.

If $\calA_0\colon \Gamma\times \calE_0\to \calE_0$ fails to have subexponential growth then, for the induced $G$-cocycle $\calA\colon G\times \calE\to \calE$, there exists a split Cartan subgroup~$A$ of~$G$ and a Borel probability measure~$\mu$ on~$X$ such that
\begin{enumerate}
\item $\mu$ is $(\calZ A)$-invariant,
\item $\mu$ projects to the Haar measure on $G/\Gamma$,
	and
	\item $\lambda_{\top,a,\mu, \calA}>0$ for some $a \in \calZ A$.
	\end{enumerate}
\end{theorem}

\begin{remark}[The case of cocompact $\Gamma$]\label{rem:cocomp}\index{Proof for $\Gamma$ cocompact}
In the case that $\Gamma$ is a cocompact lattice in $G$, we note that the induced $G$-space $X$ is compact and the cocycle $\calA$ is bounded. In this case, the conclusions of \cref{mainthm} and \cref{mainthm2} are essentially contained in \cite{BFH}. Indeed, in the setting of \cref{mainthm}, for groups $G$ with finite center, \cite[Proposition 4.6] {BFH} guarantees the existence of a split Cartan subgroup~$A$ of~$G$ and an $A$-invariant, Borel probability measure~$\mu'$ on~$X$ such that the average top Lyapunov exponent $\lambda_{\top,a,\mu', \calA}$ is positive for some $a\in A$. When $G$ has infinite center, let $\wtd K$ be a maximal subgroup such that $K=\wtd K/\calZ$ is a maximal compact subgroup of $G/\calZ$. Let $K_0\subset \wtd K$ be a compact subset that maps onto $K$. We may replace the $KAK$-decomposition of $G$ in the proof of \cite[Proposition 4.6] {BFH} with a $K_0\calZ A K_0$ decomposition as discuessed in \cref{kakaReductive}.  Up to finite index, $\calZ A$ is isomorphic to $\Z^k \times \R^\ell$. See \cref{kaka}.
Minor notational modifications to the proof of \cite[Proposition 4.6] {BFH} in the setting of \cref{mainthm2} and for groups with infinite center then yields a $(\calZ A)$-invariant, Borel probability measure~$\mu'$ on~$X$, such that the average top Lyapunov exponent $\lambda_{\top,a,\mu', \calA}$ is positive for some $a\in \calZ A.$

It remains to modify $\mu'$ to obtain a new measure $\mu$ with the above properties and whose projection to $G/\Gamma$ is the Haar measure.
Assume first that all simple factors of $G$ are higher rank (which holds, in particular, when $G$ is simple). If $\lambda_{\top,a,\mu', \calA}>0$ for some $a\in A$ then, in the setting of \cref{mainthm}, the proof of \cite[Proposition 4.7]{BFH} (see especially \cite[Lemma 6.6] {BFH}) gives a measure $\mu$ with the desired properties. In the setting of \cref{mainthm2}, we again
obtain the desired outcome by minor notational modifications to the proof of \cite[Proposition 4.7]{BFH} and \cite[Lemma 6.6] {BFH}. In the case that $\lambda_{\top,a,\mu', \calA}=0$ for all $a\in A$ but $\lambda_{\top,a,\mu', \calA}>0$ for some $a\in\calZ$, one may easily adapt the arguments in the proofs of \cite[Proposition 4.7]{BFH} and \cite[Lemma 6.6] {BFH} (which are much simpler now since $a$ commutes with all elements of $G$, see \cref{lemma:averaging} below) to obtain a measure $\mu$ with the desired properties.

When $G$ is semisimple and, in particular, when $G$ has rank-1 factors, we can combine the proof of \cite[Proposition 4.7]{BFH} with the irreducibility of $\Gamma$ to obtain the desired conclusion. Indeed, we may assume $\mu'$ is ergodic and thus fiberwise Lyapunov exponents for the $(\calZ A)$-action on $(X,\mu)$ are identified with linear functionals on $\calZ A$ (and thus vanish on torsion elements); see for example \cite[Proposition 3.11]{BFH-SLnZ} for a formulation of Oseldec's theorem for cocycles over actions of higher-rank abelian groups.
Write $G= G' \cdot G''$ where $G'$ and $G''$ are normal subgroups of positive dimension with $G'\cap G''\subset \calZ$.
Write $A'= A\cap G'$ and $A''= A\cap G''$.
Up to interchanging the roles of $G'$ and $G''$, we may assume $\lambda_{\top,a,\mu', \calA}$ is positive for some $a\in \calZ A''$. We may average $\mu'$ over a \folner sequence contained in a maximal unipotent subgroup $U$ of $G'$ normalized by $A'$ and pass to a subsequential limit to obtain a Borel probability measure $\mu''$ on $X$. We may then average $\mu''$ over a \folner sequence in $A$ and pass to a subsequential limit to obtain a Borel probability measure $\mu$ on $X$. The measure $\mu$ is $A$-invariant and, since $A$ normalizes $U$, $\mu$ is also $U$-invariant.
 Moreover, using that $a\in \calZ A''$ centralizes $G'$, \cite[Lemma 4.2]{BFH} shows that $$\lambda_{\top,a,\mu, \calA}\ge \lambda_{\top,a,\mu'', \calA}\ge \lambda_{\top,a,\mu', \calA}>0.$$
The projection of $\mu$ to $G/\Gamma$ is $(A'U)$-invariant; \cref{thm:opproot} shows that the projection of $\mu$ to $G/\Gamma$ is $G'$-invariant. Since $\Gamma$ is irreducible, the image of $\Gamma$ in $G/G'$ is dense, which implies that every $G'$-orbit on $G/\Gamma$ is dense; since $G'$ is generated by unipotent elements, Ratner's measure rigidity theorem (see for example \cite[Theorem 9]{MR1262705}) implies that the Haar measure on $G/\Gamma$ is the only $G'$-invariant, Borel probability measure on $G/\Gamma$. Thus $\mu$ projects to the Haar measure on $G/\Gamma$ and has the other properties asserted in \cref{mainthm2}.
\end{remark}

\subsection{Lyapunov exponents under averaging}
It is well known that the top Lyapunov exponent of a bounded continuous linear cocycle is an upper-semicontinuous function on the space of invariant measures (see e.g.\ \cite[Lemma 4.2]{BFH} or \cite[Lemma 9.1]{viana2014lectures}). For cocycles whose norm is tempered as in \cref{tempered}, upper-semicontinuity of the top exponent holds when restricted to families of invariant measures with uniformly exponentially small mass at $\infty$; see \cite[Lemma 3.8]{BFH-SLnZ}. This implies the following two results that we frequently use.

\begin{lemma}[{\cite[Lemma~3.9]{BFH-SLnZ}}] \label{lemma:averaging}
Let $a\in G$ and let $\mu$ be an $a$-invariant Borel probability measure on~$X$ with exponentially small mass at $\infty$.
For any amenable subgroup $H\subset C_G(a)$ and any \Folner sequence of precompact sets $\{F_n\}$ in~$H$, if the family
$\{F_n \ast \mu\}$ has uniformly exponentially small mass at $\infty$ then for any subsequential limit
$\mu_\infty$ of $\{F_n \ast \mu\}$, the measure $\mu_\infty$ is $a$-invariant and
 $$\lambda_{\top, a, \mu,\calA} \le \lambda_{\top, a, \mu_\infty, \calA}.$$
\end{lemma}

Averaging over \Folner sequences in $\calZ A$ and applying \cref{lemma:averaging} yields the following. 

\begin{corollary}\label{lem:endgameA} Let $\{a^t\}$ be a 1-parameter, $\R$-diagonalizable subgroup of~$G$ and let $\mu_0$ be an $\{a^t\}$-invariant Borel probability measure on $X$ projecting to the Haar measure on $G/\Gamma$ such that $\lambda_{\top,a^1,\mu_0, \calA} > 0$.
Then there is a Borel probability measure~$\mu$ on~$X$, such that
	\begin{enumerate}
	\item $\mu$ projects to the normalized Haar measure on $G/\Gamma$,
	\item there is a split Cartan subgroup~$A$ of~$G$ that contains~$\{a^t\}$~such that $\mu$ is $(\calZ A)$-invariant,
	and
	\item $ \lambda_{\top,a^1,\mu, \calA} > 0$.
	\end{enumerate}
\end{corollary}

\subsection{Lyapunov exponents for pre-compact orbits}\label{LyapComp}
The above lemmas are needed to control escape of mass and unboundedness of the cocycle when averaging. For orbits that remain in a compact set, we obtain the following results. The proofs of each are straight-forward modifications of the arguments in \cite [Proposition~4.6]{BFH}.

\index{Split cartan of reductive group}
In \cref{RealParabSect}, we defined the notion of a split Cartan subgroup of a connected semisimple Lie group in terms of a choice of Iwasawa decomposition. More generally, if $H$ is a connected reductive Lie subgroup of $G$, define a \emph{split Cartan subgroup} $A_H$ of $H$ to be a maximal, connected, abelian $\ad$-$\R$-diagonalizable subgroup. All such subgroups are conjugate and these definitions coincide for semisimple Lie subgroups. Moreover, there is a split Cartan subgroup $A\subset G$ such that $A_H = H\cap A$.

\begin{lemma} \label{GetExpForAOnCpct}
Let $\hat H\subset F$ be a connected reductive $\R$-subgroup such that $\hat \Gamma\!_{\hat H}$ is cocompact in~$H$ and let $H= (\phi\inv \hat H)^\circ$ be the associated subgroup of $G$.
Let $A_H$ be a split Cartan subgroup of~$H$.

Suppose the restriction $\restrict{\calA_0}{\calZ}$ of the cocycle $\calA_0$ to the center $\calZ$ of $G$ has $\len_\Gamma$-subexponential growth.
If the restriction $\restrict{\calA_0}{\Gamma\!_H}$ fails to have   $\len_\Gamma$-subexponential growth then there is an $ A_H$-invariant Borel probability measure~$\mu$ on~$X_H$ such that
$$\lambda_{\top,a,\mu,\calA} >0$$ for some $a \in A_H$.
\end{lemma}
\begin{proof}[Proof outline]
Fix a Cartan involution $\theta$ of $G$ which leaves $H$ invariant. We may assume the norm on $\calE$ is constructed relative to $\theta$.
Let $\wtd K\subset G$ be the subgroup of $\theta$-fixed points.
From \cref{moregrowth} there is $C_0>0$ such that for any compact $\calK\subset \wtd K$, all $x\in X_H$, and all $k\in \calK$,
$$ \|\calA(k,x)\|\le C_0$$
and, given $\epsilon>0$, there is $C_\epsilon$ such that for all $x\in X_H$ and $z\in \calZ$,
$$\|\calA(z,x)\|\le C_\epsilon e^{\epsilon \len_\Gamma(z)}.$$

As discussed in \cref{kakaReductive},
given $\gamma\in \Gamma\!_H$, we write $$\gamma = k'zak''$$ where $a\in A_H = A\cap H$ for some split Cartan $A\subset G$, $z\in (\calZ\cap H)$, $k',k''$ are in a bounded set,
 and
$\len_\Gamma(z)$ and $d(a,\1)$ are bounded above by linear functions of $\len_\Gamma(\gamma)$.
The assumption that $\restrict{\calA_0}{\Gamma\!_H}$ fails to have uniform $\len_\Gamma$-subexponential growth implies there exists $\kappa>0$ and sequences $\gamma_n\in \Gamma\!_H$ with $\len_\Gamma (\gamma_n)\to \infty$ and $x'_n\in X_0$ such that for all $n$,
$$\log \|\calA_0(\gamma_n,x'_n)\| \ge \kappa \len_\Gamma(\gamma_n).$$

Write $\gamma_n = k'_nz_na_nk''_n$. There is $C>0$ such that for all sufficiently small $\epsilon>0$,
for all sufficiently large $n$,
\begin{align*}
	\log \|\calA(a_n,x_n)\|
	&\ge \frac 1 C 	\log \|\calA_0(\gamma_n,x'_n)\| - \log \|\calA(z_n,x_n)\| -C\\
	&\ge \frac 1 C 	\log \|\calA_0(\gamma_n,x'_n)\| - \epsilon \len_\Gamma(z_n) - C_\epsilon -C\\	
	&\ge \frac 1 C 	 (\kappa-\epsilon) \left[ \len_\Gamma(\gamma_n) - \frac 1 C \len_\Gamma(\gamma _n )\right] - C_\epsilon -2 C\\	
	&\ge \frac 1 {C^2} 	 (\kappa-\frac 1 C \epsilon) \left[ d(a_n,\1) \right] - C_\epsilon -2 C.
\end{align*}
In particular, taking $\epsilon>0$ sufficiently small, for all sufficiently large $n$,
$$	\log \|\calA(a_n,x_n)\| \ge \frac \kappa 2 d(a_n,\1)$$
and the existence of the measure $\mu$ and $a\in A_H$ with the desired properties follows exactly as in the proof of \cite [Proposition~4.6]{BFH}.
\end{proof}

When the restriction $\restrict{\calA_0}{\calZ}$ fails to have $\len_\Gamma$-subexponential growth, we have the following which will easily lead to the desired conclusion in \cref{CentralSubExp}.
\begin{lemma} \label{GetExpOnCenter}
If the restriction $\restrict{\calA_0}{\calZ}$ fails to have $\len_\Gamma$-subexponential growth then there is a $\calZ$-invariant Borel probability measure~$\mu$ on~$X_0$ such that
$$\lambda_{\top,a,\mu,\calA_0} >0$$ for some $a\in \calZ.$
\end{lemma}
\begin{proof}[Proof sketch]
Recall that the group $\calZ$ is finitely-generated abelian and quasi-isometrically embedded in $\Gamma$. We may average the norm on $\calE_0$ over the action of all torsion elements so that every torsion element acts isometrically on $\calE_0$. Let $\{a_1, \dots, a_r\}$ generate the free abelian subgroup of $\calZ$. If the restriction $\restrict{\calA_0}{\langle a_i \rangle}$ of $\calA_0$ to the subgroup generated by each $a_i$ has $\len_\Gamma$-subexponential growth then $\restrict{\calA_0}{\calZ}$ has $\len_\Gamma$-subexponential growth. We may thus assume there is $a_i$ such that the restriction $\restrict{\calA_0}{\langle a_i \rangle}$ fails to have $\len_\Gamma$-subexponential growth. Replacing $a_i$ with $a_i\inv$ if needed, a standard exercise (see for example \cite[Proposition 3.1.2]{Brown}) yields an $a_i$-invariant, Borel probability measure $\mu_0$ on $X_0$ such that $$\lambda_{\top,a_i,\mu_0,\calA_0} >0.$$
Averaging $\mu_0$ over a \folner sequence in $\calZ$, \cref{lemma:averaging} yields a measure $\mu$ with the desired properties.
\end{proof}

A softer version of the proof of \cite[Proposition~4.6]{BFH} (see for example \cite[Proposition 3.1.2]{Brown}) also establishes the following.
\begin{lemma}\label{ExpForLimOfEmpiricals}
Let $\{a^t\}$ be a one-parameter subgroup of~$G$, let $E\subset X$ be a compact subset, and let $\{x_n\}$ be a sequence of points in $X$ such that
	\begin{enumerate}
	\item $ \limsup_{n \to \infty} \frac{1}{n} \log \| \calA(a^n, x_n) \| > 0$, and
	\item $a^t\cdot x_n\in E$ for all $0\le t\le n$ and $n\in \N$.
\end{enumerate}	
Then there exists an $\{a^t\}$-invariant, Borel probability measure $\mu$ supported on $E$ with
	$$\lambda_{\top,a^1,\mu,\calA} >0 .$$
\end{lemma}

\subsection{The projectivized cocycle and its infinitesimal generator}\label{s4.1}\index{projectivized cocycle}  \index{ infinitesimal generator of projectivized cocycle}
 Let $\P \calE\to X$ denote the projectivization of the vector-bundle $\calE\to X$.
 We represent a point $\xi\in \P\calE$ as $\xi = (x,[v])$ where $[v]$ is an equivalence class of non-zero vectors in the fiber $\calE(x)$.
The $G$-action on $\calE$ by vector-bundle automorphisms induces a natural $G$-action on $\P \calE$ which restricts to a projective transformation between each fiber and its image.

\index{}
Let $\Psi\colon G\times \P \calE\to \R$ be the function
	$$\Psi(g, \xi) = \Psi \bigl( g, (x,[v]) \bigr) = \log\left(\frac{\|\calA(g, x)(v)\|} {\|v\|}\right).$$
We have that $\Psi$ is an additive cocycle over the $G$-action on $\P \calE$:
	\begin{align} \label{cocycle}\Psi(g_1g_2,\xi) = \Psi( g_1,g_2\cdot \xi) + \Psi(g_2,\xi). \end{align}
In particular, 	$\Psi(a^n,\xi) = \sum_{k=0}^{n-1} \Psi(a, a^k \cdot \xi).$

Write $\Psi'\colon \lieg \times \P\calE\to \R$ for the {infinitesimal generator} of~$\Psi$: given $Y\in \lieg$,
	\begin{equation}\label{eq:infgen}
	\Psi'(Y,\xi) = \lim _{t\to 0} \frac{\Psi \bigl( \exp (tY), \xi \bigr) }{t}
	. \end{equation}
The function $\Psi'$ exists by the construction of the norm $\|\cdot\|$ on~$\calE$; see \fullcref{SeveralObservations}{smooth}.
The following properties hold for $\Psi'$:
	\begin{enumerate}
	\item $\Psi'(Y,\cdot)\colon \P\calE\to \R$ is continuous;
	\item $\Psi'(sY,\xi)= s\Psi'(Y,\xi)$;
	\item \label{Psi=IntPsi'}
	$\Psi \bigl( \exp(TY),\xi \bigr) = \int_0^T \Psi' \bigl( Y,\exp(sY) \cdot \xi \bigr) \ d s$.
	\end{enumerate}
Moreover, if $a^t = \exp(tY)$ and if $\td \mu$ is an $\{a^t\}$-invariant, Borel probability measure on $\P\calE$ then
\begin{enumerate}[resume]
	\item $ \int \Psi' \bigl( Y,\cdot ) \ d \td \mu(\cdot ) = \int \Psi \bigl( a,\cdot ) \ d \td \mu(\cdot ) .$
	\end{enumerate}

Note that $\P\calE\to X$ is a compact extension. If $a\in G$ and if $\mu$ is an $a$-invariant, Borel probability measure on~$X$ then there exists at least one $a$-invariant, Borel probability measure on~$\P \calE$ projecting to $\mu$.
We have the following claim which relates the top Lyapunov exponent of $\calA$ with the cocycle $\Psi$ and such measures on $\P\calE$.
\begin{claim}[{See e.g.\ \cite[Theorem 6.1]{viana2014lectures}}] \label{Lyap>Phi}
Let $a\in G$  and assume that $\mu$ is an $a$-invariant, Borel probability measure on~$X$ with exponentially small mass at $\infty$. Then:
\begin{enumerate}
\item \label{Lyap>Phi-ineq}\label{baboonfeces}
for any $a$-invariant Borel probability measure~$\wtd \mu$ on $\P\calE$ projecting to $\mu$ we have
$$\int \Psi(a, \cdot ) \ d \wtd \mu \le \lambda_{\top,a,\mu, \calA};$$
\item \label{Lyap>Phi-eq}
there exists an $a$-invariant Borel probability measure~$\wtd \mu$ on $\P\calE$ projecting to $\mu$ such that
$$\int \Psi(a, \cdot ) \ d \wtd \mu = \lambda_{\top,a,\mu, \calA}.$$
\end{enumerate}
\end{claim}

\section*{Review of notation and standing hypotheses} \label{Review}
Before we begin the proof of \cref{mainthm2,thm:whatweprove}, we review the standing hypotheses and notation that will be used in the remainder of the paper. Let $\K\in \{\R,\Q\}$.
\begin{enumerate}
	\item $G$ is a connected simply connected  semisimple Lie group with real rank at least 2 and without compact factors.
	\item $\calZ$ is the center of $G$.
	\item $\bfF$ is an algebraically simply connected, $\Q$-simple, semisimple algebraic $\Q$-group. The Lie group $F= \bfF(\R)$ is connected.
	\item $\phi\colon G\to F$ is a surjective morphism with $\ker(\phi)\subset \calZ$.
	\item $\Gamma\subset G$ is a nonuniform irreducible lattice subgroup of $G$ and $\calZ\subset \Gamma$.
	\item The image $\hat \Gamma := \phi(\Gamma)$ is a lattice subgroup of $F$ commensurable with $F_\Z$; all torsion elements in $\hat \Gamma$ are central.
	\item Given a connected subgroup $H\subset F$, let $\wtd H:= (\phi\inv (H))^\circ $ be the connected subgroup of $G$ projecting to $H$. Write $$\Gamma\!_H = \Gamma\!_{\wtd H}:=\wtd H\cap \Gamma$$ and $$\hat \Gamma\!_H:= \hat \Gamma\cap H.$$ If $H= \bfH(\R)^\circ$ for some $\Q$-subgroup $\bfH$ we say that $\wtd H$ is defined over $\Q$.
	
	\item A connected subgroup $H\subset G$ is a parabolic $\K$-subgroup if $\phi (H) = Q^\circ$ where $Q$ is a parabolic $\K$-subgroup of $F$.
	\item A connected subgroup $U\subset G$ is unipotent if $\phi(U)$ is a unipotent subgroup of $F$.
	\item An element $a\in G$ is $\R$-diagonalizable if $\phi(a)$ is $\R$-diagonalizable.
	\item Given a connected subgroup $A\subset G$, we say $A$ is a $\K$-torus if there is an $\K$-torus $\bfS\subset \bfF$ with $\phi(A) = \bfS(\R)^\circ.$
	
	If $A\subset G$ is an $\R$-torus then $A$ decomposes uniquely as $A= A'A''$ where $\phi(A')$, (resp.\ $\phi(A'')$) is a maximal $\R$-split, r(esp.\ $\R$-anisotropic) subgroup of $\phi(A)$.

	Given a one-parameter subgroup $\{a^t\}$ of $A$ write $a^t= a^t_{\comp}a^t_{\split}$ where $a^t_{\comp}\in A''$ and $a^t_{\split}\in A'$.

\item We let $\Gamma$ act by homeomorphisms on $X_0$. We assume this action lifts to a cocycle $\calA\colon \Gamma\times \calE_0\to \calE_0$ on a finite-dimensional normed vector bundle $\calE_0$ over $X_0$. 	
\item We let $X := (G \times X_0)/\Gamma$ be the suspension space with the $G$-action induced from the $\Gamma$-action on $X_0$. The space $X$ has the structure of a fiber bundle over $G/\Gamma$ and the fibers are homeomorphic to~$X_0$.

\item Write $\pi_\calE \colon \calE \to X$
 for the induced vector bundle $\calE = (G\times \calE_0)/\Gamma$   induced by $\calE_0\to X_0$. For any closed subgroup $H$ of $G$, let $\calA \colon H\times \calE_H \to \calE_H$ be the restriction of the cocycle to $(H\times X_0)/\Gamma_H$. We equip with $\calE$ with a norm as constructed in \cref{sec:norms}.

\item In the setting of \cref{thm:whatweprove}, we have $X_0= M$, $\calE$ is the fiberwise tangent bundle of $X= (G\times M)/\Gamma$, and $\calA\colon \Gamma\times TM\to TM$ is the cocycle induced by the derivative cocycle $\calA_0(\gamma,x)= D_x\alpha(\gamma)$.
\end{enumerate}

\section{Subexponential growth for  the  center} \label{Center}
We begin the proof of \cref{thm:whatweprove,mainthm2}. In the following simple proposition
we assume the failure of subexponential growth occurs for the restriction ${\calA_0}\colon \calZ \times \calE_0 \to\calE_0$ of the cocycle $\calA_0$ to the center $\calZ$ of $G$.
 In the sequel, when considering the restriction of the cocycle to other subgroups of $\Gamma$ we will quote the proposition to assume the center only contributes subexponential growth to the cocycle.
\begin{proposition} \label{CentralSubExp}
\laterdef\CentralSubExp{
Suppose 	$$\restrict{\alpha}{\calZ}\colon \calZ\to \Diff^1(M)$$
fails to have uniform a$-subexponential growth of derivatives.
Then there exists a split Cartan subgroup~$A$ of~$G$ and a $(\calZ A)$-invariant Borel probability measure~$\mu$ on~$X$ projecting to the Haar measure on $G/\Gamma$ such that
	 $\lambda_{\top,a,\mu, \calA} > 0$ for some $a \in \calZ$.}

Suppose the restriction 	$${\calA_0}\colon \calZ \times \calE_0 \to\calE_0$$
fails to have $\len_\Gamma$-subexponential growth.
Then there exists a split Cartan subgroup~$A$ of~$G$ and a $(\calZ A)$-invariant Borel probability measure~$\mu$ on~$X$ projecting to the Haar measure on $G/\Gamma$ such that
	 $\lambda_{\top,a,\mu, \calA} > 0$ for some $a \in \calZ$.
\end{proposition}
\begin{proof}
Since $\calZ\subset \Gamma$, the induced action of $\calZ $ on $G/\Gamma$ is the identity action. It follows from \cref{GetExpOnCenter} that there is a Borel probability measure $\mu_0$ on $X$ projecting to the delta mass $\delta_{\1\Gamma}$ on $G/\Gamma$ such that
$\lambda_{\top,a_0,\mu_0, \calA} > 0$ for some $a_0 \in \calZ$.

Fix any split Cartan subgroup $A\subset G$ and a one-parameter subgroup $\{b^s\}$ in $A$. Let $U=\mathcal U^+(b^s) $ be the expanding horosperical subgroup associated with $\{b^s\}$. Note that $A$ and $U$ trivially commute with $\calZ$. Averaging $\mu_0$ over an appropriate \folner sequence in $U$, by \cref{UnipSmallCusp,lemma:averaging} we obtain a $(\calZ U)$-invariant, Borel probability measure $\mu_1$ on $X$ with exponentially small mass at $\infty$ and $\lambda_{\top,a_0,\mu_1, \calA} > 0$.
Since $\mu_0$ is supported over the identity coset $\1\Gamma$ in $G/\Gamma$, $\mu_1$ is compactly supported.
 By \cref{MixingSmallCuspsCom,lemma:averaging} the sequence of measures $\{(b^s)_* \mu_1: s\in \N\}$ is uniformly tight; moreover, any subsequential limit $\mu_2$ of this sequence is $\calZ$-invariant, projects to the Haar measure on $G/\Gamma$, and satisfies $\lambda_{\top,a_0,\mu_2, \calA} > 0$.
 Finally, averaging $\mu_2$ against an appropriate \folner sequence in $A$, by \cref{lemma:averaging} we obtain a $(\calZ A )$-invariant, Borel probability measure $\mu$ on $X$ projecting to the Haar measure on $G/\Gamma$ with $\lambda_{\top,a_0,\mu, \calA} > 0$.
\end{proof}

\subsection{Additional hypothesis for \cref{UnipotentSubgroups,SubexpForParabolic,SubexpForRank1Subgroups}}\label{sec:hyp:center}
Using \cref{CentralSubExp}, we will assume in \cref{UnipotentSubgroups,SubexpForParabolic,SubexpForRank1Subgroups} below the following additional hypothesis:
\begin{assumption}\label{hyp:center}
The restriction $${\calA_0}\colon \calZ \times \calE_0 \to\calE_0$$ of the cocycle $\calA_0$ to the center $\calZ$ has $\len_\Gamma$-subexponential growth.
\end{assumption}

\section{Subexponential growth for unipotent subgroups} \label{UnipotentSubgroups}
In this section, we assume that the failure of subexponential growth occurs for the restriction of the action to some unipotent $\Q$-subgroup of~$\Gamma$.
In \Cref{SubexpForParabolic,SubexpForRank1Subgroups}, we then consider situations where the failure of subexponential growth occurs for the actions of certain other subgroups of~$\Gamma$.

Under the additional hypothesis in \cref{hyp:center} we have the following.
\begin{proposition} \label{UnipSubgrp}
\laterdef\UnipSubgrp{
If there is a connected unipotent $\Q$-subgroup $U$ of $G$ such that the restriction
	$$\restrict{\alpha}{\Gamma\!_U}\colon \Gamma\!_U \to \Diff^1(M)$$
	does not have uniform subexponential growth of derivatives then there is a split Cartan subgroup~$A$ of~$G$ and an $A$-invariant Borel probability measure~$\mu$ on~$X$ projecting to the Haar measure on $G/\Gamma$ such that
	 $\lambda_{\top,a,\mu, \calA} > 0$ for some $a \in A$.}
	 If there is a connected unipotent $\Q$-subgroup $U$ of $G$ such that the restriction
	$$\calA_0\colon \Gamma\!_U\times \calE_0 \to\calE_0$$
does not have {$\len_\Gamma$-subexponential growth}
then there is a split Cartan subgroup~$A$ of~$G$ and an $A$-invariant Borel probability measure~$\mu$ on~$X$ projecting to the Haar measure on $G/\Gamma$ such that
	 $\lambda_{\top,a,\mu, \calA} > 0$ for some $a \in A$.
\end{proposition}

The remainder of this \lcnamecref{UnipotentSubgroups} is devoted to the proof of this \lcnamecref{UnipSubgrp}.
We will prove the contrapositive of \cref{UnipSubgrp}; in particular, we will assume the following:
\begin{assumption}\label{assfortoday}
For every split Cartan subgroup~$A$ of~$G$, every $A$-invariant Borel probability measure~$\mu$ on~$X$ that projects to the Haar measure on $G/\Gamma$, and every $a \in A$, $$\lambda_{\top,a,\mu, \calA} = 0.$$
\end{assumption}

We recall the subadditive cocycles  $\calB^\pm\colon G\times G/\Gamma \to \R$ defined by \eqref{eq:pooface1} and \eqref{eq:pooface2} in \cref{TemperedSect}.
Under \cref{assfortoday}, we will show that the restrictions of $\calB^\pm$ to $U\times U/\Gamma\!_U$ are sublinear: for every $\epsilon>0$, there is $C_\epsilon$ such that for all $g\in U$ and $x\in U/\Gamma\!_U$,
$$\calB^\pm (g,x) \le \epsilon d(g,\1) + C_\epsilon.$$
Since $ \Gamma\!_U$ is cocompact in $U$, by \cref{SubexpForCompact,LMR}, this is equivalent to the following:
for every $\epsilon>0$ there is $C_\epsilon$ such that for all $\gamma\in \Gamma\!_U$,
$$\calB^\pm (\gamma,\1\Gamma) \le \epsilon d(\gamma,\1) + C_\epsilon.$$

To start the proof of \cref{UnipSubgrp},
we consider various $1$-parameter subgroups $\{a^t\}$ that normalize $U$.  For such $\{a^t\}$, we let $U^+\subset U$ be the expanding subgroup of $U$.   We will show, roughly speaking, that the restriction of $\calB^\pm$ to a subgroup $U^+$  of $U$ has sublinear growth and then deduce sublinear growth for all of~$U$ from this. However, there is a technical issue that we will actually not see immediately that the restriction of $\calB^\pm$ to $U^+$ itself has sublinear growth, but only that each element of~$U^+$ has the same image in $U/[U,U]$ as some element with sublinear growth.

Sublinear growth of the cocycle for such representatives of elements of $U^+$ in $U/[U,U]$ will be deduced from sublinear growth for generic trajectories for the 1-parameter subgroup $\{a^t\}$.
The basic idea is simply that if $\{a^t\}$ has sublinear growth and if $u_0$ is a fixed element of~$U^+$, then $u := a^t u_0 a^{-t}$ is a large element of~$U$ whose norm for the associated cocycle is sublinear in $d(u,\1)$. From  \Cref{ExpandersGenerate}, by varying $\{a^t\}$ we can then generate all of $U/[U,U]$ by  the associated expanding subgroups.   A difficulty is that sublinear growth of~$\{a^t\}$ only holds along orbits of a full measure subset of base points whereas, for elements of $U$, we need a uniform estimate that holds at every point of $U/\Gamma\!_U$. To upgrade from almost every point to every point, we work in the abelianization $U/[U,U]$ and use that sumsets of sets with high density in balls cover balls of comparable radii.

In \cref{sec81,sec82} we make precise the outline in the preceding paragraphs in the proof of \cref{UnstableIsSubexp}. We then use \cref{UnstableIsSubexp} to establish \cref{UnipSubgrp} in \cref{sec83}.

\subsection{Subexponential growth for generic \texorpdfstring{$\{a^t\}$-}{}orbits}\label{sec81}
Using \cref{GenByMaxParab}, it is sufficient to consider the case that $U$ is the unipotent radical of some maximal parabolic $\Q$-subgroup~$Q$.
We fix a choice of rational Langlands decomposition of $Q$, $$Q = (L \times S^\circ) \ltimes U.$$ Since $Q$ is a maximal $\Q$-parabolic, we have $\dim S^\circ = 1$.

To begin the proof of \cref{UnipSubgrp}, we fix the following notation:
\begin{enumerate}
	\item
	Let $\wtd A$ be a connected, $\Q$-anisotropic $\Q$-torus~$\wtd A$ of~$L$. More precisely, the image of $\wtd A$ in $F$ is a connected, $\Q$-anisotropic $\Q$-torus.
	\item \label{GenericNotation-at}
	Let $\{a^t\}$ be a one-parameter subgroup of~$\wtd A$ such that $a^1 \in \Gamma\!_{\wtd A}$.
	\item Let $a_\Z := \{a^n \mid n \in \Z\}$. Then $\Lambda:=a_\Z \Gamma\!_U$ is a cocompact lattice in $\{a^t\} U$.
	\item 	
	We write $\lieu^+, \lieu^0, \lieu^-$ for the expanding, centralized, and contracting directions of the action of $\Ad(a^1_{\split})$ on $\lieu$. More precisely, $\lieu^+, \lieu^0, \lieu^-$ are the sum of the eigenspaces corresponding to eigenvalues greater than~$1$, equal to~$1$, and less than~$1$, respectively.
	\item Each of $\lieu^+, \lieu^0, \lieu^-$ is a Lie subalgebra of~$\lieg$. Let $U^+, U^0,$ and $U^-$ denote the corresponding analytic subgroups of~$U$. Note that:
		\begin{enumerate}
		\item $U^0 = C_U (a^t_{\split})$ is the centralizer of~$\{a^t_{\split}\}$ in~$U$,
		\item each of the three subgroups is normalized by both $\wtd A$ and~$U^0$ (since $\wtd A$ and~$U^0$ each centralize $\{a^t_{\split}\}$ and normalize~$U$),
		and
		\item \label{GenericNotation-UUU}
		 $U^- U^0 U^+ = U$.
		\end{enumerate}
	\end{enumerate}
From \Cref{ExpandersGenerate} we may choose several one-parameter subgroups $\{a^t\}$ as above, such that the corresponding expanding horospherical subgroups~$U^+$ generate $U/[U,U]$. In particular, we will always consider $\{a^t\}$ as above such that $a^1_{\split}$ is non-trivial.

\begin{lemma} \label{MostGood}
Suppose \cref{assfortoday,hyp:center} hold and let $\{a^t\}$ be as above. Then for $m_U$-almost every $u\in U$ we have
\begin{equation}\label{MostGoodLimit}
\lim_{n\to \infty} \frac 1 n \calB^\pm(a^n, u \, \Gamma\!_U) = 0.
\end{equation}
\end{lemma}

\begin{proof}
Let $U^\pm$ denote the smallest $\Q$-subgroup of~$U$ that contains both $U^+$ and~$U^-$ and
let $A'/\Gamma\!_{A'}$ be the closure of the image of $\{a^t_{\split}\}$ in $\wtd A/\Gamma\!_{\wtd A}$;
that is, $A'$ is the smallest $\Q$-torus of $\wtd A$ containing $\{a^t_{\split}\}$.
We remark that $A'$ is a connected closed subgroup of~$\wtd A$ and that $\Gamma\!_{A' U}$ is cocompact in $A' U$. Let $\quot{\phantom{x}} \colon A'U \to A' U/\Gamma\!_{A' U}$ denote the natural quotient map.
Then $\{a^t_{\split}\}$ acts ergodically with respect to the Haar measure on $A' U^\pm/\Gamma\!_{A' U^\pm}$; see for example \cite[Theorem~6.1]{MR0618325}.
Since $U^0$ centralizes~$\{a^t_{\split}\}$ and since $U^0 U^\pm = U$, almost every $\{a^t_{\split}\}$-orbit in $\quot{A' U}$ equidistributes to the Haar measure on some translate $\quot{u_0 A' U^{\pm}}$ of $A' U^\pm/\Gamma\!_{A' U^\pm}$ in $\quot{ A' U}$ for some $u_0\in U^0$.

Let $\calG_0$ denote the set of points in $\quot{A' U}$ whose $\{a^t_{\split}\}$-orbit equidistributes to the Haar measure on such a translate. Let $\calG = \{\, u \in U \mid u \Gamma \in \calG_0\,\}$.
As discussed above, $\calG_0$ has full measure in $\quot{A' U}$. Since $A'$ centralizes~$\{a^t_{\split}\}$ and normalizes $U$ and $U^\pm$, the subset $\calG_0$ is $A'$-invariant which implies that $\calG$ is conull in~$U$.

To complete the proof, we show that \eqref{MostGoodLimit} holds for $u\in \calG$. Suppose that \eqref{MostGoodLimit} fails for some $u\in \calG$.
We note that $\{ a^t_{\comp}\}\calZ$ is contained in a compact subgroup of $\Ad(G)= G/\calZ$. The torus $\wtd A$ is quasi-geodesically embedded in $G$ and $\calZ$ is quasi-isometrically embedded in $G$.
The hypothesis in \cref{hyp:center} then implies for any $\epsilon>0$ there is $C$ such that for all $x$ in the compact set $\quot{\wtd A U}$,
 $$\calB^\pm(a_\comp^n, x) \le C + n\epsilon.$$
Since the $\{a^t_{\split}\}$-orbit of~$u \Gamma$ is contained in the compact set $\quot{u_0 A' U^\pm}\subset \quot{A' U}$ for some $u_0\in U^0$, we have that $$\limsup_{n \to \infty} \frac{1}{n} \calB^\pm(a_{\split}^n, u \Gamma) > 0.$$
Applying \cref{ExpForLimOfEmpiricals} to the subgroup $\{a^t_{\split}\}$,
there is an $\{a^t_\split\}$-invariant Borel probability measure $\nu$ supported on $X$ which projects to the Haar measure on the translate $\quot{u_0 A' U^\pm}$, such that
	$\lambda_{\top,a_{\split}^1,\nu,\calA} >0 $.
	We remark that since $U^0$ normalizes $U^{\pm}$, the projection $\hat \nu$ of $\nu$ to $G/\Gamma$ is $U^{\pm}$-invariant.

Let $\{F^0_j\}$ be a \folner sequence of precompact sets in~$U^0$. Observe for each $j$ that $F^0_j \ast \nu$ is supported on the subset of $X$ projecting to the compact set $\quot{U A' U^\pm}= \quot{A' U}$. Let $\nu^0$ be a subsequential limit of $\{F^0_j \ast \nu\}$.
Then $\nu^0$ is $U^0$-invariant.
In addition, since $U^0$ normalizes both $\{a_{\split}^t\}$ and~$U^\pm$, the image $\hat \nu_0$ of $\nu_0$ under the projection $X\to G/\Gamma$ remains invariant under both $\{a_{\split}^t\}$ and~$U^\pm$. In particular, $\hat \nu^0$ is invariant under $\{a_{\split}^t\} U^0 U^\pm = \{a_{\split}^t\} U$.
Furthermore, by \cref{lemma:averaging} we have $\lambda_{\top, a_{\split}^1, \nu^0, \calA} > 0$.

Recall that $U$ is the unipotent radical of a parabolic $\Q$-subgroup $Q$ and $S= \{b^s\}$ is a 1-parameter subgroup.
Also note that $S$ centralizes $\{a_{\split}^1\}$ and that $U=\mathcal U^+(b^s) $ is the expanding horosperical subgroup of $\{b^s\}$. Applying \cref{MixingSmallCuspsCom,lemma:averaging} to the sequence $\{(b^s)_*\nu_0: s\in \N\}$ and applying \cref{lem:endgameA} to any subsequential limit point $\nu_\infty$ of $\{(b^s)_*\nu_0: s\in \N\}$, we obtain a Borel probability measure $\mu$ on $X$ whose properties violate \cref{assfortoday}.
\end{proof}

\subsection{Subexponential growth for representatives of hyperbolic directions} \label{sec82}

In the remainder of this \lcnamecref{UnipotentSubgroups}, we write $\quot{\phantom{x}} \colon U \to U/[U,U]$ to denote the natural homomorphism from~$U$ to its abelianization.

We equip $U$ with a right-invariant metric $d_U$. This induces a metric on $\quot U= U/[U,U]$.
As $U$ is a unipotent subgroup of $G$ there is a constant $C_d > 1$ such that for all $u \in U$ with $d(u,\1) > 1$, we have
	$$\frac{1}{C_d}\log d_U(u, \1) \le d(u,\1) \le C_d\log d_U(u, \1) .$$
 By definition, $U\to \quot U$ is distance non-increasing.

The main lemma used to establish the inequality in \eqref{UnipSubgrpgood} is the following:
\begin{lemma}\label{UnstableIsSubexp}
There exists a constant $C_0 > 0$ such that for every $\epsilon>0$ there is~$C_{\epsilon}$ such that for every $x \in U/\Gamma\!_U$ and every $u \in U$ such that $\quot u \in \quot U^+$ there exists $v \in U$ such that
	\begin{enumlemma}
	\item $\quot v = \quot u$,
	\item $\calB^\pm(v, x) \le \epsilon \, d(u, \1) + C_{\epsilon}$, and
	\item $d( v, \1) \le C_0 d(u,\1) + C_0$.
	\end{enumlemma}
\end{lemma}

To prepare for the proof of this fundamental \lcnamecref{UnstableIsSubexp} we establish some additional notation.
For $u\in U$, write $u = wu^+$ where $u^+\in U^+$ and $w\in U^-U^0$; let $p^+ \colon U \to U^+$ be the natural map $u\mapsto u^+$. Similarly, write $\quot u = \quot w \quot u^+$ where $\quot u^+\in \quot U^+$ and $w\in \quot{U^-U^0}$ and let $\quot p^+ \colon U \to \quot U^+$ be the map $u\mapsto \quot u^+$. Observe $\quot p^+$ is a homomorphism.  For a subset $E \subset U$, let $E^+$ and~$\quot E^+$ denote, respectively, the image of~$E$ under the maps $p^+$ and $\quot p^+$. Recall that $m_{E}$ and $m_{\quot {E}^+}$ denote the Haar measures on~$U$ and~$\quot U^+$ normalized to give unit mass to $E$ and~$\quot {E}^+$, respectively.

\begin{claim} \label{GoodFolner}
There is a two-sided \folner sequence $\{F_k\}$ in~$U$ with the following properties:
	\begin{enumerate}
\item \label{GoodFolner-RN}
	For every Borel subset~$E$ of~$F_k$ and every $n \in \N$, $$ m_{a^n \quot {F_k}^+ a^{-n}}( a^n \quot E^+ a^{-n}) \ge m_{F_k}(E).$$
\item \label{GoodFolner-convex}
 Identifying the abelian group~$\quot U^+$ with some Euclidean space~$\R^\ell$, we may equip $\quot U^+$ with a norm which is invariant under conjugation by $\{a_{\comp}^t\}$ and for which $\quot{F_k}^+$ is the ball of radius~$k$ centered at the origin.
	\end{enumerate}
\end{claim}

\begin{proof} Identifying $U$ and $\{a^t\}$ with their images under the adjoint representation, equip $\quot U^+$ with a norm invariant under conjugation by $\{a_{\comp}^t\}$.
 Let $\quot{B_k}$ be the ball of radius~$k$ centered at the origin in~$\quot U^+$, let $U_0 = \ker \quot p^+$. For each $k$, let $B_k$ be a precompact Borel cross-section of~$\quot p^+$ over~$\quot{B_k}$; in particular, $\quot p^+(B_k) = \quot{B_k}$ and $v U_0 \cap B_k = \{v\}$ for all $v \in B_k$.
We may assume $1 \in B_k$.
Note that $\{\quot{B_k}\}$ is a two-sided \folner sequence in the abelian group~$\quot U^+$. We then construct a two-sided \folner sequence~$\{F_k\}$ in~$U$, such that
	$$ \text{$\quot{F_k}^+ = \quot{B_k}$ and $F_k = B_k (F_k \cap U_0)$,} $$
by induction on the nilpotency degree of~$U$ and the following observation: if $Z$ denotes the center of a locally compact group~$H$, if $\{F'_n\}$~is a \folner sequence in~$Z$, and if $\{F''_k\}$ is a sequence of precompact subsets of~$H$ such that the image of $\{F''_k\}$ in~$H/Z$ is a two-sided \folner sequence, then $\{F''_k \cdot F'_{n_k}\}$ is a two-sided \folner sequence in~$H$, provided that $n_k \to \infty$ sufficiently rapidly.

With a sequence $\{F_k\}$ constructed above, \pref{GoodFolner-convex}~holds by definition.

To show \pref{GoodFolner-RN}, we let $p \colon F_k \to \quot {F_k}^+$ be the natural map. Then for each $v \in B_k$,
	$$ p^{-1}(\quot v^+ ) \cap F_k
	= v (F_k \cap U_0) ,$$
so all fibers of~$p$ over $ \quot {F_k}^+ $ have the same measure with respect to~$m_{U_0}$, namely $m_{U_0} ( F_k \cap U_0 )$. This implies that the Radon-Nikodym derivative $d (p_*m_{F_k})/dm_{ \quot {F_k}^+ }$ is constant on~$\quot {F_k}^+$. Since both measures are normalized, the Radon-Nikodym derivative is~$1$.
For every $E \subseteq F_k $, this implies
	\begin{align*}
	m_{ \quot {F_k}^+ }(\quot E^+)
	= (p_* m_{F_k}) (\quot E^+)
	= m_{F_k } \bigl( p^{-1}(\quot E^+) \cap F_k \bigr)
	\ge m_{ F_k }(E)
	, \end{align*}
where the last inequality follows as $p^{-1}(\quot E^+) \cap F_k \supseteq E$.
Conjugation by~$a^n$ is a group automorphism and therefore does not change the relative measure of a subset within an ambient set, so
\begin{gather*} m_{a^n \quot {F_k}^+ a^{-n}}( a^n \quot E^+ a^{-n}) = m_{\quot {F_k}^+}( \quot E^+). \qedhere \end{gather*}
\end{proof}

To establish \cref{UnstableIsSubexp}, we first show that if $k$ and~$n$ are large, then a large proportion of the elements of $a^n \, \quot {F_k}^+ a^{-n}$ have the properties specified in the \lcnamecref{UnstableIsSubexp}.
We begin with the following elementary claim.
\begin{lemma}\label{whyleftfolner}
Let $m_{U/\Gamma\!_U}$ denote the normalized (left-)Haar measure on $U/\Gamma\!_U$.
Let $\rho\colon U\to U/\Gamma\!_U$ be the canonical projection and let $E\subset U$ be a right-$\Gamma\!_U$-invariant set such that $$m_{U/\Gamma\!_U}(\rho(E))\ge 1-\delta.$$
Let $\{F_k\}$ be a left-\folner sequence in $U$. Then for all sufficiently large $k$, $$m_{F_k}(E)\ge 1-2\delta.$$
\end{lemma}
\begin{proof}
Let $m_k := \rho_*(m_{F_k})$. Note that each $m_k$ is absolutely continuous with respect to $m_{U/\Gamma\!_U}.$ We claim that $m_k$ converges to $m_{U/\Gamma\!_U}$ in the strong topology: for any (Borel or Lebesgue) measurable $E\subset U/\Gamma\!_U$,
$$ m_k(E)\to m_{U/\Gamma\!_U}(E).$$

For the claim, fix a symmetric precompact neighborhood $O$ of the identity in $ U$. By left-F{\o}lnerness of $F_k$, for any $\delta>0$ there is $k_0$ such that for all $k\ge k_0$,
$$\sup_{g\in O}\frac{m_U(g\cdot F_k\triangle F_k)}{m_U(F_k)} \le \delta.$$
This implies for $k\ge k_0$ that $$\|g_*m_k - m_k\|_v \le \delta$$
where $\|\cdot \|_v$ denotes the total variation norm on signed Borel measures.

Note that $\{m_k\}$ has a weak-$*$ convergent subsequence. Any weak-$*$ limit point of $\{m_k\}$ is $U$-invariant and hence we conclude that $m_k\to m_{U/\Gamma\!_U}$ in the weak-$*$ topology.
We claim this convergence also holds in the strong topology. Indeed, let $\phi$ be a non-negative $C^\infty$ function supported in $O$ with $\int \phi (g) \ d g=1$.
Given a measurable $E\subset U/\Gamma\!_U$, let $h= \mathbbm{1}_E$.
The convolution $\phi\ast h$ is $C^0$ hence $$\int \phi\ast h \ d m_k\to \int \phi\ast h \ d m_{U/\Gamma\!_U}.$$
We note for any probability measure $\mu$ on $U/\Gamma\!_U$ that
\begin{align*}
\int (\phi\ast h)(x) \ d \mu(x)
&:=
\int \int \phi(g) h(g\inv \cdot x) \ d g \ d \mu(x)
\\&=
\int \int \phi(g) h(g\inv \cdot x) \ d \mu(x) \ d g
\\&=
\int \int \phi(g) h(x) \ d (g\inv _* \mu)(x) \ d g.
\end{align*}
In particular, by left-invariance of $m_{U/\Gamma\!_U}$, $$\int (\phi\ast h)(x) \ d m_{U/\Gamma\!_U}(x) = \int h \ d m_{U/\Gamma\!_U}.$$
On the other hand, \begin{align*}
	\bigg |\int \phi\ast h& \ d m_k - \int h \ d m_k \bigg|\\
	&=
	\left | \int \int \phi(g) h(x) \ d (g\inv _* m_k)(x) \ d g -\int h(x) \ d m_k (x) \int \phi(g) \ dg \right|
\\
	&\le \int \phi(g)
	\left |\int h(x) \ d (g\inv _* m_k)(x) -\int h(x) \ d m_k (x) \right|
	\ d g\\
	&\le \int \phi(g)
	 \| g\inv _* m_k -m_k\|_v
	\ d g\\
	&\le \delta.
\end{align*}
The claim then follows. The lemma follows as $\Gamma\!_U$-invariance of $E$ implies $m_{F_k}(E) = m_k (\rho(E)).$
\end{proof}

\begin{lemma}\label{lem:343}
For every $\epsilon>0$ and $\delta>0$, there exists $k \in \N$ and $C_\calG > 0$ such that for all sufficiently large $n$ there exists a set $G_n\subset \quot U^+$ such that:
\begin{enumerate}[label=(\alph*),ref=\alph*]
	\item \label{lem:343-subset}
	$G_n\subset a^n \quot{F_k}^+ a^{-n}$;
	\item \label{lem:343-delta}
	$m_{a^n \quot{F_k}^+ a^{-n}} (G_n) \ge 1-\delta$;
	\item \label{lem:343-v}
	for every $\quot u\in G_n$ and every $x\in U/\Gamma\!_U$ there is $v \in U$ such that
	\begin{enumerate}[label=(\roman*),ref=\roman*]
		\item \label{lem:343-subexpi} $\quot v = \quot u$;
		\item \label{lem:343-subexp}
		$\calB^\pm(v, x) \le \epsilon n$;
		\item \label{lem:343-subexpiii}$d(v, \1) \le C_\calG \, n$.
	\end{enumerate}
\end{enumerate}
\end{lemma}

\begin{proof}
Fix $\epsilon$ and~$\delta$. For $n_0 \in \N$, set
	$$ \calG(n_0) := \{\, u\in U\mid \calB^\pm(a^{\pm n}, u \Gamma\!_U) \le n\epsilon \text{ for all $n\ge n_0$} \,\}.$$
By \cref{MostGood}, for almost every $u\in U$ there exists $n_0$ such that $u \in \calG(n_0)$. Observe also that $\calG(n_0) \Gamma\!_U = \calG(n_0)$. Taking $n_0$ sufficiently large, we have
	$$m_{U/\Gamma\!_U} \bigl( \calG(n_0) / \Gamma\!_U \bigr) \ge 1- \delta/2.$$

Let $\fund_U$ be a compact, symmetric set that contains fundamental domains for both the right and left action of~$\Gamma\!_U$ on~$U$.
Fix a reference point $u_0\in \calG(n_0) \cap \fund_U$.
Consider $k\in \N$ and $n\ge n_0$. Set
	$$\calG_n:= \bigl( \calG(n_0) \, a^{-n} u_0^{-1} a^{n} \bigr) \cap F_k.$$
Since $\calG(n_0) \Gamma\!_U = \calG(n_0)$ and $\Gamma\!_U \fund_U = U$, we have
	$\calG(n_0) \, a^{-n} u_0^{-1} a^{n} = \calG(n_0) v_0$
for some $v_0 \in \fund_U$.
Then
	$$m_U( \calG_n ) = m_U \bigl( \calG(n_0) \cap (F_k v_0^{-1}) \bigr) .$$
From the lower bound on $m_{U/\Gamma\!_U} \bigl( \calG(n_0) / \Gamma\!_U \bigr)$, the fact that $\{F_k\}$ is a two-sided \folner sequence, \cref{whyleftfolner}, and using that $v_0 \in \fund_U$ is bounded, we conclude that if $k$ is sufficiently large, then for all $n\ge n_0$,
	$$m_{F_k}(\calG_n) \ge 1 - \delta .$$
Fix such $k\in \N$ for the remainder.

Take the set $G_n$ in the lemma to be the set
	$$G_n:= a^n \, \quot{\calG_n}^+ a^{-n}.$$
Since $\calG_n \subseteq F_k$, conclusion~\pref{lem:343-subset} of the \lcnamecref{lem:343} follows from definition.

From \fullcref{GoodFolner}{RN} we have
	$$m_{a^n \quot{F_k}^+ a^{-n}} (G_n)
	\ge m_{F_k} (\calG_n)
	\ge 1 - \delta .$$
This establishes~\pref{lem:343-delta}.

Consider an arbitrary element of $G_n$; by definition such an element is of the form $a^n \quot{u}^+ a^{-n}$ for some choice of $u \in \calG_n\subset F_k$. Given $x\in U/\Gamma\!_U$, select $u_x\in \fund_U$ with $u_x\Gamma\!_U = x$. Set
	\begin{equation}\label{eq:90h4klis}v = u_x u_0\inv a^n u^+ a^{-n} u_0 u_x\inv.\end{equation}

We claim that $v$ satisfies the conclusions in part~\pref{lem:343-v} of the statement of the \lcnamecref{lem:343}.
First note that, since $\quot U$ is abelian, we have $\quot v = a^n \quot{u^+} a^{-n} = a^n \quot{u}^+ a^{-n}$. Moreover, it is clear that $v \in (\fund_U)^2 (a^n F_k^+ a^{-n})(\fund_U)^2$; since $\fund_U$ and $F_k$ are precompact, this implies there are constants~$C_1$ and~$C_\calG$ such that
	$$d(v,\1) \le 2C_1 + n \, d(a, \1) + C_1 + n \, d(a, \1) + 2 C_1 \le C_\calG n .$$
This establishes \pref{lem:343-subexpi} and \pref{lem:343-subexpiii}

For \pref{lem:343-subexp}, recall that $a^1 \in \Gamma$ so $a^n \Gamma = \Gamma$. We then have
\begin{equation}\label{eq:di0kjfa} \begin{aligned}
u a^{-n} u_0 \Gamma
	&= u a^{-n} u_0 a^n \Gamma
	\in \calG_n a^{-n}u_0 a^n \Gamma
	\\&
	\subseteq \bigl( \calG(n_0) a^{-n} u_0^{-1} a^n \bigr) a^{-n} u_0 a^n \Gamma
		\\&
	= \calG(n_0)
	.\end{aligned}
\end{equation}
Since $u$ is in the precompact set~$F_k$, we may write $u^+ = w u$, where $w \in U^- U^0$ and the norm of~$w$ is uniformly bounded over the choice of $u\in \calG_n$.
Then we have $a^n u^+ a^{-n} = w_n (a^n u a^{-n})$ where the norm of $w_n := a^n w a^{-n}$ is bounded uniformly over the choice of~$n\in \N$ and $u\in \calG_n$. Since $u_x$ and $u_0$ are in the compact set~$\fund_U$, the element $v $ in \eqref{eq:90h4klis} is at bounded distance (uniform in the choice of $n$ and $u$) from $u_x u_0\inv a^n u a^{-n} u_0 u_x\inv$, respectively. Since the homogeneous space $\{a^t: t\in \R\}U/\bigl(\{a^n: n\in \Z\} \Gamma\!_U \bigr)$ is compact there is a constant~$C$ (depending only on the choice of $k$ and $\fund_U$) such that
\begin{align*}
\calB^\pm(v, x) & \le C + \calB^\pm(u_x u_0\inv a^n u a^{-n} u_0 u_x\inv, x) \\
&\le
C +
\calB^\pm(u_x u_0\inv, a^n u a^{-n} u_0 \Gamma)+
\calB^\pm( a^n, u a^{-n} u_0 \Gamma)\\
&\qquad +\calB^\pm( u , a^{-n} u_0 \Gamma)
+\calB^\pm( a^{-n} , u_0 \Gamma)
+\calB^\pm(u_0 u_x\inv, x)\\
&\le C + C + {\epsilon n} + C + \epsilon n + C \\
& \le 3 \epsilon n
\end{align*}
for all $n$ sufficiently large.
\end{proof}

We now prove \cref{UnstableIsSubexp}.

\begin{proof}[\bf Proof of \cref{UnstableIsSubexp}]
Fix $\epsilon>0$.
As in \fullcref{GoodFolner}{convex}, we identify $\quot U^+$ with some~$(\R^\ell,+)$. Let $0 < \delta < 2^{-(\ell + 1)}$ and let $k$ be as in \cref{lem:343} for this choice of~$\delta$ and $\epsilon$.

Consider any $\quot u\in \quot U^+$. Recall that $\quot{F_k}^+$ is the ball~$\quot B_k$ of radius~$k$ centered at the origin and that conjugation by~$a$ expands~$\quot U^+$.
Take the minimal $n$ such that $2 \quot u \in a^n \quot{F_k}^+ a^{-n}$. We may assume $n$ is sufficiently large so that \Cref{lem:343} holds. Let $G_n$ be as in
\Cref{lem:343}.

Write $\wtd F:= a^{-n} \quot{F_k}^+ a^n$.
Recall $\quot U^+$ is equipped with a norm so that $\quot{F_k}^+$ is a ball centered at the origin. Then
$\frac{1}{2} \wtd F\subseteq \wtd F$
 and $m_{\wtd F} \bigl( \frac{1}{2} \wtd F\bigr) = 2^{-\ell} > 2 \delta$. Since $m_{\wtd F} (G_n) > 1 - \delta$, this implies
	$$m_{\frac{1}{2} \wtd F} \bigl( G_n \cap \tfrac{1}{2} \wtd F\bigr) > \frac{1}{2} .$$
Since $\quot u \in \frac{1}{2} \wtd F$ we also have $\quot u - \frac{1}{2} \wtd F\subseteq \wtd F$. By the same argument,
	$$m_{\frac{1}{2} \wtd F} \bigl( ( \quot u - G_n) \cap \tfrac{1}{2}\wtd F\bigr) > \frac{1}{2} .$$
In particular, there is some $\quot{u'} \in \frac{1}{2} \wtd F\cap G_n \cap (\quot u - G_n)$. Then $\quot u - \quot {u'}\in G_n$ and
	$$\quot u = \quot{u'} + (\quot u - \quot{u'}) \in G_n+ G_n.$$
In particular, we may write any $\quot u = \quot{u_1} +\quot{u_2}$ with $\quot{u_i} \in G_n$.

Now consider any $x \in U/\Gamma\!_U$.
From \fullcref{lem:343}{v}, there exist $v_1,v_2 \in U$, such that $\quot{v_i} = \quot{u_i}$, $\calB^\pm(v_2,x) \le \epsilon n$, $\calB^\pm(v_1, v_2x) \le \epsilon n$, and $d(v_i,\1) \le C_\calG n$.
Let $v = v_1 v_2$. Then $\quot v = \quot{v_1} + \quot{v_2} = \quot{u_1} + \quot{u_2} = \quot u$.

Let $\lambda$ be the smallest eigenvalue of~$a^1_{\split}$ on~$\quot U^+$. Since the norm $\|\cdot \|$ on~$\quot U^+$ is chosen to be invariant under $a^t_{\comp}$, $a^i \quot{F_k}^+ a^{-i}$ contains the ball of radius $\lambda^i k$; minimality of~$n$ then implies $n \le \log_\lambda 2 \| \quot u \|$.
The map $U \to \quot U^+$ is a polynomial so there is a constant~$C$ such that $\log_\lambda 2 \|\quot u \| \le C d(u , \1)$ if $d(u,\1) > 1$. Then $n \le C d(u , \1)$ and we have
	$$ \calB^\pm(v,x)
	\le \calB^\pm(v_1,v_2x) + \calB^\pm(v_2,x)
	\le \epsilon n + \epsilon n
	\le 2 \epsilon C \, d(u, \1) $$
and
	\begin{align*}
	d(v, \1)
	\le d(v_1,\1) + d(v_2, \1)
	\le 2 C_\calG \, n
	\le 2 C_\calG C \, d(u , \1)
	. & \qedhere \end{align*}
\end{proof}

\subsection{Proof of \texorpdfstring{\cref{UnipSubgrp}}{Proposition 8.1}} \label{sec83}
We complete the proof of \cref{UnipSubgrp}. It suffices to show under \cref{assfortoday} that
for every $\epsilon>0$ there is $C_{\epsilon}$ such that for all $\gamma\in \Gamma\!_U$,
\begin{equation}\label{UnipSubgrpgood} \calB^\pm (\gamma,\1\Gamma)\le \epsilon \, d(\gamma,\1) + C_{\epsilon}.\end{equation}

We will first use \cref{UnstableIsSubexp} to show that each element of $\Gamma\!_U/[\Gamma\!_U,\Gamma\!_U]$ has a representative in $\Gamma\!_U$ for which $\calB^+$ and~$\calB^-$ have sublinear size when evaluated at the coset $\1\Gamma$.
This is the base case of an induction that establishes the same conclusion for every quotient of the descending central series of~$U$. To finish, we will observe that these estimates on the quotients easily imply the desired sublinear growth of $\calB^+$ and~$\calB^-$ evaluated at $\Gamma\!_U\times \{\1\Gamma\}$.

\begin{claim} \label{UnipSubgrpPf-Ubar}
There exists $C > 0$ such that for every $\epsilon > 0$ there is $C_{\epsilon} > 0$ such that for each $ u \in U$ and $x \in U/\Gamma\!_U$ there exists $v \in U$ such that
\begin{enumerate}
\item $\quot v = \quot u$,
\item \label{UnipSubgrpPf-Ubar-sublinear}
$\calB^\pm(v,x) < \epsilon \, d(u,\1) + C_{\epsilon}$, and
\item \label{UnipSubgrpPf-Ubar-d}
$d(v,\1) < C \, d(u, \1) + C$.
\end{enumerate}
\end{claim}

\begin{proof}
\Cref{ExpandersGenerate} provides finitely many one-parameter subgroups $a_1^t, \ldots, a_r^t$ of~$L$, such that the corresponding expanding subgroups $\quot U_1^+, \ldots, \quot U_r^+$ span~$\quot U$.
Identify $\quot U$ with some $\R^s$. Equipping $\quot U$ with some norm $\|\cdot\|$, there is $C_{\quot U}$ such that we may write
	$\quot u = \quot{u_1} \, \quot{u_2} \cdots \quot{u_r} $
where $\quot{u_i} \in \quot U_i^+$ and $\| \quot{u_i} \| \le C_{\quot U} \| \quot u \|$.
Identifying $U$ with its image under the adjoint representation, the exponential map on $U$ is a polynomial (see \cref{BasicUnip}\pref{BasicUnip-exp})
and $d(u,\1)$ is comparable to $\log \|u\|$ for all sufficiently large $u\in U$. This implies that we may choose a lift $u_i$ of each $\quot{u_i}$ such that $d(u_i, \1) \le C_U \, d(u,\1) + C_U$, for some constant~$C_U$.

Consider any $x:=x_0 \in U/\Gamma\!_U$. Recursively define $\{v_i, 1\le i\le r\}$ and $\{x_i, 1\le i\le r\}$ as follows: By \fullCref{ExpandersGenerate}{inA}, each $\{a_i^t\}$ intersects the lattice $\Gamma$. We may apply \cref{UnstableIsSubexp} with $\quot U_i^+$, $u_i$, and $x_{i-1}$ to obtain an element $v_i$. Set $x_i= v_i \cdot x_{i-1}$.
Then
	$$ \text{$\quot v_i = \quot u_i$, \
	$\calB^\pm(v_i, x_{i-1}) \le \epsilon \, d(u_i, \1) + C_{\epsilon}$, \
	and \
	$d( v_i, \1) \le C_0 \, d(u_i,\1) + C_0$}
	. $$
Let $v = v_1 v_2 \cdots v_r$. Then
	$\quot v
	= \quot{v_1} \, \quot{v_2} \cdots \quot{v_r}
	= \quot{u_1} \, \quot{u_2} \cdots \quot{u_r}
	= \quot u$.
We also have
	\begin{align*}
	 \calB^\pm(v,x)
	&= \calB^\pm(v_1 v_2 \cdots v_r,x_0)
	\le \sum_{i=1}^r \calB^\pm(v_i , x_{i-1})
	\le \sum_{i=1}^r \bigl( \epsilon \, d(u_i, \1) + C_{\epsilon} \bigr) \\
	&\le r \bigl( \epsilon \, [C_U d(u,\1) + C_U] + C_{\epsilon} \bigr)
	\le C' \epsilon \, d(u, \1) + C'
	\end{align*}
for an appropriate constant~$C'$, and
	\begin{align*}
	d(v,\1)
	&= d(v_1 v_2 \cdots v_r, \1)
	\le \sum_{i=1}^r d(v_i, \1)
	\le \sum_{i=1}^r \bigl( C_0 d(u_i,\1) + C_0 \bigr) \\
	&\le \sum_{i=1}^r \bigl( C_0 [ C_U d(u,\1) + C_U] + C_0 \bigr)
	\le C d(u,\1) + C
	. \end{align*}
This completes the proof of the claim.
\end{proof}

For the restriction of the discrete cocycle $\calA_0\colon \Gamma\times \calE_0\to \calE_0$ to $\Gamma\!_U$, we have the following immediate strengthening of \cref{UnipSubgrpPf-Ubar}.

\begin{corollary} \label{RepsForGamma}
There exists $C > 0$ such that for every $\epsilon > 0$ there is $C_{\epsilon} > 0$ such that for every $ \gamma \in \Gamma\!_U$ there exists $\lambda \in \Gamma\!_U$ such that
\begin{enumerate}
\item $ \lambda = \gamma \mod [\Gamma\!_U,\Gamma\!_U]$,
\item \label{UnipSubgrpPf-Ubar-sublinear}
$\calB^\pm(\lambda ^{\pm1},\1\Gamma) < \epsilon \, d(\gamma,\1) + C_{\epsilon}$, and
\item \label{UnipSubgrpPf-Ubar-d}
$d(\lambda ,\1) < C \, d(\gamma, \1) + C$.
\end{enumerate}
\end{corollary}
\begin{proof}
Fix $\gamma\in \Gamma\!_U$. From \cref{UnipSubgrpPf-Ubar}, there is $v\in U$ such that
 $\quot v = \quot \gamma$, $d(v,\1) < C \, d(\gamma, \1) + C$, and
$$\calB^\pm(v,\1\Gamma) < \epsilon \, d(\gamma,\1) + C_{\epsilon}.$$

Note that $[\Gamma\!_U,\Gamma\!_U]$ is a cocompact lattice in $[U,U].$ Let $\hat C$ denote the diameter of $[U,U]/{[\Gamma\!_U,\Gamma\!_U]}$. Then there is $q\in [U,U]$ such that
$qv\gamma\inv \in {[\Gamma\!_U,\Gamma\!_U]}$
and $d(q, \1)\le \hat C$.
Then $ \lambda := qv\in \Gamma\!_U$ satisfies $ {\lambda } = \gamma \mod {[\Gamma\!_U,\Gamma\!_U]}$ and $d(\lambda ,\1) < C \, d(\gamma, \1) + C +\hat C$.
By compactness of $U/\Gamma\!_U$ and using that $d(q, \1)\le \hat C$ there is some uniform $C'>0$ such that
\begin{align*}
\calB^\pm(\lambda ,\1\Gamma)
&\le \calB^\pm(v,\1\Gamma) + \calB^\pm(q,v\Gamma) \\
&\le \epsilon \, d(\gamma,\1) + C_{\epsilon} + C'.
\end{align*}
Moreover, since
\begin{align*}
\calB^\pm({\lambda }\inv ,\1\Gamma)
= \calB^\mp(\lambda , \lambda \cdot \1\Gamma)
= \calB^\mp(\lambda , \1\Gamma)
\end{align*}
we also conclude that $\calB^\pm({\lambda }\inv ,\1\Gamma) \le \epsilon \, d(\gamma,\1) + C_{\epsilon} + C'.$
\end{proof}

Let $$U = U_1 \supset U_2 \supset \cdots \supset U_{c+1} = \{1\}$$ be the descending central series of~$U$ where $U_{k+1} = [U, U_k]$.
Write $\Lambda = \Gamma\!_U$ and let $$\Gamma\!_U=\Lambda = \Lambda_1 \supset \Lambda_2 \supset \cdots \supset \Lambda_{c+1} = \{1\}$$ be the descending central series of $ \Lambda$. For each $1\le k\le c$, $\Lambda_k$ is a cocompact lattice in $U_k$ and has finite index in $\Gamma\!_{U_k}= \Gamma \cap U_{k}$.

\begin{claim} \label{UnipSubgrpPf-Uk}
There exists $C > 0$ such that every $\epsilon > 0$ there is $C_{\epsilon} > 0$ such that for every $1\le k\le c$ and $\gamma \in \Lambda_k$ there exists $\lambda \in \Lambda_k$ such that
\begin{enumerate}
\item $\lambda =\gamma \mod \Lambda_{k+1}$,
\item \label{UnipSubgrpPf-Uk-sublinear}
$\calB^\pm(\lambda ,\1\Gamma ) < \epsilon \, d(\gamma,\1) + C_{\epsilon}$ and $\calB^\pm({\lambda }\inv,\1\Gamma) < \epsilon \, d(\gamma,\1) + C_{\epsilon}$,
\item \label{UnipSubgrpPf-Uk-d}
$d(\lambda ,\1) < C \, d(\gamma, \1) + C$.
\end{enumerate}
\end{claim}

\begin{proof}
The proof is by induction on~$k$. When $k=1$, the conclusion follows from directly from \cref{RepsForGamma}. We now assume $k \ge 2$.

We recall that $\quot U= U_1/U_2$. Each quotient $U_k/U_{k+1}$ is abelian and naturally identified with $\lieu_k/\lieu_{k+1}$ via the exponential. Moreover, the map
$$(u \, U_k, wU_2) \mapsto [u, w] \, U_{k+1}$$
is a well-defined, surjective, bilinear map $(U_{k-1}/U_k) \times \quot U \to U_k/U_{k+1}$.
Fix a finite set of generators $ \{w_1,\dots, w_r\}$ for $\Lambda =\Gamma _ U$. Then the image $\{\quot w_1,\dots, \quot w_r\}$ contains a basis for $\quot U$.

For each $k\ge 2$, the collection $\{[w_i, \Lambda_{k-1}] :1\le i\le r\}$ is a generating set for $\Lambda_k$.
Moreover, there is a $C_0$ such that for every $\gamma\in\Lambda_{k}$ there are $\gamma_1,\dots, \gamma_r \in \Lambda_{k-1}$ with $$\gamma =[\gamma_1, w_1] \, [\gamma_2, w_2] \cdots [\gamma_r, w_r]$$
and $d(\gamma_i,\1) \le C_0 \, d(\gamma,\1) + C_0$ (cf.\ \fullcref{BasicUnip}{exp}).

Fix $\gamma\in\Lambda_{k} $ and write $ \gamma =[\gamma_1, w_1] \, [\gamma_2, w_2] \cdots [\gamma_r, w_r]$ for $\gamma_1,\dots, \gamma_r \in \Lambda_{k-1}$ as above.
 By the induction hypothesis, there are $C_1$ and $\lambda _1,\dots, \lambda _r \in \Lambda_{k-1}$ such that \begin{enumerate}
\item $\lambda_i = \gamma_i \mod \Lambda_k$
 \item $d(\lambda_i , \1) \le C_1 d(\gamma_i,\1) + C_1$, and
\item
	$ \calB^\pm( {\lambda_i }^{\pm 1} , \1\Gamma) \le \epsilon \, d(\gamma_i, \1) + C_\epsilon.$
 \end{enumerate}
Also, there is $C_2$ with $\max\{\calB^\pm({w_i}^{\pm 1}, \1\Gamma) : 1\le i\le r\} \le C_2.$
Let
	$$\lambda = [\lambda_1, w_1] \, [\lambda_2, w_2] \cdots [\lambda_r, w_r].$$
We check that $[\lambda_i , w_i] = [\gamma_i, w_i] \mod \Lambda_{k+1}$. Indeed, there exists $q\in \Lambda_k$ such that \begin{align*}
[\lambda_i , w_i]
&
= [\gamma_i q, w_i]
\\&
= q\inv \gamma_i\inv w_i\inv \gamma_i q w_i
\\&
= q\inv [\gamma_i, w_i] w_i \inv q w_i
\\&= q \inv [\gamma_i, w_i] q \mod \Lambda_{k+1}
\\&= [\gamma_i, w_i] \mod \Lambda_{k+1}.
\end{align*}
Thus $\lambda = \gamma \mod \Lambda_{k+1}$.
We also have	\begin{align*}
	\calB^\pm(\lambda ,\1\Gamma)
	&\le \sum_{i=1}^r \calB^\pm \bigl( [\lambda _i, w_i] , \1\Gamma \bigr) \\
	&\le \sum_{i=1}^r \bigl( \calB^\pm(\lambda _i\inv, \1\Gamma) + \calB^\pm(w_i, \1\Gamma) + \calB^\pm(\lambda _i, \1\Gamma) + \calB^\pm(w_i, \1\Gamma) \bigr) \\
	&\le 2r \bigl( \epsilon \, d(\gamma,\1) + C_\epsilon \bigr) + 2rC_2\\
	&\le C' \epsilon \, d(\gamma,\1) + C'.
	 \end{align*}
	 Similarly, 	$\calB^\pm({\lambda }\inv,\1\Gamma) \le C' \epsilon \, d(\gamma,\1) + C'.$
Moreover,
	$$ d(\lambda ,\1)
	\le \sum_{i=1}^r d(\lambda _i,\1)
	\le r \bigl( C_1 d(\gamma,\1) + C_1 \bigr)
	\le C'' d(\gamma,\1) + C''
	. $$
This completes the proof of the claim.
\end{proof}

We now complete the proof of \cref{UnipSubgrp}.

\begin{proof}[Proof of \cref{UnipSubgrp}]
Since $\Lambda_1 = \Gamma\!_U$, it suffices to show for every $\epsilon>0$ and every $1\le k\le c$ that there is $C>1$ such that for all $\gamma\in \Lambda_k$,
$$\calB^{\pm}(\gamma, \1\Gamma) \le \epsilon d(\gamma, \1) + C.$$

The proof is by backward induction on~$k$. The base case when $k = c + 1$ holds trivially.

Consider the inductive case $1\le k \le c$. Fix any $\epsilon>0$. Given $\gamma \in \Lambda_k$, take $\lambda $ as in \cref{UnipSubgrpPf-Uk}.
We have $\gamma = \lambda \mod \Lambda_{k+1}$.
From subadditivitiy, \fullcref{UnipSubgrpPf-Uk}{sublinear}, and the induction hypothesis we have
	\begin{align*}
	\calB^\pm(\gamma, \1\Gamma)
	&\le \calB^\pm(\gamma {\lambda }\inv, \1\Gamma) + \calB^\pm(\lambda , \1\Gamma)
		\\ 
	&\le \bigl( \epsilon \, d(\gamma {\lambda }\inv, \1) + C_\epsilon' \bigr) + \bigl( \epsilon \, d(\gamma,\1) + C_\epsilon \bigr). \end{align*}
By \fullcref{UnipSubgrpPf-Uk}{d}, this is bounded above by an expression of the form $C_1 \epsilon\, d(\gamma,\1) + C_2$ where $C_1$ is a constant independent of $\epsilon$. This completes the proof of the inductive case.
\end{proof}

\section{Subexponential growth for minimal parabolic \texorpdfstring{$\Q$}{Q}-subgroups} \label{SubexpForParabolic}

We establish the following proposition which guarantees the existence of a measure as in the conclusion of \cref{thm:whatweprove,mainthm2} under the assumption that the failure of subexponential growth occurs for the action of some $\Gamma\!_P$ for some minimal parabolic $\Q$-subgroup $P\subset G$.

Under the additional hypothesis in \cref{hyp:center}, we have the following.
\begin{proposition}\label{prop:cuspgroup}
\laterdef\cuspgroup{Let $P$ be a minimal parabolic $\Q$-subgroup of~$G$.
If the restriction
	$$\restrict{\alpha}{\Gamma\!_P}\colon \Gamma\!_P \to \Diff^1(M)$$
does not have uniform subexponential growth of derivatives\ there is a split Cartan subgroup~$A$ of~$G$ and an $A$-invariant Borel probability measure~$\mu$ on~$X$ projecting to the Haar measure on $G/\Gamma$ such that
	 $\lambda_{\top,a,\mu, \calA} > 0$ for some $a \in A$.}
	
Let $P$ be a minimal parabolic $\Q$-subgroup of~$G$.
If the restriction
	$$\restrict{\calA_0}{\Gamma\!_P}\colon \Gamma\!_P \times \calE_0\to \calE_0$$
does not have {$\len_\Gamma$-subexponential growth then}
there is a split Cartan subgroup~$A$ of~$G$ and an $ A$-invariant Borel probability measure~$\mu$ on~$X$ projecting to the Haar measure on $G/\Gamma$ such that
	 $\lambda_{\top,a,\mu, \calA} > 0$ for some $a \in A$.
\end{proposition}

\begin{proof}
We have the rational Langlands decomposition $P = (L \times S^\circ) \ltimes U$ where
$L$ is a connected reductive $\Q$-subgroup with $\rank_\Q L = 0$, $S$~is a maximal $\Q$-split torus of~$P$, and $U= \Rad_u(P)$ is the unipotent radical of $P$. Each of $L$, $S$, and~$U$ is defined over~$\Q$. We have that $L/\Gamma\!_L$ and $U/\Gamma\!_U$ are compact, $\Gamma\!_S$ is trivial, and, up to finite index, $\Gamma\!_P = \Gamma\!_L \ltimes \Gamma\!_U$

By \Cref{UnipSubgrp}, the conclusion of \cref{prop:cuspgroup} would follow if the restriction
$\restrict{\calA_0}{\Gamma\!_U}\colon \Gamma\!_U\times \calE_0 \to\calE_0$
failed to have $\len_\Gamma$-subexponential growth. Thus, we may assume for the remainder that
$\restrict{\calA_0}{\Gamma\!_U}\colon \Gamma\!_U\times \calE_0 \to\calE_0$
has $\len_\Gamma$-subexponential growth.
	
	Recall that the natural embedding $\Gamma\!_L \times \Gamma\!_U \to \Gamma\!_P$ is a quasi-isometry (see \cref{GammaCap}) when each of the three subgroups is equipped with the $\len_\Gamma$-metric. Thus, we may assume that the restriction
$\restrict{\calA_0}{\Gamma\!_L}\colon \Gamma\!_L\times \calE_0 \to\calE_0$
fails to have {$\len_\Gamma$-subexponential growth}.
By \cref{GetExpForAOnCpct}, it follows that there is a $\R$-diagonalizable 1-parameter subgroup $\{a^t\}\subset L$ and an $\{a^t\}$-invariant Borel probability measure~$\mu$ on~$X_L$ such that
 	$$ \lambda_{\top,a^1,\mu, \calA}
	>0.$$

Recall we write $$U_\kappa(n):= \exp_\lieu ( N_{\kappa}(n)) $$ where $ N_{\kappa}(n)= \{\, Z\in \lieu \mid \|Z\|\le e^{\kappa  n}\,\} $ is as in \cref{sec:not} for a sufficiently large $\kappa $.
Take a \Folner sequence $\{F_n\}$ in $\{a^t\}\cdot U $ of the form
	$$F_n:=\{\,a^s \mid 0\le s\le n\,\} \cdot U_\kappa(n)$$
Since $LU/\Gamma\!_{LU}$ is compact, the family $\{F_n \ast \mu\}$ is uniformly tight.
Let $\mu_\infty$ be any weak-$*$ subsequential limit of $\{F_n \ast \mu\}$. Then $\mu_\infty$ is invariant under both $\{a^t\}$ and $ U $.

\begin{claim} \label{MinParabPosExp}
We have
$\lambda_{\top,a^1,\mu_\infty, \calA} >0$.
\end{claim}
\begin{proof}
Let $\mu$ denote the measure above obtained from \cref{GetExpForAOnCpct}. There exists an $\{a^t\}$-invariant, Borel probability measure~$\wtd \mu$ on $\P\calE$ projecting to~$\mu$ with $\lambda_{\top,a^1,\mu, \calA} = \int \Psi(a,\cdot) \ d \wtd \mu$. Let
$\wtd \mu_\infty$ be any subsequential limit of $\{F_n\ast \wtd \mu\}$ that projects to~$\mu_\infty$.
Write $a^t = \exp (tY)$ for $Y\in \lieg$. Since $X_{LU}$ is compact, given any $\epsilon > 0$, for all sufficiently large~$n$ in the appropriate subsequence we have
\begin{align*}
	\lambda_{\top,a^1,\mu_\infty, \calA} +\epsilon
	&\ge \int \Psi(a^1,\xi) \ d \wtd \mu_\infty (\xi) +\epsilon\\
	&= \int \Psi'(Y,\xi) \ d \wtd \mu_\infty (\xi) +\epsilon
		&& \\
	&\ge \int \Psi'(Y,\xi) \ d(F_n\ast \wtd \mu) (\xi)
			&& \\
&= 	 \frac{1}{n}\int_0^n \int_{U_ \kappa(n)} \int \Psi'(Y, a^t\cdot u\cdot \xi)
		\, d \wtd\mu(\xi) \, d m_{U_\kappa(n)}(u) \, d t\\
&= 	 \frac{1}{n} \int_{U_\kappa(n)} \int \Psi(a^n, u\cdot \xi)
		\, d \wtd\mu(\xi) \, d m_{U_\kappa(n)}(u).\end{align*}
	Above, the first equality follows from $\{a^t\}$-invariance of $\td \mu_\infty$ as in \cref{s4.1};
the final equality is the fundamental theorem of calculus as in \cref{s4.1}.

Having taken $\kappa>0$ sufficiently large, for $u \in U_ \kappa(n)$ we have
	$a^n \, u = u' \, a^n$
for some $u' \in U_{2 \kappa}(n)$ whence
	$$ \Psi(a^n, u \cdot \xi) = \Psi(a^n, \xi) + \Psi(u', a^n \cdot \xi) - \Psi(u, \xi).$$
By definition of $\wtd \mu$, we have $$\frac{1}{n} \int \Psi(a^n, \xi)
		\, d \wtd\mu(\xi) = \lambda_{\top,a^1,\mu, \calA}.$$

We have that $U$ is normalized by~$L$ and that $LU/\Gamma\!_{LU}$ is compact. It follows from \cref{SubexpForCompact} and the assumption that $\restrict{\calA_0}{\Gamma\!_U}\colon \Gamma\!_U\times \calE_0 \to\calE_0$ has {$\len_\Gamma$-subexponential growth} that the restriction of the cocycle $\calA$ to $U \times \calE_{LU}$ has subexponential growth; in particular, for any $\epsilon > 0$, there is a constant~$C_U$ such that for all $u'' \in U$ and $\xi \in \P\calE_{LU}$,
	$$ \Psi(u'',\xi) \le C_U + \epsilon \, d(u'', \1) .$$
Identifying $U$ with its image under the adjoint representation,
the exponential map $\restrict{\exp}{\lieu}$ is polynomial; in particular, there is a constant $C_\lieu > 0$ such that
	$$ \text{$d \bigl( \1, \exp_\lieu (Z) \bigr) \le C_\lieu \, n$ \ for all $Z \in N_{2 \kappa}(n)$.}$$
Hence
	\begin{align*}
		\bigl| \Psi(u', a^n \cdot \xi) - \Psi(u, \xi) \bigr|
		&\le 2C_U + \epsilon \, d(u', \1) + \epsilon \, d(u, \1)
		\le 2C_U + 2 \epsilon C_\lieu \, n
	.\end{align*}
Then for $n$ sufficiently large,
	\begin{align*}
	\frac{1}{n} \int_{U_\kappa(n)} \int & \bigl| \Psi(u', a^n \cdot \xi) - \Psi(u, \xi) \bigr|
		\, d \wtd\mu(\xi) \, d m_{U_ \kappa(n)}(u)
	\le \frac{2C_U + 2 \epsilon C_\lieu \, n }{n}
	\le 3 \epsilon C_\lieu.
	\end{align*}

The result then follows by taking $\epsilon > 0$ sufficiently small.
\end{proof}

As $U$ is the unipotent radical of the minimal parabolic $\Q$-subgroup~$P$, $U$ is the expanding horospherical subgroup of some one-parameter subgroup~ $\{b^s\}$ of~$S$.
Recall that $\{b^s\}$ centralizes~$\{a^t\}$ whence $(b^s)_*\mu_\infty$ is $\{a^t\}$-invariant for every $s$.
From \cref{MixingSmallCuspsCom}, the sequence $\{(b^s)_*\mu_\infty: s\in \N\}$ has uniformly exponentially small mass at $\infty$ and its projections converge to the Haar measure on $G/\Gamma$. Applying \cref{lemma:averaging} and then \cref{lem:endgameA} produces a measure satisfying the conclusion of \cref{prop:cuspgroup}.
\end{proof}

\def\tempfun{\phi}
\def\heightgrowth{\omega}
\def\hcutoff{{\hat \ell_0}}

\section{Subexponential growth for \texorpdfstring{$\Q$-rank-$1$}{Q-rank-1} subgroups}
\label{SubexpForRank1Subgroups}
We prove the following result in which we consider the case that the failure of subexponential growth occurs for the restriction of the cocycle $\calA_0$ to $\Gamma\!_H$ for some standard $\Q$-rank-1 subgroup $H$ in  $G$. This, combined with \cref{CentralSubExp} and \cref{QIBddGenByRank1lift}, yields \cref{thm:whatweprove,mainthm2}.

As in the previous two sections, we keep the standing hypothesis from \cref{hyp:center}.

\begin{proposition}\label{thm:mainQR1}
\laterdef\Qrkone{
If there is a standard  $\Q$-rank-1 subgroup~$H$ of~$G$ such that the restriction
	$$\restrict{\alpha}{\Gamma\!_H}\colon \Gamma\!_H \to \Diff^1(M)$$
does not have uniform subexponential growth of derivatives\  then there is a split Cartan subgroup~$A$ of~$G$ and an $A$-invariant Borel probability measure~$\mu$ on~$X$ projecting to the Haar measure on $G/\Gamma$ such that
	 $\lambda_{\top,a,\mu, \calA} > 0$ for some $a \in A$.}
If there is a standard  $\Q$-rank-1 subgroup~$H$ of~$G$ such that the restriction
	$$\restrict{\calA_0}{\Gamma\!_H}\colon \Gamma\!_H \times \calE_0 \to \calE_0$$
does not have {$\len_\Gamma$-subexponential growth then}
there is a split Cartan subgroup~$A$ of~$G$ and an $A$-invariant Borel probability measure~$\mu$ on~$X$ projecting to the Haar measure on $G/\Gamma$ such that
	 $\lambda_{\top,a,\mu, \calA} > 0$ for some $a \in A$.
\end{proposition}

\subsection{Preliminaries, key ingredients and outline}
\subsubsection{Preliminaries}\label{SubexpForRank1Subgroups-prelims}
Fix a standard $\Q$-rank-1 subgroup~$H$ of~$G$.
From \cref{Standard=Levi} there is a parabolic $\Q$-subgroup~$Q$ of~$G$ with rational Langlands decomposition $Q^\circ = (L \times S^\circ) \ltimes N$ such that $H$~is the unique $\Q$-isotropic almost-simple factor of~$L$.
We have that $H$ normalizes $N$ and write $ \wtd H := H\ltimes N$. Since $H$ and $N$ are defined over~$\Q$, $\wtd H$ is also defined over~$\Q$.

Note that if $\rank_\Q(G)= 1$ then $N= \{e\}$ and $H= \wtd H = Q = G$;   if $\rank_\Q(G) > 1$ then $N \neq \{e\}$ and the inclusions $H \subset \wtd H \subset G$ are all proper.

Recall that we write $\Gamma\!_H:= H\cap \Gamma$, $\Gamma\!_{\wtd H}:= \wtd H\cap \Gamma$, $X_H:= (H\times M)/\Gamma\!_H$, and $X_{\wtd H}:= (\wtd H\times M)/\Gamma\!_{\wtd H}$.
The set $\wtd H/\Gamma\!_{\wtd H}$ has the structure of a fiber bundle over $H/\Gamma\!_H$ with compact fibers.
In particular, given any compact set $K\subset \wtd H/\Gamma\!_{\wtd H}$, the $N$-orbit of $K$ is compact.

 We select a minimal parabolic $\Q$-subgroup $P$ and Cartan involution $\theta$ as in \cref{SiegelForSubgroup}. We equip the vector bundle $\calE$ with a norm constructed as in \cref{sec:norms} relative to Siegel fundamental sets constructed relative to $P$ and $\theta$.  
From \cref{prop:cuspgroup}, we may assume that the restriction of~$\alpha$ to~$\Gamma\!_P$ has uniform $\len_\Gamma$-subexponential growth for every minimal parabolic $\Q$-subgroup~$P$.  From \cref{UnipSubgrp} we may assume  that $\restrict{\alpha}{\Gamma\!_N}$ has uniform $\len_\Gamma$-subexponential growth.

We will use heavily the growth controls in \cref{Ngrowthcorrect,horseforlunch}.
 We recall the two proper $C^\infty$ ``height'' functions defined earlier:  we have  $h\colon G/\Gamma\to [0,\infty) $ satisfying \eqref{heightDefnII} and $\phi\colon \wtd H/\Gamma\!_{\wtd H}\to [0,\infty) $ the function $\phi =   h\circ \rho$ where $\rho\colon  \wtd H/\Gamma\!_{\wtd H} \to H/\Gamma$ is the canonical projection as in \eqref{eqtdphi}.    We lift $h$ to $X$ and $\P\calE$ and lift $\phi$ to $X_{\wtd H}$  and $\P\calE_{\wtd H}$ via the canonical projections $\P\calE\to X\to G/\Gamma$ and  $\P\calE_{\wtd H}\to X_{\wtd H}\to \wtd H/\Gamma\!_{\wtd H} .$
We also fix $\heightgrowth$ to be the largest of the constants $\omega$ appearing in \cref{Ngrowthcorrect,horseforlunch}.

\subsubsection{Key Ingredients and Outline}
\label{subsubsec:outline}

As discussed in the introduction, the proof of Proposition \ref{thm:mainQR1} resembles arguments in  \cite{BFH} more than those of \cite{BFH-SLnZ}, particularly in the case where $\Gamma$ has $\Q$-rank one. However, to set up the arguments, the  first steps are quite similar to work done in \cite{BFH-SLnZ}. In  \cref{subsec:growthinA}, we prove various preliminaries about the growth of cocycle $\calA$  
along $A_H$ trajectories in $X_{\wtd H}$.  In particular, we show that exponential growth of derivatives is always witnessed by paths beginning and ending in a fixed compact set and that the empirical measures corresponding to such paths form tight families.  In addition we show that the exponential growth can always be witnessed along a single one-parameter subgroup.  In   \cref{TimeAveraged}, we define the crucial new tools for our proof.  Recall the norm-growth cocycle defined in  \cref{s4.1} and its infinitesimal generator for a fixed one-parameter subgroup.  This cocycle over a one-parameter subgroup is specified by a function which is continuous but unbounded.  As already explained in \cref{ProofOutline}, the key difficulty for our arguments  is that this unbounded function is not even in $L^1$ so basic tools of ergodic theory do not apply.  To resolve this difficulty we introduce two related objects.  First we introduced  a time-averaged version of the cocycle, which is measurable, bounded,  invariant under our 1-parameter subgroup, and computes Lyapunov exponents. 
However, this function is now not continuous and its integral against various measures does not behave well under weak-$*$ convergence.  To remedy this we introduce the key technical object,  a \emph{cut-off} cocycle, that interpolates between the original cocycle and the time-averaged one.  The cut-off cocycle is bounded and while not continuous has very controlled discontinuities and behaves well under weak-$*$ limits;  see \cref{s4.2}.

The proof of   \cref{thm:mainQR1} in \cref{ProofOfmainQR1} first assembles these tools and our prior results to build an $(A_HN)$-invariant measure on $X_{\wtd H}$ that has positive top Lyapunov exponent.
The rest of the proof follows arguments of \cite{BFH} in the $\Q$-rank-one case and \cite{BFH-SLnZ} in the case of higher $\Q$-rank.  Technical difficulties occur because it does not always suffice to work with the cut-off cocycle, since
it is only a cocycle over a 1-parameter subgroup but not cocycle over the full $G$-action or full $A$-action. 
We thus need to track the interplay between the original continuous norm-growth cocycle, the time-averaged cocycle, and the cut-off cocycle.

\subsection{Characterization and properties of orbits with exponential growth}
\label{subsec:growthinA}
We study the growth properties of the cocycle $\calA$ 
along $A_H$ trajectories in $X_{\wtd H}$.
From the choice of $\theta$ as in \cref{SiegelForSubgroup}, the subgroup~$H$ is $\theta$-invariant and the restriction $\theta|_H$ is a Cartan involution of~$H$.
Fix a $\theta$-invariant split Cartan subgroup~$A$ of~$G$.
We obtain an Iwasawa decomposition $H = K_H \, A_H \, N_H$ of~$H$ where $K_H = K \cap H$, $A_H = A \cap H$, and $N_H$ is the coarse root group corresponding to the $\Q$-root defining $H$.


\subsubsection{Maximal growth rate along $A_H$}
We now define a filtration of $X_{\wtd H}$ by compact sets and study the exponential growth
of the  cocycle $\calA$ 
 for trajectories beginning and ending at a fixed set  in this filtration. In particular, we show
that there is some compact set that witnesses all possible exponential growth.

Write $\liea_H$ for the Lie algebra of $A_H$.
Given $\ell>0$,
define ``thick sets'' of $X_{\wtd H}$ parameterized by the function $\phi$ as follows:
$$X_{\wtd H,\le \ell_0}:=\{ x\in X_{\wtd H} \mid \tempfun(x) \le \ell_0\}.$$
Given  $Y\in \liea_H$ set:
\begin{align*}
 \chi_\ell^+(&T,Y,  X_{\wtd H })\\&:= \frac 1 T \sup\left\{  \log \, \bigl\|\calA ( \exp(TY),x \bigr) \bigr\| : x  \in  X_{\wtd H,\le \ell_0}  \text{ and } \exp (TY)\cdot x \in  X_{\wtd H,\le \ell_0} \right\}\\
\chi_\ell^-(&T,Y, X_{\wtd H })
\\&:= \frac 1 T \sup\left\{  \log \, \bigl\|\calA ( \exp(T Y),x \bigr) \inv \bigr\| :  x \in  X_{\wtd H,\le \ell_0}  \text{ and } \exp (T Y)\cdot x \in  X_{\wtd H,\le \ell_0} \right\}
\end{align*}

Define $$\chi_\ell(Y, X_{\wtd H })= \displaystyle \limsup_{T\to \infty} \max \left\{ \chi_\ell^+(T,Y, X_{\wtd H }), \chi_\ell^-(T,Y, X_{\wtd H }) \right\}. $$
Observe that $\chi_\ell(Y, X_{\wtd H })\ge 0$ and that $\chi_\ell(cY, X_{\wtd H }) = |c| \, \chi_\ell(Y, X_{\wtd H }) $.

Fix a norm $\|\cdot \|$ on~$\lieg$ that is $\Ad(K)$ invariant and whose right $G$ translates define the symmetric metric on $K \backslash G$. Let $\liea'$ be a vector subspace of~$\liea_H$.
Given $\ell>0$ set
\begin{align*}
\chi_\ell(T,\liea', X_{\wtd H })&:= \sup \{\, \chi_\ell^+(T,Y, X_{\wtd H }) : Y\in \liea', \|Y\|= 1\,\} \\
&\phantom{:}= \sup \{\, \chi_\ell^-(T,Y, X_{\wtd H }) : Y\in \liea', \|Y\|= 1\,\}.
\end{align*}
The quantity $\chi_\ell(T,\liea', X_{\wtd H })$ measures the maximal exponential growth rate of the cocycle~$\calA$ along all orbits of the subgroup $\exp (\liea')$ in~$ X_{\wtd H }$ of length at most~$T$ that start  and end below height~$\ell$.
Define $$\chi_\ell(\liea', X_{\wtd H })= \displaystyle \limsup_{T\to \infty} \chi_\ell(T,\liea', X_{\wtd H }).$$
Note that $\chi_\ell(\liea', X_{\wtd H }) <\infty$ follows from \cref{tempered}.

Exactly as in the proof of \cref{GetExpForAOnCpct}, the decomposition discussed in \cref{kakaReductive} immediately implies the following.
\begin{claim}\label{HavePosExp}
Under \cref{hyp:center}, if the restriction of~$\calA_0$ to~$\Gamma\!_H\colon \calE_0\to \calE_0$ fails to have   subexponential growth then there exists some $\ell_1>0$ such that
	$\chi_{\ell_1}(\liea_H, X_{\wtd H })>0$.
\end{claim}
\begin{proof}
By assumption and \cref{liftQI}, there exist $\kappa>0$ and sequences $\{\gamma_j\}$ of elements of~$\Gamma\!_H$ with $\len_\Gamma(\gamma_j)\to \infty$ and $\{x_j\}$ in $X_0$ such that
$$ \log \|\calA_0(\gamma_j,{x_j})\| \ge \kappa \, d(\gamma_j, \1) .$$
Let $\overline{x_j}=[\1, x_j] \in X_{  H} \subset X_{\wtd H}$ be the point corresponding to the ordered pair $(\1,x_j) \in  H \times X_0$. Then there exists  $C$ such that for all $j$,
 	\begin{align*}
	\log \|\calA(\gamma_j, \overline{x_j})\|\ge \kappa \, d(\gamma_j,\1) - C
	.\end{align*}

Following the notation in \cref{kakaReductive}, write $\gamma_j = k'_j z_j a_j k''_j$ where $k'_j,k''_j\in K_{H,0}$, $z_j\in \calZ\cap H$, and $a_j\in A_H$.  There is $C_1$ such that for all $j$ (with $\gamma_j\neq \1$), $$\max\{d(a_j, \1), \len_\calZ(z_j)\}\le C_1 d(\1, \gamma_j).$$
By \cref{moregrowth} we have that $\|\calA(k, x)\|$ is uniformly bounded over all $k\in K_{H,0}$ and all $x\in X$. Moreover, by the assumption in \cref{hyp:center} and \cref{moregrowth}, given any $  \epsilon>0$ there is $C_\epsilon>0$ such that for all $x\in X$,
$$\log \|\calA(z_j, x)\| \le C_\epsilon +  \epsilon \len_\calZ(z_j).$$
It follows there is $\hat C$ independent of $\epsilon$ such that for all $j$ (with $\gamma_j\neq \1$),
\begin{align*}
\log \|\calA(a_j, k''_j \cdot \bar x_j)\|
& \ge \kappa \, d(\gamma_j,\1) - \hat C-  \epsilon \len_\calZ(z_j) -C_\epsilon\\
& \ge \kappa \, d(\gamma_j,\1) - \hat C-  C_1 \epsilon d(\gamma_j,\1)-C_\epsilon \\
& \ge( \kappa -C_1\epsilon) \, d(\gamma_j,\1) - \hat C-C_\epsilon\\
& \ge( \frac 1 {C_1} \kappa -\epsilon) \, d(a_j,\1) - \hat C-C_\epsilon.
\end{align*}
assuming  $\epsilon>0$ is sufficiently small so that $\frac 1 {C_1} \kappa -\epsilon>0$.
Moreover, compactness of $K_{H,0}$ implies there is $\ell>0$ so that for all $j$,
  $\phi(k''_j \cdot \overline{x_j})\le \ell$ and
  $$\phi(a_j k''_j \cdot \overline{x_j})= \phi(a_j z_j k''_j \cdot \overline{x_j}) = \phi\big((k'_j)\inv \cdot  [\1, \gamma_j\cdot x_j]\big)\le \ell.$$
It follows $\chi_\ell(\liea', X_{\wtd H })\ge \frac 1 {C_1} \kappa -\epsilon. $
\end{proof}

\subsubsection{Paths realizing exponential growth start and end in a fixed thick part}
Recall the constant~$\ell_0$ provided by \cref{SubexpInCusps}.

We prove here that if there are arbitrarily long orbits that all start and end in some arbitrary common compact thick set and all carry a certain amount of exponential growth then there are orbits that have both endpoints in the fixed compact thick set $X_{\wtd H,\le \ell_0}=\{ x\in X_{\wtd H} \mid \tempfun(x) \le \ell_0\}$ and that carry at least as much exponential growth.

\begin{lemma}\label{GetXnAndTn}
Let $\liea'$ be a subspace of~$\liea_H$.
Suppose there exist $\chi>0$, $\ell_1>0$, and sequences $x_n\in X_{\wtd H}$, $Y_n\in \liea'$ with $\|Y_n\|=1$, and $t_n\to \infty$ such that
\begin{enumerate}[label=(\alph*),ref=\alph*]
\item \label{GetXnAndTn-endpts}
for each~$n$, we have $\tempfun(x_n)\le \ell_1$ and $\tempfun\bigl((t_n Y_n) \cdot x_n \bigr) \le \ell_1$;
\item \label{GetXnAndTn-limit}
$\displaystyle \lim_{n\to\infty} \dfrac {1}{ t_n} \log \bigl\| \calA \bigl( \exp(t_nY_n),x_n \bigr) \bigr\| = \chi .$
\end{enumerate}
Then, for each $n$ there exist $\hat x_n\in X_{\wtd H}$ and $\hat t_n$ such that
\begin{enumerate}
\item \label{GetXnAndTn-tnLarge}
$\displaystyle \liminf_{n\to\infty} \frac{\hat t_n}{t_n} \ge \frac {\chi}{\chi_{\ell_1}(\liea',X_{\wtd H})}$ whence $\hat t_n \to \infty$;
\item \label{GetXnAndTn-h(xn)}
$ \tempfun(\hat x_n)\le \ell_0 $ and $\tempfun \bigl( \exp (\hat t_nY_n) \cdot \hat x_n \bigr) \le \ell_0$;
\item \label{GetXnAndTn-chiLarge}
$\displaystyle \liminf_{n\to\infty} \dfrac {1}{ \hat t_n} \log \bigl\| \calA \bigl( \exp(\hat t_nY_n),\hat x_n \bigr) \bigr\| \ge \chi .$
\end{enumerate}
In particular, for any $\liea'\subset \liea$ and any $\ell_1\ge \ell_0$, we have
\begin{equation}\label{ExpIndepOfEll}
\chi_{\ell_1}(\liea', X_{\wtd H}) = \chi_{\ell_0}(\liea',X_{\wtd H}).
\end{equation}
\end{lemma}

In light of \eqref{ExpIndepOfEll}, we will often drop the subscript and write $\chi(\liea',X_{\wtd H}) = \chi_{\ell_0}(\liea',X_{\wtd H})$ and $\chi(Y,X_{\wtd H}) = \chi_{\ell_0}(Y,X_{\wtd H})$ for any $Y\in \liea_H$.

\begin{proof}[Proof of \cref{GetXnAndTn}]
We may assume $\ell_1>\ell_0$.

We first claim for all sufficiently large $n$ that there is some $0\le s \le t_n $ such that
	$$ \tempfun \bigl( \exp (s Y_n) \cdot x_n \bigr) \le \ell_0.$$
Indeed, for each $n$, if no such $ s $ exists then by \cref{horseforlunch}
for any $\epsilon>0$ there is $C_\epsilon$ such that
	$$\log \bigl\| \calA \bigl( \exp( t_nY_n), x_n \bigr) \bigr\|
		 \le 2 \heightgrowth    \ell_1 + \epsilon t_n + 2C_\epsilon. $$
By taking $\epsilon <\chi/2$, it would follow that
	$$\log \bigl\|\calA \bigl( \exp(  t_nY_n), x_n \bigr) \bigr\| < (\chi/2 )t_n$$
for all sufficiently large~$t_n$, contradicting hypothesis~\pref{GetXnAndTn-limit}.

We may thus assume for any $n$ there exists $s_n$ such that $\tempfun\bigl(\exp (s_n Y_n)\cdot x_n\bigr) \le \ell_0$.
 For each $n\in \N$ take $$p_n = \min\{\,0\le s\le t_n \mid \exp (s Y_n)\cdot x_n \in X_{\wtd H,\le \ell_0}\,\}$$
and $$q_n = \max\{\,0\le s\le t_n \mid \exp (s Y_n)\cdot x_n \in X_{\wtd H, \le \ell_0}\,\}.$$
We have $0\le p_n \le q_n \le t_n$.
Let
	$$\text{$\hat x_n = \exp (p_n Y_n)\cdot x_n$,
		\quad $\hat t_n = q_n - p_n$,
		\ and \
		$y_n = \exp \bigl(q_n Y_n\bigr)\cdot x_n$.}$$
We claim that sequences $(\hat x_n)$ and $(\hat t_n)$ satisfy the conclusions of the \lcnamecref{GetXnAndTn}.

Given any $0<\epsilon<\chi/5$, there is $T_1>0$ such that
$$ \frac{1}{t} \log \bigl\| \calA \bigl( \exp(tY),x \bigr) \bigr\| < \chi_{\ell_1}(\liea',X_{\wtd H}) +\epsilon $$
for all $x\in X_{\wtd H}$ with $\tempfun(x) \le \ell_1, Y\in \liea'$ with $\|Y\|=1$, and $t\ge T_1$ such that $\tempfun(\exp(tY)\cdot x)\le \ell_1$.
Note also that there exists $C_0>0$ so that
$$\log \bigl\| \calA \bigl( \exp(tY),x \bigr) \bigr\| <C_0$$
for all $x\in X_{\wtd H}$ with $\tempfun(x) \le \ell_1, Y\in \liea'$ with $\|Y\|=1$, and $t\le T_1$.
Finally, from \cref{horseforlunch}
there is $C_\epsilon>0$ such that for all~$n$,
	$$\log \bigl\| \calA \bigl( (t_n -q_n) Y_n, y_n \bigr) \bigr\| \le  2\heightgrowth\ell_1 + \epsilon  (t_n-q_n)  + C_\epsilon$$
and
$$\log \| \calA(p_n Y_n, x_n)\| = \log \| \calA(-p_n Y_n, \hat x_n)\inv \| \le 2\heightgrowth \ell_1  + \epsilon  p_n + C_\epsilon .$$

Take $T_2\ge T_1$ so that $$\dfrac{ C_0+ 2 C_\epsilon + 4  \heightgrowth \ell_1}{ T_2}\le \epsilon.$$
Consider any~$n$ sufficiently large so that $t_n\ge T_2$ and
	\begin{equation} \label{tnLarge}
	\frac {1}{ t_n} \log \bigl\| \calA \bigl( \exp(t_nY_n),x_n \bigr) \bigr\| \ge \chi-\epsilon .
	\end{equation}
We claim that $\hat t_n \ge T_1$.
Indeed, were $\hat t_n \le T_1$ we would have that
\begin{align*}
\log \|\calA(t_nY_n,x_n)\| &\le
\log \| \calA ( p_nY_n,x_n ) \| +
\log \| \calA ( \hat t_nY_n,\hat x_n)\| \\
&\qquad + \log \bigl\| \calA \bigl( (t_n-q_n)Y_n,y_n) \bigr) \bigr\| \\
& \le \bigl[ 2\heightgrowth \ell_1+ \epsilon p_n + C_\epsilon \bigr] + C_0 + \bigl[ 2\heightgrowth\ell_1+ \epsilon (t_n -q_n) \bigr) + C_\epsilon \bigr] \\
& \le C_0 + 2 C_\epsilon + 4\heightgrowth \ell_1 + \epsilon t_n.
\end{align*}
Since $t_n\ge T_2$ satisfying \eqref{tnLarge} we would then obtain that
$$\frac{1}{t_n} \log \|\calA(t_nY_n,x_n)\|\le 2\epsilon<\chi-2\epsilon$$
contradicting \pref{tnLarge}.

Still considering $t_n\ge T_2$, since $\hat t_n \ge T_1$
 we have
\begin{align*}
\bigl( \chi_{\ell_1}(\liea' & ,X_{\wtd H}) + \epsilon \bigr) \hat t_n \\
&\ge \log \|\calA(\hat t_nY_n,\hat x_n)\| \\
&\ge \log \|\calA(t_nY_n,x_n)\|
- \log \|\calA(p_nY_n,x_n)\|
- \log \bigl\| \calA \bigl( (t_n-q_n)Y_n,y_n \bigr) \bigr\| \\
& \ge ( \chi-\epsilon ) t_n - \bigl[2\heightgrowth \ell_1+ \epsilon p_n) + C_\epsilon \bigr]- \bigl[2\heightgrowth \ell_1+ \epsilon (t_n-q_n)  + C_\epsilon \bigr] \\
& \ge ( \chi-\epsilon ) t_n - (2C _\epsilon+ 4 \heightgrowth \ell_1) - \epsilon t_n \\
& \ge ( \chi-3\epsilon ) t_n .
\end{align*}
It follows that
$$ \hat t_n \ge \frac {\chi-3\epsilon}{\chi_{\ell_1}(\liea',X_{\wtd H})+\epsilon} \, t_n $$
and that
$$\log \|\calA(\hat t_nY_n,\hat x_n)\| \ge ( \chi-3\epsilon ) t_n \ge ( \chi-3\epsilon ) \hat t_n. $$
Taking $\epsilon>0$ arbitrarily small, the claim follows.
\end{proof}

\subsubsection{Growth along basis vectors}
We prove that it suffices to study the growth of the  cocycle $\calA$ 
along a one-parameter subgroup of $A_H$.
Fix a vector subspace $\liea' $ of $\liea_H$ and any basis $\{Z_1, \dots, Z_r\}$ of unit vectors of $\liea'$.
Recall from \pref{ExpIndepOfEll} that $\chi_\ell(\liea',X_{\wtd H })= \chi_{\ell_0}(\liea',X_{\wtd H })$ for all $\ell\ge \ell_0$ whence we write $ \chi(\liea',X_{\wtd H }):= \chi_{\ell_0}(\liea',X_{\wtd H }).$

\begin{corollary}\label{basisvector}
There exists $C\ge 1$ depending only on the norm  $\|\cdot\|$ and the basis $ \{Z_1, \dots, Z_r\}$ such that if $\chi(\liea',X_{\wtd H })>0$ then
for some
$1\le j\le r$ we have
$$\chi (Z_j,X_{\wtd H })\ge \frac{\chi (\liea',X_{\wtd H })}{C}.$$
\end{corollary}

\begin{proof}
We induct on $r= \dim \liea'$. The result clearly holds when $r=1$.

Since $\chi (\liea',X_{\wtd H })>0$, up to reversing paths and using \cref{GetXnAndTn}, we may assume there exist sequences $x_n\in X_{\wtd H }$ with $\phi(x_n) \le \ell_0$, $Y_n\in \liea'$ with $\|Y_n \|=1$, and $t_n\in \R$ with $t_n\to \infty$, such that
writing $y_n := \exp(t_n Y_n ) \cdot x_n$, $\tempfun(y_n)\le \ell_0$ and
 $$\lim_{n\to \infty} \frac{1}{t_n}\log \|\calA( t_nY_n, x_n)\| = \chi(\liea', X_{\wtd H }) .$$

Suppose $r>1$. For each~$n$, write $t_n Y_n = \pm p_nZ_r + q_nZ'_n$ where $Z'_n$ is a unit vector in the linear span of $\{Z_1, \cdots, Z_{r-1}\}$ and $p_n, q_n\ge 0$.
There is $C_1$ depending only on the choice of basis with $$\max\{ p_n, q_n \} \le C_1 t_n$$ for all~$n$.
Let
\begin{align*}
u_n &= \sup \{ u\in [0,p_n]: \phi\bigl(\exp (u\cdot Z_r)\cdot x_n\bigr)\le \ell_0\} \\
v_n &= \sup \{ v\in [0,q_n]: \phi\bigl(\exp (-v\cdot Z'_n)\cdot y_n\bigr)\le \ell_0\} . \end{align*}
Fix $\epsilon>0$.  There is $C_\epsilon$ such that for all~$n$, using \cref{horseforlunch} we have
\begin{align*}\log \bigl\| \calA &\bigl( \exp( t_n Y_n), x_n \bigr) \bigr\|
\\& \le \log
\bigl\| \calA \bigl( \exp(u_n  Z_r ), x_n \bigr) \bigr\| + \log \bigl\| \calA \bigl(\exp(-v_n  Z'_n ), y_n \bigr)\inv \bigr\| \\
&\quad + 2\heightgrowth\ell_0 + \epsilon [(p_n-u_n) + (q_n- v_n)]   + C_\epsilon \\
&\le \log
\bigl\| \calA \bigl( \exp(u_n  Z_r ), x_n \bigr) \bigr\| + \log \bigl\| \calA \bigl(\exp(-v_n  Z'_n ), y_n \bigr)\inv \bigr\| \\
&\quad  + 2\heightgrowth\ell_0 + 2\epsilon C_1 t_n   + C_\epsilon
\end{align*}
Taking $\epsilon>0$ sufficiently small, for all $n$ sufficiently large it follows that either $\log \bigl\| \calA \bigl( \exp(u_n  Z_r ), x_n \bigr) \bigr\| $ or $\log \bigl\| \calA \bigl(\exp(-v_n  Z' _n), y_n \bigr)\inv \bigr\| $ is at least $$\frac 1 3 \bigl\| \calA \bigl( \exp( t_n Y_n), x_n \bigr) \bigr\| .$$
Let $\liea''$ be the span of $\{Z_1,\ldots, Z_{r-1}\}$.
After dividing by~$t_n$ and taking $t_n \to \infty$
we have that either
	$$ \chi(Z_r, X_{\wtd H })
	\ge \lim_{n\to \infty} \frac{t_n}{3p_n} \chi (\liea',X_{\wtd H })
	\ge \frac{1}{3C_1} \chi (\liea',X_{\wtd H })
	$$
and the conclusion follows or
	$$\chi(\liea'', X_{\wtd H })
	\ge \lim_{n\to \infty}  \frac{t_n}{3q_n} \chi (\liea',X_{\wtd H })
	\ge \frac{1}{3C_1} \chi (\liea',X_{\wtd H })
	$$
and the conclusion follows from the induction hypothesis.
\end{proof}

\subsubsection{Paths approximating maximal growth rates give a tight family of empirical measures.}
We   study empirical measures coming from certain trajectories and show they form tight families.
The families we consider start and end in a fixed compact set and witness the maximal growth rate.

\begin{lemma}[Tightness of empirical measures]\label{lem:maxpathstight}
Fix some $Y \in \liea_H$ with $\chi(Y, X_{\wtd H}) > 0$. Let $\{x_n\} \in X_{\wtd H }$ be a sequence of points in~$X_{\wtd H }$, let $t_n \to \infty$, and let $$\eta_n = \frac 1 {t_n} \int_0^{t_n} \exp(sY) _* \delta_{x_n} \, ds.$$ If
\begin{enumerate}[label=(\alph*),ref=\alph*]
 \item for each~$n$, $\tempfun(x_n)\le \ell _0$ and $\tempfun \bigl( \exp(t_nY)\cdot x_n \bigr) \le \ell_0$, and
\item $\displaystyle \lim_{n\to \infty} \frac 1 {t_n}\log \bigl\| \calA \bigl( \exp(t_nY),x_n \bigr) \bigr\| = \chi(Y,X_{\wtd H })$,
\end{enumerate}
then $\{\eta_n\}$ is a uniformly tight family of Borel probability measures on $X_{\wtd H}$.
\end{lemma}

\begin{proof}
Write $\chi = \chi(Y,X_{\wtd H })$. Without loss of generality, we may assume that $\| Y \| = 1$.

Fix some $\delta>0$. There is $T_0 > 0$ such that for all $x\in X_{\wtd H } $ with $\tempfun(x) \le \ell_0$ and all $t\ge T_0$  such that $\tempfun\bigl(\exp(tY_n)\cdot x\bigr)\le \ell_0$
we have
$$ \log \|\calA(\exp( t Y),x) ^{\pm 1}\|\le (\chi + \delta ) t.$$
Set $$C = \sup \left\{ \log \bigl\| \calA ( \exp( t Y),x \bigr)^{\pm 1} \bigr\| : x \in X_{\wtd H }, \tempfun(x) \le \ell_0, |t|\le T_0 \right\}.$$
 By \cref{horseforlunch},
there exists $T_1$ such that for all $T \ge T_1$, if $\tempfun(x) = \ell_0$, ${\tempfun \bigl( \exp(TY) \cdot x \bigr)} = \ell_0$,   and if $\tempfun \bigl( \exp(sY) \cdot x \bigr) \ge \ell_0$ for all $0\le s\le T$ then
$$\log \bigl\| \calA\bigl( \exp( T Y),x \bigr) \| \le \delta T .$$

For $\ell > 0$ and $n\in \N$, write 			
	$$ I_{> \ell,n} = \{ t \in [0,t_n] :\tempfun(\exp(tY) \cdot x_n) > \ell \} .$$
Fix some $L\ge T_1$.
For each $n$, let $$p_n = \frac{m_\R( I_{>L+\ell_0 +2,n
})}{t_n}.$$
Consider some sufficiently large~$n$ with $p_n >0$. Let $x^t = \exp(tY) \cdot x_n$ for $t \in [0,t_n]$. %
	Let $$b_0 = 0 \le a_1 <b_1<a_2<b_2 < \dots < a_J<b_J\le t_n = a_{J+1}$$
	be such that
	\begin{enumerate}
		\item $\tempfun(x^{a_{i+1}}) = \tempfun(x^{b_i}) = \ell_0$ for every $0\le i\le J$;
		\item (bounded excursions) for every $0\le i\le J$ and for all $b_i\le t \le a_{i+1}$, $$\tempfun(x^t)<L +\ell_0 +2;$$
		\item (deep excursions) for every $1\le i \le J$, $\tempfun(x^t)\ge \ell_0$ for all $a_i\le t\le b_i$ and for some $a_i< t< b_i$, $$\tempfun(x^t)\ge L +\ell_0 +2.  $$
	\end{enumerate}
Since $p_n>0$, we have $J\ge 1$.

 For each $1\le i \le J$, by the definition of $\phi$ we have $b_i - a_i \ge L \ge T_1$;  in particular, $J \le t_n/L$ and
	$$\log \| \calA \bigl( (b_i-a_i)Y, x^{a_i} \bigr) \| \le \delta(b_i - a_i) .$$
Also, for $0 \le i \le J$, we have
	$$ \log \| \calA \bigl( (a_{i+1}-b_i)Y, x^{b_i} \bigr) \| \le \max\{ C,
		(\chi + \delta ) (a_{i+1}-b_i) \}. $$
Let $l_n = \frac{1}{t_n} \sum_{i=1}^J (b_i-a_i)$. Note $p_n \le l_n$. Having taken $n$ sufficiently large, we have
	\begin{align*}
	\chi-\delta &\le \frac{1}{t_n} \log \| \calA( t_n Y, x_n) \| \\
	&\le \frac{1}{t_n} \sum_{i=1}^J \log \bigl\| \calA \bigl( (b_i-a_i)Y, x^{a_i} \bigr) \bigr\|
		+ \frac{1}{t_n} \sum_{i=0}^J \log \bigl\| \calA \bigl( (a_{i+1}-b_i)Y, x^{b_i} \bigr) \bigr\| \\
	&\le \frac{1}{t_n} \sum_{i=1}^J \delta(b_i-a_i)
		+ \frac{1}{t_n} \sum_{i=0}^J \max\bigl( C , (\chi + \delta ) (a_{i+1}-b_i)\bigr) \\
	& \le \delta l_n
		+ \frac{(J + 1) C }{t_n}
		+ (\chi + \delta) ( 1 - l_n )\\
		& \le \delta l_n
		+ \frac{C}{L} + \frac{C}{t_n}
		+ (\chi + \delta) ( 1 - l_n ).
	 \end{align*}
In particular, for sufficiently large $n$ it follows that
	 $$l_n \le \chi\inv\left[ \frac{C}{L} + \frac{C}{t_n}
		+ 2 \delta \right].$$

Thus, given any $\epsilon>0$, having taking $\delta>0$ sufficiently small, $L$ sufficiently large, and $t_n$ sufficiently large, it follows that $p_n\le l_n<\epsilon$ for all sufficiently large $n$ whence the family $\{\eta_n\}$ is uniformly tight.
\end{proof}

\subsection{Time-averaged cocycle, cut-off cocycle, and time-averaged exponents}\label{TimeAveraged}
We  define our key technical ingredients, the time-averaged and cut-off cocycles,
and study their basic properties.

Fix a unit vector $Y\in \liea_H$ with $\chi(Y, X_{\wtd H }) >0$.
Let $\{a^t\}$ denote the 1-parameter subgroup $a^t = \exp(t Y).$
Recall that $\Psi'\colon \lieg \times \P\calE\to \R$ denotes the infinitesimal generator for the norm-growth cocycle~$\Psi$ defined in \cref{s4.1}.
Let $\psi\colon\P\calE\to \R$ be the function
	\begin{align} \label{psiDefn}
	\psi(\xi) := \Psi'(Y, \xi)
	. \end{align}
Then
	\begin{align} \label{Psi=Intpsi}
	\Psi(a^T,\xi) = \int _0^T \psi (a^t \cdot \xi ) \,d t
	. \end{align}

Note that $\psi$ is continuous but need not be bounded.
 Moreover, given an $\{a^t\}$-invariant Borel probability measure~$\mu$ on $X_{\wtd H }$, it may be that $\psi$ is not $L^1(\mu)$. To deal with this lack of integrability, we replace $\psi$ with its forwards time average along the flow $\{a^t\}$. This new function will be bounded though not necessarily continuous.

\subsubsection{Time-averaged cocycle}\label{TAcoc}
We define the time-averaged cocycle and see that it is a bounded measurable function.
For any Borel function $\varphi \colon \P\calE \to \R$, write
	\begin{align*}
	\overline\varphi(\xi) &:= \limsup_{T\to \infty} \frac 1 T \int _0 ^T \varphi ( a^t \cdot \xi) \, d t \\
	\intertext{and}
	\underline\varphi(\xi) &:= \liminf_{T\to \infty} \frac 1 T \int _0 ^T \varphi ( a^t \cdot \xi) \, d t
	. \\
 \intertext{Also let}
	 \overline\Psi(a, \xi) &:= \limsup_{T\to \infty} \frac{1}{T} \Psi(a^T,\xi)
	. \end{align*}

We collect the following observations.
\begin{claim}\label{EasyPsi} \
\begin{enumerate}
	\item \label{EasyPsi-=}
	$\overline\psi(\xi) = \overline\Psi(a, \xi)$ for all $\xi \in\P \calE$.
	\item \label{EasyPsi-bdd}
	$\overline\psi$ and $\underline\psi$ are bounded measurable functions.
\end{enumerate}
Let $\calE_{\wtd H}$ denote the restriction of the vector bundle~$\calE$ to $X_{\wtd H}$.
Let $\eta$ be an $\{a^t\}$-invariant  Borel probability measure on $\P\calE_{\wtd H}$.
\begin{enumerate}[resume]
	\item \label{EasyPsi-ExplicitBoundsBigworld}
  For $\eta$-almost every $\xi\in \P\calE_{\wtd H}$
	we have
		$$-\chi(Y,X_{\wtd H})
		\le \underline\psi(\xi)
		\le \overline\psi(\xi) \le \chi(Y,X_{\wtd H}) .$$
	\item \label{EasyPsi-0}
  For $\eta$-almost every  $\xi\in \P\calE_{\wtd H}$
		 such that $\tempfun(a^t\cdot\xi)\ge \ell_0$ for all $t\ge t_0$ and some $t_0$, we have $\overline\psi(\xi) = \underline\psi(\xi) = 0$.
	\end{enumerate}
\end{claim}
\begin{proof}

Conclusion \pref{EasyPsi-=}~is immediate from~\eqref{Psi=Intpsi} and \pref{EasyPsi-bdd}~follows from \cref{tempered}.

For \pref{EasyPsi-ExplicitBoundsBigworld} and \pref{EasyPsi-0}, from \cref{tempered} there exists $k$ and, for each $\ell>0$ there exists a $C_\ell$, such that if $\phi(\xi)\le \ell$ then for all $T\ge 0$,
$$\log \|\Psi (a^T, \xi)\| \le k T + C_\ell.$$
Given an $\{a^t\}$-invariant  Borel probability measure $\eta$ on $\P\calE_{\wtd H}$, for  $\eta$-a.e.\ $\xi$ and  any $\delta>0$ there are $\ell_\delta$  and $T_\delta$ such that for all $T\ge T_\delta$,
\begin{align*}  m_\R\{0\le t \le T: \phi\bigr(a^t\cdot \xi \bigl)   \ge \ell_\delta \} &\le \delta T,\\
\log \|\Psi (a^T, \xi)\| &\le  (\chi(Y,X_{\wtd H})+\delta)  T.
\end{align*}
Fix such $\xi$ and consider any $\delta>0$.
Fix any $T\ge T_\delta$ sufficiently large so that $$T_0 =\sup\{ 0\le t \le T: \phi\bigr(\exp tY)\cdot \xi \bigl)\le \ell_\delta \}\ge T_\delta.$$
Then
\begin{align*}
\log \|\Psi (a^T, \xi)\| &= \log \|\Psi (a^{T_0}, \xi)\| + \log \|\Psi (a^{T-T_0}, a^{T_0}\cdot \xi)\|
\\&\le  (\chi(Y,X_{\wtd H})+\delta) T_0 + k(T-T_0) + C_{\ell_\delta}
\\&\le ( \chi(Y,X_{\wtd H}) +\delta)  T + k  \delta T + C_{\ell_\delta}.
\end{align*}
It follows that
$$ \overline\psi(\xi) \le \chi(Y,X_{\wtd H}) + (k+1)\delta.$$
The upper bound in \pref{EasyPsi-ExplicitBoundsBigworld} follows as $\delta>0$ was arbitrary.

Moreover, if $\tempfun(a^t\cdot\xi)\ge \ell_0$ for all $t\ge t_0$ and if $T$ is sufficiently large so that $T_0>t_0$ then, by \cref{horseforlunch}, given $\epsilon>0$ there is $C_\epsilon$ such that
\begin{align*}
\log \|\Psi &(a^T, \xi)\| = \log \|\Psi (a^{t_0}, \xi)\| + \log \|\Psi (a^{T_0-t_0}, a^{t_0}\cdot \xi)\|+ \log \|\Psi (a^{T-T_0}, a^{T_0}\cdot \xi)\|
\\&\le k  t_0 + C_{\phi(\xi)} + \omega \phi(a^{t_0}\cdot \xi) + \omega \ell_\delta + \epsilon (T_0-t_0) + C_\epsilon+ k(T-T_0) + C_{\ell_\delta}
\\&\le \hat C + \epsilon T +\omega \ell_\delta+ C_\epsilon+ k\delta T + C_{\ell_\delta}
\end{align*}
for some $\hat C$ independent of $\epsilon$ and $\delta$.  Then $\overline\psi(\xi)\le \epsilon +k \delta$ and the upper bound in \pref{EasyPsi-0}   follows as $\delta>0$ and $\epsilon>0$ are arbitrary.

The lower bounds in  \pref{EasyPsi-ExplicitBoundsBigworld} and \pref{EasyPsi-0} are similar.
\end{proof}

\subsubsection{The function \texorpdfstring{$\psi_\ell$}{psi_ell} and its properties under weak convergence} \label{s4.2}
Here we modify the function~$\psi$ to obtain a bounded function that has the same forwards time averages along orbits with good recurrence properties. This modified version of $\psi$ is our cut-off cocycle. There is a parameter choice involved in the process of cutting-off and we will eventually work with the cut-off cocycle for  well-chosen values of the parameter.    Although the new function will not be continuous, its discontinuity set will rather tame for most choices of parameters.  We also
study the time averages of the cut-off cocycle

Fix $\ell > 0$. Given $\xi\in \P\calE $, let
	\begin{align*}
	t_-( \xi) &:= \sup\{\, s\le 0 \mid a^s \cdot \xi = \ell \,\}, \\
	t_+( \xi) &:= \inf\{\, s\ge 0 \mid a^s \cdot \xi = \ell \,\}, \\
	\tau( \xi) &:= t_+( \xi)- t_-( \xi)
	\end{align*}
(with the convention that $\sup (\emptyset) = -\infty$ and $\inf(\emptyset ) = \infty $.)
Define $\psi_{\ell} \colon \P\calE\to \R$ as follows:
\begin{align} \label{psiellDefn}
\psi_\ell(\xi) = \begin{cases}
\hfil \psi(\xi) & \text{if $h(\xi)\le \ell$;} \\[\smallskipamount]
\displaystyle \frac {1 }{ \tau(\xi) }\int _{t_-(\xi)}^{t_+(\xi)} \psi(a^s \cdot \xi) \,d s
	& \vbox to 0pt{\vss\hbox{if $h(\xi) > \ell$ and} \hbox{\qquad $-\infty< t_- (\xi)< t_+ (\xi) <\infty$;}\vss} \\[\smallskipamount]
\hfil 0 & \text{otherwise}.
\end{cases}
\end{align}

From \cref{tempered}, we observe the following:
\begin{enumerate}
\item For each $\ell$, $\psi_\ell\colon \P\calE\to \R$ is a bounded measurable function;
\item moreover, the functions $\overline{\psi_\ell}\colon \P\calE\to \R$ are uniformly bounded in $\ell$.
\end{enumerate}

\begin{lemma}\label{sameinvmeas}
Let $\eta$ be an $\{a^t\}$-invariant, Borel probability measure on~$\P \calE$.

 For $\eta$-\ae~$\xi \in \P \calE$ and for all $\ell>h(\xi)$,
 $$\overline\psi(\xi) =\overline{\psi_\ell}(\xi) = \underline\psi_\ell(\xi)= \underline\psi(\xi).$$
 Moreover, for all almost every $\xi\in\P\calE$, $$\overline\psi (\xi)= \lim_{T\to \infty} \frac 1 T \int _0^T \psi(a^t \cdot \xi ) \, d t.$$
\end{lemma}
We remark that the final conclusion of \cref{sameinvmeas} does not follow directly from the pointwise ergodic theorem as $\psi$ is neither assumed to be $L^1(\eta)$ or non-negative.

\begin{proof}[Proof of \cref{sameinvmeas}]
Recall each $\psi_\ell$ is a bounded measurable function. By the pointwise ergodic theorem, for \ae $\xi\in \P\calE$ the following two conditions hold:
\begin{enumerate}
\item \label{prop1} For every rational $\ell$, $$\overline{\psi_\ell}(\xi) = \underline\psi_\ell(\xi)= \lim_{T\to \infty}\frac 1 T \int _0^T \psi_\ell(a^t \cdot \xi ) \, d t.$$
\item \label{prop2} Given $0<\delta<1$, there is a rational $\ell_\delta(\xi)>0$ with $\ell_\delta(\xi) > h(\xi)$ such that for all sufficiently large $T$,
$$m_\R \left(\{t\in [0,T]: h(a^{t}\cdot \xi)\le \ell_\delta(\xi)\}\right) \ge (1-\delta )T.$$
Moreover, we may assume $\ell_\delta(\xi)\searrow h(\xi)$ as $\delta\nearrow 1$.
\end{enumerate}

Consider $\xi$ satisfying conditions \pref{prop1} and \pref{prop2}. Fix any $0<\delta_0<1$ and set $\ell_1 := \ell_{\delta_0}(\xi)$.
For any $\ell_2\ge \ell\ge \ell_1$ and $T>0$ such that $h(a^T\cdot \xi)\le \ell$ we have \begin{equation} \label{poopforface} \int _0^{T} \psi_{\ell_2}(a^t \cdot \xi ) \, d t= \int _0^{T} \psi_{\ell}(a^t \cdot \xi ) \, d t.\end{equation}
Set $$L= \lim_{T\to \infty}\frac 1 T \int _0^T \psi_{\ell_1}(a^t \cdot \xi ) \, d t.$$

Observe that $h(\xi)<\ell_1\le \ell$ and from condition \pref{prop2} there is a sequence $\{t_n\}$ with $t_n\to \infty$ such that $h(a^{t_n}\cdot \xi)\le \ell_1 \le \ell$ for all $n$.
 Taking $\ell=\ell_1$, for any rational $\ell_2\ge \ell_1$, \eqref{poopforface} and condition \pref{prop1} imply that
 \begin{equation} \label{eq:houseoffarts} \overline{\psi_{\ell_2}}(\xi) = \underline\psi_{\ell_2}(\xi)=L.\end{equation}
 In particular, the value of $\overline{\psi_{\ell}}(\xi) = \underline\psi_{\ell}(\xi)$ is constant over all rational $\ell>h(\xi)$.
 We claim the same holds for irrational $\ell\ge h(\xi). $

\begin{claim} \label{wetandclaimy} For any $\ell\ge \ell_1$,
$\overline{\psi_\ell}(\xi) = \underline\psi_\ell(\xi)= L.$
\end{claim}
\begin{proof}
Consider any $0<T_0< T_1$ such that $$h(a^{T_0}\cdot \xi)= h(a^{T_1}\cdot \xi)= \ell$$ and $h(a^{t}\cdot \xi)>\ell$ for all $T_0<t<T_1$.
By definition, the function \begin{equation} \label{eeff}T\mapsto\frac 1 T \int_0^T \psi_{\ell} ( a^t \cdot \xi ) \, dt\end{equation} is monotonic on $[T_0, T_1]$. Explicitly, with $A= \int_0^{T_0} \psi_{\ell} ( a^t \cdot \xi ) \, dt$ and $B=\int_{T_0}^{T_1} \psi_{\ell} ( a^t \cdot \xi ) \, dt$, for $T\in [T_0, T_1]$, \eqref{eeff} is equivalent to
$$\frac A T + \frac 1 T \frac {T - T_0}{T_1 - T_0} B
= \frac 1 T\left( A- \frac{ T_0B}{T_1 - T_0}\right) + \frac B {T_1 - T_0} $$
hence the derivative of \eqref{eeff} has constant sign on $ [T_0, T_1]$. It follows that
\begin{align}
\label{eq:olivesmell}\overline {\psi_{\ell}} (\xi) & =
\lim_{T_0\to \infty} \sup\left\{\frac 1 T \int _0 ^T \psi_{\ell} ( a^t \cdot \xi) \, d t : T\ge T_0, h(a^T \cdot \xi)\le \ell\right\}\end{align}
and \begin{align*}
\underline {\psi}_{\ell} (\xi) & =
\lim_{T_0\to \infty} \inf\left\{\frac 1 T \int _0 ^T \psi_{\ell} ( a^t \cdot \xi) \, d t : T\ge T_0, h(a^T \cdot \xi)\le \ell\right\}.\end{align*}

Take any rational $\ell_2 \ge \ell$.
\eqref{eq:olivesmell} combined with \eqref{poopforface} imply
 $\overline{\psi_{\ell_2}}(\xi) \ge \overline{\psi_{\ell}}(\xi)$ and similarly $\overline{\psi_{\ell}}(\xi) \ge \overline{\psi_{\ell_1}}(\xi)$.
Similarly, $$\underline{\psi}_{\ell_2}(\xi) \le \underline{\psi}_{\ell}(\xi)\le \underline{\psi}_{\ell_1}(\xi).$$
The claim then follows from \eqref{eq:houseoffarts}.
\end{proof}

We return to the proof of the lemma. By \cref{tempered}, there is $C>0$ such that for all $\zeta \in \P\calE$ and $g\in G$,
$$ |\Psi(g,\zeta)| \le C + C h(\zeta) + C d(g,\1) .$$
Fix any $0<\delta<1$ and take $\ell\ge \ell_\delta(\xi)$. Then
$$m_\R \left(\{t\in [0,T]: h(a^{t}\cdot \xi)\ge \ell\}\right) \le \delta T$$
for all sufficiently large $T$.

Consider some large $T>0$. If $h(a^{T}\cdot \xi)\le \ell$ then,
 by definition, $$ \frac 1 T \int _0 ^T \psi ( a^t \cdot \xi) \, d t =
 \frac 1 T \int _0 ^T \psi_\ell ( a^t \cdot \xi) \, d t .$$
Given any $T_0<T_1$ such that $h(a^{T_0}\cdot \xi) = h(a^{T_1}\cdot \xi) = \ell $ and $ h(a^T\cdot \xi)\ge \ell $ for all $T_0\le T\le T_1$ then, having considered $T_0$ sufficiently large, we have
 \begin{align*}
 \Big |\frac 1 T& \int _0 ^T  \psi ( a^t \cdot \xi) \, d t -
 \frac 1 T \int _0 ^T \psi_\ell ( a^t \cdot \xi) \, d t
 \Big|
 \\ &\le
\frac 1 T \left(
\Bigl| \int _0 ^{T_0}  \psi ( a^t \cdot \xi)-  \psi_\ell ( a^t \cdot \xi)  \, d t \Bigr|
 +
  \Big |\int _{T_0} ^T  \psi ( a^t \cdot \xi) \, d t \Big| + \Big| \int _{T_0} ^T \psi_\ell ( a^t \cdot \xi) \, d t
 \Big|\right)\\
& \le \frac 1 T \left( 0 + \bigl(C + C (T- T_0) + C \ell\bigr )+ \frac {T-T_0}{T_1-T_0}\bigl( C + C (T_1- T_0) + C \ell\bigr)\right)
 \\& \le \frac {2 C +2C\ell} T +2C \frac{T- T_0}{T}
 \\& \le \frac {2 C +2C\ell} T +2C \delta.
\end{align*}
Given any $\epsilon>0$, by first taking $\delta>0$ sufficiently small and taking $\ell\ge \ell_\delta(\xi)$, by considering $T$ sufficiently large we conclude that $|\overline \psi(\xi) - \overline{\psi_\ell}(\xi)| <\epsilon$ and $|\underline \psi(\xi) - \underline{\psi}_{\ell}(\xi)| <\epsilon.$

From \cref{wetandclaimy}, for any  $\ell>h(\xi)$ we have $\overline{\psi_\ell}= \underline{\psi}_\ell=L=:\overline{\psi_{ \ell_{\delta_0(\xi)}}}$ for any choice of $0<\delta_0<1$.   The independence of $\ell$ then implies for all such $\ell$ that
 $\overline{\psi_\ell}= \overline{\psi} $, $\underline{\psi}_\ell =\underline{\psi}$, and thus $\overline{\psi}= \underline{\psi} $. The conclusions of the lemma then follow.
\end{proof}


\subsubsection{A classification of \texorpdfstring{$\{a^t\}$-}{}orbits and discontinuity points of $\psi_\ell$} \label{OrbitClassification}
Now we show that the cut-off cocycle $\psi_\ell$ has relatively tame discontinuities.  We do this by explicitly
identifying all possible discontinuities.

Fix some $\ell \in [0,\infty)$ such that the level surface $h\inv(\ell)$ is a codimension-one, $C^\infty$ submanifold in $G/\Gamma$; by Sard's theorem and the implicit function theorem
this holds for a.e.\ $\ell$ in the range of $h$.
Recall we extend $h\colon G/\Gamma\to \R$ to $h\colon \P\calE\to \R$ via the canonical projection $\pi\colon \P\calE\to G/\Gamma$. We categorize points $x\in G/\Gamma$ and thus $\xi \in \P\calE$ with $h(\xi) > \ell$ based on how its orbit under the flow~$\{a^t\}$ meets $h\inv (\ell)$.

\begin{description}
\item [Tangential orbits {\normalfont($T_\ell$)}] Write $T_\ell$ for the set of points $\xi\in \P\calE$ with $h(\xi) > \ell$, such that the forward or backwards $\{a^t\}$-orbit of~$\pi(\xi)$ in $G/\Gamma$ is tangent to $h\inv (\ell)$ at some point of the orbit.

\item [Non-recurrent {\normalfont($N_\ell$)}] We say $\xi\in \P\calE$ with $h(\xi)>\ell$ is \emph{forward non-recurrent} to height~$\ell$ if
\begin{enumerate}
\item there is $t_0<0$ such that $h(a^{t_0}\cdot \xi)= \ell$ and the $\{a^t\}$-orbit of~$\pi(\xi)$ is transverse to~$h\inv (\ell)$ in $G/\Gamma$ at~${t_0}$;
and
\item $h(a^{t}\cdot \xi) > \ell$ for all $t>t_0$.
\end{enumerate}
We similarly say that $\xi$ is \emph{backwards non-recurrent} to height~$\ell$ if there is $t_0>0$ such that $h(a^{t_0}\cdot \xi)= \ell$ and the $\{a^t\}$-orbit of~$\pi(\xi)$ is transverse to~$h\inv (\ell)$ in $G/\Gamma$ at~${t_0}$ and if $h(a^{t}\cdot \xi) > \ell$ for all $t<t_0$.
Write $N_\ell$ for the set of points that are either forward or backwards non-recurrent to height~$\ell$.

\item[Excursions {\normalfont($E_\ell$)}] We say $\xi\in \P\calE$ with $h(\xi)>\ell$ is in an \emph{excursion} if there are $t_0< 0< t_1$ such that
\begin{enumerate}
	\item $h(a^{t_0}\cdot \xi) = \ell$ and the $\{a^t\}$-orbit of $\pi(\xi)$ is transverse in $G/\Gamma$ to $h\inv (\ell)$ at ${t_0}$;
	\item $h(a^{t_1}\cdot \xi) = \ell$ and the $\{a^t\}$-orbit of $\pi(\xi)$ is transverse in $G/\Gamma$ to $h\inv (\ell)$ at ${t_1}$;
	and
	\item $h(a^t \cdot \xi) > \ell$ for all $t\in (t_0,t_1)$.
\end{enumerate}
Write $E_{\ell}$ for the set of points in excursions. We note that the transversality at the endpoints ensures that $E_{\ell}$ is open.

\item[Deep points {\normalfont($D_\ell$)}] We say $\xi$ is \emph{$\ell$-deep} if $h(a^t\cdot \xi) > \ell$ for all $t\in \R$. Write $D_{\ell}$ for the set of
$\ell$-deep points.

\end{description}

Given any $\xi\in \P\calE$ with $h(\xi) > \ell$, we have that $\xi$ is in either $ T_\ell, N_\ell, E_\ell, $ or $D_{\ell}$.

\begin{claim} \label{Measure(T)=0}
If $\eta$ is any $\sigma$-finite Borel measure on $\P\calE$ then for $m_\R$-a.e.\ $\ell \in \R$, we have $\eta \bigl( h\inv (\ell) \bigr) = 0$ and $\eta(T_\ell) = 0$.
\end{claim}

\begin{proof}
We have $\eta \bigl( h\inv (\ell) \bigr) = 0$ for all but countably many~$\ell$.

Consider the closed subset
	$$ \textstyle \left\{ (\ell, \xi, t) : \frac{\partial}{\partial s} h(a^s \cdot \xi) \big|_{s = t} = 0 , h(a^t\cdot \xi) = \ell\right\} \subseteq [0,\infty) \times \P\calE \times \R $$ and let
 $\mathcal{T}\subset [0,\infty) \times \P\calE$ denote its image under the projection $$ [0,\infty) \times \P\calE \times \R\to [0,\infty) \times \P\calE.$$
Since $[0,\infty) \times \P\calE \times \R$ is locally compact, $\mathcal{T}$ is Borel.

For  any $\xi\in \P\calE$, the function  $t\mapsto h(a^t\cdot \xi)$ is either constant or, by Sard's theorem,  $m_\R\bigl(\{\ell\in [0,\infty) : (\ell, \xi)\in \mathcal{T}\}\bigr)=0$.
For almost every $\ell$, we have $\{\ell\}\times T_\ell \subset \mathcal{T}$.
 It then follows from Fubini's Theorem that $\eta(T_\ell) = 0$ for $m_\R$-a.e.~$\ell$.
 \end{proof}

As observed above, each function $\psi_\ell$ is bounded but need not be continuous; however, its discontinuities are rather tame.
To compare values of the two height functions $\phi$ and $h$ on $\calE_{\wtd H}$ set \begin{equation}\label{eq:cutoff}\hcutoff := \sup \{h(\xi): \xi \in X_{\wtd H}, \phi(\xi)\le \ell_0\}\end{equation}
where $\ell_0$ is the critical height in \cref{horseforlunch}.
\begin{claim} \label{discontinuities}
For any $\ell>0$ such that $h\inv(\ell)$ is a codimension-one, $C^\infty$ submanifold in $G/\Gamma$, we have the following:
\begin{enumerate}
\item \label{discontinuities-TZND}
The set of discontinuity points of the function $\psi_{\ell}\colon \P \calE\to \R$ is contained in $ h\inv (\ell)\cup T_\ell \cup N_\ell \cup D_\ell$.
\item \label{discontinuities-TZ}
If   $\ell\ge \hcutoff$ then the set of discontinuity points of the restriction $\restrict{\psi_{\ell}}{\P\calE_{\wtd H}}\colon \P \calE_{\wtd H}\to \R$ is contained in $h\inv (\ell)\cup T_\ell$.
\end{enumerate}
\end{claim}
\begin{proof}
If $h(\xi)<\ell$ then, as $\psi$ is continuous, $\xi$ is a continuity point of~$\psi_\ell$.

Suppose $h(\xi)>\ell$ and $\xi \in E_\ell$. Since $E_\ell$ is open, if $\xi'\in \P\calE$ is sufficiently close to $\xi$ then $\xi'\in E_\ell$ and the $\{a^t\}$-orbit of $\xi'$ stays close to the $\{a^t\}$-orbit of $\xi$ for a sufficiently large amount of time. In particular, taking $\xi'$ sufficiently close to $\xi$ ensures that the time average defining $\psi_\ell(\xi)$ in \eqref{psiellDefn} is close to the time average defining $\psi_\ell(\xi')$. Thus $\psi_\ell$ is continuous at~$\xi$ which establishes \pref{discontinuities-TZND}.

To establish~\pref{discontinuities-TZ}, let $\xi\in N_\ell \cup D_\ell$. By definition, $\psi_\ell(\xi)=0$.
Suppose $\{\xi_n\} \subseteq \P \calE_{\wtd H}$ converges to $\xi$  and $\psi_\ell(\xi_n) \neq 0$ for all~$n$. We then have $-\infty< t_- (\xi_n)< t_+ (\xi_n) <\infty$ for all~$n$, and $t_+ (\xi_n) - t_- (\xi_n) \to \infty$. From \cref{horseforlunch},
$\psi_\ell(\xi_n) \to 0$.
Indeed let $L= \sup\{\phi(\xi) :\xi \in X_{\wtd H}, h(\xi)\le \ell\}$.
Then for any $\epsilon>0$ there is $C_\epsilon$ so that
\[\psi_\ell(\xi_n)\le \frac 1 {t_+ (\xi_n) - t_- (\xi_n) }\bigl( 2 \omega L + \epsilon(t_+ (\xi_n) - t_- (\xi_n)) + C_\epsilon\bigr). \qedhere\]
\end{proof}

\subsubsection{Sequences of measures on $\P\calE$}
We now consider a sequence $\{\eta_n\}$ of Borel probability measures on $\P \calE$ converging to a Borel probability measure~$\eta$. Under suitable hypotheses controlling the discontinuity sets of $\psi_\ell$ and using \cref{sameinvmeas,discontinuities}
 we will control the integral $\int \Psi(a, \cdot) \, d \eta$ in terms of integration against $ \eta_n$.

\begin{corollary}\label{average limit}
Let $\eta_n$ be a sequence of Borel probability measures on $\P \calE$ converging in the weak-$*$ topology to an $\{a^t\}$-invariant, Borel probability measure~$\eta$.

Suppose each $\eta_n$ is $\{a^t\}$-invariant and that either
\begin{enumerate}
\item \label{baboonfeces1} $\eta\bigl(h\inv (\ell)\cup T_\ell \cup N_\ell \cup D_\ell \big) = 0$ for $m_\R$-\ae $\ell$, or
\item \label{baboonfeces2} each $\eta_n$ is supported on $\P\calE_{\wtd H}$.
\end{enumerate}
Then $$\int \overline\Psi(a, \cdot) \,d \eta = \lim_{n\to \infty} \int \overline\Psi(a, \cdot) \,d \eta_n.$$
In the case that $\eta_n$ are possibly not $\{a^t\}$-invariant,
\begin{enumerate}[resume]

\item \label{baboonfeces3} if  for $m_\R$-a.e.\ sufficiently large   $\ell>0$ and  for every $n$,  each $\eta_n$ is supported on $\P \calE_{\wtd H}$, $\psi\in L^1(\eta_n)$, and
$$\int \psi \ d \eta_n = \int \psi_\ell \ d \eta_n$$
then $$\int \overline\Psi(a, \cdot) \,d \eta = \lim_{n\to \infty} \int \psi \ d \eta_n. $$
\end{enumerate}
\end{corollary}
We remark that if an $\{a^t\}$-invariant Borel probability measure $\eta$ on $\P\calE$ projects to the Haar measure on $G/\Gamma$ then, as Haar-\ae point in $G/\Gamma$ has dense orbit in $G/\Gamma$, the hypothesis in \pref{baboonfeces1} is automatically satisfied. We also remark that \pref{baboonfeces2} and \pref{baboonfeces3} only apply to measures supported on the restricted bundle $\calE_{\wtd H}. $
\begin{proof}
First, consider any $\{a^t\}$-invariant, Borel probability measure $\wtd \eta$ supported on $\P\calE$.
Using \cref{sameinvmeas}, dominated convergence, and the $\{a^t\}$-invariance of $ \wtd \eta$ we have
\begin{align} \label{cottagecheeselegs0}
\int \overline\Psi(a, \cdot) \,d \wtd \eta = \int \overline \psi \,d \wtd \eta = \lim_{\ell\to \infty} \int \overline{ \psi_\ell} \,d \wtd \eta = \lim_{\ell\to \infty} \int \psi_\ell \,d \wtd \eta.\end{align}

We apply \cref{Measure(T)=0} and \cref{discontinuities}. Under the hypothesis in parts \pref{baboonfeces1}, \pref{baboonfeces2}, or  \pref{baboonfeces3},  for almost every $\ell\ge \ell_0$ the restriction of each $\psi_\ell$ to $\P\calE_{\wtd H}$ is continuous on a set of full $\eta$-measure; it follows (see for example \cite[Theorem 2.7)]{MR1700749} or  \cite[Corollary 2.2.10]{MR3837546})
that
\begin{equation}\label{lasagnehouse}\int \psi_\ell \,d \eta = \lim _{n\to \infty}\int \psi_\ell \,d \eta_n.\end{equation}
 Part \pref{baboonfeces3} follows by applying \eqref{cottagecheeselegs0}   with $\wtd \eta=\eta$ and an appropriate choice of such $\ell$.

 For parts  \pref{baboonfeces1} and  \pref{baboonfeces2}, we utilize  a  uniform (in $n$) version of \eqref{cottagecheeselegs0}: Using that $\eta_n\to \eta$,
 given $\delta>0$ there is $\ell_\delta$ such that for all $\ell\ge \ell_\delta$ and all sufficiently large $n$ we have
 $$\eta_n \{\xi: h(\xi)\ge \ell\} \le \eta_n \{\xi: h(\xi)\ge \ell_\delta \} \le \delta.$$
 Using that each $\eta_n$ is $\{a^t\}$-invariant, \cref{sameinvmeas} then implies $$\eta_n\{ \xi : \overline\psi(\xi) =\overline{\psi_\ell}(\xi)\} \ge (1- \delta)$$ for all sufficiently large $n$.
There is $C$ independent of $\ell$ such that $|\overline{\psi_{\ell}}(\xi) - \overline \psi(\xi)|\le C$ for all $\xi\in \P\calE$.
Then for $\ell\ge \ell_\delta$,
\begin{align}\label{cottagecheeselegs2}
\left|\int \overline\Psi(a, \cdot) \,d \eta - \int \psi_{\ell} \,d \eta\right| \le \delta C,\quad \quad
\intertext{and for all sufficiently large $n$ and all $\ell\ge \ell_\delta$,}
\left|\int \overline\Psi(a, \cdot) \,d \eta_n - \int \psi_{\ell} \,d \eta _n\right| \le \delta C. \label{cottagecheeselegs3}
\end{align}

Under the hypothesis of parts \pref{baboonfeces1} or \pref{baboonfeces2}, there exists a sequence $\{\ell_j\}$ with $\ell_j\to \infty$ such that  \eqref{lasagnehouse} holds.
 Given any $\delta>0$, first take $\ell_j> \ell_\delta$  and then  take $n$ sufficiently large in \eqref{lasagnehouse} so that $$\left| \int \psi_{\ell_j} \,d \eta - \lim _{n\to \infty}
 \int \psi_{\ell_j} \,d \eta_n\right|<\delta.$$
From \eqref{cottagecheeselegs2} and \eqref{cottagecheeselegs3},
$$\left|\int \overline\Psi(a, \cdot) \,d \eta - \int \overline\Psi(a, \cdot) \,d \eta_n  \right| \le (2C+1) \delta$$
and conclusions \pref{baboonfeces1} and \pref{baboonfeces2} follow as $C$ is independent of $\delta$.
\end{proof}

To apply the above corollary, we often need to verify tightness of the family $\{\eta_n\}$. The following corollary of Lemma \ref{lem:maxpathstight} gives sufficient criteria for a sequence of measures supported on $\P\calE_{\wtd H}$ to be uniformly tight.
\begin{corollary}\label{cor:tight}
Let $\eta_n$ be any sequence of $\{a^t\}$-invariant, Borel probability measures on $\P\calE_{\wtd H}$ such that
$$\int \overline\psi \,d\eta_n= \chi(Y, X_{\wtd H})$$
for every~$n$. Then the family $\{\eta_n\}$ is uniformly tight.
\end{corollary}

 \begin{proof}
 Let $\eta$ be any $\{a^t\}$-invariant, Borel probability measure on $\P\calE_{\wtd H}$ with $\int \overline\psi \,d\eta= \chi(Y, X_{\wtd H}).$
From \cref{EasyPsi}\pref{EasyPsi-ExplicitBoundsBigworld} we have $\overline\psi(\xi) = \chi(Y, X_{\wtd H})$ for $\eta$-\ae $\xi\in \P\calE_{\wtd H}$.
In particular, for $\eta$-a.e. $\xi= \bigl( x,[v] \bigr) \in \P\calE_{\wtd H}$,
$$ \chi(Y, X_{\wtd H}) \ge \limsup_{T\to \infty}\frac{1}{T} \log \|\calA(a^T,x) \|\ge \limsup_{T\to \infty} \frac{1}{T}\Psi(a^T,\xi) = \overline\psi (\xi)= \chi(Y, X_{\wtd H}).$$

For $\eta$-almost every $\xi\in \P\calE_{\wtd H}$, there is $\ell_1\ge \phi(\xi)$ and a sequence $t_j\to \infty$ such that $\phi(a^{t_j}\cdot \xi)\le \ell_1$ for all $j$.
If $\overline\psi(\xi) = \chi(Y, X_{\wtd H})$, then from \cref{GetXnAndTn}, we further have that  $\phi(a^{t'_j}\cdot \xi)\le \ell_0$ for some sequence $t'_j\to \infty$.

From the pointwise ergodic theorem, for $\eta$-almost every $\xi\in \P \calE_{\wtd H}$ and every rational $\ell$ the limit $$\lim_{T\to \infty} \frac 1 T \, m_\R \bigl( \{\, t \in [0,T] \mid \tempfun ( a^t\cdot \xi ) \ge \ell \,\} \bigr)$$ exists.    \cref{lem:maxpathstight} implies the following: given any $\delta > 0$ there exists a rational $\ell_\delta > 0$ such that for $\eta$-almost every $\xi\in \P\calE$ with $\overline\psi(\xi) = \chi(Y, X_{\wtd H})$,
$$\lim_{T\to \infty} \frac 1 T \, m_\R \bigl( \{\, t \in [0,T] \mid \tempfun ( a^t\cdot \xi ) \ge \ell_\delta \,\} \bigr) < \delta .$$
Then for any $\{a^t\}$-invariant, Borel probability measure on $\P\calE_{\wtd H}$ with $\int \overline\psi \,d\eta_n= \chi(Y, X_{\wtd H}),$
the pointwise ergodic theorem implies \begin{equation*}\eta_n \bigl( \{\,\xi\in \P\calE_{\wtd H} \mid h(\xi) \ge \ell_\delta \,\} \bigr) < \delta. \qedhere \end{equation*}
\end{proof}

\subsection{Invariance and additivity of the time-averaged cocycle}
While our modifications of the norm-growth cocycle are beneficial for handling issues of convergence, a priori it destroys all connection the group action even just over $A$ or $A_H$.  Here we show that the space average of the time-averaged cocycles remains a homomorphism from $A$. This turns out to be all we need to run proofs analogous to those in  \cite{BFH} and \cite{BFH-SLnZ}.

We begin with the following consequence of the pointwise ergodic theorem and \cref{sameinvmeas}.
\begin{claim} \label{lentilvomit}
Let $\eta$ be an $\{a^t\}$-invariant, Borel probability measure on~$\P \calE$.
 If $b\in G$ is in the centralizer of $\{a^t\}$ then for $\eta$-\ae $\xi\in \P\calE$, $$\overline\Psi(a, \xi)=\overline\Psi(a, b \cdot \xi).$$
\end{claim}
\begin{proof}
As $b_*\eta$ and $\eta$ are $\{a^t\}$-invariant, from \cref{sameinvmeas}, for $\eta$-almost every $\xi\in \P\calE$ we have
$$\overline\Psi(a, b\cdot \xi) = \lim_{T \to \infty} \frac{1}{T} \Psi(a^{T} , b \cdot \xi ) \quad \text {and } \quad \overline\Psi(a, \xi) = \lim_{T \to \infty} \frac{1}{T} \Psi(a^{T} , \xi ).$$
We have the cocycle relation $$ \Psi(a^{T} , b \cdot \xi )
 = \Psi(a^{T}, \xi ) + \Psi( b, a^{T} \cdot \xi ) - \Psi(b, \xi) .$$
Moreover, for almost every $\xi$ there is a compact set $E\subset X$ and $T_n\to \infty$ such that $a^{T_n}\cdot \xi\in E$.
For such $\xi$,
$$| \Psi(b, a^{T_n} \cdot \xi ) - \Psi(b , \xi ) |$$
is uniformly bounded in $n$
whence \begin{align*}\lim_{T \to \infty} \frac{1}{T} \Psi(a^{T} , b \cdot \xi ) = \lim_{T \to \infty} \frac{1}{T} \Psi(a^{T} , \xi ). &\qedhere\end{align*}
\end{proof}

For every $a \in A$, the function $\xi\mapsto \overline\Psi(a, \cdot)$ is bounded as observed in \cref{EasyPsi}. Thus, the integral $\int \overline\Psi{(a, \cdot)} \,d \eta$ is well-defined for every Borel probability measure~$\eta$ on~$\P\calE$.
\begin{lemma}\label{claim:linear}
For any $A$-invariant, Borel probability measure~$\eta$ on $\P\calE$,
the function $A \to \R,$ $$a\mapsto \int \overline\Psi{ ( a, \cdot )} \,d \eta$$ is linear.
\end{lemma}
Above, we implicitly identify $A$ with its Lie algebra $\liea$ in the definition of linearity.
\begin{proof}
For any one-parameter subgroup~$\{a^t\}$ of~$A$ we have $\overline\Psi( a^t, \cdot) = t \, \overline\Psi( a^1, \cdot)$ for all $t > 0$.
It remains to show additivity.

For each $T>0$, \cref{lentilvomit} gives the equality \begin{equation}\label{unicornboobs}\overline\Psi(b^T, a^T \cdot \xi)= \overline\Psi(b^T, \xi)\end{equation} for almost every $\xi$. By Fubini's thoerem, for $\eta$-\ae $\xi$ equality \eqref{unicornboobs} holds for $m_\R$-a.e.\ $T>0$.

For each $a \in A$ and $\epsilon>0$, let
	$$ \mathcal{G}_{a,\epsilon,T_0} :=\left\{ \xi\in \P \calE : \text{$\big| \overline\Psi(a,\xi) - \tfrac 1 t \Psi(a^t, \xi) \big| <\epsilon$
		for all $t\ge T_0$} \right\} .$$
Fix $a,b \in A_H$ and $\epsilon>0$. For almost every $\xi$, \cref{sameinvmeas}
implies there is some~$T_0$ such that $\xi\in \mathcal{G}_{a,\epsilon,T_0} \cap \mathcal{G}_{ab,\epsilon,T_0} \cap \mathcal{G}_{b,\epsilon,T_0}$; invariance of $\eta$ implies $\xi$ may be chosen so that for some $T\ge T_0$,
$a^{T} \cdot \xi\in \mathcal{G}_{b,\epsilon,T_0}$ and \eqref{unicornboobs} holds.

 Using the cocycle relation
	$$ \Psi(a^T, \xi) +\Psi(b^T, a^T\cdot \xi) - \Psi ( (ab)^T, \xi ) = 0 ,$$
we conclude that
\begin{align*}
\bigl| \overline\Psi(a,\xi) &+ \overline\Psi(b,\xi) - \overline\Psi(ab,\xi) \bigr|\\
 &\le \left |\overline\Psi(a,\xi) - \tfrac 1 T \Psi(a^T, \xi) \right |
+ \left |\overline\Psi(b,a^T\cdot \xi) - \tfrac 1 T \Psi(b^T, a^T\cdot \xi) \right |\\
&\hskip5em+ \left |\overline\Psi ( (ab)^T, \xi) - \tfrac 1 T \Psi ( (ab)^T, \xi ) \right |\\
&< 3 \epsilon
.\end{align*}
It follows that $\overline\Psi(a,\xi) + \overline\Psi(b,\xi) = \overline\Psi(ab,\xi) $ for almost every~$\xi$; additivity of the integrals then follows.
\end{proof}

\subsection{Proof of \texorpdfstring{\cref{thm:mainQR1}}{Proposition 10.1}} \label{ProofOfmainQR1}
Recall we fix $H$ to be a standard $\Q$-rank-1 subgroup of~$G$, $N$ is a unipotent $\Q$-subgroup normalized by~$H$, and $\wtd H = H\ltimes N$. We assume that the
restriction of~$\alpha$ to~$\Gamma\!_H$ fails to have $\len_\Gamma$-subexponential growth.

To prove \cref{thm:mainQR1}, in \cref{sec:buildmeas} we first construct in \cref{TimeAvgExp>0}  a measure~$\eta$ on~$ \P \calE_{\wtd H}$ that has a positive time-averaged Lyapunov exponent.   We will then average this measure so that it projects to the Haar measure on $G/\Gamma$ while maintaining a non-zero Lyapunov exponent for the limiting measure.
The averaging step breaks up into two cases: \cref{AvgQRank=1} considers the case where $\rank_\Q(\Gamma) = 1$; \cref{AvgQRank>1} considers the case where $\rank_\Q(\Gamma) \ge 2$.

\subsubsection{Construction of a measure with positive time-averaged Lyapunov exponent}\label{sec:buildmeas}
We harvest the results of our prior efforts to construct an $(A_HN)$-invariant Borel probability
measure with positive time-averaged Lyapunov exponent.

Let $\liea_H$ be the Lie algebra of a split Cartan subgroup~$A_H$ of~$H$.
Assume as in \cref{thm:mainQR1} that the restriction
	$$\restrict{\calA_0}{\Gamma\!_H}\colon \Gamma\!_H \times \calE _0\to \calE_0$$
does not have uniform $\len_\Gamma$-subexponential growth.
 From \cref{HavePosExp}, we have $\chi(\liea_H, X_{\wtd H}) >0$. From \cref{basisvector}, there is $Y\in \liea_H$ with
 $\|Y\|=1$ and $\chi(Y, X_{\wtd H}) >0$. Write $a^t = \exp (tY)$ and set $ a = a^1.$

From \Cref{GetXnAndTn}, there exist points $x_n\in X_{\wtd H}$ with $\tempfun(x_n) \le \ell_0$ and $t_n \to \infty$ such that for every $n$, $\tempfun(a^{t_n}\cdot x_n) \le \ell_0$ and
$$ \lim_{n\to \infty } \frac 1 { t_n }\log \|\calA(a^{t_n}, x_n)\| = \chi(Y, X_{\wtd H}).$$
For each $n$, take $\xi_n\in \P \calE_{\wtd H}$ with $\pi_\calE(\xi_n) =x_n$ such that
	\begin{align} \label{LimPhi=chi}
	\lim_{n\to \infty } \frac 1 { t_n }\Psi(a^{t_n} ,\xi_n)
	= \chi(Y, X_{\wtd H} ).
	\end{align}

There is a \Folner sequence in $\{a^t\}\cdot N $ of the form
	$$ F_\kappa(t_n) = \bigl\{a^s \mid 0\le s\le t_n \bigr\} \cdot N_\kappa(t_n) ,$$
where $N_\kappa(t) := \exp_\lien\bigl(\{\, Z\in \lien \mid \|Z\|\le e^{\kappa t}\,\}\bigr) $ and $\kappa>0$ is chosen sufficiently large as in \cref{sec:not}.
For $n\in \N$, set
	$$\eta_n:= F_\kappa(t_n)\ast \delta_{\xi_n}.$$

\begin{lemma} \label{LimEta=chi}
We have
$$ \lim_{n\to \infty} \int \psi \,d \eta_n = \chi(Y, X_{\wtd H} ) .$$
\end{lemma}
\begin{proof}
From \pref{LimPhi=chi} it suffices to show
$$\displaystyle \lim_{n\to \infty} \int \psi \,d \eta_n = \lim _{ n\to \infty} \frac 1 { t_n }\Psi(a^{t_n}, \xi_n)  .$$
For all $g \in N_\kappa(t_n)$ we have
\begin{equation}\label{thatdontmakeyouaho}
(a^{t_n}g ) \cdot \xi_n = (g' a^{t_n})\cdot \xi_n	\end{equation}
where $g' \in N_{2\kappa}(t_n)$.

Recall that $\tempfun(x_n)\le \ell_0$ and $\tempfun(a^{t_n}\cdot x_n)\le \ell_0$. Recall also from {\cref{eggface}}  that the action of~$N$ on $X_{\wtd H}$ has $\tempfun$-tempered subexponential growth:  for any $\epsilon > 0$, there is a constant~$C_\epsilon$ such that
	$$ \text{$\Psi(g'',\xi) \le C_\epsilon + \heightgrowth  \tempfun(\xi) + \epsilon  d(g'', \1) $ for all $g'' \in N$ and $\xi \in \calE_{\wtd H}$} .$$
Since $N$ is unipotent, identifying $N\subset G$ with its image in $F$ or in $\Ad(G)$, the exponential map $\exp_{\lien}\colon \lien\to N$ from $\lien$ to $N$ is polynomial and there is a constant $C_\lien > 0$ such that for $t_n \ge 1$,
	$$ \text{$d \bigl( \1, \exp_\lien (Z) \bigr) \le C_\lien t_n$ \ for all $Z \in N_{2\kappa}(t_n)$} .$$

We have the cocycle relation,
	$$ \Psi(a^{t_n}, g \cdot \xi_n) - \Psi(a^{t_n}, \xi_n) = \Psi(g', a^{t_n} \cdot \xi_n) - \Psi(g, \xi_n).$$
Hence for $g\in N_\kappa(t_n)$,
	\begin{align*}
		\left| \int _0^{t_n} \psi (a^s g \cdot \xi_n) \,d s- \int _0^{t_n} \psi (a^s \cdot \xi_n) \,d s\right |
		&= \bigl| \Psi(a^{t_n}, g \cdot \xi_n) - \Psi(a^{t_n}, \xi_n) \bigr| \\
		&= \bigl| \Psi(g', a^{t_n} \cdot \xi_n) - \Psi(g, \xi_n) \bigr| \\
		&\le 2C_\epsilon+ \heightgrowth \ell_0 + \epsilon  d(g', \1)  +\heightgrowth   \ell_0 + \epsilon d(g, \1)  \\
		&\le 2C_\epsilon + 2 \heightgrowth  \ell_0 + 2 \epsilon C_\lien t_n \\
		&\le \hat C_\epsilon + \hat \epsilon \, t_n
	\end{align*}
 where $g'$ is as in \eqref{thatdontmakeyouaho} and $\hat \epsilon>0$ can be made arbitrarily small by choosing $\epsilon>0$ sufficiently small.

We then have
\begin{align*}
		\biggl|\int \psi \,d \eta_n - & \frac 1 { t_n } \Psi(a^{t_n} , \xi_n) \biggr|
		\\&= \biggl|\int \psi \,d \eta_n - \frac 1 { t_n }\int_0^{t_n} \psi(a^{s}\cdot \xi_n) \,d s \biggr|
		\\&\le
		 \int_{N_\kappa(t_n)} \frac 1{ t_n} \left| \int _0^{t_n}
			\psi (a^s g\cdot \xi_n) - \psi(a^{s}\cdot \xi_n) \,d s \right| \,d m_{N_\kappa(t_n)}(g)
		\\&\le \int_{N_\kappa(t_n)} \frac 1{ t_n}
		\bigl( \hat C_\epsilon + \hat \epsilon \, t_n \bigr) \, d m_{N_\kappa(t_n)}(g) \\
		&=\frac 1{ t_n} (\hat C_\epsilon + \hat \epsilon \, t_n )
	 \end{align*}	
which can be made arbitrarily small by taking $t_n$ sufficiently large and $\hat \epsilon $ sufficiently small.
\end{proof}

\begin{claim} \label{etanTight}
The family $\{ \eta_n\}$ is a uniformly tight family of measures.
\end{claim}
\begin{proof}
Recall that $\wtd H/\Gamma\!_{\wtd H}$ is a bundle over $H/\Gamma\!_H$ with compact fibers.
 In particular, for any compact subset $E\subset \wtd H/\Gamma\!_{\wtd H}$, its $N$-orbit $N\cdot E\subset \wtd H/\Gamma\!_{\wtd H}$ is compact.
From, \cref{lem:maxpathstight}, given $\delta>0$ there is a $\ell_\delta>0$ such that for
every $n$, $$m_\R\{ 0\le t\le t_n : \tempfun(a^t \cdot x_n)> \ell_\delta\} < \delta t_n.$$
Recall that $\phi$ is $N$-invariant.  Then
It follows that $$\eta_n (\{\xi\in \P\calE_{\wtd H} : \tempfun(\xi)>   \ell_\delta\}) < \delta$$ for every $n$.
\end{proof}

Recall the choice of $x_n, t_n$ from which the measures $\eta_n$ are constructed. Since
$\tempfun(x_n)\le \ell_0$ and $\tempfun(a^{t_n}\cdot x_n)\le \ell_0$,  we have
$$\text{$h(g \cdot x_n)\le \hcutoff$ and $h\bigl((a^{t_n}g)\cdot x_n\bigr)\le \hcutoff$}$$
for all $g\in N$ where $\hcutoff$ is as in \eqref{eq:cutoff}. Recall $\psi_\ell$ as defined in~\pref{psiellDefn}. Taking $\ell\ge \hcutoff$ sufficiently large, it follows from the definition of $\eta_n$ that for every $n$,
$$ \int \psi_\ell \,d \eta_n = \int \psi \,d \eta_n.$$

With the above observations, we obtain an $A_H$-invariant, Borel probability measure on $\P\calE_{\wtd H}$ with non-zero time-averaged Lyapunov exponent.
\begin{proposition} \label{TimeAvgExp>0}
There is an $(A_H N)$-invariant, Borel probability measure~$\eta$ on $\P \calE_{\wtd H}$ such that
	$ \int \overline \Psi(a,\cdot) \,d \eta = \chi(Y,X_{\wtd H}) > 0 $.
\end{proposition}

\begin{proof}
From \cref{etanTight}, the family $\{\eta_n\}$ is uniformly tight. Let $\eta_\infty$ be any weak-$*$ limit of the family $\{\eta_n\}$.
Then $\eta_\infty$ is invariant under both $\{a^t\}$ and~$N$.

Take for almost every $\ell>\hcutoff$ we have that $h\inv(\ell)$ is a codimension-one hypersurface, $\eta \bigl( h\inv (\ell) \bigr) = \eta(T_\ell) = 0$, and $$ \int \psi_\ell \,d \eta_n = \int \psi \,d \eta_n$$ for every $n$.
From \cref{average limit}\pref{baboonfeces3} and \cref{LimEta=chi}, we have
$$\int \overline\Psi(a,\cdot) \,d \eta_\infty = \chi(Y, X_{\wtd H}).$$

Let $\{F'_j\}$ be a \Folner sequence in~$A_H$, let $\bar\eta_j:= F'_j \ast \eta_\infty$. Since $A_H$ is abelian and $\eta_\infty$ is $\{a^t\}$-invariant, by \cref{lentilvomit} we have for every $j$ that $$ \int \overline \Psi(a,\cdot) \,d \bar\eta_j = \int \overline\Psi(a,\cdot) \,d \eta_\infty = \chi(Y, X_{\wtd H}) .$$
By \cref{cor:tight}, the family $\{\bar\eta_j\}$ is uniformly tight. Let $\eta$ be a limit point of $\{\bar\eta_j\}$.

We have that $\eta$ is $A_H$-invariant and, since $A_H$ normalizes $N$, $\eta$ is $N$-invariant.
From \cref{average limit}\pref{baboonfeces2}
we have
\begin{align*} \int \overline \Psi(a,\cdot) \,d \eta = \lim_{j\to \infty} \int \overline \Psi(a,\cdot) \,d \bar \eta_j = \chi(Y, X_{\wtd H}). &
 \qedhere\end{align*}
 \end{proof}

\subsubsection{Proof of \texorpdfstring{\cref{thm:mainQR1}}{the proposition} when \texorpdfstring{$\rank_\Q(\Gamma) = 1$}{the Q-rank is 1}} \label{AvgQRank=1}
We   use the $(A_HN)$-invariant measure just constructed to prove Proposition \ref{thm:mainQR1}.
We assume for the time being that $\rank_\Q(\Gamma) = 1$; in this case we have $H= G= \wtd H$ so $A_H=A$ and $N$~is trivial.
We adapt the averaging arguments from \cite[\S6.3 and \S6.4]{BFH}.

Let $A$~be a split Cartan subgroup of~$G$
and let $\beta_1,\ldots, \beta_r$ be the simple $\R$-roots of $\Phi(A,G)$ relative to some ordering to be specified later.
Let $\eta=\eta_0$ denote the $A$-invariant Borel probability measure guaranteed by \cref{TimeAvgExp>0}. From \cref{claim:linear}, the map $A \to \R$ given by $ a \mapsto \int \overline \Psi \bigl( a,\cdot \bigr) \,d \eta $ is linear and from \cref{TimeAvgExp>0} is not identically zero.

The action of the real Weyl group of a simple group on the Lie algebra of a split Cartan subgroup is known to be irreducible.
Thus, after permuting the simple factors of~$G$ and conjugating by an element of the Weyl group of a simple factor, we may assume the ordering of the roots is such that the map $ a \mapsto \int \overline \Psi \bigl( a,\cdot \bigr) \,d \eta $ is not proportional to any
$\beta_i$ for $i > 1$.
In particular, we may choose a one-parameter subgroup~$\{a^t\}$ in~$A$ such that
	$$ \int \overline \Psi(a^1,\cdot) \,d \eta > 0$$
and $\beta_i(a^t)= 0$ for all $i > 1$.
Let $U$ denote the unipotent subgroup generated by the coarse root groups~$U^{[\beta_i]}$ for $i > 1$; then $U$ centralizes the one-parameter subgroup~$\{a^t\}$.

Let $\{\wtd F_j\}$ be a \Folner sequence of centered intervals in $U$ as in \cref{def:UnipInterval}.
By \cref{lem:unipotenttight}, the family $\{\wtd F_j \ast \eta_0\}$ is a uniformly tight. Let $ \eta_1$ be any weak-$*$ limit of $\{\wtd F_j \ast \eta_0\}$.
Since $\{\wtd F_j\}$ is F{\o}lner, the measure $ \eta_1$ is $U$-invariant and, as $U$ centralizes~$\{a^t\}$, $ \eta_1$ is $\{a^t\}$-invariant. By \cref{lentilvomit} and \cref{average limit}\pref{baboonfeces2},
	$$\int \overline\Psi(a^1,\cdot) \,d \eta_1
	= \lim_{j \to \infty} \int \overline\Psi(a^1,\cdot) \,d \wtd F_j \ast \eta_0
	= \int \overline \Psi(a^1,\cdot) \,d \eta_0 .$$
Let $\hat \eta_1$ denote the image of $\eta_1$ under the projection $\pi_\calE\colon \calE\to G/\Gamma$. Since $\eta_0$ is $A$-invariant, it follows from \cref{thm:averaginghomo2} that the image $\hat \eta_1$ is also $A$-invariant.

Let $\{\hat F_j\}$ be a \Folner sequence in~$A$.
As $\hat \eta_1$ is $A$-invariant, each $\hat F_j \ast \eta_1$ projects to the same measure in $G/\Gamma$; in particular, the family $\{\hat F_j \ast \eta_1\}$ is uniformly tight. Let $ \eta_2$ be any weak-$*$ limit of $\{\hat F_j \ast \eta_1\}$.
From \cref{lentilvomit} and \cref{average limit}\pref{baboonfeces2},
we have
	$$\int \overline\Psi(a^1,\cdot) \,d \eta_2
	= \int \overline\Psi(a^1,\cdot) \,d (\hat F_j \ast \eta_1)
	= \int \overline\Psi(a^1,\cdot) \,d \eta_1
	= \int \overline\Psi(a^1,\cdot) \,d \eta
	> 0. $$

We have that $\eta_2$ is $A$-invariant. Additionally, $\eta_2$ is $U$-invariant since $ \eta_1$ is $U$-invariant and $A$ normalizes~$U$. In particular, $\eta_2$ and its image $\hat \eta_2$ in $G/\Gamma$ are invariant under the coarse root groups $U^{[\beta_i]}$ for each $i > 1$. From \cref{thm:opproot}, the projection $\hat \eta_2$ of $\eta_2$ to $G/\Gamma$ is also invariant under the coarse root groups $U^{[-\beta_i]}$ for all $i > 1$.

Suppose first that $G$ is not simple. Then $\langle\, U^{[\beta_i]}, U^{[-\beta_i]} \mid i > 1 \,\rangle$ contains a (non-compact) simple factor~$G_1$ of~$G$. From the preceding paragraph, we know that the projection $\hat \eta_2$ is $G_1$-invariant. The lattice $\Gamma$ is assumed irreducible in $G$ and thus $\Gamma$ projects densely to the group $G/G_1$. This implies that every $G_1$-orbit in $G/\Gamma$ is dense.
Since $G_1$ is generated by unipotent elements, Ratner's measure rigidity theorem (see for example \cite[Theorem 9]{MR1262705}) implies that the Haar measure on $G/\Gamma$ is the only $G_1$-invariant, Borel probability measure on $G/\Gamma$.
In particular, the projection $ \hat \eta_2$ of $\eta_2$ to $G/\Gamma$ is Haar. Taking $\mu$ to be the projection of $\eta_2$ to $X$, we have that $\mu$ is $A$-invariant and projects to the Haar measure on $G/\Gamma$. Since the Haar measure has exponentially small mass at $\infty$ we have $ \Psi(a, \cdot) \in L^1(\eta_2)$ whence
$$\int \Psi(a, \cdot) \,d \eta_2 = \int \overline\Psi(a, \cdot) \,d \eta_2>0$$
and the conclusion of \cref{thm:mainQR1} holds by \cref{Lyap>Phi}\pref{baboonfeces}.

Now consider the case that $G$ is simple. Following \cite[Section~5.3]{BFH}, we assume that the roots $\beta_1,\ldots,\beta_r$ are enumerated as in \cite[Appendix A]{BFH} and we let
	$$\hat \beta := \begin{cases}
	\text{$\delta =$ highest root in $\Phi(A,G)$} & \text{if $G$ is of type $A_\ell$, $B_\ell$, $D_\ell$, $E_6,$ or $E_7$;} \\
	\text{$\delta' =$ second highest root} & \text{if $G$ is of type $C_\ell$, $(BC)_\ell$, $E_8$, $F_4$, or $G_2$}
	. \end{cases} $$
As $ \eta_2$ is $A$-invariant, from \cref{claim:linear} the function $a\mapsto \int \overline\Psi \bigl( a,\cdot \bigr) \,d \eta_2$ is linear on $A$.
Moreover, it is not proportional to at least one $\wtd\beta \in \{\beta_1, \hat \beta\}$.
Thus, there is a one-parameter subgroup $\{\wtd a^t\}$ of $A$ such that $\int \overline\Psi ( \wtd a^1, \cdot ) \,d \hat \eta_\infty > 0$ and $U^{\wtd\beta}(\wtd a^t) = 0$.
Let $\{ F_k\}$ be a \Folner sequence of centered intervals in~$U^{\wtd\beta}$.
From \cref{lem:unipotenttight}, the family $\{\wtd F_k \ast \eta_2 \}$ is uniformly tight. Let $ \eta_3$ be a weak-$*$ limit point of this family.
The proof of \cite[Proposition~6.4]{BFH} establishes that the projection of $ \eta_3$ to $G/\Gamma$ is Haar.
Moreover, \cref{lentilvomit} and \cref{average limit}\pref{baboonfeces2},
imply
	$$\int \overline\Psi(\wtd a^1,\cdot) \,d \eta_3
	= \int \overline\Psi(\wtd a^1,\cdot) \,d ( \wtd F_k \ast \eta_2)
	= \int \overline\Psi(\wtd a^1,\cdot) \,d \eta_2
	> 0. $$
If $\{\hat F_k\}$ is a \Folner sequence in $A$ then we have that $ \{\hat F_k\ast \eta_3\}$ is uniformly tight and from \cref{lentilvomit} and
\cref{average limit}\pref{baboonfeces2}
$$\int \overline\Psi(\wtd a^1,\cdot) \,d \eta_4 = 	 \int \overline\Psi(\wtd a^1,\cdot) \,d (\hat F_k \ast \eta_3)= \int \overline\Psi(\wtd a^1,\cdot) \,d \eta_3>0.$$
Then $\eta_4$ is $A$-invariant and  projects to the Haar measure on $G/\Gamma$.
Since the Haar measure on $G/\Gamma$ has exponentially small mass at $\infty$, $ \Psi(a, \cdot) $ is $L^1(\eta_4)$ and the pointwise ergodic theorem implies
$$\int \Psi(a, \cdot) \,d \eta_4>0. $$
The conclusion of \cref{thm:mainQR1} holds by \cref{Lyap>Phi}\pref{baboonfeces}. \qed

\subsubsection{Proof of \texorpdfstring{\cref{thm:mainQR1}}{the proposition} when \texorpdfstring{$\rank_\Q(\Gamma) \ge 2$}{the Q-rank is large}} \label{AvgQRank>1}

Assume now that $\rank_\Q(\Gamma) \ge 2$. This implies that the standard $\Q$-rank-1 subgroup~$H$ is a proper subgroup of~$G$.
In this case, the group $N$ is non-trivial; we have that $N$ is the unipotent radical of a standard parabolic $\Q$-subgroup which is normalized by~$H$ as in \cref{SubexpForRank1Subgroups-prelims}.
Let $A\subset G$ be a split Cartan subgroup containing $A_H$ such that $N$ is the expanding horospherical subgroup for some one-parameter subgroup $\{b^s\}\subset A$.
Let $\eta$ be the $A_H N$-invariant Borel probability measure on $\calE_{\wtd H}$ guaranteed by \cref{TimeAvgExp>0}.
For $n \in \Z^+$ set
	$$\eta_n := \frac 1 n \int _0 ^n (b^s)_*\eta \, ds .$$

\begin{claim}
 The family $\{ \eta_n \}$ is uniformly tight.
\end{claim}

\begin{proof}
Let $ \mu$ denote the image of $\eta$ under the projection $\P\calE_{\wtd H}\to G/\Gamma$.
Recall that $N$-orbits in $\wtd H/\Gamma\!_{\wtd H}$ are compact. In particular, given $\delta>0$ there is a $N$-invariant subset $C\subset
 \wtd H/\Gamma\!_{\wtd H}$ with $\mu(C)>1-\delta$. Let $\widehat\mu$ denote the restriction of $\mu$ to $C$.

As $N$ is the expanding horospherical subgroup for~$b^s$, by \cref{MixingSmallCuspsCom} the family $\bigl\{ (b^s)_* \widehat \mu : s\in [0,\infty) \bigr\}$ has uniformly exponentially small mass at $\infty$. Therefore, there is some~$\ell_1$ such that for all $s>0$,
	$$(b^s)_* \widehat \mu \bigl( \{\, x\in G/\Gamma \mid h(x) > \ell_1 \,\} \bigr) < \delta$$
whence
	$$ (b^s)_* \mu \bigl( \{\, x\in G/\Gamma \mid h(x) > \ell_1 \,\} \bigr) < 2 \delta .$$
Since $\ell_1$ is independent of~$s$, the sequence $\{\frac 1 n \int_0^n(b^s)_* \mu \, ds\}$ is uniformly tight. As $\P\calE \to G/\Gamma$ is a proper map, the family $\{\eta_n\}$ is uniformly tight.
\end{proof}

Let $\eta_\infty = \lim_{k \to \infty} \eta_{n_k}$ for some weak-$*$ convergent subsequence $\{ \eta_{n_k} \}$.

It is clear from construction that $\eta_\infty$ is $b^s$-invariant, $\{a^t\}$-invariant, and also $N$-invariant because $\eta$ is $N$-invariant and $\{b^s\}$ normalizes~$N$. Since $N$ is the expanding horospherical subgroup associated to~$b^s$, this implies that $\eta_\infty$ projects to the Haar measure on $G/\Gamma$ by \cref{MixingSmallCuspsCom}. Applying \cref{average limit}\pref{baboonfeces1} we have that $$\int \overline\Psi(a, \cdot) \,d \eta_\infty>0.$$
As $\eta_\infty$ is $\{a^t\}$-invariant, we may average over a \folner sequence in $A$, apply \cref{average limit}\pref{baboonfeces1} to obtain an $A$-invariant Borel probability measure $\eta'$ on $\P\calE$ projecting to the Haar measure on $G/\Gamma$ with
$$\int \overline\Psi(a, \cdot) \,d \eta'>0.$$
Since the Haar measure on $G/\Gamma$ has exponentially small mass at $\infty$, $ \Psi(a, \cdot) $ is $L^1(\eta')$ and the pointwise ergodic theorem implies
$$\int \Psi(a, \cdot) \,d \eta'>0$$
and the proposition follows from \cref{Lyap>Phi}\pref{baboonfeces}.   \qed

\section{Sketch of alternate proof in higher $\Q$-rank} \label{AlternateHighRank}
This section provides an alternate approach that provides a shorter proof of \cref{thm:mainQR1} in the special case where $\rank_\Q G \ge 2$ and $G$ has finite center. Mainly we avoid the technology of cut-off functions and time averages that make up the bulk of the previous section of the paper. The methods do not seem to apply in $\Q$-rank~one, and therefore do not yield \cref{thm:mainQR1} in full generality. Since this result is subsumed by \cref{thm:mainQR1}, and is therefore not required for a complete proof of our main theorems, we do not provide details. The proof given here is much more similar to the proof in \cite{BFH-SLnZ}

\begin{proposition}
Assume $\rank_\Q G \ge 2$ and let $\alpha\colon \Gamma\to \Diff^1(M)$ be an action.
If there is a standard  $\Q$-rank-1 subgroup~$H$ of~$G$, such that the restriction
	$$\restrict{\alpha}{\Gamma\!_H}\colon \Gamma\!_H \to \Diff^1(M)$$
does not have {uniform subexponential growth of derivatives}
there is a split Cartan subgroup~$A$ of~$G$ and an $A$-invariant Borel probability measure~$\mu$ on~$X$ projecting to the Haar measure on $G/\Gamma$ such that
	 $\lambda_{\top,a,\mu, \calA} > 0$ for some $a \in A$.
\end{proposition}

\begin{proof}[Sketch of proof]
We assume all results attained prior to \cref{SubexpForRank1Subgroups}, and will also assume some results proved in the early parts of that \lcnamecref{SubexpForRank1Subgroups}.
Since $\restrict{\alpha}{\Gamma\!_H}$ does not have uniform subexponential growth of derivatives, there is an $\R$-diagonalizable one-parameter subgroup $\{a^t\}$ of~$A_H$, a sequence $\{\xi_n\}$ of points in~$\calE$, and $t_n \to \infty$ such that $\{\xi_n\}$ and $\{a^{t_n} \cdot \xi_n\}$ are bounded, and
	$$ \lim_{n \to \infty} \frac{1}{t_n} \calA(a^{t_n} , \xi_n) > 0. $$
By assuming this limit is as large as possible, we may assume that
	\begin{align} \label{AlternatePfSublinearHeight}
	 \lim_{n \to \infty} \frac{1}{t_n} \max_{t \in [0,t_n]} h( a^t \cdot \xi_n ) = 0
	. \end{align}
This follows from the proof of \cref{lem:maxpathstight} and is more explicit in the proof of \cite[Lem.~5.3]{BFH-SLnZ}.

Let $\{b^s\}$ be a nontrivial $\Q$-diagonalizable one-parameter subgroup of~$G$ that centralizes~$H$. The existence of $\{b_s\}$ is where we use the assumption that $\rank_\Q G \ge 2$.
Let $N$ be the expanding horospherical subgroup of~$\{b^s\}$, so $N$ is a unipotent $\Q$-subgroup. It is normalized by~$H$, because $H$ centralizes~$\{b^s\}$. Choose a \folner sequence $\{F_n\}$ in~$N$, such that
	$$ \limsup_{n \to \infty} \frac{1}{t_n} \max_{u \in F_n} \log \| u \| < \infty ,$$
and such that $\{a^{t'_n} F_n a^{-t'_n}\}$ is also a \folner sequence in $N$, for every sequence $\{t'_n\}$ with $|t'_n| \le t_n$.

Note that letting $F'_n := \{a^t\}_{0 \le t \le t_n} F_n$ yields a \Folner sequence in $\{a^t\} N$.
From \pref{AlternatePfSublinearHeight}, we see that if we let $h_n = \max_{t \in [0,t_n]} h( a^t \cdot \xi_n )$, then there is a sequence $\{s_n\}$, such that $h_n/s_n \to 0$ and $s_n/t_n \to 0$; for example, let $s_n = \sqrt{h_n t_n}$. Therefore, it follows from exponential mixing that if we let $\mu_n := b^{s_n} F'_n * \delta_{\xi_n}$, then the projection of $\{\mu_n\}$ to $G/\Gamma$ converges to Haar by applying \cref{ExpMixing}. There is a technical issue that \cref{ExpMixing} requires the support of~$f$ to be in the unit ball of $N$, which can be addressed by covering $\{a^{t'_n} F_n a^{-t'_n}\}$ by small cubes contained in translates of the unit ball of $N$ and by considering a partition of unity subordinate to this cover.

The sequence of measures $\{\mu_n\}$ has also exponentially small mass at $\infty$, this follows because for a fixed $t$ in $[0,t_n]$, the set $b^{s_n}a_t F_n$ is a $U$-orbit intersecting the thick part of $G/\Gamma$ nontrivially and so the non-divergence results for unipotent subgroups apply and one can use \cite[Lem.~3.3]{BFH-SLnZ}. Formally, \cite[Lem.~3.3]{BFH-SLnZ} is proven for one-dimensional unipotent subgroups, but for higher dimensional subgroups it follows by partition $U$ into one-dimensional unipotent orbits.

In addition, any subsequential limit~$\mu_\infty$ of $\{\mu_n\}$ is $\{a^t\}$-invariant, because $F'_n$ is a \Folner set and $b^{s_n}$ centralizes~$\{a^t\}$.
Finally, note by an easy computation that:
\begin{equation}\label{desi1}
\liminf_{n \to \infty} \lambda_{\top,a^1,\mu_n, \calA} \geq \frac{1}{t_n} \int_{F_n} \calA(a^{t_n}, b^{s_n} u \cdot \xi_n) \, dm_{F_n}(u),
\end{equation}
and we also have
\begin{equation}\label{desi2}	
\liminf_{n \to \infty}  \frac{1}{t_n} \int_{F_n} \calA(a^{t_n}, b^{s_n} u \cdot \xi_n) \, dm_{F_n}(u) \geq \lim_{n \to \infty} \frac{1}{t_n} \calA(a^{t_n} , \xi_n) >0
\end{equation}	
because $s_n/t_n \to 0$ and by using the fact that $\restrict{\alpha}{\Gamma\!_N}$ has {uniform subexponential growth of derivatives} by Proposition \ref{UnipSubgrp} and making use of \ref{AlternatePfSublinearHeight}. Therefore $\lambda_{\top,a^1,\mu_\infty, \calA} > 0$ because the top Lyapunov exponent of a bounded continuous linear cocycle is an upper-semicontinuous function of the measure.
\end{proof}

\appendix

\section{Numerology tables associated with Zimmer's conjecture}\label{sec:table}\label{appendix}
We compute the numbers $n(G), d(G), v(G)$ and $r(G)$ for all simple real Lie groups.  We note that such numbers depend only on the Lie algebra of $G$.  We primarily follow the naming conventions in \cite{MR1920389}.

\setlength\tabcolsep{0pt}
\newcommand{\LLalt}{\LL\addlinespace[.2em]} 

\newcolumntype{x}[1]{>{\centering\arraybackslash}m{ #1}}
\newcommand{\ftform}[1]{$(#1)$}

{\scriptsize
\ctable[topcap, notespar,
caption =  {Numerology appearing in Zimmer's conjecture for  classical $\R$-split simple Lie algebras.}\label{T1},
mincapwidth = \textwidth,
footerwidth,
maxwidth=\textwidth,
pos=H
]
{ x{.12\textwidth}   x{.13\textwidth}    x{.08\textwidth}   x{.12\textwidth}   x{.15\textwidth}  x{.13\textwidth}  x{.08\textwidth} }
{
\tnote[\ftform{a}]{$\liesl(4,\R)=\so(3,3)$}
\tnote[\ftform{b}]{$\so(1,2)=\liesl(2,\R)$ and $\so(2,3)=\liesp(4,\R)$}
\tnote[\ftform{c}]{$\so(2,2)$ is not simple and $\so(3,3)=\liesl(4,\R)$}
}
{
   $\lieg$ & \makecell {restricted  \\ root  system} &\makecell{real\\ rank}& $n(G)$& $d(G)$& $v(G)$ & $r(G)$\\ \FL
 \makecell{$\liesl(n,\R)$ \\  $n\ge 2$}&$A_{n-1}$ & $n-1$ & $n$&   \makecell{ $2n-2$, $n\neq 4$ \\ $5$,  $n= 4$\tmark[\ftform{a}]}   & $n-1$&$n-1$ \ML

\makecell{$\liesp(2n,\R)$\\$n\ge 2$}&$C_n$ & $n$ & $2n$& $ 4n-4 $& $2n-1$ &$2n-1$ \ML

 \makecell{$\so(n,n+1)$\\  $n\ge 3$\tmark[\ftform{b}]} &$B_n$& $n$ & $2n+1$& $ 2n $& $2n-1$ &$2n-1$\ML

\makecell{$\so(n,n)$\\  $n\ge 4 $\tmark[\ftform{c}]}&$D_n$& $n$ & $2n$& $ 2n-1 $& $2n-2$ &$2n-2$
\LLalt
}}

{\scriptsize
\ctable[notespar,
caption =  {Numerology appearing in Zimmer's conjecture for  complex simple Lie   algebras} \label{T2},
mincapwidth = \textwidth,
footerwidth,
maxwidth=\textwidth,
pos=H
]
{ x{.13\textwidth}   x{.15\textwidth}    x{.08\textwidth}   x{.12\textwidth}   x{.15\textwidth}  x{.14\textwidth}  x{.08\textwidth} }
{
\tnote[\ftform{d}]{$\liesl(4,\C)=\so(6,\C)$}
\tnote[\ftform{e}]{$\so(5,\C) =\liesp(4,\C) $  and $\so(3,\C)=\liesl(2,\C)$}
\tnote[\ftform{f}]{$\so(6,\C)=\liesl(4,\C) $  and $\so(4,\C)$ is not simple}
}
{
   $\lieg$ & \makecell {restricted  \\ root  system} &\makecell{real\\ rank}& $n(G)$& $d(G)$& $v(G)$ & $r(G)$\\ \FL

    \makecell{$\liesl(n,\C) $\\  $n\ge 2$} 
        &$A_{n-1}$ & $n-1$ & $2n$& \makecell{ $2n-2$,  $ n\neq 4 $\\$ 5$, { $  n=4$}\tmark[\ftform{d}]}
        & $2n-2$&$n-1$ \ML
                    \makecell{  $ \liesp(2n,\C)$\\  $n\ge 2$}
                &$C_{n}$ & $n$ & $4n$& $4n-4 $& $4n-2$&$2n-1$ \ML
        \makecell{  $\so(2n+1,\C)$ \\ $ n\ge  3$\tmark[\ftform{e}]}  &$B_{n}$ & $n$ & $4n+2$& $2n $& $4n-2$&$2n-1$ \ML
        \makecell{      $\so(2n,\C)$ \\  $ n\ge 4$\tmark[\ftform{f}]}  &$D_{n}$ & $n$ & $4n$& $2n -1$& $4n-4$&$2n-2$ \ML
        $\mathfrak {e}_6(\C)$ &$E_6$& $6$ & $54$ & $26$ & $32$ & $16$  \ML
    $\mathfrak {e}_7(\C)$ &$E_7$& $7$ & $112$ & $54$ & $54$ & $27$  \ML
    $\mathfrak {e}_8(\C)$ &$E_8$& $8$ & $496$ & $112$ & $114$ & $57$  \ML
    $\mathfrak {f}_4(\C)$ &$F_4$& $4$ & $52$ & $16$ & $30$ & $15$  \ML
    $\mathfrak {g}_2(\C)$ &$G_2$& $2$ & $14$ & $6$ & $10$ & $5$  \LLalt
   }}

{\scriptsize
\ctable[notespar,
caption =  {Numerology appearing in Zimmer's conjecture for classical non-$\R$-split simple real forms} \label{T3},
mincapwidth = \textwidth,
footerwidth,
maxwidth=\textwidth,
pos=H
]
{ x{.2\textwidth}   x{.18\textwidth}    x{.08\textwidth}   x{.12\textwidth}   x{.15\textwidth}  x{.14\textwidth}  x{.08\textwidth} }
{\tnote[\ftform{g}]{$\liesl(2,\mathbb{H})=\so(1,5)$}
\tnote[\ftform{h}]{
$\so(1,3)=\liesl(2,\C)$}
\tnote[\ftform{i}]{$\su(2,2)=\so(4,2)$}
\tnote[\ftform{j}]{$\liesp(2,2)=\so(1,4)$}
\tnote[\ftform{k}]{$\so^*(4)$ is not simple}
\tnote[\ftform{\ell}]{$\so^*(6)=\su(1,3)$}
}
{
   $\lieg$ & \makecell {restricted  \\ root  system} &\makecell{real\\ rank}& $n(G)$& $d(G)$& $v(G)$ & $r(G)$\\ \FL
	\makecell{$\liesl(n,\mathbb{H})$\\  $n\ge2$}& $A_{n-1}$ & $ n-1 $ &   \makecell{$4n  $,   $n\neq2$ \\$6$,  $n= 2$\tmark[\ftform{g}]}    &  \makecell{$4n - 2 $, $n\neq2$ \\   $5$, $n= 2$}
	& $ 4n - 4 $ & $ n-1  $\ML\makecell{$\so(n,m)$\\    $2 \le n\le n+2 \le m$\\    $n=1$, $m\ge 6$\tmark[\ftform{g, h, j}]} 
&
$B_n$ 
& $n$ & $n +m $& $ n+m -1$& $n+m-2$ &$2n-1$\ML 	\makecell{$\su(n,m)$\\ $1\le n\le m$\\   $(n,m)\neq (2,2)$\tmark[\ftform{i}]}  
	& \makecell{$(BC)_n$, { $n<m$}\\$C_n$, {  $n=m$}} & $n$ &$2n +2m $ &  $2n +2m -2$&$ 2n+2m-3 $&$ 2n-1 $ \ML

$\su(2,2)\tmark[\ftform{i}] $
	& $C_2$ & $2$ &$6$ &  $5$&$ 4 $&$ 3 $ \ML
	\makecell{ $\liesp(2n,2m)$ \\  $1 \le n\le m$\\   $(n,m)\neq (1,1)$\tmark[\ftform{j}]}  & \makecell{$(BC)_n$, {  $n<m$}\\$C_n$, {  $n=m$} }& $ n $&  $4n +4m$& $4n+4m-4 $ & $ 4n+4m-5$ & $ 2n-1$\ML

$\liesp(2,2)\tmark[\ftform{j}]$   &  $A_1$& $ 1 $&  $5$& $4 $ & $ 3$ & $ 1$\ML
   \makecell{   $\so^*(2n)$\\  $n\ge 4$ even\tmark[\ftform{k}] }&  $C_{\frac 1 2 n}$ & $   \dfrac n 2   $ & $  4n $ & $  2n-1 $ & $  4n-7 $ & $  n-1 $ \ML
     \makecell{ $\so^\ast(2n)$ \\ $n\ge 5$ odd\tmark[\ftform{\ell}]}  & $(BC)_{\frac 1 2 (n-1)}$ & $  \dfrac{n-1} 2   $ & $ 4n  $ & $ 2n-1  $ & $  4n-7 $ & $ n-2  $ \LLalt
   }}

\vfill

{\scriptsize
\ctable[notespar,
caption = {Numerology   appearing in Zimmer's conjecture for  real forms of exceptional  Lie algebras\tmark[\ftform{m}]}\label{T4},
mincapwidth = 1\textwidth,
footerwidth,
pos=H
]
{ x{.07\textwidth}   x{.13\textwidth}    x{.07\textwidth}   x{.07\textwidth}   x{.07\textwidth}  x{.07\textwidth}  x{.07\textwidth} }
{\tnote[\ftform{m}]{Naming conventions follow \cite[Appendix C.4]{MR1920389}}}
{
   $\lieg$ & \makecell {restricted  \\ root  system} &\makecell{real\\ rank}& $n(G)$& $d(G)$& $v(G)$ & $r(G)$\\ \FL
          $E_{I}$ &$E_6$& $6$ & $27$ & $26$ & $16$ & $16$  \ML
          $E_{II}$ &$F4$& $4$ & $27$ & $26$ & $21$ & $15$  \ML
          $E_{III}$ &$(BC)_2$& $2$ & $27$ & $26$ & $21$ & $3$  \ML
          $E_{IV}$ &$A_2$& $2$ & $27$ & $26$ & $16$ & $3$  \ML
    $E_{V}$ &$E_7$& $7$ & $56$ & $54$ & $27$ & $27$  \ML
          $E_{VI}$ &$F_4$& $4$ & $56$ & $54$ & $33$ & $15$  \ML
          $E_{VII}$ &$C_3$& $3$ & $56$ & $54$ & $27$ & $5$  \ML
    $E_{VIII}$ &$E_8$& $8$ & $248$ & $112$ & $57$ & $57$  \ML
              $E_{IX}$ &$F_4$& $4$ & $248$ & $112$ & $57$ & $15$  \ML
    $F_{I}$ &$F_4$& $4$ & $26$ & $16$ & $15$ & $15$  \ML
     $F_{II}$ &$(BC)_1 $& $1 $ & $26 $ & $16 $ & $15 $ & $1 $  \ML
    $G$ &$G_2$& $2$ & $7$ & $6$ & $5$ & $5$  \LLalt
     }}

\end{document}